\documentclass[12pt]{amsart}
\usepackage{amssymb,latexsym,mathtools}
\usepackage[margin=1.25in]{geometry}
\usepackage{mathrsfs}
\usepackage{mathtools}
\usepackage{graphicx}
\usepackage{esint}
\usepackage{braket}
\usepackage{url, hyperref}
\newtheorem{theorem}{Theorem}
\newtheorem{lemma}{Lemma}
\newtheorem{definition}{Definition}

\newtheorem{prop}{Proposition}

\newtheorem{remark}{Remark}
\newtheorem{corollary}{Corollary}
\newtheorem{question}{Question}

\begin{document}

\date{\today}

\title[Variational Principle and Einstein Relation]{A Variational Principle for Pulsating Standing Waves and an Einstein Relation in the Sharp Interface Limit}
\author[P.S.\ Morfe]{Peter S.\ Morfe}

\maketitle

\begin{abstract}  This paper investigates the connection between the effective, large scale behavior of Allen-Cahn energy functionals in periodic media and the sharp interface limit of the associated $L^{2}$ gradient flows.  By introducing a Percival-type Lagrangian in the cylinder $\mathbb{R} \times \mathbb{T}^{d}$, we establish a link between the $\Gamma$-convergence results of Anisini, Braides, and Chiad\`{o} Piat and the sharp interface limit results of Barles and Souganidis.  In laminar media, we prove a sharp interface limit in a graphical setting, making no assumptions other than sufficient smoothness of the coefficients, and we prove that the effective interface velocity and surface tension satisfy an Einstein relation.  A number of pathologies are presented to highlight difficulties that do not arise in the spatially homogeneous setting.  \end{abstract}

\tableofcontents

\section{Introduction}  

In this work, we revisit the analysis of the large scale behavior of Van der Waals-Cahn-Hilliard phase transitions in periodic media.  We are interested in the relationship between the asymptotics as $\epsilon \to 0^{+}$ of the energy functionals
		\begin{equation} \label{E: functional}
			\mathcal{F}^{a}_{\epsilon}(u^{\epsilon};\Omega) = \int_{\Omega} \left( \frac{1}{2} \epsilon \langle a(\epsilon^{-1}x)Du^{\epsilon}(x),Du^{\epsilon}(x)\rangle + \epsilon^{-1}W(u^{\epsilon}(x))\right) \, dx
		\end{equation} 
	and the associated $L^{2}$ gradient flows, namely,
		\begin{equation} \label{E: main}
			u_{t}^{\epsilon} - \text{div}(a(\epsilon^{-1}x)Du^{\epsilon}) + \epsilon^{-2} W'(u^{\epsilon}) = 0 \quad \text{in} \, \, \mathbb{R}^{d} \times (0,T).
		\end{equation}
The Van der Waals-Cahn-Hilliard theory \cite{van der waals,cahn hilliard} provides a phenomenological, mesoscopic-scale description of phase coexistence in a material composed of a mixture of two phases (say, crystal and melt); see \cite{langer,scaling limits,kobayashi,glicksman} for introductory accounts.  To take into account the effect of the thin interfacial layer between the two phases, the function $u^{\epsilon}$, called a phase field, is used to model the local state of the material.  It is close to $1$ in the bulk of the crystal phase and $-1$ in the melt phase, and transitions smoothly between $-1$ and $1$ in the region separating the two.  The basic postulate of the theory is the free energy \eqref{E: functional} determines the possible equilibrium values of $u^{\epsilon}$ and, out of equilibrium, $u^{\epsilon}$ evolves in time so as to decrease its energy, \eqref{E: main} being one possible choice of dynamics.

To model the effect of periodically arranged heterogeneities or impurities in the material, we consider energies \eqref{E: functional} having the standard form except for the presence of a uniformly elliptic, $\mathbb{Z}^{d}$-periodic matrix field $a$ in the gradient term, which models the interaction energy between phases.  The function $W : [-1,1] \to [0,\infty)$, which describes the free energy density of a given phase, is a double well potential with zeros at $1$ and $-1$.  The $\epsilon$ scaling in \eqref{E: functional} and \eqref{E: main} arises after blowing down space by a factor $\epsilon^{-1}$ and time by $\epsilon^{-2}$.  Our interest is in the macroscopic description of phase coexistence that emerges when $\epsilon \to 0^{+}$.

When the matrix field $a$ is constant, and hence the material properties are spatially homogeneous, the asymptotics of \eqref{E: functional} and \eqref{E: main} are by now well-known.  In the transition from mesoscopic to macroscopic scales, there is a reduction of complexity: the phase field $u^{\epsilon}$ concentrates at the equilibrium phases $1$ and $-1$ and the width of the interface between them becomes infinitessimal.  Accordingly, this scaling regime is referred to as the \emph{sharp interface limit} and the objective is to analyze the limiting interface.  At the level of the energy, Modica and Mortola \cite{modica mortola,modica} proved that $\mathcal{F}^{\text{Id}}_{\epsilon}$ $\Gamma$-converges to a multiple of the perimeter functional.  That is, roughly speaking, there is a constant $c_{W} > 0$ such that if $u^{\epsilon} \approx \chi_{E} - \chi_{\mathbb{R}^{d} \setminus E}$ for some smooth open set $E$, then 
	\begin{equation*}
		c_{W} \mathcal{H}^{d-1}(\partial E \cap \Omega) \leq \liminf_{\epsilon \to 0^{+}} \mathcal{F}^{\text{Id}}_{\epsilon}(u^{\epsilon}; \Omega)
	\end{equation*}
and the lower bound is achieved for well-prepared choices of the data $(u^{\epsilon})_{\epsilon > 0}$.  The limit of \eqref{E: main} is completely analogous (see \cite{barles souganidis} and the references therein): if $u^{\epsilon}(\cdot,0) \approx \chi_{E} - \chi_{\mathbb{R}^{d} \setminus E}$ at the initial time, then $u^{\epsilon}(\cdot,t) \approx \chi_{E_{t}} - \chi_{\mathbb{R}^{d} \setminus E_{t}}$ with the family $(E_{t})_{t \geq 0}$ being the sets whose boundaries evolve by mean curvature flow and for which $E_{0} = E$.  

Once $a$ is allowed to vary periodically, characterizing the macroscopic behavior becomes a question for homogenization theory.  By analogy with the spatially homogeneous case, one can imagine that the effect of the mesoscale variation of $a$ will average out at the macroscale.  The question is then how and why averaging occurs.  The natural expectation is that the energy and its gradient flow converge to an anisotropic surface energy and an anisotropic curvature flow, respectively, ideally in the same manner discussed by Spohn in \cite{spohn paper}.  Investigating the extent to which this is true is the principal aim of the present paper.

\subsection{Motivation} Our approach to the study of \eqref{E: functional} and \eqref{E: main} is inspired by two seminal papers in the area.  In \cite{ansini braides}, Ansini, Braides, and Chiad\`{o} Piat proved that $\mathcal{F}^{a}_{\epsilon}$ $\Gamma$-converges to an anisotropic perimeter functional $\tilde{F}^{a}$, which can be expressed in the following form:
	\begin{equation} \label{E: anisotropic perimeter}
		\tilde{F}^{a}(E;\Omega) = \int_{\partial E \cap A} \tilde{\varphi}^{a}(\nu_{\partial E}(\xi)) \, \mathcal{H}^{d-1}(d \xi).
	\end{equation}
Here $\tilde{\varphi}^{a}$ is a Finsler norm, which, following the mathematical physics convention, we will henceforth refer to as the \emph{(macroscopic) surface tension} associated with \eqref{E: functional}.  

Some years before \cite{ansini braides} appeared, Barles and Souganidis \cite{barles souganidis} analyzed the behavior of \eqref{E: main} under the assumption that it possessed a smooth family of special solutions we will refer to as \emph{pulsating standing waves}.  These are functions $U_{e}$ indexed by a unit vector $e \in S^{d-1}$ (see \eqref{E: sphere notation}) satisfying the following degenerate elliptic PDE in the cylinder $\mathbb{R} \times \mathbb{T}^{d}$:
	\begin{align} \label{E: pulsating standing wave}
		\mathcal{D}^{*}_{e}(a(x) &\mathcal{D}_{e} U_{e}) + W'(U_{e}) = 0 \, \, \text{in} \, \, \mathbb{R} \times \mathbb{T}^{d}, \quad (\mathcal{D}_{e} := e \partial_{s} + D_{x})\\
		&\lim_{s \to \pm \infty} U_{e}(s,x) = \pm 1, \quad \partial_{s} U_{e} \geq 0. \nonumber
	\end{align}
Under this assumption, the authors showed that \eqref{E: main} does indeed approximate a certain geometric flow in the sharp interface limit.  Specifically, they showed that the limiting interfaces $(\partial E_{t})_{t \geq 0}$ between the two phases evolve with normal velocity $V_{\partial E_{t}}$ given by
	\begin{equation} \label{E: anisotropic curvature flow}
		\tilde{M}^{a}(\nu_{\partial E_{t}}) V_{\partial E_{t}} = \text{tr} \left( \tilde{\mathcal{S}}^{a}(\nu_{\partial E_{t}}) A_{\partial E_{t}} \right).
	\end{equation}
Here $\nu_{\partial E_{t}}$ and $A_{\partial E_{t}}$ are the normal vector and second fundamental form associated with $\partial E_{t}$, respectively, and $\tilde{M}^{a}$ and $\tilde{\mathcal{S}}^{a}$ are effective coefficients computable in terms of certain integrals of the pulsating standing waves.  We will refer to $\tilde{M}^{a}$ as the \emph{mobility}.

It is natural to wonder how the matrix $\tilde{\mathcal{S}}^{a}$ in \eqref{E: anisotropic curvature flow} relates to the surface tension $\tilde{\varphi}^{a}$ in \eqref{E: anisotropic perimeter}.  Indeed, questions of this type are important in homogenization and the calculus of variations \cite{variational evolution book,serfaty} and are fundamental in the statistical mechanics of interfaces \cite{spohn paper,bellettini butta presutti}.  This is the first question the paper seeks to address:

\begin{question}  How is $\tilde{\mathcal{S}}^{a}$ related to \eqref{E: anisotropic perimeter}?  \end{question}

That \eqref{E: main} can be analyzed using special solutions connecting the equilibria $\pm 1$ is reminiscent of what is known in the spatially homogeneous setting.  In that case, standing wave solutions are well known to play an important role not only in the asymptotics of the gradient flow \cite{crazy de masi presutti paper,katsoulakis souganidis isotropic,katsoulakis souganidis anisotropic,barles souganidis} but also in the determination of the surface tension \cite{non-local_isotropic,scaling limits,alberti guide}.  By contrast, in the periodic setting, since \cite{barles souganidis} relatively little attention has been devoted to the pulsating standing wave equation \eqref{E: pulsating standing wave}.  (However, see \cite{ducrot,giletti rossi}.)  This leads to the second question treated here:

\begin{question}  Do pulsating standing waves admit a variational interpretation?  Do they exist in general?  Are they smooth?  \end{question}  

Finally, in view of known examples in the homogenization of geometric flows (cf.\ \cite{cesaroni novaga valdinoci,novaga valdinoci,braides gelli novaga}), it is far from clear what kind of estimates or regularity can be expected from the coefficients $\tilde{M}^{a}$ and $\tilde{\mathcal{S}}^{a}$.  Following \cite{ducrot}, for instance, an optimistic first guess suggests this question might be approachable through an elliptic regularization of \eqref{E: pulsating standing wave}.  These issues constitute our last question:

\begin{question}  What can be said about the coefficients $\tilde{\mathcal{S}}^{a}$ and $\tilde{M}^{a}$ where estimates and regularity are concerned?  Can they be obtained through elliptic regularization?  \end{question}

\subsection{Overview of the Results}  In this section, we give an informal overview of the main results.  More precise versions of the results stated here appear in Section \ref{S: results}.

To address the questions above, we begin by introducing a Lagrangian $\mathscr{T}^{a}_{e}$ of the form
	\begin{equation} \label{E: lagrangian}
		\mathscr{T}^{a}_{e}(U) = \int_{\mathbb{R} \times \mathbb{T}^{d}} \left( \frac{1}{2} \langle a(x) \mathcal{D}_{e}U, \mathcal{D}_{e}U \rangle + W(U) \right) \, dx \, ds \quad (\mathcal{D}_{e} := e \partial_{s} + D_{x}),
	\end{equation}
which has the pulsating standing wave equation \eqref{E: pulsating standing wave} as its Euler-Lagrange equation.  Starting with our second question, we prove 

\begin{theorem} \label{T: standin}  Fix a unit vector $e \in S^{d-1}$ (see \eqref{E: sphere notation}).  If $a$ is a uniformly elliptic, $\mathbb{Z}^{d}$-periodic matrix field and $W$ is a non-negative continuous function in $[-1,1]$ with $W^{-1}(\{0\}) = \{-1,1\}$, then:
	\begin{itemize}
		\item[(i)] There is a weak solution $U_{e}$ of \eqref{E: pulsating standing wave} which is minimizing for $\mathscr{T}^{a}_{e}$ and satisfies
			\begin{equation*}
				\mathscr{T}^{a}_{e}(U_{e}) = \tilde{\varphi}^{a}(e).
			\end{equation*}
		\item[(ii)] If $U$ is a continuous solution of \eqref{E: pulsating standing wave} and $a$ and $W$ are regular enough, then $U$ is minimizing for $\mathscr{T}^{a}_{e}$.
		\item[(iii)] There exist smooth coefficients $a$ and $W$ such that, for some direction $e \in S^{d-1}$, $\mathscr{T}^{a}_{e}$ has no continuous minimizers.
	\end{itemize}
\end{theorem}

The theorem shows that the pulsating standing wave equation \eqref{E: pulsating standing wave} has a natural variational interpretation and that this leads to obstructions to smoothness.  This is actually very natural considering the close similarity between $\mathscr{T}^{a}_{e}$ and the so-called Percival Lagrangian of classical Aubrey-Mather theory.  

In Aubrey-Mather theory, the Percival Lagrangian acts like a generating function for certain plane-like minimizers (see \cite{moser_old_paper,bessi}). $\mathscr{T}^{a}_{e}$ plays the same role here.  If $U$ is a minimizer, we prove below that the functions $\{u_{\zeta}\}_{\zeta \in \mathbb{R}}$ defined by 
	\begin{equation} \label{E: transformation}
		u_{\zeta}(x) = U(\langle x,e \rangle - \zeta, x)
	\end{equation}
form a monotone family of minimizing critical points of the energy $\mathcal{F}^{a}_{1}$, heteroclinic between $1$ and $-1$ in the $e$ direction.  From this point of view, we see that the question of smoothness of pulsating standing waves can be reformulated as one about foliations of minimizers of $\mathcal{F}^{a}_{1}$.  In view of classical results in Aubrey-Mather theory, this leads naturally to the expectation that minimizers of $\mathscr{T}^{a}_{e}$ are generically discontinuous.

Smoothness of pulsating standing waves is a key element of the analysis in \cite{barles souganidis}.  Pulsating waves are used in both an asymptotic expansion for the solution $u^{\epsilon}$ of \eqref{E: main} and in the computation of the coefficients $\tilde{M}^{a}$ and $\tilde{S}^{a}$ in \eqref{E: anisotropic curvature flow}, both of which utilize smoothness to justify computations.  As we will see below, the lack of smoothness in a particular direction $e$ can manifest itself as pathological behavior in the coefficients $\tilde{M}^{a}(e)$ and $\tilde{\mathcal{S}}^{a}(e)$ at $e$.  Thus, it seems likely that new ideas will be needed to characterize the asymptotics of \eqref{E: main} in general.

For now, by restricting to laminar media, that is, the setting when $a$ depends on only $k < d$ of the variables, we can modify the approach of \cite{barles souganidis} to make some progress on our first question concerning \eqref{E: anisotropic curvature flow}.  Here we are able to regain smoothness for directions $e$ that cross the laminations.  Roughly speaking, this happens because the additional symmetry of $a$ combines with the maximum principle to force the family $\{u_{\zeta}\}_{\zeta \in \mathbb{R}}$ to be generated by translation and hence to form a smooth foliation (see Section \ref{S: laminar symmetry} below). 

Exploiting the regularity that is gained in the laminar setting, we show that certain solutions of the gradient flow \eqref{E: main} are described by \eqref{E: anisotropic curvature flow} in the sharp interface limit, and we relate $\tilde{\mathcal{S}}^{a}$ to the surface tension $\tilde{\varphi}^{a}$ consistently with the picture in \cite{spohn paper}.

\begin{theorem}  If $a(x + y) = a(x)$ for $y \in \mathbb{Z}^{k} \times \mathbb{R}^{d-k}$ for some $k < d$ and $a$ and $W$ are smooth enough, then the coefficient matrix $\tilde{\mathcal{S}}^{a}$ of \eqref{E: anisotropic curvature flow} is well-defined in $S^{d-1} \setminus (S^{k-1} \times \{0\})$ and satisfies
	\begin{equation} \label{E: einstein relation}
		\tilde{\mathcal{S}}^{a}(e) = D^{2}\tilde{\varphi}^{a}(e) \quad \text{for} \, \, e \in S^{d-1} \setminus (S^{k-1} \times \{0\}).
	\end{equation}

Furthermore, there is a nonempty class of initial data $\mathcal{A}$ such that if $(u^{\epsilon})_{\epsilon > 0}$ are the solutions of \eqref{E: main} with initial datum $u^{\epsilon}(\cdot,0) = u_{0} \in \mathcal{A}$, then there is a family of open sets $(E_{t})_{t \geq 0}$ with $E_{0} = \{u_{0} > 0\}$, moving with normal velocity given by \eqref{E: anisotropic curvature flow}, and such that, for all $t \geq 0$,
	\begin{equation*}
		u^{\epsilon}(\cdot,t) \to 1 \quad \text{in} \, \, E_{t}, \quad u^{\epsilon}(\cdot,t) \to -1 \quad \text{in} \, \, \mathbb{R}^{d} \setminus \overline{E}_{t}.
	\end{equation*}
\end{theorem}

In statistical mechanics, the fact that the effective interface velocity $V$ is related to $\tilde{\varphi}^{a}$ through \eqref{E: anisotropic curvature flow} and \eqref{E: einstein relation} is referred to as an \emph{Einstein relation} \cite{spohn paper}. 

Concerning our third question, we demonstrate a number of pathologies that are unique to the spatially heterogeneous setting.  The unifying principle behind the pathologies is the lack of smoothness of the pulsating standing waves gives rise to discontinuities and degeneracies at the level of the effective coefficients $\tilde{\varphi}^{a}$ and $\tilde{M}^{a}$.

\subsection{Related literature}  Concerning level-set PDE, sharp interface limits, and related mathematical results, we refer to \cite{barles souganidis} and the references therein.  For motivation for the study of diffuse interface models like \eqref{E: functional} and \eqref{E: main} from a mathematical physics point of view, see \cite{scaling limits}.  The inspiration to revisit \cite{ansini braides} and \cite{barles souganidis} came from related work by Butt\`{a} \cite{butta} and Bellettini, Butt\`{a}, and Presutti \cite{bellettini butta presutti} in the study of the (spatially homogeneous) Lebowitz-Penrose functional.

The Lagrangian \eqref{E: lagrangian} can be understood as a phase transition or heteroclinic version of the Percival Lagrangian in the Aubry-Mather and Moser-Bangert theories.  That the fundamental so-called ``WSI" solutions of those theories can be encoded using a degenerate elliptic energy functional was highlighted by Moser \cite{moser_old_paper}.  Detailed proofs of this have been given by Bessi \cite{bessi} and de la Llave and Su \cite{de la llave su}.  

Even if the connection between \eqref{E: pulsating standing wave} and Aubry-Mather theory does not seem to have been observed previously in the literature, it is by now well known that \eqref{E: functional} can be fit into the framework of Moser-Bangert theory (see \cite{junginger-gestrich valdinoci,rabinowitz stredulinsky book}). 

 As a corollary of the analysis of \eqref{E: lagrangian}, we give a new proof of the existence of plane-like minimizers of \eqref{E: functional}.  That question was previously treated by Alessio, Jeanjean, and Montecchiari \cite{alessio jeanjean montecchiari} and Rabinowitz and Stredulinsky \cite{rabinowitz stredulinsky} in the rational case and Valdinoci \cite{valdinoci} in general, and ultimately can be deduced from Bangert's study of heteroclinics as well \cite{minimal_laminations}.  

The approach of \cite{rabinowitz stredulinsky}, like our analysis of \eqref{E: lagrangian}, boils down to an application of the direct method of the calculus of variations.  However, the energy functional they use only makes sense when $e$ is rational.  From that point of view, \eqref{E: lagrangian} shows how to regain compactness in the irrational case (see the remark following Proposition \ref{P: remove constraint} below).  

In previous works, such as \cite{rabinowitz stredulinsky} and \cite{bessi}, the existence of minimizers has been proved assuming a certain amount of smoothness of the coefficients so that the strong maximum principle applies.  Here we give an existence proof that entirely avoids this, thereby allowing us to obtain existence under weaker regularity assumptions without appealing to approximation arguments.

Concerning the sharp interface limit of \eqref{E: main}, our idea to consider phase field approximations of graphical solutions of \eqref{E: anisotropic curvature flow} was inspired by the work of Barles, Cesaroni, and Novaga \cite{barles cesaroni novaga}, who showed that certain graphical solutions of a sharp interface analogue of \eqref{E: main} converge to the solutions of a homogenized geometric flow, again in laminar media.  They were able to show that in dimension two, the effective velocity vanishes in the direction of the laminations (see Remark 5.2 in that work, especially the last sentence).  We obtain a similar, but weaker, conclusion here, again in dimension two (see Theorem \ref{T: dim_2_stuff}, (iii) below).  

Our results highlight the fact that the macroscopic quantities describing the large scale behavior of functionals like \eqref{E: functional} can have singularities.  This question has previously been treated in Moser-Bangert theory by Senn, who showed that the analogue of the surface tension in that setting is always strictly convex \cite{senn 1} but need not be $C^{1}$ \cite{senn 2}.  Recently, Chambolle, Goldman, and Novaga \cite{chambolle goldman novaga} proved the same result for surface energies with periodic coefficients using a Lagrangian formulation morally similar to \eqref{E: lagrangian}, which had been previously introduced by Chambolle and Thouroude \cite{chambolle thouroude}.  The author plans to address the corresponding questions for $\tilde{\varphi}^{a}$ in future work.

Finally, this paper was partly influenced by recent developments in the study of pulsating wave solutions of reaction-diffusion equations such as \eqref{E: main}.  The proof of the existence of minimizers of \eqref{E: lagrangian}, which is complicated by the degeneracy of the gradient term, was inspired by advances made by Ducrot \cite{ducrot} (see Proposition \ref{P: bv_estimate} below).  It would be nice to see applications of the ideas presented here in the study of pulsating \emph{traveling} waves.  Although variational constructions of traveling waves exist in homogeneous media \cite{elliptic systems phase transitions}, it's not clear that this can be extended to the periodic set-up.  Nonetheless, we prove an integral identity for functions in $\mathbb{R} \times \mathbb{T}^{d}$ (see Theorem \ref{T: ergodic_lemma}) that is new to the best of the author's knowledge and may be of interest to experts.

\subsection{Organization of the Paper}  In the next section, we give precise statements of the main results of the paper.  Section \ref{S: notation} explains the notation used throughout.  Section \ref{S: preliminary_analysis} provides a dictionary for translating between the Lagrangian $\mathscr{T}^{a}_{e}$ in $\mathbb{R} \times \mathbb{T}^{d}$, on the one hand, and the energy $\mathcal{F}^{a}_{1}$ in $\mathbb{R}^{d}$, on the other.  The existence of minimizers of $\mathscr{T}^{a}_{e}$ is treated in Section \ref{S: existence},  and the connections between the minimizers of $\mathscr{T}^{a}_{e}$, solutions of \eqref{E: pulsating standing wave}, and plane-like minimizers of \eqref{E: functional} are treated in Section \ref{S: physical_coordinates}.  Section \ref{S: einstein relation laminar media} provides the link between the results of \cite{barles souganidis} and \cite{ansini braides} in laminar media, establishing the Einstein relation and also deriving general regularity results for the surface tension in this context.  Since the arguments are similar, Section \ref{S: einstein relation laminar media} also includes results on an elliptic regularization of $\mathscr{T}^{a}_{e}$.  Pathological examples are discussed in Section \ref{S: example_2D}.  Lastly, a sharp interface limit for a class of graphical initial data in laminar media is proved in Section \ref{S: sharp_interface_limit}. 

There are three appendices treating ancillary technical results needed in the paper.  Appendix A treats a few technical results used elsewhere in the paper.   The key change-of-variable formulas required to relate $\mathscr{T}^{a}_{e}$ to $\mathcal{F}^{a}_{1}$ are proved in Appendix B.  Appendix C covers a tubular neighborhood theorem that is needed in the proof of the sharp interface limit.

\section{Statement of Main Results} \label{S: results} 

In the majority of the work, we operate under the following assumptions: we always assume
	\begin{align}
		&a : \mathbb{T}^{d} \to \mathcal{S}_{d} \, \, \text{measurable,} \quad \lambda \text{Id} \leq a \leq \Lambda \text{Id} \label{A: a_assumption_1}, \\
		&W: [-1,1] \to [0,\infty) \, \, \text{continuous,} \, \, W^{-1}(\{0\}) = \{-1,1\} \label{A: W_assumption_1},
	\end{align}
and, when proving the Einstein relation in Section \ref{S: einstein relation laminar media}, we also assume that there is an $\alpha \in (0,1)$ such that
	\begin{align}
		&a \in C^{1,\alpha}(\mathbb{T}^{d}; \mathcal{S}_{d}), \label{A: a_assumption_2} \\
		&W \in C^{2,\alpha}([-1,1]), \quad W''(1) \wedge W''(-1) \geq \alpha. \label{A: W_assumption_2}
	\end{align}
	
In Section \ref{S: sharp_interface_limit}, when we study the sharp interface limit, we impose further regularity on $W$ and add some standard assumptions concerning the sign of $W'$ in $[-1,0]$ and $[0,1]$ (see \eqref{A: additional_regularity} and \eqref{A: sign}).

Since we will be interested in variational arguments involving \eqref{E: functional}, it is convenient to extend $W$ from $[-1,1]$ to a larger interval.  For concreteness, we will define the extension (abusing notation) $W : [-2,2] \to \mathbb{R}$ by 
	\begin{equation*}
		W(u) = \left\{ \begin{array}{r l}
			W(1) + u - 1, & \text{if} \, \, u \in [1,2], \\
			W(-1) + (1 -u), & \text{if} \, \, u \in [-2,-1].
					\end{array} \right.
	\end{equation*}  

\subsection{Minimizers of $\mathscr{T}^{a}_{e}$} \label{S: calculus of variations results} To start with, we show that the Lagrangian $\mathscr{T}^{a}_{e}$ defined in \eqref{E: lagrangian} has minimizers.  Throughout the paper, we denote by $S^{d-1}$ the unit sphere in $\mathbb{R}^{d}$.  For each $e \in S^{d-1}$, we hereafter let $\mathscr{E}^{a}(e)$ be given by
	\begin{equation*}
		\mathscr{E}^{a}(e) = \inf \left\{ \mathscr{T}^{a}_{e}(U) \, \mid \, |U| \leq 1, \, \, U(\cdot + s,\cdot) \to \pm 1 \, \, \text{locally as} \, \, \text{as} \, \, s \to \pm \infty \right\}.
	\end{equation*}

	\begin{theorem} \label{T: existence}  (1) If $a$ and $W$ satisfy \eqref{A: a_assumption_1} and \eqref{A: W_assumption_1}, then, for each $e \in S^{d -1}$, there is a minimizer $U_{e}$ of the variational principle $\mathscr{E}^{a}(e)$ satisfying $\partial_{s} U_{e} \geq 0$.  
	
	(2) If, in addition, $a$ and $W$ satisfy \eqref{A: a_assumption_2} and \eqref{A: W_assumption_2} and $e \notin \mathbb{R} \mathbb{Z}^{d}$, then $U_{e}$ is unique up to translations in the $s$ variable. 
	
	(3) $\mathscr{E}^{a} = \tilde{\varphi}^{a}$ in $S^{d-1}$, where $\tilde{\varphi}^{a}$ is the integrand in \eqref{E: anisotropic perimeter}. \end{theorem}
	
	As a consequence of the theorem, we have a new proof of the existence of plane-like minimizers of $\mathcal{F}^{a}_{1}$.  Let us recall that, following, for instance, \cite{caffarelli de la llave}, a function $u : \mathbb{R}^{d} \to [-1,1]$ is called a Class A minimizer of $\mathcal{F}_{1}^{a}$ if, no matter the choice of open set $\Omega \subseteq \mathbb{R}^{d}$, we have
		\begin{equation*}
			\mathcal{F}_{1}^{a}(u; \Omega) \leq \mathcal{F}_{1}^{a}(u + f; \Omega) \quad \text{for each} \, \, f \in C^{\infty}_{c} (\Omega; [-1,1]). 
		\end{equation*}

By a plane-like minimizer, we mean a Class A minimizer $u$ of $\mathcal{F}^{a}_{1}$ for which the transition region between the equilibrium phases $1$ and $-1$ looks like a plane at the macroscopic scale, that is, for some $e \in S^{d-1}$, we have
	\begin{equation*}
		\lim_{\epsilon \to 0^{+}} u(\epsilon^{-1}x) = \left\{ \begin{array}{r l}
												1, & \text{if} \, \, \langle x,e \rangle > 0, \\
												-1, & \text{if} \, \, \langle x,e \rangle < 0.
											\end{array} \right.
	\end{equation*}
The next corollary shows that $\mathcal{F}^{a}_{1}$ has a very special class of these:
	
	\begin{corollary} \label{C: new_existence} If $a$ and $W$ satisfy \eqref{A: a_assumption_1} and \eqref{A: W_assumption_1}, then, for each $e \in S^{d-1}$, there is a family $\{u_{\zeta}\}_{\zeta \in \mathbb{R}}$ of functions in $\mathbb{R}^{d}$ taking values in $[-1,1]$ such that:
		\begin{itemize}
			\item[(i)] For each $\zeta \in \mathbb{R}$, $u_{\zeta}$ is a Class A minimizer of $\mathcal{F}^{a}_{1}$ and $\lim_{\langle x, e \rangle \to \pm \infty} u_{\zeta}(x) = \pm 1$ uniformly in $\langle e \rangle^{\perp}$.
			\item[(ii)] The map $\zeta \mapsto u_{\zeta}$ from $\mathbb{R}$ into $C_{\text{loc}}(\mathbb{R}^{d})$ (with the topology of local uniform convergence) is non-increasing and has at most countably many discontinuities.
			\item[(iii)] If $k \in \mathbb{Z}^{d}$ and $\zeta \in \mathbb{R}$, then $u_{\zeta + \langle k,e \rangle} = u_{\zeta}(\cdot -k)$.
			\item[(iv)] $\{u_{\zeta}\}_{\zeta \in \mathbb{R}}$ is uniformly $\gamma$-H\"{o}lder continuous in $\mathbb{R}^{d}$ for some $\gamma \in (0,1)$.
			\item[(v)] For almost every $\zeta \in \mathbb{R}$, the following limit exists and satisfies:
				\begin{equation*}
					\lim_{R \to \infty} R^{1 - d} \mathcal{F}^{a}_{1}(u_{\zeta}; Q^{e}(0,R) \oplus_{e} \mathbb{R}) = \tilde{\varphi}^{a}(e).
				\end{equation*}
			If $e \in \mathbb{R} \mathbb{Z}^{d}$ or \eqref{A: a_assumption_1} and \eqref{A: W_assumption_2} both hold, then we can assume without loss of generality this is true for every $\zeta \in \mathbb{R}$.  
		\end{itemize}
	\end{corollary}  

Notice that (iii) implies that each function in the family $\{u_{\zeta}\}_{\zeta \in \mathbb{R}}$ satisfies the Birkhoff property (see Section \ref{S: birkhoff} for the definition).  

As noted in the literature review above, results like Corollary \ref{C: new_existence} on the existence of plane-like minimizers with the Birkhoff property have previously been proved in \cite{alessio jeanjean montecchiari} and \cite{rabinowitz stredulinsky} in the rational case and \cite{valdinoci} in general.

Next, we revisit the pulsating standing waves of \cite{barles souganidis}.  As we will see below, these functions, if smooth, generate a foliation of $\mathbb{R}^{d} \times [-1,1]$ by plane-like critical points of $\mathcal{F}^{a}_{1}$.  Thus, each of these functions should be a Class A minimizer by extremal field theory-type arguments.  When $a$ and $W$ are sufficiently regular, this is indeed the case.

	\begin{prop} \label{P: continuous_minimizers}  Suppose that $a$ and $W$ satisfy \eqref{A: a_assumption_1}, \eqref{A: W_assumption_1}, \eqref{A: a_assumption_2}, and \eqref{A: W_assumption_2}.  If $e \in S^{d-1}$ and $U \in C(\mathbb{R} \times \mathbb{T}^{d}; [-1,1])$ is a solution of the pulsating standing wave equation \eqref{E: pulsating standing wave} satisfying
		\begin{equation*}
			\partial_{s} U \geq 0, \quad \lim_{s \to \pm \infty} U(\cdot + s,\cdot) = \pm 1 \quad \text{in} \, \, L^{1}_{\text{loc}}(\mathbb{R} \times \mathbb{T}^{d}),
		\end{equation*} 
	then $U$ is a minimizer of $\mathscr{E}^{a}(e)$ and the critical points $\{u_{\zeta}\}_{\zeta \in \mathbb{R}}$ of $\mathcal{F}^{a}_{1}$ generated by $U$ are Class A minimizers.  In particular, this applies to the pulsating standing waves of \cite[Equation 6.8]{barles souganidis}.   \end{prop}  
		
Note that the proposition shows that the assumptions of \cite[Theorem 6.3]{barles souganidis} actually impose a rather strong constraint on the energy functional \eqref{E: functional}.  We will see below that there are smooth $a$ and $W$ for which $\mathscr{E}^{a}(e)$ does not have continuous minimizers.

\subsection{Differentiability of $\tilde{\varphi}^{a}$ and the Einstein Relation}  The remainder of the paper concerns the sharp interface limit, revisiting what was done in \cite{barles souganidis}.  To start with, we prove a regularity result for the surface tension in laminar media and relate the effective interface velocity found in \cite{barles souganidis} to the surface tension.

Before we state the result, we remark that in the laminar setting, it is instructive to alter the domain of integration in the Lagrangian $\mathscr{T}^{a}_{e}$.  In the results that follow, we assume that $a$ is $\mathbb{Z}^{k} \times \mathbb{R}^{d- k}$-periodic and, thus, is the extension of a function in $\mathbb{T}^{k}$ to one in $\mathbb{T}^{d}$.  In this case, we replace $\mathbb{R} \times \mathbb{T}^{d}$ by $\mathbb{R} \times \mathbb{T}^{k}$ and define
	\begin{equation*}
		\mathscr{T}^{a}_{e}(U) = \int_{\mathbb{R} \times \mathbb{T}^{k}} \left(\frac{1}{2} \langle a(x) \mathcal{D}_{e} U, \mathcal{D}_{e}U \rangle + W(U) \right)\, dx \, ds, \quad \mathcal{D}_{e} := e \partial_{s} + D_{x},
	\end{equation*}
where the function $U$ has domain $\mathbb{R} \times \mathbb{T}^{k}$, $D_{x} = (\partial_{x_{1}},\partial_{x_{2}},\dots,\partial_{x_{k}},0,\dots,0)$, and $e \in S^{d-1} \subseteq \mathbb{R}^{d}$.  We take this approach because it makes the benefits of the laminarity assumption more readily apparent (see Section \ref{S: laminar discussion}).  

With the modified definition of $\mathscr{T}^{a}_{e}$ now in hand, we proceed with the main result.

	\begin{theorem} \label{T: differentiability} If $a$ and $W$ satisfy \eqref{A: a_assumption_1}, \eqref{A: W_assumption_1}, \eqref{A: a_assumption_2} and \eqref{A: W_assumption_2}, and if $a$ is $\mathbb{Z}^{k} \times \mathbb{R}^{d- k}$-periodic for some $k < d$, then:
		\begin{itemize}
			\item[(i)] $\tilde{\varphi}^{a} \in C^{2}(\mathbb{R}^{d} \setminus (\mathbb{R}^{k} \times \{0\}))$.
			\item[(ii)] For each $e \in S^{d -1} \setminus (S^{k - 1} \times \{0\})$, there is a unique, smooth $U_{e}$ minimizing the problem $\mathscr{E}^{a}(e)$ subject to the monotonicity condition $\partial_{s} U_{e} \geq 0$ and the constraint $\int_{\mathbb{T}^{k}} U_{e}(s,x) \, dx = 0$.
			\item[(iii)]  The map $e \mapsto U_{e}$ is continuously Fr\'{e}chet differentiable from $S^{d -1} \setminus (S^{k - 1} \times \{0\})$ into the space $BC(\mathbb{R} \times \mathbb{T}^{d})$.
			\item[(iv)]
			For each $e \in S^{d-1} \setminus (S^{k-1} \times \{0\})$, 
			$\partial_{s} U_{e} \in L^{2}(\mathbb{R} \times \mathbb{T}^{d})$ 
			and the derivative of $\tilde{\varphi}$ at $e$ is given by 
				\begin{equation*}
			D\tilde{\varphi}^{a}(e) = \int_{\mathbb{R} \times \mathbb{T}^{k}}  \partial_{s}U_{e} a(x) \mathcal{D}_{e}U_{e}  \, dx \, ds.
				\end{equation*}
			\item[(v)] Given $\xi \in \mathbb{R}^{d}$ and $e$ as in (iv), if we define functions $R_{e}^{\xi}$ and $\Psi_{e}^{\xi}$ by $R_{e}^{\xi} = \langle D_{e} U_{e}, \xi \rangle$ and $\Psi_{e}^{\xi} = (\partial_{s}U_{e})^{-1} R_{e}^{\xi}$, then
		\begin{align*}
			\langle D^{2}\tilde{\varphi}^{a}(e) \xi, \xi \rangle &= \int_{\mathbb{R} \times \mathbb{T}^{k}} \langle a(x) \xi, \xi \rangle |\partial_{s} U_{e}|^{2} \, dx \, ds \\
				&\qquad- \int_{\mathbb{R} \times \mathbb{T}^{k}} (\langle a(x) \mathcal{D}_{e}R_{e}^{\xi}, R_{e}^{\xi} \rangle + W''(U_{e}) |R_{e}^{\xi}|^{2}) \, dx \, ds \\
			&= \int_{\mathbb{R} \times \mathbb{T}^{k}} \langle a(x)(\xi + \mathcal{D}_{e}\Psi_{e}^{\xi}), \xi + \mathcal{D}_{e} \Psi_{e}^{\xi} \rangle |\partial_{s}U_{e}|^{2} \, dx \, ds.
		\end{align*}
		\end{itemize}
	\end{theorem}  

Notice that the expression for $D^{2} \tilde{\varphi}^{a}(e)$ in terms of $\Psi^{\xi}_{e}$ is reminiscent of the equation for the effective diffusion matrix in linear elliptic homogenization.  

Manipulating the representation of $D^{2}\tilde{\varphi}^{a}$ from the theorem, we obtain the Einstein relation.

	\begin{corollary} \label{C: einstein_relation}  If the hypotheses of Theorem \ref{T: differentiability} are satisfied,  then the matrix $\tilde{\mathcal{S}}^{a}(e)$ in \eqref{E: anisotropic curvature flow}, which was originally defined in \cite[Section 6]{barles souganidis}, is well-defined as long as $e \in S^{d-1} \setminus (S^{k -1} \times \{0\})$, and, in that case, $\tilde{\mathcal{S}}^{a}(e) = D^{2}\tilde{\varphi}^{a}(e)$.  
	
	Additionally, the matrix $(\tilde{M}^{a})^{-1} D^{2}\tilde{\varphi}^{a}$ appearing in \eqref{E: anisotropic curvature flow} satisfies the following bound:
		\begin{equation*}
			\tilde{M}^{a}(e)^{-1} D^{2}\tilde{\varphi}^{a}(e) \leq \Lambda (\text{Id} - e \otimes e) \quad \text{if} \, \, e \in S^{d-1} \setminus (S^{k-1} \times \{0\}).
		\end{equation*}
	\end{corollary}  
	
\subsection{A counter-example in 2D and other pathologies} \label{S: counterexamples} We show that, in general, even if $a$ and $W$ are smooth, the minimizers of the problem $\mathscr{E}^{a}(e)$ may not be smooth (or even continuous).    

	\begin{theorem} \label{T: non_integrable}  There is a smooth function $a_{0} : \mathbb{T} \to (0,\infty)$ such that if $W(u) = \frac{1}{4}(1 - u^{2})^{2}$, then, for $d=  1$, $\mathscr{E}^{a}(e)$ does not have a continuous minimizer in either direction in $S^{0}$.  We can extend this to any dimension $d \in \mathbb{N}$, thereby obtaining examples of laminar media in which there are at least two directions where $\mathscr{T}^{a}_{e}$ does not have continuous minimizers.  \end{theorem} 
	
The idea here goes back to the transformation \eqref{E: transformation} taking a minimizer $U$ of $\mathscr{T}^{a}_{e}$ to the functions $\{u_{\zeta}\}_{\zeta \in \mathbb{R}}$ it generates.  In dimension $d = 1$, each of these functions satisfies $\tilde{\varphi}^{a}(e) = \int_{-\infty}^{\infty} \left(\frac{1}{2} a_{0}(x) u_{\zeta}'(x)^{2} + W(u_{\zeta}(x)) \right) \, dx$.  If we choose $a_{0}$ to have a very sudden bump upwards in a neighborhood of a certain point $x_{0} \in \mathbb{R}$, then it is natural to expect that the functions $\{u_{\zeta}\}_{\zeta \in \mathbb{R}}$, trying to minimize their energy, will avoid the bump and, for instance, 
	\begin{equation*}
		\min\{u_{\zeta}(x_{0}) \, \mid \, \zeta \in \mathbb{R}\} < 0 < \max\{u_{\zeta}(x_{0}) \, \mid \, \zeta \in \mathbb{R}\}.
	\end{equation*}
Given that we know that $\lim_{\zeta \to \pm \infty} u_{\zeta}(x_{0}) = \pm 1$, this can only happen if $\zeta \mapsto u_{\zeta}(x_{0})$ is discontinuous, and hence $U$ is also.

Theorem \ref{T: non_integrable} is proved below precisely by constructing such an $a_{0}$.

	In dimension two, we carry the analysis further, providing, in particular, an example where the coefficients in \eqref{E: anisotropic curvature flow} become arbitrarily small in the degenerate directions.  The reader should contrast the pathologies evinced in the next theorem with what occurs in the spatially homogeneous case.  If, for instance, $a \equiv \lambda \text{Id}$, then there is a constant $c_{W}$  depending only on $W$ \cite{alberti guide,barles souganidis} such that, for each $e \in S^{d-1}$,
		\begin{align*}
			\tilde{\varphi}^{\lambda \text{Id}}(e) = \tilde{M}^{\lambda \text{Id}}(e) = c_{W} \sqrt{\lambda}, \quad \tilde{M}^{\lambda \text{Id}}(e)^{-1} D^{2} \tilde{\varphi}^{\lambda \text{Id}}(e) = \text{Id} - e \otimes e.
		\end{align*}
	By contrast, the next result shows that, in periodic media, $\tilde{M}^{a}$ need not be bounded or continuous and $(\tilde{M}^{a})^{-1} \|D^{2}\tilde{\varphi}^{a}\|$ can approach zero.
	
	\begin{theorem} \label{T: dim_2_stuff} In dimension $d = 2$ with the potential $W(u) = \frac{1}{4} (1 - u^{2})^{2}$, there is a smooth function $a_{0} : \mathbb{T} \to (0,\infty)$ such that if $a = a_{0} \text{Id}$, then the surface tension $\tilde{\varphi}^{a}$ and the mobility $\tilde{M}^{a}$ have the following properties:
		\begin{itemize}
			\item[(i)] $\tilde{\varphi}^{a}$ is not differentiable in the directions $e_{1} = (1,0)$ or $-e_{1} = (-1,0)$ (i.e.\ in the direction of the laminations).
			\item[(ii)] The asymptotic behavior of the mobility $\tilde{M}^{a}$ as $e \to \pm e_{1}$ is given by
				\begin{equation*}
					0 < \liminf_{e \to \pm e_{2}} |\langle e, e_{2} \rangle| \tilde{M}^{a}(e) \leq \limsup_{e \to \pm \bar{e}_{1}} |\langle e, e_{2} \rangle| \tilde{M}^{a}(e) < \infty,
				\end{equation*}
			where $e_{2} = (0,1)$.
			\item[(iii)] $\liminf_{e \to \pm e_{1}} \tilde{M}^{a}(e)^{-1} \|D^{2}\tilde{\varphi}^{a}(e)\| = 0$.
		\end{itemize}
	\end{theorem}   

It is worth emphasizing that Theorem \ref{T: dim_2_stuff} shows that pathological behavior can occur in the sharp interface limit even if the gradient term in $\mathcal{F}^{a}_{1}$ is isotropic (i.e.\ $a$ is a positive function times the identity).  The difficulties are caused by the periodicity of $a$, not anisotropic effects arising from its matrix character.
	
The previous theorem suggests that, unlike the spatially homogeneous setting (cf.\ \cite[Theorem 2.1]{bellettini butta presutti}), even if we can prove a sharp interface limit for arbitrary initial data, the coefficients in the effective equation might not satisfy a bound like
		\begin{equation*}
			\tilde{M}^{a}(e)^{-1} D^{2}\tilde{\varphi}^{a}(e) \geq c (\text{Id} - e \otimes e),
		\end{equation*}
where $c > 0$ is independent of $e$.  

Using the well-known elliptic regularization approach (cf.\ \cite{moser_old_paper}, \cite{ducrot}, \cite{xin}), we prove a result that can be interpreted as an obstruction to such bounds.  In the statement of the theorem, the smooth functions $(\tilde{\varphi}^{a,\delta})_{\delta > 0}$ and $(\tilde{M}^{a,\delta})_{\delta > 0}$ are obtained through elliptic regularization and converge to $\tilde{\varphi}^{a}$ and $\tilde{M}^{a}$, respectively, as $\delta \to 0^{+}$.  It is natural to guess that whenever $\tilde{M}^{a}(e)^{-1} D^{2}\tilde{\varphi}^{a}(e)$ makes sense at some point $e \in S^{d-1}$, we have
	\begin{equation*}
		\lim_{\delta \to 0^{+}} \tilde{M}^{a,\delta}(e)^{-1} D^{2}\tilde{\varphi}^{a,\delta}(e) = \tilde{M}^{a}(e)^{-1} D^{2}\tilde{\varphi}^{a}(e).
	\end{equation*}
This leads to the question of $\delta$-independent bounds on $(\tilde{M}^{a,\delta})^{-1} D^{2} \tilde{\varphi}^{a,\delta}$.  The next theorem shows there is an obstruction to lower bounds. 

\begin{theorem} \label{T: no lower bound}  Assume $a$ and $W$ satisfy \eqref{A: a_assumption_1}, \eqref{A: W_assumption_1}, \eqref{A: a_assumption_2}, and \eqref{A: W_assumption_2} and $d \geq 2$.  If there is an $\mathcal{H}^{d-1}$-measurable set $E \subseteq S^{d-1}$ and constants $c, \delta_{0} > 0$ such that, for each $\delta \in (0,\delta_{0})$ and $e \in E$,
		\begin{equation} \label{E: lower bound}
			\tilde{M}^{a,\delta}(e)^{-1} D^{2}\tilde{\varphi}^{a,\delta}(e) \geq c(\text{Id} - e \otimes e),
		\end{equation}
	then, for $\mathcal{H}^{d -1}$-a.e.\ $e \in E$, there is a minimizer $U_{e} \in UC(\mathbb{R} \times \mathbb{T}^{d})$ of $\mathscr{E}^{a}(e)$ with $\partial_{s}U_{e} \in L^{2}(\mathbb{R} \times \mathbb{T}^{d})$.  		\end{theorem}  
	
	Taken together with Theorem \ref{T: dim_2_stuff}, Theorem \ref{T: no lower bound} suggests that lower bounds on the effective interface velocity only hold under very special circumstances.
	
\subsection{Sharp interface limit for a special class of initial data}  In view of the results described above, it is natural to wonder whether or not there are examples of periodic media and initial data in which homogenization occurs and is described by \eqref{E: anisotropic curvature flow}.  Here we give an affirmative answer.  

In the result that follows, we add the following two assumptions on $W$:
	\begin{align}
		&W \in C^{4,\alpha}([-1,1]), \label{A: additional_regularity} \\
		(-1,0) &\subseteq \{W' > 0\}, \quad (0,1) \subseteq \{W' < 0\}. \label{A: sign}
	\end{align}
The second assumption is standard in the literature on sharp interface limits; the first is an artifact of the proof.

	\begin{theorem} \label{T: sharp_interface_limit_graphs} Assume that $a$ is $\mathbb{Z}^{k} \times \mathbb{R}^{d - k}$-periodic for some $k < d$, $a$ and $W$ satisfy \eqref{A: a_assumption_1}, \eqref{A: W_assumption_1}, \eqref{A: a_assumption_2}, \eqref{A: W_assumption_2}, \eqref{A: additional_regularity}, and \eqref{A: sign}, and $e \in S^{d-1} \cap (\mathbb{R}^{k} \times \{0\})^{\perp}$.  Suppose that $u_{0} \in UC(\mathbb{R}^{d} ; [-1,1])$ satisfies the following two conditions:
		\begin{itemize}
			\item[(i)] There is a $\mathcal{U} \in UC(\mathbb{R}^{d - 1})$ and a $q \in \mathbb{R}^{d-1}$ such that 
				\begin{equation*}
					\sup \left\{ |\mathcal{U}(x') - \langle q, x' \rangle| \, \mid \, x' \in \mathbb{R}^{d-1} \right\} < \infty 
				\end{equation*} 
			and 
				\begin{equation*}
					\{x \in \mathbb{R}^{d} \, \mid \, u_{0}(x) = 0\} = \{x \in \mathbb{R}^{d} \, \mid \, \langle x,e \rangle = \mathcal{U}(x - \langle x,e \rangle e)\}.
				\end{equation*}
			\item[(ii)] There is a $\delta_{0} > 0$ and an $R > 0$ such that
				\begin{align*}
					d_{\mathcal{H}}(\{x \in \mathbb{R}^{d} \, \mid \, u_{0}(x) = 0\}, \{x \in \mathbb{R}^{d} \, \mid \, |u_{0}(x)| < \delta_{0}\}) &< R.
				\end{align*}
		\end{itemize}
	There is a continuous function $h : \mathbb{R}^{d-1} \times [0,\infty) \to \mathbb{R}$ with $h(\cdot,0) = \mathcal{U}$ such that if $(u^{\epsilon})_{\epsilon > 0}$ are the solutions of \eqref{E: main} with initial datum $u^{\epsilon}(\cdot,0) = u_{0}$ and $(E_{t})_{t \geq 0}$ are the epigraphs defined by $E_{t} = \{x \in \mathbb{R}^{d} \, \mid \, \langle x,e \rangle > h(x - \langle x,e \rangle e, t)\}$, then
		\begin{equation*}
			\lim_{\epsilon \to 0^{+}} u^{\epsilon} = \left\{ \begin{array}{r l}
													1, & \text{locally uniformly in} \, \, \bigcup_{t > 0} E_{t}, \\
													-1, & \text{locally uniformly in} \, \, \bigcup_{t > 0} (\mathbb{R}^{d} \setminus \overline{E}_{t}).
												\end{array} \right.
		\end{equation*}
	Furthermore, $(E_{t})_{t \geq 0}$ is a solution of the geometric flow \eqref{E: anisotropic curvature flow}.
	\end{theorem}
	
Certainly, we have considerably restricted the class of initial data compared to \cite{barles souganidis}.  Nevertheless, the theorem shows that there are initial data that homogenize in the sharp interface limit, without making any assumptions other than sufficient smoothness of the coefficients and laminarity.  Significantly, Theorem \ref{T: sharp_interface_limit_graphs} provides the first class of examples in which the effective behaviors of \eqref{E: functional} and \eqref{E: main} are known to be related through the Einstein relation \eqref{E: einstein relation}.  Put another way, the theorem shows that (up to a mobility factor) the gradient flow and homogenization ``commute" in this very special case.

\section{Notation}  \label{S: notation}

\subsection{General}  If $a,b \in \mathbb{R}$, we define $a \vee b$ and $a \wedge b$ by
	\begin{equation*}
		a \vee b = \max\{a,b\}, \quad a \wedge b= \min \{a,b\}.
	\end{equation*}   

We define $\text{sgn}(s) = \frac{s}{|s|}$ if $s \neq 0$.

If $X$ is a metric space with metric $d$, we denote by $B(x,\epsilon) = \{y \in X \, \mid \, d(x,y) < \epsilon\}$.  The Hausdorff distance $d_{\mathcal{H}}(A,B)$ between two sets $A,B$ contained in $X$ is defined by 
	\begin{equation*}
		d_{\mathcal{H}}(A,B) = \inf \left\{\epsilon \geq 0 \, \mid \, A \subseteq \bigcup_{b \in B} B(b,\epsilon), \, \, B \subseteq \bigcup_{a \in A} B(a,\epsilon) \right\}.
	\end{equation*}	
	
\subsection{Euclidean Space}  If $v \in \mathbb{R}^{d}$, then $\langle v \rangle = \{\alpha v \, \mid \, \alpha \in \mathbb{R}\}$.

The Euclidean inner product between two vectors $\xi,\zeta \in \mathbb{R}^{d}$ is denoted by $\langle \xi, \zeta \rangle$.  If $A \subseteq \mathbb{R}^{d}$, then $A^{\perp} = \{x \in \mathbb{R}^{d} \, \mid \, \langle a, x \rangle = 0 \, \, \text{if} \, \, a \in A \}$.  

We write $\|\cdot\|$ for the norm induced by the inner product $\langle \cdot, \cdot \rangle$.  $S^{d-1}$ is the $(d-1)$-dimensional sphere in $\mathbb{R}^{d}$, that is, 
	\begin{equation}
		S^{d-1} = \{e \in \mathbb{R}^{d} \, \mid \, \|e\| = 1\}. \label{E: sphere notation}
	\end{equation}
	
In $\mathbb{R}^{d}$, we denote by $\{e_{1},\dots,e_{d}\}$ the standard orthonormal basis given by $e_{1} = (1,0,0,\dots,0),e_{2} = (0,1,0,\dots,0),\dots,e_{d} = (0,\dots,0,1)$.

If $e \in S^{d-1}$, $A\subseteq \langle e \rangle^{\perp}$, and $E \subseteq \mathbb{R}$, then we define $A \oplus_{e} E$ by
	\begin{equation*}
		A \oplus_{e} E = \{a + \alpha e \, \mid \, a \in A, \, \, \alpha \in E \}.
	\end{equation*}
	
Frequently we will be interested in cubes contained in some hyperplane of $\mathbb{R}^{d}$.  In this case, given $e \in S^{d-1}$, we fix an orthonormal basis $\{v_{1},\dots,v_{d-1}\}$ of $\langle e \rangle^{\perp}$ and, for each $R > 0$, let $Q^{e}(0,R)$ be the cube in $\langle e \rangle^{\perp}$ determined by this basis, of side length $R$, and centered at $0$.  In other words,
	\begin{equation*}
		Q^{e}(0,R) = \left\{ y \in \langle e \rangle^{\perp} \, \mid \, \max\{|\langle y, v_{1}\rangle|,\dots,|\langle y, v_{d-1}\rangle|\} \leq R/2 \right\}.
	\end{equation*}  
The specific choice of basis is irrelevant where the results of the paper are concerned.
	
\subsection{Linear Algebra}  $M_{d}$ is the space of real $d \times d$-matrices.  $\mathcal{S}_{d}$ is the subspace consisting of symmetric matrices.  

If $A,B \in \mathcal{S}_{d}$, we write $A \leq B$ if $\langle (A - B) \xi, \xi \rangle \leq 0$ for all $\xi \in \mathbb{R}^{d}$.  

If $A \in M_{d}$, then $A^{*}$ denotes its transpose.

Given $\xi, \zeta \in \mathbb{R}^{d}$, $\xi \otimes \zeta$ is the linear operator on $\mathbb{R}^{d}$ defined by $(\xi \otimes \zeta)(v) = \langle \zeta,v \rangle \xi$.  Given matrices $A,B \in M_{d}$, $A \otimes B$ is the linear operator on $M_{d}$ defined by 
	\begin{equation*}
		(A \otimes B)(v \otimes w) = Av \otimes Bw
	\end{equation*} 
and extended to the entire space by linearity.   

\subsection{Functions} \label{S: function notation}  If $V$ is a function on $\mathbb{R} \times \mathbb{T}^{d}$, $s_{0} \in \mathbb{R}$, and $x_{0} \in \mathbb{R}^{d}$, we define functions $T_{s_{0}}V$ and $S_{x_{0}}V$ in $\mathbb{R} \times \mathbb{T}^{d}$ by 
	\begin{equation*}
		T_{s_{0}}V(s,x) = V(s - s_{0}, x), \quad S_{x_{0}}V(s,x) = V(s,x - x_{0}).
	\end{equation*}

Given functions $f$ and $g$ on the same domain, we define $f \vee g$ and $f \wedge g$ by 
	\begin{equation*}
		(f \vee g)(x) = \max \left\{f(x),g(x) \right\}, \quad (f \wedge g)(x) = \min \left\{f(x),g(x)\right\}.
	\end{equation*}

Given a family of functions $(f^{\epsilon})_{\epsilon > 0}$, each defined on a metric space $X$ with metric $d$, we define the upper and lower half-relaxed limits $\limsup^{*} f^{\epsilon}$ and $\liminf_{*} f^{\epsilon}$, respectively, by 
	\begin{align}
		\limsup \nolimits^{*} f^{\epsilon}(x) &= \lim_{\delta \to 0^{+}} \sup \left\{ f^{\epsilon}(y) \, \mid \, d(x,y) + \epsilon < \delta \right\}, \label{E: upper half relaxed} \\
		\liminf \nolimits_{*} f^{\epsilon}(x) &= \lim_{\delta \to 0^{+}} \inf \left\{ f^{\epsilon}(y) \, \mid \, d(x,y) + \epsilon < \delta \right\}. \label{E: lower half relaxed}
	\end{align}

\subsection{Measure Theory}  The $k$-dimensional Lebesgue measure in $\mathbb{R}^{k}$ is denoted by $\mathcal{L}^{k}$.  We will abuse notation and write $\mathcal{L}^{d + 1}$ also for the product measure on $\mathbb{R} \times \mathbb{T}^{d}$ obtained from $\mathcal{L}^{1}$ on $\mathbb{R}$ and $\mathcal{L}^{d}$ on $\mathbb{T}^{d}$.  The $m$-dimensional Hausdorff measure is $\mathcal{H}^{m}$.  

\subsection{Derivatives and Related}  In $\mathbb{R} \times \mathbb{T}^{d}$, we write $\partial_{s}$ for the derivative operator with respect to $s$ and $D_{x}$ for the derivative with respect to $x$.  In $\mathbb{R}^{d}$, we usually indicate (scalar) partial derivatives using subscripts.  

We will use the same notation for classical, weak, and distributional derivatives.  In particular, $\partial_{s}$ is often intended in the distributional sense.  

If $U \in L^{1}_{\text{loc}}(\mathbb{R} \times \mathbb{T}^{d})$ and $\Omega$ is a bounded open subset of $\mathbb{R} \times \mathbb{T}^{d}$, we define $TV(U; \Omega)$ by 
	\begin{align*}
		TV(U; \Omega) &= \sup \bigg \{ \int_{\Omega} U(s,x) (\text{div} \, \Psi)(s,x) \, dx \, ds \, \mid \, \Psi \in C^{1}(\Omega; \mathbb{R}^{d + 1}),\\
			&\qquad \qquad \|\Psi\|_{L^{\infty}(\Omega)} \leq 1 \bigg \}.
	\end{align*}
	
Given a convex function $\psi : \mathbb{R}^{d} \to \mathbb{R}$, we denote its second derivative (as a Radon measure) by $D^{2}\psi$.  That is, if $g \in C^{\infty}_{c}(\mathbb{R}^{d})$, then
	\begin{equation*}
		\int_{\mathbb{R}^{d}} g(y) \, D^{2}\psi(dy) = \int_{\mathbb{R}^{d}} D^{2}g(x) \psi(x) \, dx.
	\end{equation*}  

\section{Analysis in $\mathbb{R} \times \mathbb{T}^{d}$} \label{S: preliminary_analysis}

We begin by collecting some facts and terminology that will be useful in the analysis of functions in $\mathbb{R} \times \mathbb{T}^{d}$.  The most important result is Theorem \ref{T: ergodic_lemma}, which provides the link between the Lagrangian \eqref{E: lagrangian} and the energy \eqref{E: functional}.  

\subsection{Rational versus Irrational Directions in $S^{d - 1}$}  In the study of \eqref{E: lagrangian}, the choice of the direction $e$ plays an important role.  As we will see, there is a marked difference between rational and irrational directions.  By a rational direction, we mean an element $e \in S^{d-1} \cap \mathbb{R} \mathbb{Z}^{d}$.  In what follows, it is helpful to define $M_{e} \subseteq \mathbb{Z}^{d}$ to be the subgroup 
	\begin{equation*}
		M_{e} = \{k \in \mathbb{Z}^{d} \, \mid \, \langle k,e \rangle = 0\}.
	\end{equation*}

Next, we state an algebraic fact that is at the heart of the distinction between rational and irrational directions:
	
	\begin{prop} \label{P: algebra} Suppose $e \in S^{d - 1}$.  The following are equivalent: 
		\begin{itemize}
			\item[(i)] $e \in \mathbb{R} \mathbb{Z}^{d}$,
			\item[(ii)] $M_{e}$ is a subgroup of $\mathbb{Z}^{d}$ of rank $d - 1$,
			\item[(iii)] $\{\langle k, e \rangle \, \mid \, k \in \mathbb{Z}^{d}\}$ is a discrete sub-group of $\mathbb{R}$ (and, thus, has rank one).
		\end{itemize}\end{prop}  
		
	The proof that (ii) and (iii) imply (i) is provided in Appendix \ref{A: transformation properties}.    

We will be interested in functions $v : \mathbb{R}^{d} \to \mathbb{R}$ invariant under the action of $M_{e}$ on $\mathbb{R}^{d}$, that is, those functions for which 
	\begin{equation} \label{E: periodicity}
		v(x + k) = v(x) \quad \text{if} \, \, k \in M_{e}.
	\end{equation}
Depending on the rank of $M_{e}$, such functions are periodic in a certain number of directions.  Since we will exploit some of the particular structure associated to the cases when $e \in \mathbb{R} \mathbb{Z}^{d}$ or not, we will explain it in the remainder of this section.

\subsection{From Cylindrical to Physical Coordinates}  

While we are ultimately interested in functions in $\mathbb{R}^{d}$, the Lagrangian $\mathscr{T}^{a}_{e}$ takes as inputs functions in $\mathbb{R} \times \mathbb{T}^{d}$.  To streamline what comes later, let us lay the groundwork for the correspondence between $\mathscr{T}^{a}_{e}$ and $\mathbb{R} \times \mathbb{T}^{d}$ (the problem in ``cylindrical coordinates") and $\mathcal{F}^{a}_{1}$ and $\mathbb{R}^{d}$ (in ``physical coordinates").  

	\begin{definition} \label{D: generators}  If $U : \mathbb{R} \times \mathbb{T}^{d} \to \mathbb{R}$ and $e \in S^{d - 1}$, the family of functions $\{u_{\zeta}\}_{\zeta \in \mathbb{R}}$ generated by $U$ in the direction of $e$ are the functions in $\mathbb{R}^{d}$ given by
		\begin{equation*}
			u_{\zeta}(x) = U(\langle x, e \rangle - \zeta, x).
		\end{equation*}
	\end{definition}
	
	When the direction $e$ is understood, we will not mention it and simply refer to the ``family of functions generated by $U$."
	
Among the advantages of this perspective, a simple computation gives

	\begin{prop}  If $U : \mathbb{R} \times \mathbb{T}^{d} \to \mathbb{R}$ and $e \in S^{d - 1}$, the family of functions $\{u_{\zeta}\}_{\zeta \in \mathbb{R}}$ generated by $U$ in the direction of $e$ satisfies \eqref{E: periodicity}.  \end{prop}

\subsection{Periodicity when $e \in \mathbb{R} \mathbb{Z}^{d}$}

Let us begin by choosing a fundamental domain for the action of $M_{e}$ on $\mathbb{R}^{d}$.  Given $e \in \mathbb{R} \mathbb{Z}^{d}$, fix a $\mathbb{Z}$-basis $\{v_{1}^{e},\dots,v_{d- 1}^{e}\}$ of $M_{e}$.  We will work with this assignment of basis throughout the paper.

	In what follows, we define the simplex $Q_{e}$ by 
		\begin{equation*}
			Q_{e} = \left\{ \sum_{i =1}^{d-1} \lambda_{i} v^{e}_{i} \, \mid \, (\lambda_{1},\dots,\lambda_{d-1}) \in [0,1)^{d-1} \right\}.
		\end{equation*}

	\begin{prop}  $Q_{e} \oplus_{e} \mathbb{R}$ is a fundamental domain for the action of $M_{e}$ on $\mathbb{R}^{d}$, that is, if $x \in \mathbb{R}^{d}$, then there is a unique $k \in M_{e}$ such that $x - k \in Q_{e} \oplus_{e} \mathbb{R}$.  \end{prop}

Note that, in general, $Q_{e}$ need not be a cube.  For example, if $d = 3$ and $e = 3^{-\frac{1}{2}}(1,1,1)$, then a basis for $M_{e}$ is given by $v^{e}_{1} = (1,0,-1)$ and $v^{e}_{2} = (0,1,-1)$.  Since $v_{1}$ and $v_{2}$ are not orthogonal, $Q_{e}$ is a rhombus rather than a square.  

For convenience, we make the following definition:

	\begin{definition}  The quotient space $\mathbb{T}^{d - 1}_{e} \oplus_{e} \mathbb{R}$ is defined by $\mathbb{T}^{d -1}_{e} \oplus_{e} \mathbb{R} = \mathbb{R}^{d}/M_{e}$, where the quotient is intended in the algebraic sense.  \end{definition}  
	
By what came before, $\mathbb{T}^{d -1}_{e} \oplus_{e} \mathbb{R}$ is in bijective correspondence with $Q_{e} \oplus_{e} \mathbb{R}$ and we will regard (e.g.\ measurable) functions on it as functions in $\mathbb{R}^{d}$ satisfying \eqref{E: periodicity}.

	Now we are prepared to discuss the correspondence between functions on $\mathbb{R} \times \mathbb{T}^{d}$ and those in $\mathbb{R}^{d}$ when $e \in \mathbb{R} \mathbb{Z}^{d}$.  Here there is an additional element of periodicity, this time in the $e$ direction rather than orthogonal to it.  Namely, Proposition \ref{P: algebra} implies there is an $m_{e} > 0$ such that 
		\begin{equation*}
			\{\langle k,e \rangle \, \mid \, k \in \mathbb{Z}^{d}\} = \{\ell m_{e} \, \mid \, \ell \in \mathbb{Z}\}.
		\end{equation*}
	In particular, $m_{e}$ is characterized by the formula
		\begin{equation} \label{E: rational period}
			m_{e} = \inf \left\{ \langle k, e \rangle \, \mid \, k \in \mathbb{Z}^{d} \right\} \cap (0,\infty).
		\end{equation}
		
	\begin{prop} \label{P: continuity_rational_directions}  	If $e \in \mathbb{R} \mathbb{Z}^{d}$ and $U : \mathbb{R} \times \mathbb{T}^{d} \to \mathbb{R}$, then the functions $\{u_{\zeta}\}_{\zeta \in \mathbb{R}}$ generated by $U$ in the $e$ direction descend to functions on $\mathbb{T}^{d - 1}_{e} \oplus_{e} \mathbb{R}$ and $u_{\zeta}(\cdot - k) = u_{\zeta + \langle k, e \rangle}$ if $k \in \mathbb{Z}^{d}$.  In particular, the map $\zeta \mapsto u_{\zeta}$ is periodic modulo translations with period $m_{e}$. 
	\end{prop}  

\subsection{Quasi-periodicity when $e \notin \mathbb{R} \mathbb{Z}^{d}$}  When $e$ is irrational, the transformation between a function $U$ and the functions $\{u_{\zeta}\}_{\zeta \in \mathbb{R}}$ it generates in that direction is harder to visualize, but enjoys some analytical advantages, as we will see.

An important analytical property of the transformation $U \mapsto \{u_{\zeta}\}_{\zeta \in \mathbb{R}}$ in the irrational case is stated in the next result.

	\begin{prop} \label{P: continuity_irrational_directions}  Suppose $U$ is a real-valued function in $\mathbb{R} \times \mathbb{T}^{d}$, $e \in S^{d -1} \setminus \mathbb{R} \mathbb{Z}^{d}$.  Let $\{u_{\zeta}\}_{\zeta \in \mathbb{R}}$ be the family of functions generated by $U$ in the $e$ direction.  The following are equivalent:
		\begin{itemize}
			\item[(i)] $U \in BUC(\mathbb{R} \times \mathbb{T}^{d})$ (resp.\ $U \in BC(\mathbb{R} \times \mathbb{T}^{d})$).
			\item[(ii)] $\{u_{\zeta}\}_{\zeta \in \mathbb{R}} \subseteq BUC(\mathbb{R}^{d})$ (resp.\ $BC(\mathbb{R}^{d})$) and, for each $(k_{n})_{n \in \mathbb{N}} \subseteq \mathbb{Z}^{d}$ and $\gamma \in \mathbb{R}$, 
				\begin{equation*}
					\text{if} \quad \gamma = \lim_{n \to \infty} \langle k_{n},e \rangle, \quad \text{then} \quad u_{\zeta - \gamma} = \lim_{n \to \infty} u_{\zeta}(\cdot + k_{n})
				\end{equation*} 
			uniformly (resp.\ locally uniformly) in $\mathbb{R}^{d}$.
	\end{itemize}
Conversely, if $\{u_{\zeta}\}_{\zeta \in \mathbb{R}}$ is any family of functions satisfying (ii), then the function $U(s,x) = u_{\langle x,e \rangle - s}(x)$ is well-defined in $\mathbb{R} \times \mathbb{T}^{d}$, generates $\{u_{\zeta}\}_{\zeta \in \mathbb{R}}$, and satisfies (i). \end{prop} 

\subsection{Transformation Properties of the Lebesgue Measure}  The previous correspondence between special classes of quasi-periodic functions in $\mathbb{R}^{d}$ and continuous functions in $\mathbb{R} \times \mathbb{T}^{d}$ is not very promising since we are studying an integral functional on $\mathbb{R} \times \mathbb{T}^{d}$.  Fortunately, the Lebesgue measure on $\mathbb{R} \times \mathbb{T}^{d}$ enjoys very nice transformation properties of its own.

\begin{theorem} \label{T: ergodic_lemma}  Let $U \in L^{1}(\mathbb{R} \times \mathbb{T}^{d})$ and $e \in S^{d - 1}$.  If $\{u_{\zeta}\}_{\zeta \in \mathbb{R}}$ is the family of functions generated by $U$ in the direction $e$, then
	\begin{equation*}
		\int_{\mathbb{R} \times \mathbb{T}^{d}} U(s,x) \, ds \, dx = \lim_{T \to \infty} T^{-1} \int_{0}^{T} \left( \lim_{R \to \infty} R^{1 - d} \int_{Q^{e}(0,R) \oplus_{e} \mathbb{R}} u_{\zeta}(x) \, dx \right) \, d\zeta.
	\end{equation*} 
Moreover, if $e \in \mathbb{R} \mathbb{Z}^{d}$ and $m_{e} > 0$ is given by \eqref{E: rational period}, then 
	\begin{equation*}
		\int_{\mathbb{R} \times \mathbb{T}^{d}} U(s,x) \, ds \, dx = m_{e}^{-1} \mathcal{H}^{d - 1}(Q_{e})^{-1} \int_{0}^{m_{e}} \int_{Q_{e} \oplus_{e} \mathbb{R}} u_{\zeta}(x) \, dx \, d \zeta.
	\end{equation*}
On the other hand, if $e \notin \mathbb{R} \mathbb{Z}^{d}$, then, for $\mathcal{L}^{1}$-a.e.\ $\zeta \in \mathbb{R}$,
	\begin{equation*}
		\int_{\mathbb{R} \times \mathbb{T}^{d}} U(s,x) \, ds \, dx = \lim_{R \to \infty} R^{1 - d} \int_{Q^{e}(0,R) \oplus_{e} \mathbb{R}} u_{\zeta}(x) \, dx.
	\end{equation*}
\end{theorem}  

The proof is given in Appendix \ref{A: transformation properties}.  

\subsection{Birkhoff Property} \label{S: birkhoff}  The Birkhoff property is fundamental in the study of plane-like minimizers of \eqref{E: functional}.  Let us define it now and then state and prove a lemma that can be used to explain why at the very least some plane-like minimizers of $\mathcal{F}^{a}_{1}$ have the property, even if the strong maximum principle is not available.  (The lemma is used in Section \ref{S: monotonicity_constraint}.).

\begin{definition} \label{D: birkhoff} Given $M \in \mathbb{N}$, a function $v : \mathbb{R}^{d} \to \mathbb{R}$ is said to possess the \emph{$M \mathbb{Z}^{d}$-Birkhoff property} with respect to the direction $e \in S^{d -1}$ provided $v(x + k) \geq v(x)$ whenever $k \in M \mathbb{Z}^{d}$ and $\langle k, e \rangle \geq 0$.  \end{definition}    

The next lemma identifies a natural class of functions with the Birkhoff property.

\begin{lemma} \label{L: asymptotic_birkhoff}  Suppose $e \in S^{d-1} \cap \mathbb{R} \mathbb{Z}^{d}$ and $N \in \mathbb{N}$.  If $v : \mathbb{T}^{d - 1}_{e} \oplus \mathbb{R} \to [-1,1]$ satisfies $v(x) = 1$ for $\langle x, e \rangle \geq N$ and $v(x) = -1$ for $\langle x, e \rangle \leq 0$, then there is an $M \in \mathbb{N}$ such that $v$ is $M\mathbb{Z}^{d}$-Birkhoff with respect to $e$.  \end{lemma}  

\begin{proof}  First, we claim that if $k \in \mathbb{Z}^{d}$ and $\langle k, e \rangle \geq N$, then $v(x + k) \geq v(x)$ for each $x \in \mathbb{R}^{d}$.  Indeed, if $\langle x, e \rangle > 0$, then $\langle x + k, e \rangle > \langle k, e \rangle \geq N$ so $v(x + k) = 1 \geq v(x)$.  On the other hand, if $\langle x, e \rangle \leq 0$, then $v(x) = -1 \leq v(x + k)$.   
	
	Next, recall that Proposition \ref{P: algebra}, (iii) implies that $\{\langle k, e \rangle \, \mid \, k \in \mathbb{Z}^{d}\} = \{\ell m_{e} \, \mid \, \ell \in \mathbb{Z}\}$ with $m_{e}$ given by \eqref{E: rational period}.  Fix a $k_{0} \in \mathbb{Z}^{d}$ such that $\langle k_{0},e \rangle = m_{e}$.  
	
	We claim that $M = \lceil \frac{N}{m_{e}} \rceil$ has the desired properties.  Indeed, by the definition of $m_{e}$ and the choice of $M$, 
		\begin{equation*}
			\inf (\left\{ \langle k, e \rangle \, \mid \, k \in M \mathbb{Z}^{d} \right\} \cap (0,\infty)) = \inf \{ M \ell m_{0} \, \mid \, \ell \in \mathbb{N}\} = M m_{e} \geq N.
		\end{equation*}
Thus, if $k \in M \mathbb{Z}^{d}$ and $\langle k, e \rangle > 0$, then $\langle k, e \rangle \geq N$ and our previous work implies
	\begin{equation*}
		v(x + k) \geq v(x) \quad \text{if} \, \, x \in \mathbb{R}^{d}.
	\end{equation*}
At the same time, if $k \in M \mathbb{Z}^{d}$ and $\langle k, e \rangle < 0$, then $-k \in M \mathbb{Z}^{d}$, $\langle -k, Ne \rangle > 0$, and $x = (x + k) - k$ so $v \geq v(\cdot + k)$ follows from the previous computation.  Finally, if $k \in M \mathbb{Z}^{d}$ and $\langle k, e \rangle = 0$, then $k \in \mathbb{Z}^{d} \cap \langle e \rangle^{\perp} = M_{e}$ and $v$ is a function on $\mathbb{T}^{d - 1}_{e} \oplus \mathbb{R}$ so $v(\cdot + k) = v$.  We conclude that $v$ is $M \mathbb{Z}^{d}$-Birkhoff with respect to $e$. \qed        \end{proof}

\section{Existence of Minimizers of $\mathscr{T}^{a}_{e}$} \label{S: existence}

Here we prove the existence of minimizers of the variational principle $\mathscr{E}^{a}(e)$ introduced in Section \ref{S: calculus of variations results}.  The main result of this section is

	\begin{theorem} \label{T: existence theorem main} For each $e \in S^{d - 1}$, there is a $U \in L^{\infty}(\mathbb{R} \times \mathbb{T}^{d})$ satisfying $|U| \leq 1$, $\partial_{s} U \geq 0$, and $\lim_{s \to \pm \infty} T_{-s}U = \pm 1$ such that
		\begin{equation*}
			\mathscr{T}^{a}_{e}(U) = \mathscr{E}^{a}(e).
		\end{equation*}
	\end{theorem}  
	
The proof proceeds in three steps.  In the first step, we prove that a minimizer exists subject to the monotonicity constraint $\partial_{s} U \geq 0$.  Next, we show that there is no ``symmetry breaking," that is, that the value of $\mathscr{E}^{a}(e)$ is unchanged if we replace $\mathbb{Z}^{d}$ by $M \mathbb{Z}^{d}$ for some $M \in \mathbb{N}$.  In the last step, we observe, using the unbroken symmetry and Lemma \ref{L: asymptotic_birkhoff}, that unconstrained candidates do no better than constrained ones. 

As usual, we will be interested in functions connecting $1$ and $-1$ at either end of the cylinder.  Precisely, we will study functions $U$ in $\mathbb{R} \times \mathbb{T}^{d}$ such that
	\begin{equation} \label{E: asymptotic}
		T_{-s}U\to \pm 1 \quad \text{in} \, \, L^{1}_{\text{loc}}(\mathbb{R} \times \mathbb{T}^{d}) \quad \text{as} \ \, s \to \pm \infty.
	\end{equation}

Here are the function spaces we will use in the sequel:
	\begin{align*}
		\tilde{\mathscr{X}} &= \left\{U \in  L^{\infty}(\mathbb{R} \times \mathbb{T}^{d}) \, \mid \ U \, \, \text{satisfies} \, \, \eqref{E: asymptotic}\right\}, \\
		\mathscr{X} &= \left\{U \in \tilde{\mathscr{X}} \, \mid \, -1 \leq U \leq 1\right\}, \\
		\mathscr{X}_{+} &= \left\{U \in \mathscr{X} \, \mid \, \partial_{s} U \geq 0 \right\}, \\
		C^{\infty}_{\text{sgn}}(\mathbb{R} \times \mathbb{T}^{d}) &= \{U \in C^{\infty}(\mathbb{R} \times \mathbb{T}^{d}) \, \mid \, \exists M_{U} > 0 \, \, \text{such that} \\
			&\qquad \text{sgn}(s)U(s,x) = 1 \, \, \text{if} \, \, |s| \geq M_{U}\}.
	\end{align*} 
	
\subsection{Properties of Functions in $\mathscr{X}_{+}$} \label{S: properties}

Our first objective is to find minimizers in the constrained set $\mathscr{X}_{+}$.  In this section, we state properties of $\mathscr{X}_{+}$ that will aid us in this task.  Many of the proofs only require elementary real analysis arguments and are therefore ommited.

The first fact we will use is that a function in $\mathscr{X}_{+}$ can be studied line-by-line.

\begin{prop} \label{P: key_monotonicity}  If $U$ is a measurable function on $\mathbb{R} \times \mathbb{T}^{d}$ such that $\partial_{s} U \geq 0$ and $|U| \leq 1$ a.e., then there is a $G_{U} \subseteq \mathbb{T}^{d}$ satisfying $\mathcal{L}^{d}(G_{U}) = 1$ with the following property: if $x \in G_{U}$, then there is a unique non-decreasing, left-continuous function $U_{x} : \mathbb{R} \to [-1,1]$ such that $U(s,x) = U_{x}(s)$ for a.e.\ $s \in \mathbb{R}$.
\end{prop}  

Since the functions $\{U_{x}\}_{x \in G_{U}}$ are bounded and non-decreasing, they have limits at infinity.  Henceforth, we define $U_{x}^{+}$ and $U_{x}^{-}$ by 
	\begin{equation*}
		U_{x}^{+} = \lim_{s \to \infty} U_{x}(s), \quad U_{x}^{-} = \lim_{s \to -\infty} U_{x}(s).
	\end{equation*}

\begin{prop} \label{P: important_recovery} (i) If $U$ satisfies the hypotheses of Proposition \ref{P: key_monotonicity}, then we have $\int_{\mathbb{R} \times \mathbb{T}^{d}} \partial_{s} U \, dx \, ds \leq 2$.  If, in addition, $\mathscr{T}^{a}_{e}(U) < \infty$, then $U_{x}^{+} = 1$ and $U_{x}^{-} = -1$ for a.e.\ $x \in G_{U}$.  

(ii) If $U \in \mathscr{X}_{+}$, then $\int_{\mathbb{R} \times \mathbb{T}^{d}} \partial_{s}U(s,x) \, dx \, ds = 2$ and $U_{x}^{\pm} = \pm 1$ a.e.\ in $G_{U}$.
\end{prop}

In addition to the lack of coercivity in the $s$ variable, $\mathscr{T}^{a}_{e}$ has another degeneracy: it is invariant under translations in the $s$ variable.  This is not hard to correct, however.  Toward that end, we use

\begin{prop} \label{P: existence_average_magnetization}  If $U$ is a measurable function on $\mathbb{R} \times \mathbb{T}^{d}$ such that $|U| \leq 1$ a.e.\ and $\partial_{s} U \geq 0$, then there is a unique non-decreasing, left-continuous function $\psi_{U} : \mathbb{R} \to [-1,1]$ such that $\psi_{U}(s) = \int_{\mathbb{T}^{d}} U(s,x) \, dx$ for a.e.\ $s \in \mathbb{R}$.  \end{prop}

We will find minimizers of $\mathscr{T}^{a}_{e}$ in $\mathscr{X}_{+}$ using the direct method.  Therefore, we will want to know that sequences in $\mathscr{X}_{+}$ with uniformly bounded energy do not escape $\mathscr{X}_{+}$.  A first step in that direction is the following easy characterization of $\mathscr{X}_{+}$:

\begin{prop} \label{P: asymptotic_magnetization}  If $U$ is a measurable function on $\mathbb{R} \times \mathbb{T}^{d}$ such that $|U| \leq 1$ a.e.\ and $\partial_{s} U \geq 0$, then the following are equivalent:
	\begin{itemize}
		\item[(i)] $U \in \mathscr{X}_{+}$.
		\item[(ii)] The function $\psi_{U}$ defined in Proposition \ref{P: existence_average_magnetization} satisfies $\lim_{s \to \pm \infty} \psi_{U}(s) = \pm 1$.
		\item[(iii)] $\lim_{R \to \pm \infty} \int_{[R,R + 1] \times \mathbb{T}^{d}} U \, dx \, ds = \pm 1$.
	\end{itemize}
\end{prop}  

The next result will be used to guarantee that sequences with bounded energy do not escape $\mathscr{X}_{+}$.

\begin{prop} \label{P: fun}  If $U$ is a measurable function on $\mathbb{R} \times \mathbb{T}^{d}$ such that
	\begin{itemize}
		\item[(i)] $|U| \leq 1$ a.e.,
		\item[(ii)] $\partial_{s} U \geq 0$,
		\item[(iii)] $\mathscr{T}^{a}_{e}(U) < \infty$,
		\item[(iv)] $\psi_{U}$ is non-constant or there is an $s \in \mathbb{R}$ such that $\psi(s) \in (-1,1)$,
	\end{itemize}
then, for a.e.\ $x \in G_{U}$, $U_{x}^{\pm} = \lim_{s \to \pm \infty} U_{x}(s) = \pm 1$.  In particular, $U \in \mathscr{X}_{+}$.  
\end{prop}  

	The ideas used in the proof are already present in \cite{rabinowitz stredulinsky}.  

	\begin{proof}  Define $U^{+}$ and $U^{-}$ on $\mathbb{R} \times G_{U}$ by $U^{+}(s,x) = U^{+}_{x}$ and $U^{-}(s,x) = U^{-}_{x}$.  Since $\mathcal{L}^{d}(\mathbb{T}^{d} \setminus G_{U}) = 0$, we can consider $U^{+}$ and $U^{-}$ as measurable functions in $\mathbb{R} \times \mathbb{T}^{d}$.  
	
	  A straightforward application of Fubini's Theorem implies that $\partial_{s} U^{\pm} = 0$.  
	  
	  Further, lower semi-continuity implies $\mathscr{T}^{a}_{e}(U^{\pm}) < \infty$. In fact, $\mathscr{T}^{a}_{e}(U^{\pm}) = 0$.  To see this, notice that if $R > 0$, then
	  	\begin{align*}
			\mathscr{T}^{a}_{e}(U^{\pm}; [-R,R] \times \mathbb{T}^{d}) &\leq \liminf_{n \to \infty} \mathscr{T}^{a}_{e}(T_{\mp n}U; [-R,R] \times \mathbb{T}^{d}) \\
				&= \lim_{n \to \infty} \mathscr{T}^{a}_{e}(U; [-R \pm n, R \pm n] \times \mathbb{T}^{d}) = 0.
		\end{align*}
	As a consequence, $\mathcal{D}_{e}U^{\pm} = 0$ a.e.\ in $\mathbb{R} \times \mathbb{T}^{d}$.  
	
	Since $(\partial_{s} U^{\pm}, \mathcal{D}_{e}U^{\pm}) = 0$ a.e.\ in $\mathbb{R} \times \mathbb{T}^{d}$, it follows that $U^{+}$ and $U^{-}$ are almost everywhere constant in $\mathbb{R} \times \mathbb{T}^{d}$.  From the inequality $U^{-} \leq U^{+}$, we deduce that either $U^{+} \equiv U^{-} \equiv 1$, $U^{+} \equiv U^{-} \equiv -1$, or $U^{-} \equiv -1$ and $U^{+} \equiv 1$.    
	
If $U^{+} \equiv U^{-} \equiv 1$ a.e., then $U_{x} \equiv 1$ in $\mathbb{R}$ for a.e.\ $x \in G_{U}$.  This is impossible, however, as it contradicts assumption (iv).  We similarly arrive at a contradiction if both functions are identically $-1$.  Therefore, we are left to conclude that $U^{+} \equiv 1$ and $U^{-} \equiv -1$ a.e.\ in $\mathbb{R} \times \mathbb{T}^{d}$. \qed \end{proof}

\subsection{Minimizers in $\mathscr{X}_{+}$}  The main result of this section concerns the existence of minimizers in $\mathscr{X}_{+}$.
	
\begin{prop} \label{P: existence}  For each $e \in S^{d - 1}$, there is a $U_{e} \in \mathscr{X}_{+}$ such that 
	\begin{equation*}
		\mathscr{T}_{e}^{a}(U_{e}) = \inf \left\{ \mathscr{T}^{a}_{e}(U) \, \mid \, U \in \mathscr{X}_{+} \right\} =: \mathscr{E}^{a}_{+}(e).
	\end{equation*}
\end{prop}  

To prove this, we start with preliminary regularity estimates.  Once this is done, we can apply the direct method.  

The first result shows $\psi_{U}$ inherits regularity from $U$:

\begin{prop} \label{P: apriori} Suppose $e \in S^{d - 1}$ and $U \in \mathscr{X}_{+}$ satisfies $\mathscr{T}^{a}_{e}(U) < \infty$.  Then the function $\psi_{U}$ of Proposition \ref{P: existence_average_magnetization} satisfies
\begin{equation} \label{E: uniform_holder}
|\psi_{U}(s) - \psi_{U}(t)| \leq \sqrt{2} \lambda^{-\frac{1}{2}} \sqrt{\mathscr{T}^{a}_{e}(U)} |s - t|^{\frac{1}{2}} \quad \text{for each} \, \, s, t \in \mathbb{R}.
\end{equation}
\end{prop}

\begin{proof}  Let $\psi_{U}$ be the non-decreasing, left-continuous function defined in Proposition \ref{P: existence_average_magnetization}.  Assume that $U$ is smooth; otherwise, apply Proposition \ref{P: smooth_approx} of Appendix \ref{A: elementary_properties}.  

Given $s, t \in \mathbb{R}$, the identity $\int_{\mathbb{T}^{d}} D_{x}U(\cdot,x) \, dx = 0$ and the Cauchy-Schwarz inequality give
\begin{align*}
|\psi_{U}(t) - \psi_{U}(s)| &= \left| \int_{s}^{t} \int_{\mathbb{T}^{d}} \langle e \partial_{s} U(r,x) + D_{x}U(r,x), e \rangle \, dx \, dr \right| \\
		&\leq \sqrt{2} \lambda^{-\frac{1}{2}} \sqrt{\mathscr{T}^{a}_{e}(U)} |s - t|^{\frac{1}{2}}.
\end{align*} \qed
\end{proof}  

The previous a priori estimate will be useful in the sequel.  In addition, the following $BV$ estimate is crucial.  First, we introduce a convenient notation.  If $e \in S^{d - 1}$, we define a norm $|\cdot|_{e} : \mathbb{R} \times \mathbb{R}^{d} \to [0,\infty)$ by
\begin{equation*}
|(q,p)|_{e} = \sqrt{q^{2} + \|qe + p\|^{2}}.
\end{equation*}
Notice that since $\|p\|^{2} \leq 2(\|p + qe\|^{2} + q^{2})$, the following inequality holds:
	\begin{equation*}
		\|(q,p)\| = \sqrt{q^{2} + \|p\|^{2}} \leq \sqrt{3} |(q,p)|_{e}.
	\end{equation*}

\begin{prop} \label{P: bv_estimate} If $U \in L^{\infty}(\mathbb{R} \times \mathbb{T}^{d})$ satisfies $|U| \leq 1$ a.e., $\partial_{s} U \geq 0$, and $\mathscr{T}^{a}_{e}(U) < \infty$, then $U \in \text{BV}_{\text{loc}}(\mathbb{R} \times \mathbb{T}^{d})$.  Specifically, if $s < t$, then
\begin{equation*}
TV(U; (s,t) \times \mathbb{T}^{d}) \leq \sqrt{3}\left(2 + \sqrt{2} \lambda^{-\frac{1}{2}} |s - t|^{\frac{1}{2}} \sqrt{\mathscr{T}^{a}_{e}(U)}\right).
\end{equation*} \end{prop}  

Proposition \ref{P: bv_estimate} is inspired by arguments appearing in \cite{ducrot}.  The same idea also appears in \cite{bessi}.  

\begin{proof}  This is a direct computation.  We will assume that $U$ is smooth.  The general case follows by approximation, as in the last proof.  The key fact we need is $\int_{\mathbb{R} \times \mathbb{T}^{d}} \partial_{s} U(s,x) \, ds \leq 2$, which was observed already in Proposition \ref{P: important_recovery}.  From it and the triangle inequality, we obtain
\begin{align*}
\int_{s}^{t} \int_{\mathbb{T}^{d}} |(\partial_{s}U,D_{x}U)|_{e} \, dx \, ds &\leq \int_{s}^{t} \int_{\mathbb{T}^{d}} \left( \partial_{s}U(s,x) + \|\mathcal{D}_{e}U(s,x) \| \right) \, dx \, ds \\
	&\leq 2 + \sqrt{2} \lambda^{-\frac{1}{2}} \sqrt{\mathscr{T}^{a}_{e}(U)} |s - t|^{\frac{1}{2}}.
	\end{align*}
Since $\|(\partial_{s}U,D_{x}U)\| \leq \sqrt{3} |(\partial_{s}U,D_{x}U)|_{e}$ pointwise, we conclude by appealing to the definition of $TV(U; (s,t) \times \mathbb{T}^{d})$.  \qed
\end{proof}  

Putting together Propositions \ref{P: apriori} and \ref{P: bv_estimate}, we obtain the desired compactness result:

\begin{prop} \label{P: compactness}  Suppose $\{e\}, (e_{n})_{n \in \mathbb{N}} \subseteq S^{d - 1}$ and $(U_{n})_{n \in \mathbb{N}} \subseteq \mathscr{X}_{+}$ satisfy
\begin{itemize}
\item[(i)] $\mathscr{E} := \sup \left\{\mathscr{T}^{a}_{e_{n}}(U_{n}) \, \mid \, n \in \mathbb{N} \right\} < \infty$,
\item[(ii)] $\psi_{U_{n}}(0) = 0$ independently of $n \in \mathbb{N}$,
\item[(iii)] $e = \lim_{n \to \infty} e_{n}$.
\end{itemize}  
If $(n_{k})_{k \in \mathbb{N}} \subseteq \mathbb{N}$ is any subsequence, then there is a $U \in \mathscr{X}_{+}$ and a further subsequence $(n_{k_{j}})_{j \in \mathbb{N}}$ such that $U = \lim_{j \to \infty} U_{n_{k_{j}}}$ pointwise a.e.\ in $\mathbb{R} \times \mathbb{T}^{d}$.  Moreover, $U$ satisfies
\begin{itemize}
\item[(iv)]  $\partial_{s} U \geq 0$, $\mathcal{D}_{e}U \in L^{2}(\mathbb{R} \times \mathbb{T}^{d})$,
\item[(v)] $\psi_{U}(0) = 0$,
\item[(vi)]  $\mathscr{T}^{a}_{e}(U) < \infty$, and, in particular,
\begin{equation*}
\mathscr{T}^{a}_{e}(U) \leq \liminf_{j \to \infty} \mathscr{T}^{a}_{e_{n_{k_{j}}}}(U_{n_{k_{j}}}).
\end{equation*}    
\end{itemize}
\end{prop}  

\begin{proof}  Fix a sub-sequence $(n_{k})_{k \in \mathbb{N}} \subseteq \mathbb{N}$.  By Proposition \ref{P: bv_estimate} and the compactness of $BV$ in $L^{1}$ in bounded domains, there is a function $U \in \text{BV}_{\text{loc}}(\mathbb{R} \times \mathbb{T}^{d})$ and a sub-sequence $(n_{k_{j}})_{j \in \mathbb{N}}$ such that $\lim_{j \to \infty} U_{n_{k_{j}}} = U$ in $L^{1}_{\text{loc}}(\mathbb{R} \times \mathbb{T}^{d})$ and pointwise a.e.  Evidently $- 1 \leq U \leq 1$ almost everywhere and $\partial_{s} U \geq 0$.  

Since $\mathscr{T}^{a}_{e}$ controls the $L^{2}$-norm of $\mathcal{D}_{e}$, it is clear that $\mathcal{D}_{e} U \in L^{2}(\mathbb{R} \times \mathbb{T}^{d})$.  The inequality in (vi) then follows from lower semi-continuity.

Now we verify that $U$ satisfies (v).  Invoking Proposition \ref{P: apriori} and assumption (ii) and passing to a further subsequence if necessary, there is a non-decreasing, continuous function $\psi : \mathbb{R} \to [-1,1]$ such that $\psi_{U_{n}} \to \psi$ locally uniformly.  Since $U_{n_{k_{j}}} \to U$ in $L^{1}_{\text{loc}}(\mathbb{R} \times \mathbb{T}^{d})$, Fubini's Theorem shows that $\psi = \psi_{U}$ and, thus, (v) holds. 

We appeal to Proposition \ref{P: fun} to conclude that $U \in \mathscr{X}_{+}$.  \qed \end{proof}

Finally, in the proof of existence, we will use the following observation:

\begin{prop} \label{P: s_symmetry}  If $V \in \mathscr{X}_{+}$ and $s_{0} \in \mathbb{R}$, then $T_{s_{0}}V \in \mathscr{X}_{+}$ and $\mathscr{T}^{a}_{e}(T_{s_{0}}V) = \mathscr{T}^{a}_{e}(V)$.  \end{prop}  

Now we have all the ingredients necessary to apply the direct method and obtain minimizers in $\mathscr{X}_{+}$.

\begin{proof}[Proof of Proposition \ref{P: existence}]  Let $(U_{n})_{n \in \mathbb{N}} \subseteq \mathscr{X}_{+}$ be such that 
\begin{equation} \label{E: minimizing}
\mathscr{E}^{a}_{+}(e) = \lim_{n \to \infty} \mathscr{T}^{a}_{e}(U_{n}).
\end{equation}
In view of Propositions \ref{P: apriori} and \ref{P: s_symmetry}, there is no loss of generality if we assume that $\psi_{U_{n}}(0) = 0$ for all $n \in \mathbb{N}$.  In particular, this assumption implies that the hypotheses of Proposition \ref{P: compactness} all hold.

By that result, there is a $U \in \mathscr{X}_{+}$ such that $\mathscr{T}^{a}_{e}(U) \leq \lim_{n \to \infty} \mathscr{T}^{a}_{e}(U_{n}) = \mathscr{E}^{a}_{+}(e)$.  Clearly, $\mathscr{T}^{a}_{e}(U) \geq \mathscr{E}^{a}_{+}(e)$ so it is a minimizer. \qed \end{proof}  

\subsection{No Symmetry Breaking} \label{S: symmetry breaking}

In the next step, we prove that the constrained energy is unchanged if we replace $\mathbb{Z}^{d}$ by $M\mathbb{Z}^{d}$.  The proof we give here uses a very weak form of the maximum principle that allows us to continue the analysis in spite of the low regularity assumptions on $a$ and $W$.  

Given an $M \in \mathbb{N}$, we study the problem in $\mathbb{R} \times M \mathbb{Z}^{d}$ using the following definitions:
	\begin{gather*}
		\mathscr{X}^{(M)}_{+} = \left\{U \in L^{\infty}(\mathbb{R} \times M \mathbb{T}^{d}) \, \mid \, |U| \leq 1, \, \, \partial_{s} U \geq 0, \, \, U \, \, \text{satisfies} \, \, \eqref{E: asymptotic} \right\}, \\
		\mathscr{T}^{a}_{e,M}(U) = M^{-d} \int_{\mathbb{R} \times M \mathbb{T}^{d}} \left(\frac{1}{2} \langle a(x) \mathcal{D}_{e}U, \mathcal{D}_{e}U \rangle + W(U) \right) \, ds \, dx, \\
		\mathscr{E}^{a}_{+,M}(e) = \inf \left\{ \mathscr{T}^{a}_{e,M}(U) \, \mid \, U \in \mathscr{X}^{(M)}_{+} \right\}, \\
		\mathcal{M}_{e}(\mathbb{R} \times M \mathbb{T}^{d}) = \left\{ U \in \mathscr{X}_{+}^{(M)} \, \mid \, \mathscr{T}^{a}_{e,M}(U) = \mathscr{E}^{a}_{M}(e) \right\}.
	\end{gather*}

\begin{theorem} \label{T: no_symmetry_breaking}  For all $M \in \mathbb{N}$ and $e \in S^{d-1}$, we have $\mathscr{E}^{a}_{+,M}(e) = \mathscr{E}^{a}_{+}(e)$.  \end{theorem}  

Here the inequality $\mathscr{E}^{a}_{+,M}(e) \leq \mathscr{E}^{a}_{+}(e)$ is immediate.  Our proof of the opposite bound $\mathscr{E}^{a}_{+,M}(e) \geq \mathscr{E}^{a}_{+}(e)$ is inspired by the approach of \cite{caffarelli de la llave}.    

We begin with a familiar lemma.

\begin{lemma} \label{L: submodularity} Given $e \in S^{d-1}$ and $M \in \mathbb{N}$, if $U, V \in \mathcal{M}_{e}(\mathbb{R} \times M\mathbb{T}^{d})$, then 
	\begin{equation*}
		U \wedge V, \, \, U \vee V \in \mathcal{M}_{e}(\mathbb{R} \times M\mathbb{T}^{d}).
	\end{equation*}  
\end{lemma}  

Since the lemma follows as in the classical, uniformly elliptic case, we leave the details to the interested reader.  With that result in hand, we're prepared for the

\begin{proof}[Proof of Theorem \ref{T: no_symmetry_breaking}]  First, we prove $\mathscr{E}^{a}_{+,M}(e) \leq \mathscr{E}^{a}_{+}(e)$.  Suppose $U_{e} \in \mathcal{M}_{e}(\mathbb{R} \times \mathbb{T}^{d})$.  Since $\mathscr{X}_{+}$ naturally includes into $\mathscr{X}_{+}^{(M)}$, we can consider $U_{e}$ as an element of $\mathscr{X}_{+}^{(M)}$.  In so doing, we find
	\begin{align*}
		&\mathscr{E}^{a}_{+,M}(e) \leq M^{-d} \int_{\mathbb{R} \times M\mathbb{T}^{d}} \left( \frac{1}{2} \langle a(x) \mathcal{D}_{e}U_{e}, \mathcal{D}_{e} U_{e} \rangle + W(U_{e}) \right) \, dx \, ds \\
			\quad &= M^{-d} \sum_{k \in \mathbb{Z}^{d} \cap [0,M)^{d}} \int_{\mathbb{R} \times (\mathbb{T}^{d} + k)} \left( \frac{1}{2} \langle a(x) \mathcal{D}_{e}U_{e}, \mathcal{D}_{e} U_{e} \rangle + W(U_{e}) \right) \, dx \, ds \\
			\quad &= M^{-d} \sum_{k \in \mathbb{Z}^{d} \cap [0,M)^{d}} \int_{\mathbb{R} \times \mathbb{T}^{d}} \left( \frac{1}{2} \langle a(x) \mathcal{D}_{e}U_{e}, \mathcal{D}_{e} U_{e} \rangle + W(U_{e}) \right) \, dx \, ds = \mathscr{E}^{a}_{+}(e).
	\end{align*}

Next, we prove $\mathscr{E}^{a}_{+,M}(e) \geq \mathscr{E}^{a}_{+}(e)$.  Fix $\tilde{U}_{e} \in \mathcal{M}_{e}(\mathbb{R} \times M \mathbb{T}^{d})$.  Since $a$ is $\mathbb{Z}^{d}$-periodic, a quick computation shows that $S_{m}\tilde{U}_{e} \in \mathcal{M}_{e}(\mathbb{R} \times M\mathbb{T}^{d})$ if $m \in \mathbb{Z}^{d}$.  Therefore, by Lemma \ref{L: submodularity}, the function $U_{e}$ given by 
	\begin{align*}
		U_{e}(s,x) = \min \left\{S_{m}\tilde{U}_{e}(s,x) \, \mid \, m \in \mathbb{Z}^{d} \cap [0,M)^{d}\right\}
	\end{align*} 
is also in $\mathcal{M}_{e}(\mathbb{R} \times M \mathbb{T}^{d})$.  Another straightforward computation shows that $U_{e}$ is $\mathbb{Z}^{d}$-periodic and, in fact, $U_{e} \in \mathscr{X}_{+}$.  Thus, $\mathscr{T}^{a}_{e}(U_{e}) \geq \mathscr{E}^{a}_{+}(e)$.  Finally, we use the definition of $\mathscr{T}^{a}_{e,M}(U_{e})$ to compare $\mathscr{E}^{a}_{+,M}(e)$ to $\mathscr{E}^{a}_{+}(e)$:
	\begin{align*}
		\mathscr{E}^{a}_{+,M}(e) 
			&= M^{-d} \sum_{k \in \mathbb{Z}^{d} \cap [0,M)^{d}} \int_{\mathbb{R} \times (\mathbb{T}^{d}+ k)} \left(\frac{1}{2} \langle a(x) \mathcal{D}_{e}U_{e}, \mathcal{D}_{e} U_{e} \rangle + W(U_{e}) \right) \, dx \, ds \\
			&= M^{-d} \sum_{k \in \mathbb{Z}^{d} \cap [0,M)^{d}} \mathscr{T}^{a}_{e}(U_{e}) \geq \mathscr{E}^{a}(e).
	\end{align*}  \qed\end{proof}  
	
\begin{corollary} \label{C: minimizer_symmetry_unbroken} If $U_{e} \in \mathcal{M}_{e}(\mathbb{R} \times \mathbb{T}^{d})$ and $M \in \mathbb{N}$, then $U_{e} \in \mathcal{M}_{e}(\mathbb{R} \times M \mathbb{T}^{d})$. \end{corollary} 

\begin{proof}  This follows from the string of inequalities leading to $\mathscr{E}^{a}_{+,M}(e) \leq \mathscr{E}^{a}_{+}(e)$ and the fact that equality actually holds.  \qed \end{proof}  

\subsection{Removing the constraint}  \label{S: monotonicity_constraint}

We now show that the monotonicity constraint is superfluous.  First, we notice that, in rational directions, the functions generated by any candidate are close to $M\mathbb{Z}^{d}$-Birkhoff functions provided $M$ is sufficiently large.  This is an application of the next proposition and Lemma \ref{L: asymptotic_birkhoff}.  

\begin{prop} \label{P: approximation_energy_need} Let $e \in S^{d-1} \cap \mathbb{R} \mathbb{Z}^{d}$.  Suppose $u : \mathbb{T}^{d -1}_{e} \oplus_{e} \mathbb{R} \to [-1,1]$ satisfies the following two properties:
	\begin{itemize}
		\item[(a)] $\mathcal{F}^{a}_{1}(u; Q_{e} \oplus \mathbb{R}) < \infty$.
		\item[(b)] For each $\delta > 0$, 
			\begin{align}
				\lim_{R \to \infty} \mathcal{L}^{d}(\{u \leq 1 - \delta\} \cap \{R \leq \langle x,e \rangle \leq R + 1\}) = 0, \label{E: heteroclinic part 1}\\
				\lim_{R \to \infty} \mathcal{L}^{d}(\{u \geq -1 + \delta\} \cap \{-(R + 1) \leq \langle x, e \rangle \leq - R\}) = 0. \label{E: heteroclinic part 2}
			\end{align}
	\end{itemize}
Then, for each $\epsilon > 0$, there is a $u_{\epsilon} : \mathbb{T}^{d - 1}_{e} \oplus_{e} \mathbb{R} \to [-1,1]$ and an $N_{\epsilon} \in \mathbb{N}$ such that
	\begin{itemize}
		\item[(i)] $u_{\epsilon}(x) = 1$ if $\langle x, e \rangle \geq N_{\epsilon}$,
		\item[(ii)] $u_{\epsilon}(x) = -1$ if $\langle x, e \rangle \leq 0$,
		\item[(iii)] $\mathcal{F}^{a}_{1}(u_{\epsilon}; Q_{e} \oplus \mathbb{R}) \leq \mathcal{F}^{a}_{1}(u; Q_{e} \oplus \mathbb{R}) + \epsilon$.
	\end{itemize}
\end{prop}  

	\begin{proof}  This follows by arguing exactly as in Proposition \ref{P: cut_off} below.  \qed \end{proof}  
	
Next, we observe that $M \mathbb{Z}^{d}$-Birkhoff functions in $\mathbb{T}^{d-1}_{e} \oplus_{e} \mathbb{R}$ can always be generated by a function in $\mathbb{R} \times M \mathbb{Z}^{d}$ with the same energy.  More precisely, we have

\begin{prop} \label{P: extension}  Suppose $e \in \mathbb{R} \mathbb{Z}^{d}$.  If $v : \mathbb{T}^{d - 1}_{e} \oplus_{e} \mathbb{R} \to [-1,1]$ is $M \mathbb{Z}^{d}$-Birkhoff with respect to $e$ for some $M \in \mathbb{N}$ and $\lim_{s \to \infty} v(\cdot + se) = \pm 1$ in $L^{1}_{\text{loc}}(Q_{e} \oplus_{e} \mathbb{R})$, then there is a $V \in \mathscr{X}^{(M)}_{+}$ (see Section \ref{S: symmetry breaking}) such that 
	\begin{equation*}
		\mathscr{T}^{a}_{e,M}(V) = \mathcal{H}^{d - 1}(Q_{e})^{-1} \mathcal{F}^{a}_{1}(v; Q_{e} \oplus \mathbb{R}).
	\end{equation*}
\end{prop}  

	\begin{proof}  First, let $m_{e} > 0$ be the constant defined in \eqref{E: rational period} and choose a $k_{0} \in \mathbb{Z}^{d}$ such that $\langle k_{0}, e \rangle = m_{e}$.  Next, define a family of functions $\{v_{\zeta}\}_{\zeta \in \mathbb{R}}$ on $\mathbb{T}^{d -1}_{e} \oplus \mathbb{R}$ by 
		\begin{equation*}
			v_{\zeta}(x) = v(x + M\ell k_{0}) \quad \text{if} \, \, \ell = \left \lceil -\frac{\zeta}{Mm_{e}} \right \rceil.
		\end{equation*}
Finally, define $V$ in $\mathbb{R} \times \mathbb{R}^{d}$ by 
		\begin{equation*}
			V(s,x) = v_{\langle x,e \rangle - s}(x).
		\end{equation*}

The properties of $v$, the definition of $m_{e}$, and an application of Theorem \ref{T: ergodic_lemma} together show that $V$ has the desired properties. \qed \end{proof}

Finally, we prove that the unconstrained and constrained energies coincide.

\begin{prop} \label{P: remove constraint}  If $e \in S^{d - 1}$, then $\mathscr{E}^{a}(e) = \inf \left\{ \mathscr{T}^{a}_{e}(U) \, \mid \, U \in \mathscr{X}_{+} \right\}$.  \end{prop}

	\begin{proof}  Since $\mathscr{X}_{+} \subseteq \mathscr{X}$, the inequality $\mathscr{E}^{a}(e) \leq \inf  \left\{ \mathscr{T}^{a}_{e}(U) \, \mid \, U \in \mathscr{X}_{+} \right\}$ is immediate.  We prove the complementary inequality in two steps.
	
	\textbf{Step 1: $e \in \mathbb{R}\mathbb{Z}^{d}$}  
	
	First, we claim that if $u : \mathbb{T}^{d - 1}_{e} \oplus \mathbb{R} \to [-1,1]$ is such that \eqref{E: heteroclinic part 1} and \eqref{E: heteroclinic part 2} hold, then 
		\begin{equation} \label{E: key_birkhoff_finally}
			\mathscr{E}_{+}^{a}(e) \leq \mathcal{H}^{d - 1}(Q_{e})^{-1} \mathcal{F}_{1}^{a}(u; Q_{e} \oplus_{e} \mathbb{R}).
		\end{equation}
	Indeed, given $\epsilon > 0$, let $v = u_{\epsilon}$ be the function defined in Proposition \ref{P: approximation_energy_need}.  By Lemma \ref{L: asymptotic_birkhoff} and Proposition \ref{P: extension}, there is an $M \in \mathbb{N}$ and a $V \in \mathscr{X}^{(M)}_{+}$ such that
		\begin{equation*}
			\mathscr{T}^{a}_{e}(V) = \mathcal{H}^{d - 1}(Q_{e})^{-1} \mathcal{F}_{1}^{a}(v; Q_{e} \oplus_{e} \mathbb{R}).		
		\end{equation*}
Since $V \in \mathscr{X}^{(M)}_{+}$, this yields
		\begin{equation*}
			\mathscr{E}_{+}^{a}(e) - \epsilon \leq \mathcal{H}^{d - 1}(Q_{e})^{-1} \mathcal{F}^{a}_{1}(v; Q_{e} \oplus_{e} \mathbb{R}) - \epsilon < \mathcal{H}^{d - 1}(Q_{e})^{-1} \mathcal{F}^{a}_{1}(u; Q_{e} \oplus_{e} \mathbb{R}).
		\end{equation*}
Sending $\epsilon \to 0^{+}$ gives \eqref{E: key_birkhoff_finally}.  

Next, suppose that $U \in \mathscr{X}$.  Define $\{u_{\zeta}\}_{\zeta \in \mathbb{R}}$ on $\mathbb{T}^{d - 1}_{e} \oplus_{e} \mathbb{R}$ by $u_{\zeta}(x) = U(\langle x,e \rangle - \zeta,x)$. Since $U \in \mathscr{X}$, $u_{\zeta}$ satisfies \eqref{E: heteroclinic part 1} and \eqref{E: heteroclinic part 2} for a.e.\ $\zeta \in \mathbb{R}$.  Thus, using what we just proved and Theorem \ref{T: ergodic_lemma}, we find
	\begin{align*}
		\mathscr{T}^{a}_{e}(U) &= \frac{1}{m_{0}} \int_{0}^{m_{0}} \mathcal{H}^{d - 1}(Q_{e})^{-1} \mathcal{F}_{1}^{a}(u_{\zeta}; Q_{e} \oplus_{e} \mathbb{R}) \, d \zeta \\
			&\geq \frac{1}{m_{0}} \int_{0}^{m_{0}} \mathscr{E}^{a}_{+}(e) \, d \zeta = \mathscr{E}^{a}_{+}(e).
	\end{align*}    
We conclude that $\mathscr{E}^{a}(e) \geq \mathscr{E}^{a}_{+}(e)$.  

\textbf{Step 2: $e \in S^{d - 1} \setminus \mathbb{R} \mathbb{Z}^{d}$}  

First, we apply Proposition \ref{P: smooth_approximation} from Appendix \ref{A: approximation} to find that, for each $e' \in S^{d - 1}$,
	\begin{equation} \label{E: smoothing_approximation}
		\mathscr{E}^{a}(e') = \inf \left\{ \mathscr{T}^{a}_{e'}(U) \, \mid \, U \in \mathscr{X} \cap C^{\infty}_{\text{sgn}}(\mathbb{R} \times \mathbb{T}^{d}) \right\}.
	\end{equation}
Next, fix $U \in \mathscr{X} \cap C^{\infty}_{\text{sgn}}(\mathbb{R} \times \mathbb{T}^{d})$ and $(e_{n})_{n \in \mathbb{N}} \subseteq S^{d-1} \cap \mathbb{R} \mathbb{Z}^{d}$ such that $e = \lim_{n \to \infty} e_{n}$.  Given $n \in \mathbb{N}$, we have $\mathscr{E}^{a}_{+}(e_{n}) \leq \mathscr{T}^{a}_{e_{n}}(U)$ by the previous step.  Moreover, since $U \in C^{\infty}_{\text{sgn}}(\mathbb{R} \times \mathbb{T}^{d})$, it's easy to see that $e' \mapsto \mathscr{T}^{a}_{e'}(U)$ is continuous.  Arguing using Propositions \ref{P: compactness} and \ref{P: s_symmetry}, it's not hard to show that $e \mapsto \mathscr{E}^{a}_{+}(e)$ is itself lower semi-continuous.  Therefore,
	\begin{equation*}
		\mathscr{E}_{+}^{a}(e) \leq \liminf_{n \to \infty} \mathscr{E}_{+}^{a}(e_{n}) \leq \lim_{n \to \infty} \mathscr{T}^{a}_{e_{n}}(U) = \mathscr{T}^{a}_{e}(U).
	\end{equation*}
Invoking \eqref{E: smoothing_approximation}, we conclude $\mathscr{E}_{+}^{a}(e) \leq \mathscr{E}^{a}(e)$.   \qed \end{proof}  

It is worth observing at this stage that the functional $u \mapsto \mathcal{F}^{a}_{1}(u; Q_{e} \oplus_{e} \mathbb{R})$ that appeared repeatedly in the first step of the proof is precisely the one that is minimized in \cite{rabinowitz stredulinsky}.  In fact, slightly more can be said.

\begin{remark} \label{R: periodic representation}  The proof of Proposition \ref{P: remove constraint} shows that $\mathscr{E}^{a}(e)$ can be represented in the following way when $e \in \mathbb{R} \mathbb{Z}^{d}$:
	\begin{align*}
		\mathscr{E}^{a}(e) &= \mathcal{H}^{d-1}(Q_{e})^{-1} \min \bigg \{ \mathcal{F}^{a}_{1}(u; Q_{e} \oplus_{e} \mathbb{R}) \, \mid \, u : \mathbb{T}^{d-1}_{e} \oplus_{e} \mathbb{R} \to [-1,1],  \\
			&\qquad \quad \lim_{\langle x,e \rangle \to \pm \infty} u(x) = \pm 1 \quad \text{in} \, \, L^{1}_{\text{loc}}(Q_{e} \oplus_{e} \mathbb{R}) \bigg \}.
	\end{align*}
\end{remark}  
	
We make one last remark that will be needed later.

\begin{remark} \label{R: extension_to_space} The functional $\mathscr{T}^{a}_{e}$ can be extended to vectors $v \in \mathbb{R}^{d} \setminus \{0\}$ by 
	\begin{equation*}
		\mathscr{T}^{a}_{v}(U) = \int_{\mathbb{R} \times \mathbb{T}^{d}} \left( \frac{1}{2} \langle a(x) \mathcal{D}_{v}U, \mathcal{D}_{v}U \rangle + W(U) \right) \, dx \, ds,
	\end{equation*}
where $\mathcal{D}_{v} = v \partial_{s} + D_{x}$.  Letting $\mathscr{E}^{a}(v) = \min \left\{ \mathscr{T}^{a}_{v}(U) \, \mid \, U \in \mathscr{X}\right\}$, it is not hard to show that $\mathscr{E}^{a}(v) = \|v\| \mathscr{E}^{a}(\|v\|^{-1} v)$ and $U$ is a minimizer of $\mathscr{E}^{a}(v)$ if and only if the function $(s,x) \mapsto U(\|v\|^{-1}s,x)$ is a minimizer of $\mathscr{E}^{a}(\|v\|^{-1}v)$.  \end{remark}  

\section{Back to $\mathcal{F}^{a}_{1}$} \label{S: physical_coordinates}

In this section, we show that a minimizer $U_{e} \in \mathcal{M}_{e}(\mathbb{R} \times \mathbb{T}^{d})$ generates a family $\{u_{\zeta}\}_{\zeta \in \mathbb{R}}$ of plane-like minimizers of $\mathcal{F}^{a}_{1}$.  We also investigate consequences of this, including a uniqueness theorem when $e$ is irrational and a proof that continuous pulsating standing waves are necessarily in $\mathcal{M}_{e}(\mathbb{R} \times \mathbb{T}^{d})$.

\subsection{Plane-like minimizers}

To change from cylindrical coordinates in $\mathbb{R} \times \mathbb{T}^{d}$ back to the coordinates in $\mathbb{R}^{d}$, it is convenient to define the following transformation.  Given $e \in S^{d - 1}$, let $\mathcal{T}_{e} : \mathbb{R} \times \mathbb{R}^{d} \to \mathbb{R} \times \mathbb{R}^{d}$ be the map 
$\mathcal{T}_{e}(s,x) = (\langle x,e \rangle - s, x)$.
Notice that $\mathcal{T}_{e}$ is smooth and $\mathcal{T}_{e} \circ \mathcal{T}_{e} = \text{Id}$.   

Using $\mathcal{T}_{e}$ and what has already been proved, we find

\begin{prop} \label{P: plane_like_minimizers}  If $e \in S^{d-1}$, $U \in \mathcal{M}_{e}(\mathbb{R} \times \mathbb{T}^{d})$, and $\{u_{\zeta}\}_{\zeta \in \mathbb{R}}$ are the functions generated by $U$, then, for a.e.\ $\zeta \in \mathbb{R}$, $u_{\zeta}$ is a Class A minimizer of $\mathcal{F}^{a}_{1}$ in $\mathbb{R}^{d}$.  Moreover, there is a $\gamma \in (0,1)$ such that these minimizers are uniformly $\gamma$-H\"{o}lder continuous in $\mathbb{R}^{d}$.  \end{prop} 

\begin{proof}  To start with, we claim that if $g \in C^{\infty}_{c}(\mathbb{R}^{d}; [-1,1])$ is supported in $B(0,R)$ for some $R > 0$, then there is an $A_{g} \subseteq \mathbb{R}$ such that $\mathcal{L}^{1}(\mathbb{R} \setminus A_{g}) = 0$ and
	\begin{equation} \label{E: minimizer_with_g}
		\mathcal{F}^{a}_{1}(u_{\zeta} + g; B(0,R)) \geq \mathcal{F}^{a}_{1}(u_{\zeta}; B(0,R)) \quad \text{if} \, \, \zeta \in A_{g}.
	\end{equation}

To see this, first, let $B$ be a Lebesgue measurable, bounded subset of $\mathbb{R}$.  We will prove that
	\begin{equation*}
		\int_{B} \mathcal{F}^{a}_{1}(u_{\zeta} + g; B(0,R)) \, d\zeta \geq \int_{B} \mathcal{F}^{a}_{1}(u_{\zeta}; B(0,R)) \, d \zeta.
	\end{equation*}
To do this, we begin by fixing a family $(\varphi_{\epsilon})_{\epsilon > 0} \subseteq C^{\infty}_{c}(\mathbb{R})$ such that $0 \leq \varphi_{\epsilon} \leq 1$ and $\varphi_{\epsilon} \to \chi_{B}$ a.e.\ in $\mathbb{R}$.  Since $B$ is bounded, we can assume there is an $S > 0$ such that the union of the supports of $(\varphi_{\epsilon})_{\epsilon > 0}$ is contained in $(-S,S)$.

To prove the claim, we will change variables using $\mathcal{T}_{e}$ and invoke Corollary \ref{C: minimizer_symmetry_unbroken}.  From the definition of $\{u_{\zeta}\}_{\zeta \in \mathbb{R}}$, we know that $u_{\zeta}(x) = U(\mathcal{T}_{e}(\zeta,x))$ and $Du_{\zeta}(x) = \mathcal{D}_{e}U(\mathcal{T}_{e}(\zeta,x))$ a.e.\ in $\mathbb{R} \times \mathbb{R}^{d}$.  This implies, in particular, that $(x,\zeta) \mapsto \|Du_{\zeta}(x)\|^{2}$ is locally integrable, a fact we will need shortly.

Define $(\Phi_{\epsilon})_{\epsilon > 0} \subseteq C^{\infty}(\mathbb{R} \times \mathbb{R}^{d})$ by
	\begin{equation*}
		\Phi_{\epsilon}(s,x) = \varphi_{\epsilon}(\langle x, e \rangle - s) g(x).  
	\end{equation*} 
We have defined $\Phi_{\epsilon}$ so that, when we change variables, we have $\varphi_{\epsilon}(\zeta) g(x) = \Phi_{\epsilon}(\mathcal{T}_{e}(\zeta,x))$.  Notice that $(\Phi_{\epsilon})_{\epsilon > 0} \subseteq C^{\infty}_{c}(\mathbb{R} \times \mathbb{R}^{d})$ since $\mathcal{T}_{e}$ is a diffeomorphism.

In fact, by the choice of $S$, there is an $M \in \mathbb{N}$ such that the support of $\Phi_{\epsilon}$ is contained in $\mathbb{R} \times [-M,M)^{d}$ independently of $\epsilon > 0$.  Defining $m= (M,M,\dots,M) \in \mathbb{Z}^{d}$, we see that $\hat{\Phi}_{\epsilon} := S_{m}\Phi_{\epsilon}$ is supported in $\mathbb{R} \times [0,2M)^{d}$.  Therefore, extending $\hat{\Phi}_{\epsilon}$ to a $2M \mathbb{Z}^{d}$-periodic function $\tilde{\Phi}_{\epsilon} \in C^{\infty}_{c}(\mathbb{R} \times 2M\mathbb{T}^{d})$ and applying Corollary \ref{C: minimizer_symmetry_unbroken}, we find
	\begin{equation*}
		\mathscr{T}^{a}_{e,2M_{\epsilon}}(U + \tilde{\Phi}_{\epsilon}) \geq \mathscr{T}^{a}_{e,2M_{\epsilon}}(U),
	\end{equation*}
which can be rewritten in terms of $\hat{\Phi}_{\epsilon}$ as
	\begin{align*}
		\int_{\mathbb{R} \times \mathbb{R}^{d}} &\left(\frac{1}{2} \langle a(x) \mathcal{D}_{e}(U + \hat{\Phi}_{\epsilon}), \mathcal{D}_{e}(U + \hat{\Phi}_{\epsilon}) \rangle - \frac{1}{2} \langle a(x) \mathcal{D}_{e}U, \mathcal{D}_{e}U \rangle \right. \\
		&\qquad \left. + W(U + \hat{\Phi}_{\epsilon}) - W(U) \right) \, dx \, ds \geq 0.
	\end{align*} 
Since $\hat{\Phi}_{\epsilon} = S_{m}\Phi_{\epsilon}$ and $U = S_{m} U$, the previous inequality remains true if $\hat{\Phi}_{\epsilon}$ is replaced by $\Phi_{\epsilon}$.

Changing back to the $(x,\zeta)$ coordinates using $\mathcal{T}_{e}$, we deduce that
	\begin{align*}
		0 \leq \int_{\mathbb{R} \times \mathbb{R}^{d}} &\left(\frac{1}{2} \langle a(x)[Du_{\zeta}(x) + \varphi_{\epsilon}(\zeta) Dg(x)], Du_{\zeta}(x) + \varphi_{\epsilon}(\zeta) Dg(x) \rangle \right. \\
		&\left. - \frac{1}{2} \langle a(x) Du_{\zeta}(x), Du_{\zeta}(x) \rangle \right) \, dx \, d\zeta \\
		& + \int_{\mathbb{R} \times \mathbb{R}^{d}} \left(W(u_{\zeta}(x) + \varphi_{\epsilon}(\zeta) g(x)) - W(u_{\zeta}(x)) \right) \, dx \, d\zeta.
	\end{align*}
In view of the local integrability of $(x,\zeta) \mapsto \|Du_{\zeta}(x)\|^{2}$, the integral and the limit $\epsilon \to 0^{+}$ can be interchanged to obtain, by Fubini's Theorem,
	\begin{equation*}
		\int_{B} \left(\mathcal{F}^{a}_{1}(u_{\zeta} + g; B(0,R)) - \mathcal{F}^{a}_{1}(u_{\zeta}; B(0,R))\right) \, d\zeta \geq 0.
	\end{equation*}
$B$ was an arbitrary bounded measurable set so we are left to conclude that there is an $A_{g} \subseteq \mathbb{R}$ such that $\mathcal{L}^{1}(\mathbb{R} \setminus A_{g}) = 0$ and \eqref{E: minimizer_with_g} holds.

Since $H^{1}_{\text{loc}}(\mathbb{R}^{d})$ is a Fr\'{e}chet space (and, in particular, it is separable), we conclude that we can find a Lebesgue measurable set $A$ such that $\mathcal{L}^{1}(\mathbb{R} \setminus A) = 0$ and, for each $R > 0$ and $g \in C^{\infty}_{c}(B(0,R);[-1,1])$, 
	\begin{equation*}
		\mathcal{F}^{a}_{1}(u_{\zeta} + g; B(0,R)) \geq \mathcal{F}^{a}_{1}(u_{\zeta}; B(0,R)) \quad \text{if} \, \, \zeta \in A.
	\end{equation*}  
Therefore, for each $\zeta \in A$, $u_{\zeta}$ is a Class A minimizer of $\mathcal{F}^{a}_{1}$.  
	
We have proved that $(u_{\zeta})_{\zeta \in A}$ is a (uniformly bounded) family of Class A minimizers of $\mathcal{F}^{a}_{1}$.  Therefore, \cite[Theorem 3.1]{giaquinta giusti} implies there is a $\gamma \in (0,1)$ and a $C_{\gamma} > 0$ depending only on $a$ and $W$ such that
	\begin{equation*}
		|u_{\zeta}(x) - u_{\zeta}(y)| \leq C_{\gamma} \|x - y\|^{\gamma} \quad \text{if} \, \, x,y \in \mathbb{R}^{d}, \, \, \zeta \in A.
	\end{equation*}  \qed \end{proof}   

Combining the monotonicity of $U_{e}$ and the equicontinuity of the plane-like minimizers in $\{u_{\zeta}\}_{\zeta \in \mathbb{R}}$, we obtain

	\begin{prop} \label{P: continuity of the minimizers}  If $U_{e} \in \mathcal{M}_{e}(\mathbb{R} \times \mathbb{T}^{d})$, then there are functions $U^{+}_{e}, U^{-}_{e} \in \mathcal{M}_{e}(\mathbb{R} \times \mathbb{T}^{d})$ with the following properties:
		\begin{itemize}
			\item[(i)] $U^{+}_{e} = U^{-}_{e} = U_{e}$ a.e.\ in $\mathbb{R} \times \mathbb{T}^{d}$.
			\item[(ii)] If $\{u_{\zeta}^{+}\}_{\zeta \in \mathbb{R}}$ (resp.\  $\{u_{\zeta}^{-}\}_{\zeta \in \mathbb{R}}$) denotes the family of functions generated by $U^{+}_{e}$ (resp.\ $U^{-}_{e}$), then the map $\zeta \mapsto u_{\zeta}^{+}$ (resp.\ $\zeta \mapsto u_{\zeta}^{-}$) is right-continuous (resp.\ left-continuous) with respect to the topology of local uniform convergence.
			\item[(iii)] For each $\zeta \in \mathbb{R}$, $u_{\zeta}^{+} = \lim_{\mu \to \zeta^{+}} u_{\mu}^{+} = \lim_{\mu \to \zeta^{+}} u_{\mu}^{-}$ and $u_{\zeta}^{-} = \lim_{\mu \to \zeta^{-}} u_{\mu}^{+} = \lim_{\mu \to \zeta^{-}} u_{\mu}^{-}$ locally uniformly in $\mathbb{R}^{d}$.
			\item[(iv)] The set $\mathscr{D} = \{\zeta \in \mathbb{R} \, \mid \, u_{\zeta}^{+} \neq u_{\zeta}^{-}\}$ is countable.
			\item[(v)] For each $\zeta \in \mathbb{R}$ and $\bullet \in \{+,-\}$, $\lim_{r \to \pm \infty} u_{\zeta}^{\bullet}(re + x^{\perp}) = \pm 1$ uniformly with respect to $x^{\perp} \in \langle e \rangle^{\perp}$.
			\item[(vi)] If $\mathscr{D}$ is empty, then $U_{e}^{+} = U_{e}^{-} \in UC(\mathbb{R} \times \mathbb{T}^{d})$ and $\zeta \mapsto u_{\zeta}$ is continuous in the topology of uniform convergence.
		\end{itemize}
	\end{prop}  
	
\begin{remark}  A closer look at Proposition \ref{P: continuity of the minimizers} and its proof shows that $U_{e} \in C(\mathbb{R} \times \mathbb{T}^{d})$ if and only if $U_{e} \in UC(\mathbb{R} \times \mathbb{T}^{d})$.  \end{remark}  

Since we will need it later, we start by proving one of the assertions in Proposition \ref{P: continuity of the minimizers} as a separate lemma.

	\begin{lemma} \label{L: tedious lemma} If $u : \mathbb{R}^{d} \to [-1,1]$ satisfies the $\mathbb{Z}^{d}$-Birkhoff property with respect to the $e$ direction and
		\begin{equation} \label{E: local convergence at infinity}
			\lim_{s \to \pm \infty} u(\cdot + se) = \pm 1 \quad \text{locally uniformly in} \, \, \mathbb{R}^{d},
		\end{equation}
	then the convergence is actually uniform in the following sense:
		\begin{align}
			\lim_{R \to \infty} \inf \left\{ u(re + x^{\perp}) \, \mid \, r \geq R, \, \, x^{\perp} \in \langle e \rangle^{\perp} \right\} &= 1, \label{E: uniform convergence along the line 1}\\
			\lim_{R \to -\infty} \sup \left\{ u(re + x^{\perp}) \, \mid \, r \geq R, \, \, x^{\perp} \in \langle e \rangle^{\perp} \right\} &= -1. \label{E: uniform convergence along the line 2}
		\end{align}
	\end{lemma}  
	
		\begin{proof}  Given $\delta > 0$, \eqref{E: local convergence at infinity} implies we can fix $r > 0$ such that 
			\begin{equation*}
				\inf \left\{ u_{\zeta}(re + x) \, \mid \, x \in [0,1)^{d} \right\} \geq 1 - \delta.
			\end{equation*}
		Now observe that if $k \in \mathbb{Z}^{d}$ and $\langle k, e \rangle \geq 0$, then the Birkhoff property implies
			\begin{equation*}
				\inf \left\{ u_{\zeta}(re + x + k) \, \mid \, x \in [0,1)^{d} \right\} \geq 1 - \delta.
			\end{equation*}
		Therefore,
			\begin{equation*}
				u_{\zeta}(re + x) \geq 1 - \delta \quad \text{if} \, \, x \in \bigcup_{k \in \mathbb{Z}^{d} : \langle k, e \rangle \geq 0} k + [0,1)^{d}.
			\end{equation*}
		Observing now that $\bigcup_{k \in \mathbb{Z}^{d} : \langle k,e \rangle \geq 0} k + [0,1)^{d} \supseteq \langle e \rangle^{\perp} \oplus_{e} [\sqrt{d},\infty)$, we obtain
			\begin{equation*}
				\inf\left\{ u_{\zeta}( se + x^{\perp}) \, \mid \, s \geq r + \sqrt{d}, \, \, x^{\perp} \in \langle e \rangle^{\perp}\right\} \geq 1 - \delta.
			\end{equation*}
		Since $\delta$ was arbitrary, we conclude that $\lim_{r \to \infty} u_{\zeta}(re + x^{\perp}) = 1$ uniformly with respect to $x^{\perp} \in \langle e \rangle^{\perp}$.  
		
		The limit $\langle x,e \rangle \to -\infty$ can be treated via symmetrical arguments.\qed
\end{proof}  

	\begin{proof}[Proof of Proposition \ref{P: continuity of the minimizers}]  Let $\{u_{\zeta}\}_{\zeta \in \mathbb{R}}$ be the family of functions generated by $U_{e}$.  By Proposition \ref{P: plane_like_minimizers} and Theorem \ref{T: ergodic_lemma}, we can fix a set $A \subseteq \mathbb{R}$ such that $\mathcal{L}^{1}(\mathbb{R} \setminus A) = 0$ and, for each $\zeta \in A$, 
		\begin{itemize}
			\item[(a)] $u_{\zeta}$ is a Class A minimizer of $\mathcal{F}^{a}_{1}$,
			\item[(b)] $\limsup_{R \to \infty} R^{1 - d} \mathcal{F}^{a}_{1}(u_{\zeta}; Q^{e}(0,R) \oplus_{e} \mathbb{R}) < \infty$.
		\end{itemize} 
	Since $\mathcal{L}^{1}(\mathbb{R} \setminus A) = 0$, $A$ is necessarily dense in $\mathbb{R}$.  
	
	\textbf{Step 1:} Convergence as $|\langle x,e \rangle| \to \infty$
	
		Fix $\zeta \in A$.  By (b) above, $\mathcal{F}^{a}_{1}(u_{\zeta}; Q^{e}(0,R) \oplus_{e} \mathbb{R})< \infty$ for each $R > 0$.  Thus, the uniform continuity of $u_{\zeta}$ implies
			\begin{equation*}
				\lim_{r \to \pm \infty} W(u_{\zeta}(r e + x^{\perp})) = 0 \quad \text{if} \, \,  x^{\perp} \in \langle e \rangle^{\perp}.
			\end{equation*}  
		Thus, $|u_{\zeta}(\cdot + re)| \to 1$ locally uniformly in $\mathbb{R}^{d}$ as $r \to \pm \infty$.  Further, by appealing to the uniform continuity of $u_{\zeta}$ and the Birkhoff property, it is not hard to show that \eqref{E: local convergence at infinity} holds.  Hence by Lemma \ref{L: tedious lemma}, $\lim_{r \to \pm \infty} u_{\zeta}(re + x^{\perp}) = \pm 1$ uniformly with respect to $x^{\perp} \in \langle e \rangle^{\perp}$.  
			
	\textbf{Step 2:} Defining $\{u_{\zeta}^{+}\}_{\zeta \in \mathbb{R}}$ and $\{u_{\zeta}^{-}\}_{\zeta \in \mathbb{R}}$
	
	Given $\zeta \in \mathbb{R}$, define $u_{\zeta}^{+} = \lim_{A \ni \mu \to \zeta^{+}} u_{\mu}$ and $u_{\zeta}^{-} = \lim_{A \ni \mu \to \zeta^{-}} u_{\mu}$, which both exist locally uniformly in $\mathbb{R}^{d}$ by monotonicity and equicontinuity of $\mu \mapsto u_{\mu}$.   Also notice that $u_{\zeta}^{+} \leq u_{\zeta}^{-}$.  
	
	Define $U^{\pm}_{e}$ in $\mathbb{R} \times \mathbb{T}^{d}$ by  $U^{\pm}_{e}(s,x) = u^{\pm}_{\langle x,e \rangle - s}(x)$.  We leave it to the reader to verify that $U^{\pm}_{e}(s,x + k) = U^{\pm}_{e}(s,x)$ if $k \in \mathbb{Z}^{d}$.
		
		\textbf{Step 3:} One-sided continuity
		
		Define $\mathscr{D} = \{\zeta \in \mathbb{R} \, \mid \, u_{\zeta}^{+} \neq u_{\zeta}^{-}\}$.
		
		If $\zeta \in \mathscr{D}$, then there is an $x \in \mathbb{R}^{d}$ such that $u_{\zeta}^{+}(x) < u_{\zeta}^{-}(x)$.  Thus, for each $\zeta \in \mathscr{D}$, we can fix a nonempty, compact set $K_{\zeta} \subseteq \mathbb{R}^{d}$ such that $u_{\zeta}^{+} < u^{-}_{\zeta}$ in $K_{\zeta}$.  Let $U_{\zeta}$ be the open subset of $C_{\text{loc}}(\mathbb{R}^{d})$ defined by $U_{\zeta} = \{w \, \mid \, u^{+}_{\zeta} < w < u^{-}_{\zeta} \, \, \text{in} \, \, K_{\zeta}\}$.  
		
		Observe that since $\zeta \mapsto u_{\zeta}$ is non-increasing, it follows that $u_{\zeta}^{-} \geq u_{\mu}^{+}$ if $\zeta < \mu$.  Thus, $\{U_{\zeta}\}_{\zeta \in \mathscr{D}}$ is a disjoint family of open sets in $C_{\text{loc}}(\mathbb{R}^{d})$.  Since this is a separable metric space, we conclude that $\mathscr{D}$ must be countable.  
		
		Now $u_{\zeta}^{+} = u_{\zeta}^{-} = u_{\zeta}$ for a.e.\ $\zeta \in \mathbb{R}$, and, thus, Theorem \ref{T: ergodic_lemma} implies $U^{+}_{e} = U^{-}_{e} = U_{e}$ a.e.\ in $\mathbb{R} \times \mathbb{T}^{d}$.  
				
		\textbf{Step 4:} Uniform convergence when $\mathscr{D} = \emptyset$
		
		Assume now that $\mathscr{D}$ is empty.  Fix $\zeta \in \mathbb{R}$.  We claim that $u_{\mu} \to u_{\zeta}$ uniformly in $\mathbb{R}^{d}$ as $\mu \to \zeta$.  To see this, we argue by contradiction.  Suppose instead there is a $\delta > 0$ and two sequences $(y_{n})_{n \in \mathbb{N}} \subseteq \mathbb{R}^{d}$ and $(\mu_{n})_{n \in \mathbb{N}} \subseteq \mathbb{R}$ such that $|u_{\mu_{n}}(y_{n}) - u_{\zeta}(y_{n})| \geq \delta$ independently of $n$ even though $\lim_{n \to \infty} \mu_{n} = \zeta$.  Since we already know that $u_{\mu_{n}} \to u_{\zeta}$ locally uniformly, it follows that $\lim_{n \to \infty} \|y_{n}\| = \infty$.  
		
		For each $n \in \mathbb{N}$, let $[y_{n}] \in \mathbb{Z}^{d}$ be the vector such that $y - [y_{n}] \in [0,1)^{d}$.  Passing to a subsequence if necessary, fix $\xi \in [0,1]^{d}$ and $\gamma \in [-\infty,\infty]$ such that $\lim_{n \to \infty} \langle [y_{n}], e \rangle = \gamma$ in the topology of the extended reals and $\xi = \lim_{n \to \infty} (y_{n} - [y_{n}])$.  If $\gamma \in \{-\infty,\infty\}$, then $\lim_{n \to \infty} |u_{\mu_{n}}(y_{n}) - u_{\zeta}(y_{n})| = 0$ by the boundedness of $(\mu_{n})_{n \in \mathbb{N}}$ and Step 1.  Therefore, we can assume $|\gamma| < \infty$.  
		
		If we write $u_{\mu_{n}}(y_{n}) = u_{\mu_{n} - \langle [y_{n}], e \rangle}(y_{n} - [y_{n}])$ and $u_{\zeta}(y_{n}) = u_{\zeta - \langle [y_{n}], e \rangle}(y_{n} - [y_{n}])$, then the limits $\mu_{n} - \langle [y_{n}], e \rangle \to \zeta - \gamma$, $\zeta - \langle [y_{n}],e \rangle \to \zeta - \gamma$, and $y_{n} - [y_{n}] \to \xi$ imply
			\begin{align*}
				\delta &= \lim_{n \to \infty} |u_{\mu_{n} - \langle [y_{n}], e \rangle}(y_{n} - [y_{n}]) - u_{\zeta - \langle [y_{n}], e \rangle}(y_{n} - [y_{n}])| \\
				&= |u_{\zeta - \gamma}(\xi) - u_{\zeta - \gamma}(\xi)| = 0.
			\end{align*}     
		This contradiction shows that $u_{\zeta} = \lim_{\mu \to \zeta} u_{\mu}$ uniformly in $\mathbb{R}^{d}$.  From this and Proposition \ref{P: continuity_irrational_directions}, $U_{e}^{+} = U_{e}^{-} \in UC(\mathbb{R} \times \mathbb{T}^{d})$.  \qed
\end{proof}  

For later use, the next result establishes that when $a$ and $W$ are more regular, the plane-like minimizer $u_{\zeta}$ converges exponentially to $\pm 1$ in the $C^{1}$ topology as $\langle x,e \rangle \to \pm \infty$.

	\begin{prop} \label{P: annoying exponential convergence trick} Suppose that $a$ and $W$ satisfy \eqref{A: a_assumption_1}, \eqref{A: W_assumption_1}, \eqref{A: a_assumption_2}, and \eqref{A: W_assumption_2} and $e \in S^{d-1}$. If $u : \mathbb{R}^{d} \to [-1,1]$ is a Class A minimizer of $\mathcal{F}^{a}_{1}$ satisfying the $\mathbb{Z}^{d}$-Birkhoff property with respect to the $e$ direction and \eqref{E: local convergence at infinity} holds,
	then there is a constant $C_{0}$ depending only on $a$, $W$, and $u$ such that, for each $r \in \mathbb{R}$, 
		\begin{align}
			\sup \left\{ |u(x) - 1| + \|Du(x)\| \, \mid \, x \in re + \langle e \rangle^{\perp} \right\} &\leq C_{0} \exp(-C_{0}^{-1}r) \label{E: tedious exponential argument 1} \\
			\sup \left\{ |u(x) + 1| + \|Du(x)\| \, \mid \, x \in re + \langle e \rangle^{\perp} \right\} &\leq C_{0} \exp(C_{0}^{-1}r). \label{E: tedious exponential argument 2}
		\end{align}
	\end{prop}   
	
		\begin{proof}  Since $u$ is a bound solution of the Euler-Lagrange equation associated with $\mathcal{F}^{a}_{1}$, we know $Du$ is bounded by a constant depending only on $a$ and $W$.  Thus, it suffices to prove the existence of an $R_{0} > 0$ such that \eqref{E: tedious exponential argument 1} holds for all $r \geq R_{0}$ and \eqref{E: tedious exponential argument 2} holds for all $r \leq -R_{0}$.
		
		 The first observation is that since $u$ is $\mathbb{Z}^{d}$-Birkhoff in the $e$ direction and \eqref{E: local convergence at infinity} holds, Lemma \ref{L: tedious lemma} applies.  In particular, by \eqref{E: uniform convergence along the line 1}, \eqref{E: uniform convergence along the line 2}, and \eqref{A: W_assumption_2}, there is an $R_{0} > 0$ such that 
			\begin{align*}
				W'(u(x)) &\leq \frac{\alpha}{2} (u(x) - 1) \quad \text{if} \, \, \langle x,e \rangle \geq R_{0}, \\
				W'(u(x)) &\geq \frac{\alpha}{2}(u(x) + 1) \quad \text{if} \, \, \langle x,e \rangle \leq -R_{0}.
			\end{align*}
		
	Let us now prove that $|u(x) - 1|$ decays exponentially $\langle x,e \rangle \to \infty$.  Define $v(x) = 1 - u(x)$ and observe that, since $u$ is a Class A minimizer, we can compute, for each $y \in \mathbb{R}^{d}$ with $\langle y, e \rangle \geq R_{0}$,
		\begin{equation*}
			- \text{div}(a(y) Dv(y)) + \alpha v(y) = - \left(-\text{div}(a(y) Du(y)) + \alpha (u(y) - 1)\right) \leq 0.
		\end{equation*}
	At the same time, if $C_{0} > 0$ is sufficiently large, then the assumptions on $a$ imply the function $\overline{v}(y) = C_{0} \exp\left(-C_{0}^{-1} \langle y,e \rangle\right)$ satisfies
		\begin{align*}
			- \text{div}(a(y) D\overline{v}(y)) + \alpha \overline{v}(y) &= - \text{tr} \left( a(y) D^{2}\overline{v}(y)\right) - \langle (\text{div} \, a)(y), D\overline{v}(y) \rangle + \alpha \overline{v}(y) \\
			&= \left(- \langle a(y)e, e \rangle C_{0}^{-2} - \langle (\text{div} \, a)(y), e \rangle C_{0}^{-1} + \alpha \right) \overline{v}(y) \\
			&\geq \frac{\alpha}{2} \overline{v}(y).
		\end{align*}
Up to making $C_{0}$ larger if necessary, we can assume that $\overline{v}(y) \geq 2 \geq v(y)$ if $\langle y, e \rangle = R_{0}$, and, thus, by the maximum principle, $\overline{v}(y) \geq v(y)$ holds for all $y \in \mathbb{R}^{d}$ with $\langle y,e \rangle \geq R_{0}$.  

	Next, notice that by writing $W'(u) = \int_{0}^{1} W''(1 + t (u - 1) \, dt (u-1)$, we see that $v$ satisfies an equation of the form $-\text{div}(a(y) Dv(y)) + \tilde{\alpha}(y) v(y) = 0$ with a bounded, H\"{o}lder continuous coefficient $\tilde{\alpha}$.  Therefore, by Schauder estimates (see, e.g., \cite[Chapter 6]{Gilbarg Trudinger}), there is a $C_{1} > 0$ such that
		\begin{equation*}
			\|Dv(y)\| \leq C_{1}|v(y)| \quad \text{if} \, \, \langle y,e \rangle \geq R_{0}.
		\end{equation*}
	Hence since $Du = -Dv$, we deduce that $Du$ converges exponentially as $\langle x,e \rangle \to \infty$.  In particular, increasing $C_{0}$ if necessary, we conclude that \eqref{E: tedious exponential argument 1} holds for all $r \geq R_{0}$, as claimed.
	
	A similar argument establishes \eqref{E: tedious exponential argument 2} for $r \leq -R_{0}$. \qed
	\end{proof}  

\subsection{Surface tension}  Finally, we relate $\mathscr{E}^{a}$ to $\mathcal{F}^{a}_{1}$:

	\begin{prop} \label{P: surface tension} $\mathscr{E}^{a} = \tilde{\varphi}^{a}$.  Moreover, if $U_{e} \in \mathcal{M}_{e}(\mathbb{R} \times \mathbb{T}^{d})$, then, for a.e.\ $\zeta \in \mathbb{R}$,
		\begin{equation} \label{E: average_energy}
			\lim_{R \to \infty} R^{1 - d} \mathcal{F}^{a}_{1}(u_{\zeta}; Q^{e}(0,R) \oplus_{e} \mathbb{R}) = \tilde{\varphi}^{a}(e)
		\end{equation}
	Replacing $U_{e}$ by a right- or left-continuous representative as in Proposition \ref{P: continuity of the minimizers}, \eqref{E: average_energy} holds true for every $\zeta \in \mathbb{R}$ if $e \in \mathbb{R} \mathbb{Z}^{d}$ or \eqref{A: a_assumption_2} and \eqref{A: W_assumption_2} both hold.
	\end{prop}    
	
	\begin{proof}  We start by proving \eqref{E: average_energy} when $e \in \mathbb{R} \mathbb{Z}^{d}$, then prove it in the irrational case, and we conclude by discussing the situation when \eqref{A: a_assumption_2} and \eqref{A: W_assumption_2} are in effect.
	
	\textbf{Step 1: Rational Directions}  
	
	Fix $e \in S^{d-1} \cap \mathbb{R} \mathbb{Z}^{d}$.  We claim that $\mathscr{E}^{a}(e) = \tilde{\varphi}^{a}(e)$.  Let $U_{e} \in \mathcal{M}_{e}(\mathbb{R} \times \mathbb{T}^{d})$ and denote by $\{u_{\zeta}\}_{\zeta \in \mathbb{R}}$ the functions generated by $U_{e}$.  By Proposition \ref{P: plane_like_minimizers} and Theorem \ref{T: ergodic_lemma}, we can fix $C \subseteq \mathbb{R}$ such that $\mathcal{L}^{1}(\mathbb{R} \setminus C) = 0$ and, for each $\zeta \in C$, $u_{\zeta}$ is a Class A minimizer of $\mathcal{F}^{a}_{1}$ and
		\begin{align} \label{E: energy_assumption}
			\mathscr{E}^{a}(e) &= \lim_{R \to \infty} R^{1-d} \mathcal{F}^{a}_{1}(u_{\zeta}; Q^{e}(0,R) \oplus_{e} \mathbb{R}) \\
			&= \mathcal{H}^{d-1}(Q_{e})^{-1} \mathcal{F}^{a}_{1}(u_{\zeta}; Q_{e} \oplus_{e} \mathbb{R}) \label{E: lower semicontinuous}
		\end{align}
This proves \eqref{E: average_energy} holds for almost every $\zeta$, as claimed.

Next, consider the case when $U_{e}$ is right-continuous.  In this case, even if $\zeta \in \mathbb{R} \setminus C$, there is nonetheless a sequence $(\zeta_{n})_{n \in \mathbb{N}} \subseteq C$ such that $\zeta_{n} \to \zeta$ as $n \to \infty$ and $\zeta_{n} \geq \zeta$ independently of $n$.  Note, at the same time, that the functional appearing in \eqref{E: lower semicontinuous} is lower semi-continuous with respect to the argument $u$, and $u_{\zeta_{n}} \to u_{\zeta}$ locally uniformly by Proposition \ref{P: continuity of the minimizers}.  Therefore, Remark \ref{R: periodic representation} implies
	\begin{align*}
		\mathscr{E}^{a}(e) &\leq \mathcal{H}^{d-1}(Q_{e})^{-1} \mathcal{F}^{a}_{1}(u_{\zeta}; Q_{e} \oplus_{e} \mathbb{R}) \\
		&\leq \liminf_{n \to \infty} \mathcal{H}^{d-1}(Q_{e})^{-1} \mathcal{F}^{a}_{1}(u_{\zeta_{n}}; Q_{e} \oplus_{e} \mathbb{R}) = \mathscr{E}^{a}(e).
	\end{align*}
Thus, when $U_{e}$ is right-continuous, \eqref{E: lower semicontinuous} holds for all $\zeta \in \mathbb{R}$.  The left-continuous case follows similarly.
		
	As far as rational directions are concerned, it only remains to prove that $\mathscr{E}^{a}(e) = \tilde{\varphi}^{a}(e)$.  Toward that end, let us fix a smooth function $q : \mathbb{R} \to [-1,1]$ satisfying
		\begin{equation*}
			\int_{-\infty}^{\infty} \left( \frac{\Lambda}{2} q'(s)^{2} + W(q(s)) \right) \, ds < \infty, \quad \lim_{s \to \pm \infty} q(s) = \pm 1.
		\end{equation*}
	Define $q_{e} : \mathbb{R}^{d} \to [-1,1]$ by $q_{e}(x) = q(\langle x, e \rangle)$.  
	
	Next, fix $\zeta \in C$ and define $\tilde{Q}_{n} = Q(0,n) \oplus_{e} [-n,n]$ for $n \in \mathbb{N}$.  By \cite{ansini braides}, we can fix $(u_{n})_{n \in \mathbb{N}} \subseteq H^{1}_{\text{loc}}(\mathbb{R}^{d})$ such that
		\begin{align*}
			u_{n} &= q_{e} \quad \text{in} \, \, \mathbb{R}^{d} \setminus \tilde{Q}_{n} \\
			\tilde{\varphi}^{a}(e) &= \lim_{n \to \infty} n^{1-d} \mathcal{F}^{a}_{1}(u_{n}; \tilde{Q}_{n}).
		\end{align*} 
	At this stage, it is convenient to move to macroscopic coordinates: define $(\tilde{u}_{n})_{n \in \mathbb{N}}$ by $\tilde{u}_{n}(x) = u_{n}(nx)$ and $\tilde{v}_{n}(x) = u_{\zeta}(nx)$.
	
	Fix $\alpha \in (0,1)$ and notice that $\tilde{u}_{n}, \tilde{v}_{n} \to \chi_{\{\langle x, e \rangle > 0\}} - \chi_{\langle x, e \rangle < 0\}}$ in $L^{1}((1 + \alpha) \tilde{Q}_{1} \setminus \tilde{Q}_{1})$ as $n \to \infty$.  Thus, arguing as in \cite{ansini braides}, we can find cut-off functions $(\psi_{n})_{n \in \mathbb{N}} \subseteq C^{\infty}_{c}((1 + \alpha) \tilde{Q}_{1}; [0,1])$ such that
		\begin{align*}
			\limsup_{n \to \infty} \mathcal{F}^{a}_{n^{-1}}( &(1 - \psi_{n}) \tilde{u}_{n} + \psi_{n} \tilde{v}_{n} ; (1 + \alpha) \tilde{Q}_{1}) \leq \lim_{n \to \infty} n^{1 - d} \mathcal{F}^{a}_{1}(u_{n} ; \tilde{Q}_{n}) + \omega(\alpha), \\
			\omega(\alpha) &= ((1 + \alpha)^{d -1} - 1) \int_{-\infty}^{\infty} \left( \frac{\Lambda}{2} q'(s)^{2} + W(q(s)) \right) \, ds \\
			&\qquad+ \lim_{n \to \infty} n^{1 - d} \mathcal{F}^{a}_{1}(u_{\zeta}; (1 + \alpha) \tilde{Q}_{n} \setminus \tilde{Q}_{n}).
		\end{align*}
	Since $u_{\zeta}$ and its derivative are invariant under translations by $M_{e}$, it follows that
		\begin{equation*}
			\lim_{n \to \infty} n^{1 - d} \mathcal{F}^{a}_{1}(u_{\zeta}; (1 + \alpha) \tilde{Q}_{n} \setminus \tilde{Q}_{n}) = ((1 + \alpha)^{d-1} - 1) \mathcal{H}^{d-1}(Q_{e})^{-1} \mathcal{F}^{a}_{1}(u_{\zeta}; Q_{e} \oplus_{e} \mathbb{R}).
		\end{equation*}
	Thus, $\omega(\alpha) \to 0$ as $\alpha \to 0^{+}$.  
	
	Now we invoke the Class A minimization property of $u_{\zeta}$ to find
		\begin{align*}
			(1 + \alpha)^{d-1} \mathcal{H}^{d-1}(Q_{e})^{-1} &\mathcal{F}^{a}_{1}(u_{\zeta}; Q_{e} \oplus_{e} \mathbb{R}) =  \lim_{n \to \infty} n^{1 - d} \mathcal{F}^{a}_{1}(u_{\zeta}; (1 + \alpha) \tilde{Q}_{n} \oplus_{e} \mathbb{R}) \\
			&\leq \limsup_{n \to \infty} \mathcal{F}^{a}_{n^{-1}}( (1 - \psi_{n}) \tilde{u}_{n} + \psi_{n} \tilde{v}_{n} ; (1 + \alpha) \tilde{Q}_{1}) \\
			&\leq \tilde{\varphi}^{a}(e) + \omega(\alpha)
		\end{align*}
	Thus, in the limit $\alpha \to 0^{+}$, we find $\mathcal{H}^{d-1}(Q_{e})^{-1} \mathcal{F}^{a}_{1}(u_{\zeta}; Q_{e} \oplus_{e} \mathbb{R}) \leq \tilde{\varphi}^{a}(e)$.  
	
	To see that $\tilde{\varphi}^{a}(e) \leq \mathcal{H}^{d-1}(Q_{e})^{-1} \mathcal{F}^{a}_{1}(u_{\zeta}; Q_{e} \oplus_{e} \mathbb{R})$, we reverse the roles of the boundary condition $q_{e}$ and $u_{\zeta}$.  
	
\textbf{Step 2: Irrational Directions}  

Now we turn to irrational directions.  Assume that $e \in S^{d-1} \setminus \mathbb{R} \mathbb{Z}^{d}$.  Since $\mathscr{E}^{a}$ is determined by its value in $C^{\infty}_{\text{sgn}}(\mathbb{R} \times \mathbb{T}^{d})$ by Proposition \ref{P: smooth_approximation}, it follows that $\mathscr{E}^{a}$ is upper semi-continuous.  At the same time, the existence of minimizers and their sequential compactness (i.e.\ Proposition \ref{P: compactness}) implies that $\mathscr{E}^{a}$ is lower semi-continuous.  Therefore, it is a continuous function.  Since $\mathscr{E}^{a}(e') = \tilde{\varphi}^{a}(e')$ for all $e' \in S^{d-1} \cap \mathbb{R} \mathbb{Z}^{d}$ and $\tilde{\varphi}^{a}$ is also continuous, we conclude by density that $\mathscr{E}^{a}(e) = \tilde{\varphi}^{a}(e)$.    
	
	Next, assume that $U_{e} \in \mathcal{M}_{e}(\mathbb{R} \times \mathbb{Z}^{d})$.  Arguing as in the first paragraph of Step 1 above, we see that
	\begin{equation*}
		\tilde{\varphi}^{a}(e) = \mathscr{E}^{a}(e) =  \lim_{R \to \infty} R^{1-d} \mathcal{F}^{a}_{1}(u_{\zeta}; Q^{e}(0,R) \oplus_{e} \mathbb{R}) \quad \text{for a.e.} \, \, \zeta \in \mathbb{R}.
	\end{equation*}
	
	\textbf{Step 3: Regular coefficients}  
	
	If \eqref{A: a_assumption_2} and \eqref{A: W_assumption_2} both hold, then even if $e \in S^{d-1} \setminus \mathbb{R} \mathbb{Z}^{d}$, we have a $\zeta$-independent bound on $\|Du_{\zeta}\|_{L^{\infty}(\mathbb{R}^{d})}$ and uniform exponential decay of $Du_{\zeta}$ as $\langle x,e \rangle \to \pm \infty$ by Proposition \ref{P: annoying exponential convergence trick}.  Therefore, straightforward adjustments show that the approach of Step 1 applies even if we do not know that $Du_{\zeta}$ is periodic (or almost periodic).  In particular, in this case, no matter the choice of $e$ or $\zeta$, the limiting energy density exists and satisfies \eqref{E: average_energy}.  \qed \end{proof}
	
\subsection{Uniqueness when $e \notin \mathbb{R} \mathbb{Z}^{d}$} When $a$ and $W$ are regular enough, the following uniqueness result applies.  In what follows, recall that $\{T_{s}\}_{s \in \mathbb{R}}$ are the translation operators of Section \ref{S: function notation}.

\begin{prop} \label{P: uniqueness} Fix $e \in S^{d - 1} \setminus \mathbb{R} \mathbb{Z}^{d}$ and assume that $a$ and $W$ satisfy \eqref{A: a_assumption_1}, \eqref{A: W_assumption_1}, \eqref{A: a_assumption_2}, and \eqref{A: W_assumption_2}.  Suppose $U_{e}^{(1)}, U_{e}^{(2)} \in \mathcal{M}_{e}(\mathbb{R} \times \mathbb{T}^{d})$ are defined in such a way that the functions $\{u^{(i)}_{\zeta}\}_{\zeta \in \mathbb{R}}$ generated by $U^{(i)}$ for $i \in \{1,2\}$ are both right-continuous or left-continuous with respect to $\zeta$.  If $\int_{\mathbb{T}^{d}} U^{(1)}_{e}(s,x) \, dx = \int_{\mathbb{T}^{d}} U^{(2)}_{e}(s,x) \, dx$ for some $s \in \mathbb{R}$, then $U^{(1)}_{e} = U^{(2)}_{e}$.

In particular, under these assumptions, there is a $U_{e}$ such that 
	\begin{equation*}
		\mathcal{M}_{e}(\mathbb{R} \times \mathbb{T}^{d}) = \{T_{s}U_{e}\}_{s \in \mathbb{R}}.
	\end{equation*}    \end{prop}  
	
		\begin{proof}  We begin by proving the first statement.  Note that, by replacing $U^{(i)}_{e}$ by $T_{s}U^{(i)}_{e}$ if necessary, we can assume that $s = 0$. 
		
		Define $\overline{U}_{e} = U^{(1)}_{e} \vee U^{(2)}_{e}$ and $\underline{U}_{e} = U^{(1)}_{e} \wedge U^{(2)}_{e}$.  By Lemma \ref{L: submodularity}, $\overline{U}_{e}, \underline{U}_{e} \in \mathcal{M}_{e}(\mathbb{R} \times \mathbb{T}^{d})$, and, thus, every element of the families $\{\bar{u}_{\zeta}\}_{\zeta \in \mathbb{R}}$ and $\{\underbar{u}_{\zeta}\}_{\zeta \in \mathbb{R}}$ is a stationary solution of \eqref{E: main}.  Since $\underbar{u}_{\zeta} \leq \bar{u}_{\zeta}$, the strong maximum principle implies that either $\underbar{u}_{\zeta} = \bar{u}_{\zeta}$ or $\underbar{u}_{\zeta} < \bar{u}_{\zeta}$ in $\mathbb{R}^{d}$ (cf.\ \cite[Corollary A.3]{valdinoci de la llave}).  To conclude, we will assume that there is a $\zeta' \in \mathbb{R}$ so that $u^{(1)}_{\zeta'} = \underbar{u}_{\zeta'} < \bar{u}_{\zeta'} = u^{(2)}_{\zeta'}$ and show that this leads to a contradiction.  (If instead $u_{\zeta'}^{(2)} = \underline{u}_{\zeta'} < \bar{u}_{\zeta'} = u_{\zeta'}^{(1)}$, switch the roles of $(1)$ and $(2)$.)
		
		From the identity $u_{\zeta'}^{(1)} < u_{\zeta'}^{(2)}$, we know that 
			\begin{equation*}
				u_{\zeta' + \langle k,e \rangle}^{(1)} = u_{\zeta'}^{(1)}(\cdot - k) < u_{\zeta'}^{(2)}(\cdot - k) = u_{\zeta' + \langle k,e \rangle}^{(2)} \quad \text{for each} \, \, k \in \mathbb{Z}^{d}.
			\end{equation*}
		Thus, since $\{\langle k,e \rangle \, \mid \, k \in \mathbb{Z}^{d}\}$ is dense by Proposition \ref{P: algebra} and $\zeta \mapsto u_{\zeta}$ is right-continuous, we deduce that $u_{\zeta}^{(1)} \leq u_{\zeta}^{(2)}$ for each $\zeta \in \mathbb{R}$.  Of course, by the strong maximum principle, for any given $\zeta \in \mathbb{R}$, we must have either $u_{\zeta}^{(1)} < u_{\zeta}^{(2)}$ or $u_{\zeta}^{(1)} = u_{\zeta}^{(2)}$.  But we know the inequality is strict when $\zeta = \zeta'$ so it must always be strict.  (Otherwise, if $u_{\zeta''}^{(1)} = u_{\zeta''}^{(2)}$ for some $\zeta'' \in \mathbb{R}$, we use the density of $\{\zeta'' + \langle k,e \rangle \, \mid \, k \in \mathbb{Z}^{d}\}$ to obtain the contradiction $u_{\zeta'}^{(1)} = u_{\zeta'}^{(2)}$.)  
		
		Integrating over $\{0\} \times \mathbb{T}^{d}$, we obtain
			\begin{align*}
				0 &< \int_{\mathbb{T}^{d}} \left(u^{(2)}_{\langle x,e \rangle}(x) - u^{(1)}_{\langle x,e \rangle}(x) \right) \, dx \\
					&= \int_{\mathbb{T}^{d}} (U_{e}^{(2)}(0,x) - U_{e}^{(1)}(0,x)) \, dx = 0.
			\end{align*}
		This contradiction shows $u^{(1)}_{\zeta} = u^{(2)}_{\zeta}$ for all $\zeta \in \mathbb{R}$, hence $U^{(1)}_{e} = U^{(2)}_{e}$. 
		
	Finally, we prove that $\mathcal{M}_{e}(\mathbb{R} \times \mathbb{T}^{d})$ equals the translates of a single minimizer in this setting.  By Theorem \ref{T: existence theorem main} and Proposition \ref{P: continuity of the minimizers}, we can fix a $U_{e} \in \mathcal{M}_{e}(\mathbb{R} \times \mathbb{T}^{d})$ for which the associated functions $\{u_{\zeta}^{e}\}_{\zeta \in \mathbb{R}}$ are right-continuous with respect to $\zeta$.  If $U \in \mathcal{M}_{e}(\mathbb{R} \times \mathbb{T}^{d})$, then, up to redefining $U$ on a set of measure zero, we can assume that it also has this right-continuity property.  By Proposition \ref{P: apriori}, there is an $\bar{s} \in \mathbb{R}$ such that $\int_{\mathbb{T}^{d}} T_{\bar{s}}U_{e}(0,x) \, dx = \int_{\mathbb{T}^{d}} U(0,x) \, dx$.  Hence what we just proved implies $U = T_{\bar{s}} U_{e} \in \{T_{s} U\}_{s \in \mathbb{R}}$.  \qed
	\end{proof}  
	
\subsection{Non-uniqueness when $e \in S^{d-1} \cap \mathbb{R} \mathbb{Z}^{d}$} \label{S: nonunique}  In the rational case, uniqueness does not hold in general.  It is worth stressing that this issue already manifests itself in the spatially homogeneous case, that is, when $a$ is constant; see the remark below.

\begin{prop}  Suppose $e \in S^{d-1} \cap \mathbb{R} \mathbb{Z}^{d}$ and let $m_{e} > 0$ be the constant given by \eqref{E: rational period}.  Fix $U \in \mathcal{M}_{e}(\mathbb{R} \times \mathbb{T}^{d})$ for which the corresponding functions $\{u_{\zeta}\}_{\zeta \in \mathbb{R}}$ are right-continuous with respect to $\zeta$.  If $U$ generates more than one function, that is, if there is a $\zeta_{1} \in (0,m_{e})$ such that $u_{0} \neq u_{\zeta_{1}} \neq u_{m_{e}}$, then $\mathcal{M}_{e}(\mathbb{R} \times \mathbb{T}^{d}) \setminus \{T_{s}U\}_{s \in \mathbb{R}}$ is nonempty.  \end{prop}  

Note that such a $\zeta_{1}$ certainly exists if $U$ is continuous.

\begin{proof} Let $h : \mathbb{R} \to \mathbb{R}$ be any right-continuous, non-decreasing function such that $h(\zeta + m_{e}) = h(\zeta) + m_{e}$.  Define functions $\{u^{(h)}_{\zeta}\}_{\zeta \in \mathbb{R}}$ in $\mathbb{R}^{d}$ by 
	\begin{equation*}
		u^{(h)}_{\zeta}(x) = u_{h(\zeta)}(x).
	\end{equation*}
A straightforward computation shows that the function $U^{(h)}$ defined in $\mathbb{R} \times \mathbb{T}^{d}$ by 
	\begin{equation} \label{E: time change}
		U^{(h)}(s,x) = u^{(h)}_{\langle x,e \rangle - s}(x)
	\end{equation}
determines an element of $\mathscr{X}_{+}$ that generates $\{u^{(h)}_{\zeta}\}_{\zeta \in \mathbb{R}}$.  Furthermore, using Theorem \ref{T: ergodic_lemma} and Proposition \ref{P: surface tension}, we can compute its energy:
	\begin{equation*}
		\mathscr{T}^{a}_{e}(U^{(h)}) = m_{e}^{-1} \mathcal{H}^{d-1}(Q_{e})^{-1} \int_{0}^{m_{e}} \mathcal{F}^{a}_{1}(u_{h(\zeta)}; Q_{e} \oplus_{e} \mathbb{R}) \, d\zeta = \mathscr{E}^{a}(e).
	\end{equation*}
Thus, $U^{(h)} \in \mathcal{M}_{e}(\mathbb{R} \times \mathbb{T}^{d})$.  

Define $U^{(h)}$ as above with $h(\zeta) = \lfloor m_{e}^{-1} \zeta \rfloor m_{e}$.  If $\{u^{(h)}_{\zeta}\}_{\zeta \in \mathbb{R}}$ are the functions generated by $U^{(h)}$, then $\{u^{(h)}_{\zeta}\}_{\zeta \in [0,m_{e})} = \{u_{0}\}$.  At the same time, given any $s \in \mathbb{R}$, if $\{u^{(s)}_{\zeta}\}_{\zeta \in \mathbb{R}}$ are the functions generated by $T_{s}U$, then 
	\begin{equation*}
		\{u^{(s)}_{\zeta}\}_{\zeta \in [0,m_{e})} \supseteq \{u_{\zeta_{1} + \lceil m_{e}^{-1}(s - \zeta_{1}) \rceil m_{e}}, u_{\lceil m_{e}^{-1}s \rceil m_{e}} \}.
	\end{equation*}
Hence the assumption on $\zeta_{1}$ implies that $\# \{u^{(s)}_{\zeta}\}_{\zeta \in [0,m_{e})} \geq 2$ and, thus, $\{u^{(h)}_{\zeta}\}_{\zeta \in [0,m_{e})} \neq \{u^{(s)}_{\zeta}\}_{\zeta \in [0,m_{e})}$.  It follows that $U^{(h)} \neq T_{s}U$.  Since $s$ was arbitrary, $U^{(h)} \in \mathcal{M}_{e}(\mathbb{R} \times \mathbb{T}^{d}) \setminus \{T_{s}U\}_{s \in \mathbb{R}}$. \qed
\end{proof}  

\begin{remark} \label{R: symmetry breaking} In the spatially homogeneous case, say, when $a \equiv \lambda \text{Id}$, there is a $U \in \mathcal{M}_{e}(\mathbb{R} \times \mathbb{T}^{d})$ given by $U(s,x) = q(s)$ for some $q : \mathbb{R} \to [-1,1]$.  Further, if $W$ satisfies \eqref{A: W_assumption_2}, then $q$ is unique up to translations and $q' > 0$ (see, e.g., \cite{alberti guide}).  This symmetry is not necessarily shared by all minimizers. 

Indeed, if we define $U^{(h)} \in \mathcal{M}_{e}(\mathbb{R} \times \mathbb{T}^{d})$ as in \eqref{E: time change}, then 
	\begin{equation*}
		U^{(h)}(s,x) = q(\langle x,e \rangle - h(\langle x,e \rangle - s)).
	\end{equation*} 
Here $U^{(h)}$ will have a non-trivial dependence on the $x$ variable if $h$ is nonlinear.  \end{remark}	

\subsection{On laminar media}  \label{S: laminar symmetry} At this stage, it is appropriate to mention an important distinction between the general and laminar settings.  Recall that $a$ is laminar if there is a $k \in \{1,2,\dots,d-1\}$ such that 
	\begin{equation}
		a(x + y) = a(x) \quad \text{for} \, \, x \in \mathbb{R}^{d}, \, \, y \in \mathbb{Z}^{k} \times \mathbb{R}^{d-k}. \label{E: laminar}
	\end{equation}
When $a$ and $W$ are regular enough, this symmetry is passed on to the pulsating standing waves.  We will see in the next sections that this has interesting ramifications.

The consequences of the laminarity assumption can be stated particularly easily in cylindrical coordinates, as we now show.

\begin{prop} \label{P: laminar symmetry} Suppose that $a$ is laminar, that is, \eqref{E: laminar} holds for some $k \in \{1,2,\dots,d-1\}$, and that assumptions \eqref{A: a_assumption_1}, \eqref{A: W_assumption_1}, \eqref{A: a_assumption_2}, and \eqref{A: W_assumption_2} hold.  If $e \in (S^{d-1} \setminus \mathbb{R} \mathbb{Z}^{d}) \cup (S^{d-1} \cap (\mathbb{R}^{k} \times \{0\}))$ and $U_{e} \in \mathcal{M}_{e}(\mathbb{R} \times \mathbb{T}^{d})$, then 
	\begin{equation} \label{E: laminar symmetry}
		\partial_{x_{i}} U_{e} \equiv 0 \quad \text{for each} \, \, i \in \{k + 1,\dots,d\}.
	\end{equation}
\end{prop}  

In other words, if $a$ only depends on the first $k$ coordinates, then any element of $\mathcal{M}_{e}(\mathbb{R} \times \mathbb{T}^{d})$ only depends on $(s,x_{1},\dots,x_{k})$.

\begin{proof}  Notice that it suffices to show that, for each $y \in \{0\} \times \mathbb{R}^{d-k}$, 
	\begin{equation} \label{E: desired symmetry property}
		U_{e}(s,x + y) = U_{e}(s,x) \quad \text{for a.e.} \, \,  (s,x) \in \mathbb{R} \times \mathbb{T}^{d}.
	\end{equation}
Furthermore, by Proposition \ref{P: continuity of the minimizers}, we can assume without loss of generality that the functions $\{u^{e}_{\zeta}\}_{\zeta \in \mathbb{R}}$ generated by $U_{e}$ are right-continuous.  

We split the proof into two cases: (i) $e \in S^{d-1} \setminus \mathbb{R} \mathbb{Z}^{d}$ and (ii) $e \in \mathbb{R} \mathbb{Z}^{k} \times \{0\}$.

\textbf{Case (i): $e \in S^{d-1} \setminus \mathbb{R} \mathbb{Z}^{d}$}

Fix a $y \in \{0\} \times \mathbb{R}^{d-k}$.  Since $a(x + y) = a(x)$ for all $x \in \mathbb{R}^{d}$, it follows that the function $U_{e}^{(y)}(s,x) = U_{e}(s,x+y)$ is an element of $\mathcal{M}_{e}(\mathbb{R} \times \mathbb{T}^{d})$.  At the same time,
	\begin{equation*}
		\int_{\mathbb{T}^{d}} U_{e}^{(y)}(0,x) \, dx = \int_{\mathbb{T}^{d}} U_{e}(0,x + y) \, dx = \int_{\mathbb{T}^{d}} U_{e}(0,x) \, dx.
	\end{equation*}
Therefore, by Proposition \ref{P: uniqueness}, $U_{e}^{(y)} = U_{e}$.

\textbf{Case (ii): $e \in S^{d-1} \cap (\mathbb{R} \mathbb{Z}^{k} \times \{0\})$} 

As in the previous step, fix $y \in \{0\} \times \mathbb{R}^{d-k}$, define $U_{e}^{(y)}(s,x) = U_{e}(s,x+y)$, and let $\{u_{\zeta}\}_{\zeta \in \mathbb{R}}$ and $\{u^{(y)}_{\zeta}\}_{\zeta \in \mathbb{R}}$ denote the functions generated by $U_{e}$ and $U_{e}^{(y)}$ respectively.  We claim that $u_{\zeta}^{(y)} \equiv u_{\zeta}$ in $\mathbb{R}^{d}$ independently of $\zeta \in \mathbb{R}$.

To see this, fix $\zeta \in \mathbb{R}$ and recall from Remark \ref{R: periodic representation} that $u_{\zeta}$ and $u_{\zeta}^{(y)}$ are minimizers of the functional $\mathcal{F}_{1}^{a}(\cdot; Q_{e} \oplus_{e} \mathbb{R})$.  Therefore, $u_{\zeta} \wedge u_{\zeta}^{(y)}$ and $u_{\zeta} \vee u_{\zeta}^{(y)}$ are also minimizers.  By the strong maximum principle (cf.\ \cite[Corollary A.3]{valdinoci de la llave}), these two functions must be strictly ordered.  Hence $u_{\zeta} > u_{\zeta}^{(y)}$, $u_{\zeta} < u_{\zeta}^{(y)}$, or $u_{\zeta} \equiv u_{\zeta}^{(y)}$.

At this stage, let us notice that $\langle y,e \rangle = 0$ and, thus, $u^{(y)}_{\zeta} = u_{\zeta}(\cdot + y)$ by definition.

Next, observe that if $u_{\zeta} > u_{\zeta}^{(y)}$, then this implies $u_{\zeta}(x) > u_{\zeta}(x + ny)$ for each $n \in \mathbb{N}$ and $x \in \mathbb{R}^{d}$.  To see that this is impossible, recall that $u_{\zeta}$ is a function in $\mathbb{T}^{d-1}_{e} \oplus_{e} \mathbb{R}$, and it is continuous due to the regularity of $a$ and $W'$.  Further, there are sequences $(n_{j})_{j \in \mathbb{N}} \subseteq \mathbb{N}$ and $(k_{j})_{j \in \mathbb{N}} \subseteq \{0\} \times \mathbb{Z}^{d - k}$ such that $n_{j} y - k_{j} \to 0$ as $j \to \infty$.  Since, by the choice of $e$, we know that $\{k_{j} \, \mid \, j \in \mathbb{N}\} \subseteq \{0\} \times \mathbb{Z}^{d -k} \subseteq M_{e}$, we obtain a contradiction as follows: 
	\begin{equation*}
		u_{\zeta}(x) > u_{\zeta}(x + y) \geq \lim_{j \to \infty} u_{\zeta}(x + n_{j} y) = \lim_{j \to \infty} u_{\zeta}(x + n_{j} y - k_{j}) = u_{\zeta}(x).
	\end{equation*}
Interchanging $y$ and $-y$, we deduce that the alternative $u_{\zeta} < u_{\zeta}^{(y)}$ is similarly impossible.  Therefore, $u_{\zeta} \equiv u_{\zeta}^{(y)}$.

We showed that $u_{\zeta}^{(y)} = u_{\zeta}$ for each $y \in \{0\} \times \mathbb{R}^{d-k}$ and $\zeta \in \mathbb{R}$.  Therefore, \eqref{E: desired symmetry property} holds. \qed  \end{proof}

As a direct consequence of the previous result, we deduce that, in the laminar setting, it is always possible to find minimizers with the same symmetry properties as $a$.

\begin{prop} \label{P: weak symmetry} If $a$ satisfies \eqref{E: laminar} and assumptions \eqref{A: a_assumption_1}, \eqref{A: W_assumption_1}, \eqref{A: a_assumption_2}, and \eqref{A: W_assumption_2} all hold, then, for each $e \in S^{d-1}$, there is a $U_{e} \in \mathcal{M}_{e}(\mathbb{R} \times \mathbb{T}^{d})$ that only depends on $(s,x_{1},\dots,x_{k})$, i.e., for which \eqref{E: laminar symmetry} holds.  \end{prop}  

As shown by Remark \ref{R: symmetry breaking}, in the generality of Proposition \ref{P: weak symmetry}, we cannot assert that \emph{all} elements of $\mathcal{M}_{e}(\mathbb{R} \times \mathbb{T}^{d})$ are independent of $(x_{k + 1},\dots,x_{d})$ as in Proposition \ref{P: laminar symmetry}, only that some are.

	\begin{proof}  The previous result shows this is true when $e \notin \mathbb{R} \mathbb{Z}^{d}$.  If $e \in \mathbb{R} \mathbb{Z}^{d}$, we can nonetheless choose $(e_{n})_{n \in \mathbb{N}} \subseteq S^{d-1} \setminus \mathbb{R} \mathbb{Z}^{d}$ such that $e = \lim_{n \to \infty} e_{n}$.  Fix $(U_{e_{n}})_{n \in \mathbb{N}}$ so that, for each $n$, $U_{e_{n}} \in \mathcal{M}_{e_{n}}(\mathbb{R} \times \mathbb{T}^{d})$ and $\int_{\mathbb{R} \times \mathbb{T}^{d}} U_{e_{n}}(0,x) \, dx = 0$.  By Proposition \ref{P: compactness}, there is a $U \in \mathcal{M}_{e}(\mathbb{R} \times \mathbb{T}^{d})$ such that, up to the extraction of a subsequence, $U_{e_{n}} \to U$ a.s.\ in $\mathbb{R} \times \mathbb{T}^{d}$.  Since Proposition \ref{P: laminar symmetry} implies $\partial_{x_{i}} U_{e_{n}} = 0$ independently of $i \in \{k + 1,\dots,d\}$ and $n \in \mathbb{N}$, $U$ inherits the same property in the limit. \qed \end{proof}  

The last two propositions show that, in effect, when $a$ is laminar, the Lagrangian can be lowered to the cylinder $\mathbb{R} \times \mathbb{T}^{k}$.   Indeed, if $U_{e}$ is a minimizer depending only on $(s,x_{1},\dots,x_{k})$, then its energy can be computed using an integral in $\mathbb{R} \times \mathbb{T}^{k}$:
	\begin{equation*}
		\mathscr{T}^{a}_{e}(U_{e}) = \int_{\mathbb{R} \times \mathbb{T}^{k}} \left(\frac{1}{2} \langle a(x) \mathcal{D}_{e} U_{e}, \mathcal{D}_{e} U_{e} \rangle + W(U_{e}) \right) \, dx \, ds.
	\end{equation*}
The fact of the matter is, at the expense of additional technical complications, we could have worked in $\mathbb{R} \times \mathbb{T}^{k}$ and defined the Lagrangian in this way from the beginning.

Of course, as we already saw in the proof of Proposition \ref{P: laminar symmetry}, if $U_{e}$ depends only on $(s,x_{1},\dots,x_{k})$, then this symmetry will be reflected by the functions $\{u^{e}_{\zeta}\}_{\zeta \in \mathbb{R}}$ it generates.  The next remark comments on one of the useful consequences of this.

\begin{remark} \label{R: symmetry birkhoff}  If $U_{e} \in L^{\infty}(\mathbb{R} \times \mathbb{T}^{d})$, $\partial_{s} U_{e} \geq 0$, and $U_{e}$ depends only on the variables $(s,x_{1},\dots,x_{k})$, then the functions $\{u^{e}_{\zeta}\}_{\zeta \in \mathbb{R}}$ generated by $U_{e}$ in the $e$ direction satisfy the $(\mathbb{Z}^{k} \times \mathbb{R}^{d- k})$-Birkhoff property.  More precisely, for each $\zeta \in \mathbb{R}$, the function $u^{e}_{\zeta}$ satisfies Definition \ref{D: birkhoff} with $M \mathbb{Z}^{d}$ replaced by $\mathbb{Z}^{k} \times \mathbb{R}^{d-k}$.  

This applies, in particular, to the functions of Propositions \ref{P: laminar symmetry} and \ref{P: weak symmetry}.    \end{remark}

The final result of this section shows that it is always possible to find continuous pulsating standing waves in certain directions in the laminar setting. 

	\begin{prop} \label{P: laminar foliations}  Fix $k < d$ and suppose that $e \in S^{d-1} \setminus (S^{k-1} \times \{0\})$.  If $a$ satisfies \eqref{E: laminar}, \eqref{A: a_assumption_1} and \eqref{A: W_assumption_1} hold, and $U \in \mathcal{M}_{e}(\mathbb{R} \times \mathbb{T}^{d})$ depends only on $(s,x_{1},\dots,x_{k})$, that is, \eqref{E: laminar symmetry} holds, then $U \in UC(\mathbb{R} \times \mathbb{T}^{d})$.  Moreover, if $\{u_{\zeta}\}_{\zeta \in \mathbb{R}}$ are the functions generated by $U$, then there is a $v \in \{0\} \times \mathbb{R}^{d-k}$ such that, for each $\zeta \in \mathbb{R}$,
		\begin{equation} \label{E: translation generation}
			u_{\zeta} = u_{0}(\cdot - \zeta v) \quad \text{in} \, \, \mathbb{R}^{d}.
		\end{equation}  \end{prop}  
		
			\begin{proof}  Fix $e' \in S^{d-1} \cap (\mathbb{R}^{k} \times \{0\})$ such that $\langle e', e \rangle > 0$.  Using \eqref{E: laminar symmetry}, we we can write
				\begin{align*}
					u_{\zeta}(x) &= U(\langle x,e \rangle - \zeta, x) \\
						&= U(\langle x - \zeta \langle e, e' \rangle^{-1}e', e \rangle, x - \zeta \langle e, e' \rangle^{-1} e') = u_{0}(x - \zeta \langle e, e' \rangle^{-1} e').
				\end{align*}
			This gives the desired expression for $u_{\zeta}$ with $v = \langle e, e' \rangle^{-1} e'$.
			
			We know that $u_{0}$ is uniformly continuous in $\mathbb{R}^{d}$ by Proposition \ref{P: continuity of the minimizers}.  Thus, the previous formula implies that $\zeta \mapsto u_{\zeta}$ is continuous.  Invoking Proposition \ref{P: continuity of the minimizers} again, we conclude that $U \in UC(\mathbb{R} \times \mathbb{T}^{d})$. \qed   \end{proof}  
  
\subsection{Continuous Pulsating Standing Waves as Minimizers}  Next, we prove Proposition \ref{P: continuous_minimizers} concerning continuous pulsating standing waves.  The argument given below was presented, in a slightly different context, by Cabr\'{e} during the conference ``Calculus of Variations and Nonlinear Partial Differential Equations" at Columbia University in 2016; he attributes it to Caffarelli.  This trick is sometimes called a ``sliding argument."

	\begin{proof}[Proof of Proposition \ref{P: continuous_minimizers}]  Assume that $U \in C(\mathbb{R} \times \mathbb{T}^{d}; [-1,1])$ satisfies $\mathcal{D}_{e}^{*}(a(x) \mathcal{D}_{e} U) + W'(U) = 0$ and $\partial_{s} U \geq 0$ in the distributional sense in $\mathbb{R} \times \mathbb{T}^{d}$ and
		\begin{equation*}
			\lim_{s \to \pm \infty} T_{-s}U = \pm 1 \quad \text{in} \, \, L^{1}_{\text{loc}}(\mathbb{R} \times \mathbb{T}^{d}).
		\end{equation*}  
	Let $\{u_{\zeta}\}_{\zeta \in \mathbb{R}}$ be the functions generated by $U$.  Arguing as in the proof of Proposition \ref{P: plane_like_minimizers}, we find that for a.e.\ $\zeta \in \mathbb{R}$, the function $u_{\zeta}$ is a distributional solution of $- \text{div}(a(x) Du_{\zeta}) + W'(u_{\zeta}) = 0$ in $\mathbb{R}^{d}$.  Since $(x,\zeta) \mapsto u_{\zeta}(x)$ is bounded and continuous, every member of the family is necessarily a distributional solution, and assumptions \eqref{A: a_assumption_2} and \eqref{A: W_assumption_2} together imply $\{u_{\zeta}\}_{\zeta \in \mathbb{R}} \subseteq C^{2,\alpha}(\mathbb{R}^{d})$. 
	
	Fix $\zeta \in \mathbb{R}$.  We claim that $u_{\zeta}$ is a Class A minimizer of $\mathcal{F}^{a}_{1}$.  To see this, fix $R > 0$ and pick $w \in H^{1}(B(0,R);[-1,1])$ such that
		\begin{align*}
			\mathcal{F}^{a}_{1}(w; B(0,R)) &= \inf \left\{ \mathcal{F}^{a}_{1}(u_{\zeta} + f; B(0,R)) \, \mid \, f \in C^{\infty}_{c}(B(0,R);[-1,1]) \right\}, \\
			w &= u_{\zeta} \, \, \text{on} \, \, \partial B(0,R).
		\end{align*}  
	Since \eqref{A: a_assumption_2} and \eqref{A: W_assumption_2} are in force, $w$ extends to a continuous function in $\overline{B(0,R)}$ (cf.\ \cite[Theorem 8.34]{Gilbarg Trudinger}), and the strong maximum principle (cf.\ \cite[Corollary A.3]{valdinoci de la llave}) implies 
		\begin{equation*}
			-1 < \min \left\{w(x) \, \mid \, x \in \overline{B(0,R)} \right\} \leq \max \left\{w(x) \, \mid \, x \in \overline{B(0,R)} \right\} < 1.
		\end{equation*}  
	
	Henceforth, let $\zeta_{1}, \zeta_{2} \in [-\infty,\infty]$ be defined by 
		\begin{align*}
			\zeta_{1} = \sup \left\{ \zeta' \in \mathbb{R} \, \mid \, u_{\zeta'} > w \, \, \text{in} \, \, B(0,R) \right\}, \\
			\zeta_{2} = \inf \left\{ \zeta' \in \mathbb{R} \, \mid \, u_{\zeta'} < w \, \, \text{in} \, \, B(0,R) \right\}.
		\end{align*}
	Since $u_{\zeta} \to \pm 1$ locally uniformly as $\zeta \to \pm \infty$, it follows that $- \infty < \zeta_{1} \leq \zeta_{2} < \infty$.  
	
	From the inequality $u_{\zeta_{1}} \geq u_{\zeta}$ on $\partial B(0,R)$ and the ordering of $\{u_{\zeta}\}_{\zeta \in \mathbb{R}}$, we know that $\zeta_{1} \leq \zeta$.  Similarly, $\zeta \leq \zeta_{2}$.  In particular, this implies $u_{\zeta_{2}} \leq u_{\zeta} \leq u_{\zeta_{1}}$ in $\mathbb{R}^{d}$.
	
	We claim that $u_{\zeta_{1}} = u_{\zeta} = u_{\zeta_{2}}$ in $\mathbb{R}^{d}$.  To see this, observe that the map $(x,\zeta) \mapsto u_{\zeta}(x)$ is continuous, and, thus, there is an $x_{1} \in \overline{B(0,R)}$ such that $u_{\zeta_{1}}(x_{1}) = w(x_{1})$.  Since $u_{\zeta_{1}} \geq w$, the strong maximum principle implies that we can assume without loss of generality that $x_{1} \in \partial B(0,R)$ (see, e.g., \cite[Corollary A.3]{valdinoci de la llave}).  Thus, $u_{\zeta_{1}}(x_{1}) = u_{\zeta}(x_{1})$.   Since $u_{\zeta}$ and $u_{\zeta_{1}}$ are solutions and $u_{\zeta_{1}} \geq u_{\zeta}$ in the whole space, we conclude that $u_{\zeta} = u_{\zeta_{1}}$.  A similar argument shows $u_{\zeta_{2}} = u_{\zeta}$.  Since $u_{\zeta_{2}} \leq w \leq u_{\zeta_{1}}$ in $\overline{B(0,R)}$ by construction, this implies $u_{\zeta} = w$.      
	
	We showed that if $\zeta \in \mathbb{R}$ and $R > 0$, then 
		\begin{equation*}
			\mathcal{F}^{a}_{1}(u_{\zeta}; B(0,R)) = \inf \left\{ \mathcal{F}^{a}_{1}(u_{\zeta} + f; B(0,R)) \, \mid \, f \in C^{\infty}_{c}(B(0,R); [-1,1]) \right\}.
		\end{equation*}
	Therefore, $\{u_{\zeta}\}_{\zeta \in \mathbb{R}}$ is a family of Class A minimizers of $\mathcal{F}^{a}_{1}$.  
	
	Finally, we prove that $U \in \mathcal{M}_{e}(\mathbb{R} \times \mathbb{T}^{d})$.  To start with, recall that $\partial_{s} U \geq 0$, hence each function in $\{u_{\zeta}\}_{\zeta \in \mathbb{R}}$ satisfies the $\mathbb{Z}^{d}$-Birkhoff property in the $e$ direction.  Moreover, due to elliptic regularity, the limit $\lim_{s \to \pm \infty} U(s,x) = \pm 1$ implies that each of these functions also satisfies \eqref{E: local convergence at infinity}.  Therefore, we can invoke Proposition \ref{P: annoying exponential convergence trick} to deduce that each member of the family $\{u_{\zeta}\}_{\zeta \in \mathbb{R}}$ converges to $\pm 1$ in $C^{1}$ at an exponential rate as $\langle x,e \rangle \to \pm \infty$.  Hence the functions $\{Du_{\zeta}\}_{\zeta \in \mathbb{R}}$ decay exponentially to zero as $\langle x,e \rangle \to \pm \infty$.  Since \eqref{A: W_assumption_2} is in effect, the same can be said of $\{W(u_{\zeta})\}_{\zeta \in \mathbb{R}}$.  Thus, for each $\zeta \in \mathbb{R}$, we have $\limsup_{R \to \infty} R^{1-d} \mathcal{F}^{a}_{1}(u_{\zeta};B(0,R)) < \infty$.  Since $\{u_{\zeta}\}_{\zeta \in \mathbb{R}}$ consists of Class A minimizers, we can argue as in Proposition \ref{P: surface tension} to find
		\begin{equation*}
			\lim_{R \to \infty} R^{1- d} \mathcal{F}^{a}_{1}(u_{\zeta}; B(0,R)) = \tilde{\varphi}^{a}(e) = \mathscr{E}^{a}(e) \quad \text{for each} \, \, \zeta \in \mathbb{R}.
		\end{equation*}
	Thus, by Theorem \ref{T: ergodic_lemma},
		\begin{equation*}
			\mathscr{T}^{a}_{e}(U) = \mathscr{E}^{a}(e)
		\end{equation*}
	and we conclude that $U \in \mathcal{M}_{e}(\mathbb{R} \times \mathbb{T}^{d})$. \qed
	     \end{proof}  
	     
\subsection{Proof of Theorem \ref{T: existence}}  For the reader's convenience, we show how the results of the previous two sections imply Theorem \ref{T: existence} and its corollary.

\begin{proof}[Proofs of Theorem \ref{T: existence} and Corollary \ref{C: new_existence}]  The existence of minimizers was proved in Proposition \ref{P: existence} and \ref{P: remove constraint}.  Statement (2) of the theorem is the result of Proposition \ref{P: uniqueness}, and statement (3) is that of Proposition \ref{P: surface tension}.  

Concerning Corollary \ref{C: new_existence}, Proposition \ref{P: continuity of the minimizers} covers statements (i)-(iv).  Statement (v) follows from Proposition \ref{P: surface tension}.  \qed 
\end{proof}  

\subsection{The mobility} \label{S: mobility}  In this section, we define the mobility $\tilde{M}^{a}$.  Following \cite{barles souganidis}, we would like to define it by $\tilde{M}^{a}(e) = \|\partial_{s} U_{e}\|^{2}_{L^{2}(\mathbb{R} \times \mathbb{T}^{d})}$ for some $U_{e} \in \mathcal{M}_{e}(\mathbb{R} \times \mathbb{T}^{k})$.  When no such minimizer exists with $\partial_{s} U_{e} \in L^{2}(\mathbb{R} \times \mathbb{T}^{d})$, we set $\tilde{M}^{a}(e) = \infty$.  However, this is not well-defined if $e \in S^{d-1} \cap \mathbb{R}\mathbb{Z}^{d}$.  

The difficulty becomes apparent following the construction in Section \ref{S: nonunique}.  Recall the definition of the period $m_{e}$ in \eqref{E: rational period}.  Suppose that $U \in \mathcal{M}_{e}(\mathbb{R} \times \mathbb{T}^{d})$ and $\|\partial_{s} U\|_{L^{2}(\mathbb{R} \times \mathbb{T}^{d})}^{2} < \infty$.  If we fix an $h \in H^{1}_{\text{loc}}(\mathbb{R})$ such that $h(\zeta + m_{e}) = h(\zeta) + m_{e}$, and we define $U^{(h)} \in \mathcal{M}_{e}(\mathbb{R} \times \mathbb{T}^{d})$ according to \eqref{E: time change}, then
	\begin{equation*}
		\|\partial_{s} U^{(h)}\|_{L^{2}(\mathbb{R} \times \mathbb{T}^{d})} = m_{e}^{-1} \int_{0}^{m_{e}} \left\{ \mathcal{H}^{d-1}(Q_{e})^{-1} \int_{Q_{e} \oplus_{e} \mathbb{R}} \partial_{\zeta} u_{\zeta}(x)^{2} \, dx \right\} |h'(\zeta)|^{2} \, d\zeta.
	\end{equation*}
Since the only restriction on the choice of $h$ is that $m_{e}^{-1} \int_{0}^{m_{e}} h'(\zeta) \, d\zeta = 0$, this quantity appears to depend on this choice in general.

Nonetheless, the following definition turns out to be useful.  To start with, given $e \in S^{d-1}$, define $\tilde{M}^{a}(e)$ by 
	\begin{equation}
		\tilde{M}^{a}(e) = \inf \left\{ \|\partial_{s} U\|_{L^{2}(\mathbb{R} \times \mathbb{T}^{d})}^{2} \, \mid \, U \in \mathcal{M}_{e}(\mathbb{R} \times \mathbb{T}^{d}) \right\}.
	\end{equation}
Let us emphasize that, in this definition, the infimum is only necessary when $e \in \mathbb{R} \mathbb{Z}^{d}$ since otherwise the minimizer is unique up to translation.

More generally, for $v \in \mathbb{R}^{d} \setminus \{0\}$, we extend the definition consistently with Remark \ref{R: extension_to_space}: 
	\begin{align}
		\tilde{M}^{a}(v) &= \|v\|^{-1} \tilde{M}^{a}(\|v\|^{-1} v) \nonumber \\
		&= \inf \left\{ \|\partial_{s} U\|_{L^{2}(\mathbb{R} \times \mathbb{T}^{d})}\|^{2} \, \mid \, U \in \mathscr{X}_{+}, \, \, \mathscr{E}^{a}(v) = \mathscr{T}^{a}_{v}(U) \right\}. \label{E: mobility}
	\end{align}
In Proposition \ref{P: regularization} below, we show that this is the mobility selected by elliptic regularization.  

\section{Analysis of the Laminar Case and Elliptic Regularization}  \label{S: einstein relation laminar media}  

The aim of this section is two-fold.  First, we analyze pulsating standing waves in the laminar setting, that is, when $a$ depends on only $k \leq d - 1$ of the $d$ variables.  As we will see, the laminar assumption unlocks a certain amount of regularity in specific directions, which can be exploited to prove a number of results of interest to us.  Most notably, this analysis will prepare the way for the analysis of the sharp interface limit and hence the proof of Theorem \ref{T: sharp_interface_limit_graphs}.

In addition to the laminar case, at the same time, we will also study the elliptic regularization procedure obtained by defining
	\begin{equation*}
		\mathscr{T}^{a,\delta}_{e}(V) = \int_{\mathbb{R} \times \mathbb{T}^{d}} \left( \frac{\delta}{2} |\partial_{s}V|^{2} + \frac{1}{2} \langle a(x) \mathcal{D}_{e} V, \mathcal{D}_{e} V \rangle + W(V) \right) \, dx \, ds.
	\end{equation*}
This is a classical approximation that has been studied in other contexts, and it is natural to seek to use it to draw conclusions about $\mathscr{T}^{a}_{e}$ and its minimizers.  

We prove below that the elliptically regularized functional has smooth minimizers and these provide approximations for both the surface tension $\tilde{\varphi}^{a}$ and mobility $\tilde{M}^{a}$.  We also provide evidence for an a priori upper bound on $(\tilde{M}^{a})^{-1} D^{2} \tilde{\varphi}^{a}$, and we prove Theorem \ref{T: no lower bound}, which establishes an obstruction to a matching lower bound.

The overarching reason that we study the laminar setting and elliptic regularization simultaneously in this section is, as we will see, the mathematical structures involved are closely related.

Throughout this section, assumptions \eqref{A: a_assumption_1}, \eqref{A: W_assumption_1}, \eqref{A: a_assumption_2}, and \eqref{A: W_assumption_2} are all in force. 

\subsection{Laminar Media} \label{S: laminar discussion}  We fix a $k \in \{1,2,\dots,d-1\}$ and assume that $a$ is $\mathbb{T}^{k} \times \mathbb{R}^{d - k}$-periodic, that is, \eqref{E: laminar} holds.
As was already mentioned previously, in this set-up it is convenient to replace $\mathbb{T}^{d}$ by $\mathbb{T}^{k}$ in cylindrical coordinates:
	\begin{equation} \label{E: hidden symmetry}
		\mathscr{T}^{a}_{e}(U) = \int_{\mathbb{R} \times \mathbb{T}^{k}} \left(\frac{\langle a(x) \mathcal{D}_{e} U, \mathcal{D}_{e}U \rangle}{2} + W(U) \right) \, dx \, ds, \quad (\mathcal{D}_{e} = e \partial_{s} + D_{x}).
	\end{equation}
Here $e \in S^{d -1} \subseteq \mathbb{R}^{d}$, $D_{x} = (\partial_{x_{1}},\dots,\partial_{x_{k}},0,\dots,0)$, and, by a slight abuse of notation, we treat $a$ as a function defined in $\mathbb{T}^{k}$.  In effect, we are identifying $\mathbb{T}^{k}$ everywhere with $\mathbb{T}^{k} \times \{0\} \subseteq \mathbb{T}^{d}$.  
	
Notice that if $e \notin S^{k - 1} \times \{0\}$, then $\mathscr{T}^{a}_{e}$ is actually uniformly elliptic in this setting.  Precisely, letting $P_{k}, P_{k}^{\perp} : \mathbb{R}^{d} \to \mathbb{R}^{d}$ denote the orthogonal projections onto $\mathbb{R}^{k} \times \{0\}$ and $\{0\} \times \mathbb{R}^{d - k}$, respectively, we can write
	\begin{align*}
		\frac{\langle a(x) \mathcal{D}_{e}U, \mathcal{D}_{e}U \rangle}{2} &\geq \frac{\lambda}{2} \left(\|P_{k}(e)\partial_{s}U + D_{x}U\|^{2} + \|P_{k}^{\perp}(e)\|^{2} |\partial_{s}U|^{2}\right)
			\geq 0
	\end{align*}
and equality holds if and only if $(\partial_{s}U,D_{x}U) = 0$.  

\subsection{Elliptic Regularization}  More generally, for $k \in \{1,2,\dots,d\}$ and $a$ satisfying \eqref{E: laminar}, $\delta \geq 0$, and $v \in \mathbb{R}^{d} \setminus \{0\}$, we will consider the elliptically regularized functional $\mathscr{T}^{a,\delta}_{v}(U)$ defined by 
	\begin{equation*}
		\mathscr{T}^{a,\delta}_{v}(U) = \int_{\mathbb{R} \times \mathbb{T}^{k}} \left(\frac{\delta}{2} |\partial_{s}U|^{2} + \frac{\langle a(x) \mathcal{D}_{v} U, \mathcal{D}_{v}U \rangle}{2} + W(U) \right) \, dx \, ds.
	\end{equation*}
Let us stress that here we include the case $k = d$.  

Define $\tilde{\varphi}^{a,\delta} : \mathbb{R}^{d} \setminus \{0\} \to [0,\infty)$ by 
	\begin{align}
		\tilde{\varphi}^{a,\delta}(v) &= \inf \left\{ \mathscr{T}_{v}^{a,\delta}(V) \, \mid \, V \in \mathscr{X}^{(k)} \right\}, \label{E: approximate surface tension definition} \\
		\mathscr{X}^{(k)} &= \left\{V \in L^{\infty}(\mathbb{R} \times \mathbb{T}^{k}; [-1,1]) \, \mid \, \lim_{s \to \pm \infty} T_{-s}V = \pm 1 \, \, \text{in} \, \, L^{1}_{\text{loc}}(\mathbb{R} \times \mathbb{T}^{k}) \right\}. \nonumber
	\end{align}
A straightforward argument involving Proposition \ref{P: weak symmetry} shows that $\tilde{\varphi}^{a,0} = \tilde{\varphi}^{a}$.

\subsection{Preliminaries}

In what follows, in order to tackle both the laminar and elliptic regularization contexts simultaneously, the following parameters will be useful.  We will fix a $k \in \{1,2,\dots,d\}$ for which \eqref{E: laminar} holds and choose a $\delta \geq 0$ and a compact set $K \subset \subset \mathbb{R}^{d} \setminus \{0\}$.  We have the following two sets of assumptions: 
	\begin{gather}
		\delta > 0, \quad \text{or}  \label{E: elliptic regularization assumption} \\
		k < d \quad \text{and} \quad K \subset \subset \mathbb{R}^{d} \setminus (\mathbb{R}^{k} \times \{0\}). \label{E: laminar assumption}
	\end{gather}
The laminar setting is covered by \eqref{E: laminar assumption}, while \eqref{E: elliptic regularization assumption} corresponds to elliptic regularization.  

In view of the preceding considerations, if one of \eqref{E: elliptic regularization assumption} or \eqref{E: laminar assumption} holds, then there are constants $\mu(\delta,K), M(\delta,K) > 0$ such that, for each $(t,p) \in \mathbb{R} \times (\mathbb{R}^{k} \times \{0\}) \subseteq \mathbb{R} \times \mathbb{R}^{d}$,
		\begin{equation} \label{E: extra ellipticity}
			\mu(\delta,K)(t^{2} + \|p\|^{2}) \leq \delta t^{2} + \langle a(x) (t v + p), tv + p \rangle \leq M(\delta,K) (t^{2} + \|p\|^{2}).
		\end{equation}
This uniform ellipticity will be essential in the results that follow.

First, due to uniform ellipticity, $\mathscr{T}^{a,\delta}_{v}$ fits into the framework of functionals considered in previous works, such as \cite{rabinowitz stredulinsky}.  Thus, as we prove next, it has smooth minimizers.  It is also important to notice that its coefficients are invariant under translations in the $s$ variable.  Hence, as we will see, these minimizers are unique up to translations.  The most subtle point that we will exploit, which leverages the regularity of the coefficients, is we can localize the zone in which minimizers transition from $1$ to $-1$ in a quantitative way.  This will be useful later when we differentiate minimizers with respect to the direction.

\begin{prop} \label{P: exponential decay pulsating} Fix $\delta \geq 0$, $k \in \{1,2,\dots,d\}$, and $K \subset \subset \mathbb{R}^{d} \setminus \{0\}$ satisfying \eqref{E: elliptic regularization assumption} or \eqref{E: laminar assumption}.  There is a unique $U_{v}^{\delta} \in UC(\mathbb{R} \times \mathbb{T}^{k} ; [-1,1])$ such that
	\begin{gather*}
		\mathscr{T}^{a,\delta}_{v}(U^{\delta}_{v}) = \tilde{\varphi}^{a,\delta}(v), \quad \lim_{s \to \pm \infty} U^{\delta}_{v}(s,x) = \pm 1, \quad \text{and} \quad \int_{\mathbb{T}^{k}} U_{v}^{\delta}(0,x) \, dx = 0.
	\end{gather*}
The map $v \mapsto U^{\delta}_{v}$ sends $K$ continuously into $BC(\mathbb{R} \times \mathbb{T}^{k})$ with the uniform norm topology.  Further, there is a constant $C > 1$ depending only on $K$, $a$, $W$, and $\delta$ such that 
	\begin{align}
		1 - U^{\delta}_{v}(s,x) &\leq C e^{-C^{-1}s}, \label{E: exponential pulsating estimate 1}\\
		U^{\delta}_{v}(s,x) +1  &\leq C e^{C^{-1}s}, \label{E: exponential pulsating estimate 2} \\
		|U^{\delta}_{v}(s,x) - U^{\delta}_{v}(s',x')| &\leq C \left(\|x - x'\| + |s - s'|\right)^{\frac{1}{C}}. \label{E: holder estimate}
	\end{align}
\end{prop}  

The proposition follows from \cite[Theorem 2.3]{valdinoci de la llave}, which, in turn, is a consequence of results in \cite{valdinoci}.  The main ingredient that we need that is not made explicit in \cite{valdinoci de la llave} is the dependence of the constants.  Accordingly, for the reader's convenience, we give a complete proof based on \cite{valdinoci} in Appendix \ref{A: annoying uniformly elliptic part}.

Notice that the minimization property of $U^{\delta}_{v}$ implies that it is a weak solution of the uniformly elliptic PDE
	\begin{equation}
		- \delta^{2} \partial_{s}^{2} U^{\delta}_{v} + \mathcal{D}_{v}^{*}(a(x) \mathcal{D}_{v} U^{\delta}_{v}) + W'(U^{\delta}_{v}) = 0 \quad \text{in} \, \, \mathbb{R} \times \mathbb{T}^{k}. \label{E: pulsating standing wave with regularization}
	\end{equation}
Much of the remainder of the section proceeds by exploiting properties of the linearized version of this equation.

To start with, we observe that \eqref{A: W_assumption_2} implies an exponential decay property for solutions of the linearized equation.

	\begin{prop} \label{P: linear equation estimate}  Fix $\delta \geq 0$, $k \in \{1,2,\dots,d\}$, and $K \subset \subset \mathbb{R}^{d} \setminus \{0\}$ satisfying \eqref{E: elliptic regularization assumption} or \eqref{E: laminar assumption}.  Suppose that $F \in C_{0}(\mathbb{R} \times \mathbb{T}^{k})$ and $G \in BC(\mathbb{R} \times \mathbb{T}^{k})$ are such that, for some $\Gamma > 0$ and $\alpha_{0} \in (0,\alpha)$,
		\begin{align*}
			|F(s,x)| &\leq \Gamma^{-1} \|F\|_{L^{\infty}(\mathbb{R} \times \mathbb{T}^{k})} e^{-\Gamma |s|} \quad \text{for} \, \, (s,x) \in \mathbb{R} \times \mathbb{T}^{k}
		\end{align*}
and $\|G\|_{L^{\infty}(\mathbb{R} \times \mathbb{T}^{k})} \leq \alpha_{0}$.  If $Q \in C_{0}(\mathbb{R} \times \mathbb{T}^{k}) \cap L^{2}(\mathbb{R} \times \mathbb{T}^{k})$ is a viscosity solution of the PDE
		\begin{equation} \label{E: linearized PDE}
			- \delta \partial_{s}^{2} Q + \mathcal{D}_{v}^{*}(a(x) \mathcal{D}_{v}Q) + (W''(U_{v}^{\delta}) + G) Q = F \quad \text{in} \, \, \mathbb{R} \times \mathbb{T}^{k},
		\end{equation}
	then there is a $\beta > 0$ depending only on $\delta$, $W$, $\alpha_{0}$, $K$, $\Gamma$, the ellipticity constants $\lambda$ and $\Lambda$ of \eqref{A: a_assumption_1}, and the constant $C$ of Proposition \ref{P: exponential decay pulsating} such that 
		\begin{equation} \label{E: better estimate}
			|Q(s,x)| \leq \beta^{-1} (\|F\|_{L^{\infty}(\mathbb{R} \times \mathbb{T}^{k})} \vee 1) \|Q\|_{L^{\infty}(\mathbb{R} \times \mathbb{T}^{k})} e^{-\beta |s|} \quad \text{for} \, \, (s,x) \in \mathbb{R} \times \mathbb{T}^{k}.
		\end{equation}
	\end{prop}  
	
	Note that it is not possible to remove the $\|Q\|_{L^{\infty}(\mathbb{R} \times \mathbb{T}^{k})}$ term from the right-hand side of \eqref{E: better estimate}.  Indeed, as we will see below, $\partial_{s} U^{\delta}_{v}$ solves \eqref{E: linearized PDE} with $F = G = 0$.  Hence, when $G = 0$, we can always replace $Q$ by $Q + c \partial_{s} U^{\delta}_{v}$ for some $c \in \mathbb{R}$, which makes $\|Q\|_{L^{\infty}(\mathbb{R} \times \mathbb{T}^{d})}$ arbitrarily large without changing $F$.  
	
		\begin{proof}  First, observe that it suffices to prove the result when $\|F\|_{L^{\infty}(\mathbb{R} \times \mathbb{T}^{k})} \leq 1$.  Indeed, this is the case since, by linearity, we can replace $F$ by $(\|F\|_{L^{\infty}(\mathbb{R} \times \mathbb{T}^{k})} \vee 1)^{-1}F$ and $Q$ by $(\|F\|_{L^{\infty}(\mathbb{R} \times \mathbb{T}^{k})} \vee 1)^{-1} Q$.
		
		By the exponential estimates \eqref{E: exponential pulsating estimate 1} and \eqref{E: exponential pulsating estimate 2} of Proposition \ref{P: exponential decay pulsating}, assumption \eqref{A: W_assumption_2}, and the assumption that $\|G\|_{L^{\infty}(\mathbb{R} \times \mathbb{T}^{k})} < \alpha$, there is an $M > 0$ depending only on $C$ and $W$ such that 
			\begin{equation*}
				W''(U^{\delta}_{v}(s,x)) + G(s,x) \geq \frac{1}{2}(\alpha - \alpha_{0}) \quad \text{if} \, \, (s,x) \in (\mathbb{R} \setminus [-M,M]) \times \mathbb{T}^{k}.
			\end{equation*}
		We will use this fact to obtain an upper bound on $Q$ in $(\mathbb{R} \setminus [-M,M]) \times \mathbb{T}^{k}$.  It is convenient to start by working in $[M,\infty) \times \mathbb{T}^{k}$.  
		
		We seek a bound of the form $\overline{Q}(s) = \beta^{-1} \|Q\|_{L^{\infty}(\mathbb{R} \times \mathbb{T}^{k})} e^{-\beta s}$.  Differentiating, we note that
			\begin{align*}
				-\delta \partial_{s}^{2} \overline{Q} &+ \mathcal{D}_{v}^{*}(a(x) \mathcal{D}_{v} \overline{Q}) + (W''(U^{\delta}_{v}) + G(s,x)) \overline{Q} \\
				&\geq \left(-\delta \beta^{2} - \Lambda \|v\|^{2} \beta^{2} - \|\text{div} \, a\|_{L^{\infty}(\mathbb{T}^{k})} \|v\| \beta + \frac{1}{2} \left(\alpha - \alpha_{0} \right) \right) \overline{Q} \\ 
				&\quad \text{in} \, \, (M,\infty) \times \mathbb{T}^{k}.
			\end{align*}
		Hence, fixing $\beta > 0$ small enough depending only on $\delta$, $\Lambda$, $\|\text{div} a\|_{L^{\infty}(\mathbb{T}^{k})}$, $\|F\|_{L^{\infty}(\mathbb{R} \times \mathbb{T}^{k})}$, $\alpha - \alpha_{0}$, and $\Gamma$, we obtain
			\begin{align*}
				-\delta \partial_{s}^{2} \overline{Q} + \mathcal{D}_{v}^{*}(a(x) \mathcal{D}_{v} \overline{Q}) + (W''(U^{\delta}_{v}) + G) \overline{Q} &> \Gamma^{-1} e^{-\Gamma s} \quad \text{in} \, \, (M,\infty) \times \mathbb{T}^{k}.
			\end{align*}
		Further, making $\beta$ smaller if necessary, we have $\overline{Q} \geq Q$ on $\{M\} \times \mathbb{T}^{k}$.  Applying the maximum principle to $\overline{Q} - Q$, we conclude that $Q \leq \overline{Q}$ in $[M,\infty) \times \mathbb{T}^{k}$.  
		
	A similar argument applies in the interval $(-\infty,-M] \times \mathbb{T}^{k}$.  \qed \end{proof}  
	
When it comes to studying solutions of the linearized equation, it will be important to note that the derivative $\partial_{s} U^{\delta}_{v}$ is a positive element of the kernel of the linearized operator.  The next result contains the main properties of this function.
	
	\begin{prop} \label{P: eigenfunction_estimate}  If $\delta \geq 0$, $k \in \{1,2,\dots,d\}$, and $K \subset \subset \mathbb{R}^{d} \setminus \{0\}$ are chosen such that \eqref{E: elliptic regularization assumption} or \eqref{E: laminar assumption} hold and $U^{\delta}_{v}$ is the minimizer of $\mathscr{T}^{\delta}_{v}$ from Proposition \ref{P: exponential decay pulsating}, then $\partial_{s} U^{\delta}_{v} \in L^{2}(\mathbb{R} \times \mathbb{T}^{k})$ and $\partial_{s} U_{v}^{\delta} > 0$ in $\mathbb{R} \times \mathbb{T}^{k}$.  Furthermore, there is a constant $\gamma > 0$ depending only on $\delta$, $W$, $K$, the ellipticity constants of \eqref{A: a_assumption_1}, and the constant $C$ of Proposition \ref{P: exponential decay pulsating} such that
		\begin{equation}
			\partial_{s}U_{v}^{\delta}(s,x) \leq \gamma^{-1} \sqrt{\mu(\delta,K)^{-1} \tilde{\varphi}^{a,\delta}(v)} e^{-\gamma |s|} \quad \text{for} \, \, (s,x) \in \mathbb{R} \times \mathbb{T}^{k}. \label{E: eigenfunction estimate}
		\end{equation}
	\end{prop}
	
	\begin{proof}  We start by proving that $\partial_{s} U^{\delta}_{v}$ is non-negative, then we show it is in $L^{2}(\mathbb{R} \times \mathbb{T}^{k})$, and we conclude by proving positivity and the exponential estimate.  
	
	\textbf{Step 1: $\partial_{s}U^{\delta}_{v} \geq 0$}  
	
	First, we show that $\partial_{s} U_{v}^{\delta} \geq 0$ in $\mathbb{R} \times \mathbb{T}^{k}$.  Indeed, given $s_{0} > 0$, if we define $U(s,x) = U^{\delta}_{v}(s + s_{0},x)$, then the corresponding properties of $U^{\delta}_{v}$ imply that
		\begin{equation*}
			\mathscr{T}^{a,\delta}_{v}(U) < \infty \quad \text{and} \quad \mathscr{T}^{a,\delta}_{v}(U) \leq \mathscr{T}^{a,\delta}_{v}(U + f) \quad \text{if} \, \, f \in C^{\infty}_{c}(\mathbb{R} \times \mathbb{T}^{k}; [-1,1]).
		\end{equation*}
	Furthermore, a straightforward computation shows that the minimum $U_{m} = U^{\delta}_{v} \wedge U$ and the maximum $U_{M} = U^{\delta}_{v} \vee U$ have the same properties.  In particular, $U_{m}$ and $U_{M}$ are weak solutions of the Euler-Lagrange equation $- \delta \partial_{s}^{2} U_{\bullet} + \mathcal{D}_{v}^{*}(a(x) \mathcal{D}_{v} U_{\bullet}) + W'(U_{\bullet}) = 0$ in $\mathbb{R} \times \mathbb{T}^{k}$ and $U_{m} \leq U_{M}$.  By the strong maximum principle, either $U_{m} < U_{M}$ everywhere or else $U_{m} = U_{M}$.  
	
	If $U_{m} = U_{M}$, then $U^{\delta}_{v}(s + n s_{0},x) = U^{\delta}_{v}(s,x)$ for each $(s,x) \in \mathbb{R} \times \mathbb{T}^{k}$ and $n \in \mathbb{N}$.  This is impossible since $U^{\delta}_{v}(s,x) \to \pm 1$ as $s \to \pm \infty$.  Therefore, $U_{m} < U_{M}$ holds everywhere in $\mathbb{R} \times \mathbb{T}^{k}$.  Hence either $U^{\delta}_{v}(s,x) < U^{\delta}_{v}(s + s_{0},x)$ in $\mathbb{R} \times \mathbb{T}^{k}$ or $U^{\delta}_{v}(s,x) > U^{\delta}_{v}(s + s_{0},x)$.  Again, the asymptotic behavior of $U^{\delta}_{v}$ as $s \to \pm \infty$ rules out all but one option.  We conclude that $U^{\delta}_{v}(s + s_{0},x) > U^{\delta}_{v}(s,x)$ for all $(s,x) \in \mathbb{R} \times \mathbb{T}^{k}$.  
	
	Since $s_{0} > 0$ was arbitrary, we conclude that $\partial_{s} U^{\delta}_{v} \geq 0$ holds in $\mathbb{R} \times \mathbb{T}^{k}$.  
	
	\textbf{Step 2: $\partial_{s} U^{\delta}_{v} \in L^{2}(\mathbb{R} \times \mathbb{T}^{k})$} 
	
	We can invoke \eqref{E: extra ellipticity} directly to obtain the estimate
		\begin{equation} \label{E: easy ellipticity thing}
			\mu(\delta,K) \int_{\mathbb{R} \times \mathbb{T}^{k}} |\partial_{s} U^{\delta}_{v}|^{2} \, dx \,ds \leq \mathscr{T}^{a,\delta}_{v}(U^{\delta}_{v}) < \infty.
		\end{equation}
	This proves $\partial_{s}U_{v}^{\delta} \in L^{2}(\mathbb{R} \times \mathbb{T}^{k})$.
	
	 \textbf{Step 3: Positivity and Exponential estimate}  
	 
	 Finally, observe that $V^{\delta}_{v} := \partial_{s}U_{v}^{\delta}$ is a weak solution of the uniformly elliptic PDE 
	 	\begin{equation*}
			-\delta_{s}^{2}V^{\delta}_{v} + \mathcal{D}_{v}^{*}(a(x)\mathcal{D}_{v} V_{v}^{\delta}) + W''(U_{v}^{\delta}) V_{v}^{\delta} = 0 \quad \text{in} \, \, \mathbb{R} \times \mathbb{T}^{k}.
		\end{equation*}
	Therefore, the Harnack inequality \cite[Theorem 8.20]{Gilbarg Trudinger} implies that $V^{\delta} > 0$.
	
	By the local maximum principle (see \cite[Theorem 4.1]{Han Lin} or \cite[Theorem 8.17]{Gilbarg Trudinger}), there is a constant $D > 0$ depending only on $W$ and the ellipticity constants $\mu(\delta,K)$ and $M(\delta,K)$ such tha
		\begin{equation} \label{E: local maximum principle}
			\|V^{\delta}_{v}\|_{L^{\infty}(\mathbb{R} \times \mathbb{T}^{k})} \leq D \|V^{\delta}_{v}\|_{L^{2}(\mathbb{R} \times \mathbb{T}^{k})}.
		\end{equation}
	This proves $V^{\delta}_{v} \in L^{\infty}(\mathbb{R} \times \mathbb{T}^{k})$.  Elliptic regularity now implies that $V^{\delta}_{v} \in UC(\mathbb{R} \times \mathbb{T}^{k}) \cap L^{2}(\mathbb{R} \times \mathbb{T}^{k})$, and hence $V^{\delta}_{v} \in C_{0}(\mathbb{R} \times \mathbb{T}^{k})$.  Combining this and \eqref{E: local maximum principle} with Proposition \ref{P: linear equation estimate}, we obtain an estimate of the form
		\begin{equation*}
			V_{v}^{\delta}(s,x) \leq \gamma^{-1} \|V_{v}^{\delta}\|_{L^{2}(\mathbb{R} \times \mathbb{T}^{k})} e^{-\gamma |s|} \quad \text{for} \, \, (s,x) \in \mathbb{R} \times \mathbb{T}^{k}
		\end{equation*}
	We conclude by observing that \eqref{E: easy ellipticity thing} implies that
		\begin{equation*}
			\mu(\delta,K) \|V_{v}^{\delta}\|_{L^{2}(\mathbb{R} \times \mathbb{T}^{k})}^{2} \leq \mathscr{T}^{a,\delta}_{v}(U^{\delta}_{v}) = \tilde{\varphi}^{a,\delta}(v).
		\end{equation*}   \qed	\end{proof}  
	
\subsection{Analysis of $\mathcal{L}^{\delta}_{v}$}  In this section, we analyze the operator $\mathcal{L}_{v}^{\delta}$ obtained by linearizing the pulsating wave equation around $U_{v}^{\delta}$.  More precisely, for $v \in \mathbb{R}^{d} \setminus \{0\}$, we define the unbounded operator $\mathcal{L}_{v}^{\delta}$ in $L^{2}(\mathbb{R} \times \mathbb{T}^{k})$ as follows:
	\begin{equation*}
		\left\{ \begin{array}{l}
			D(\mathcal{L}_{v}^{\delta}) = H^{2}(\mathbb{R} \times \mathbb{T}^{k}), \\
			\mathcal{L}^{\delta}_{v}\Phi = -\delta \partial_{s}^{2} \Phi + \mathcal{D}_{v}^{*}(a(x) \mathcal{D}_{v}\Phi) + W''(U_{v}^{\delta}) \Phi.
		\end{array} \right.
	\end{equation*}

Throughout the remainder of this section, we will write $V_{v}^{\delta} := \partial_{s} U_{v}^{\delta}$ for convenience.  As was noted already in the proof of Proposition \ref{P: eigenfunction_estimate}, this function is in the kernel of $\mathcal{L}^{\delta}_{v}$, that is, $\mathcal{L}^{\delta}_{v} V^{\delta}_{v} = 0$.  

The first result we will need is useful representation of the quadratic form determined by $\mathcal{L}^{\delta}_{v}$.

\begin{prop} \label{P: non-negative_operator}  If $\Phi \in H^{2}(\mathbb{R} \times \mathbb{T}^{k})$ and $\Psi = (V_{v}^{\delta})^{-1} \Phi$, then
	\begin{align*}
		\int_{\mathbb{R} \times \mathbb{T}^{k}} &\left( \delta |\partial_{s} \Phi|^{2} + \langle a(x) \mathcal{D}_{v} \Phi, \mathcal{D}_{v}\Phi \rangle + W''(U_{v}^{\delta}) |\Phi|^{2} \right) \, dx \, ds \\
			&\qquad \qquad \qquad \qquad = \int_{\mathbb{R} \times \mathbb{T}^{k}} \left( \delta |\partial_{s} \Psi|^{2} + \langle a(x) \mathcal{D}_{v} \Psi, D_{v}\Psi \rangle \right) |V_{v}^{\delta}|^{2} \, dx \, ds.
	\end{align*}
\end{prop}  
	
	\begin{proof}  When $\Phi$ is smooth, this is a classical argument involving integration-by-parts and the fact that $V_{v}^{\delta}$ is a positive eigenfunction of $\mathcal{L}_{v}^{\delta}$.  The general case follows by approximation. \qed \end{proof}

Finally, we will need the following result to construct the correctors used in the analysis of the sharp interface limit.

	\begin{prop} \label{P: linear_eqn} $\mathcal{L}_{v}^{\delta}$ is closed, self-adjoint, and $\text{Ker}(\mathcal{L}_{v}^{\delta}) = \langle V_{v}^{\delta} \rangle$.  Moreover, $\text{Ran}(\mathcal{L}_{v}^{\delta}) = \langle V_{v}^{\delta} \rangle^{\perp}$.  \end{prop}

\begin{proof}	Define $\tilde{\alpha} : \mathbb{R} \setminus \{0\} \to (0,\infty)$ by $\tilde{\alpha}(s) = W''(\text{sgn}(s))$ and let $\mathcal{H}^{\delta}_{v}$ on $L^{2}(\mathbb{R} \times \mathbb{T}^{k})$ be the unbounded operator with domain $H^{2}(\mathbb{R} \times \mathbb{T}^{k})$ given by
		\begin{equation*}
			\mathcal{H}^{\delta}_{v}\Phi = -\delta \partial_{s}^{2} \Phi + \mathcal{D}_{v}^{*}(a(x) \mathcal{D}_{v}\Phi) + \tilde{\alpha}(s) \Phi.
		\end{equation*}

	From \eqref{A: a_assumption_1}, \eqref{A: a_assumption_2}, \eqref{E: elliptic regularization assumption}, and \eqref{E: laminar assumption}, we can show that $\mathcal{H}^{\delta}_{v}$ is a closed operator.  Indeed, by $L^{2}$ estimates for uniformly elliptic equations (cf.\ \cite[Theorem 9.11]{Gilbarg Trudinger} or \cite[Section 6.3.1]{evans}), 
		\begin{equation*}
			\|\Phi\|^{2}_{H^{2}([n,n+1] \times \mathbb{T}^{k})} \leq C(\|\Phi\|_{L^{2}([n-1,n+2] \times \mathbb{T}^{k})}^{2} + \|\mathcal{H}^{\delta}_{v}\Phi\|_{L^{2}([n-1,n+2] \times \mathbb{T}^{k})}^{2})
		\end{equation*}  
	Summing over $n$, we find
		\begin{equation*}
			\|\Phi\|^{2}_{H^{2}(\mathbb{R} \times \mathbb{T}^{k})} \leq C(\|\Phi\|_{L^{2}(\mathbb{R} \times \mathbb{T}^{k})}^{2} + \|\mathcal{H}^{\delta}_{v}\Phi\|_{L^{2}(\mathbb{R} \times \mathbb{T}^{k})}^{2})
		\end{equation*}
	Thus, the graph of $\mathcal{H}^{\delta}_{v}$ is closed in $L^{2}(\mathbb{R} \times \mathbb{T}^{k}) \times L^{2}(\mathbb{R} \times \mathbb{T}^{k})$, and $\mathcal{H}^{\delta}_{v}$ is a closed operator.  
		
		Since $W''(1) \wedge W''(-1) > 0$ by \eqref{A: W_assumption_2}, the operator $(\mathcal{H}^{\delta}_{v})^{-1} : L^{2}(\mathbb{R} \times \mathbb{T}^{k}) \to H^{2}(\mathbb{R} \times \mathbb{T}^{k})$ exists and is bounded.
	
	Observe that we can write $\mathcal{L}_{v}^{\delta} = \mathcal{H}^{\delta}_{v} + M^{\delta}_{v}$, where $M^{\delta}_{v} \Phi = (W''(U_{v}^{\delta}) - \tilde{\alpha}(s))\Phi$ is a bounded linear operator on $L^{2}(\mathbb{R} \times \mathbb{T}^{k})$.  In particular, $\mathcal{L}^{\delta}_{v} = (\text{Id} + M_{v}^{\delta} (\mathcal{H}^{\delta}_{v})^{-1}) \mathcal{H}^{\delta}_{v}$.  Since $(\mathcal{H}^{\delta}_{v})^{-1}$ takes $L^{2}(\mathbb{R} \times \mathbb{T}^{k})$ continuously into $H^{2}(\mathbb{R} \times \mathbb{T}^{k})$ and $W''(U_{e}) - \tilde{\alpha}(s) \to 0$ uniformly as $|s| \to \infty$, it follows that $M_{\alpha}^{\delta} (\mathcal{H}^{\delta}_{v})^{-1}$ is compact.  Therefore, by the Fredholm alternative, $\text{Id} + M_{\alpha}^{\delta} (\mathcal{H}^{\delta}_{v})^{-1}$ is a closed operator with closed range.  Since $\mathcal{L}^{\delta}_{v} = (\text{Id} + M_{\alpha}^{\delta} (\mathcal{H}^{\delta}_{v})^{-1})\mathcal{H}^{\delta}_{v}$, we deduce that $\mathcal{L}_{v}^{\delta}$ is also.  
	
	$\mathcal{L}_{v}^{\delta}$ is clearly symmetric.  Therefore, to prove it is self-adjoint, it is only necessary to show that $D((\mathcal{L}_{v}^{\delta})^{*}) = D(\mathcal{L}_{v}^{\delta})$.  This follows, for example, by mollification.  
	
	The previous proposition showed $\text{Ker}(\mathcal{L}_{v}^{\delta}) = \langle V_{v}^{\delta} \rangle$.  Finally, since $\mathcal{L}_{v}^{\delta}$ is self-adjoint with closed range, the identity
$\text{Ran}(\mathcal{L}_{v}^{\delta}) = \langle V_{v}^{\delta} \rangle^{\perp}$ follows. \qed
\end{proof}

\subsection{Derivatives with respect to $v$}

Since our eventual goal is to obtain formulas for the first and second derivatives of $\tilde{\varphi}^{a,\delta}$, in this section, we differentiate the map $v \mapsto U^{\delta}_{v}$.  To start with, we fix $\xi \in \mathbb{R}^{d}$ and $h \in \mathbb{R}$ and define $R^{\delta,\xi}_{v,h}$ by 
	\begin{equation*}
		R^{\delta,\xi}_{v,h}(s,x) = \frac{U^{\delta}_{v + h \xi}(s,x) - U^{\delta}_{v}(s,x)}{h}.
	\end{equation*}
	
The following result follows from a direct manipulation of the equations solved by $U^{\delta}_{v + h \xi}$ and $U^{\delta}_{v}$.

	\begin{prop}  $R^{\delta,\xi}_{v,h}$ satisfies the PDE
		\begin{equation*}
			\left\{ \begin{array}{r l}
				\mathcal{L}^{\delta}_{e}R^{\delta,\xi}_{v,h} = B^{\delta}_{h} R^{\delta,\xi}_{v,h} + K^{\delta}_{h} & \text{in} \, \, \mathbb{R} \times \mathbb{T}^{k}, \\
				\int_{\mathbb{T}^{k}} R^{\delta,\xi}_{v,h}(0,x) \, dx = 0,
				\end{array} \right.
		\end{equation*}
	where $B^{\delta}_{h}$ and $K^{\delta}_{h}$ are given by 
		\begin{align*}
			B^{\delta}_{h} &= - \int_{0}^{1} \{W''(U^{\delta}_{v} + t[U^{\delta}_{v + h \xi} - U^{\delta}_{v}]) - W''(U^{\delta}_{v})\} \, dt, \\
			K^{\delta}_{h} &= \langle \xi, a(x) \mathcal{D}_{v}V^{\delta}_{v + h \xi}\rangle + \langle \xi, a(x) \mathcal{D}_{v + h \xi}V^{\delta}_{v + h \xi}\rangle + \langle \text{div}\, a, \xi \rangle V^{\delta}_{v + h \xi}.
		\end{align*}
	The sequence $(K^{\delta}_{h})_{h \in (-1,1)}$ is uniformly bounded in $C_{0}(\mathbb{R} \times \mathbb{T}^{k})$ and 
		\begin{equation*}
			\lim_{h \to 0} \|B^{\delta}_{h}\|_{L^{\infty}(\mathbb{R} \times \mathbb{T}^{k})} = 0.
		\end{equation*}
	\end{prop}  
	
Now we use uniform ellipticity to pass to the limit $h \to 0$.

	\begin{prop} \label{P: differentiability_of_waves}  The limit $R^{\delta,\xi}_{v} = \lim_{h \to 0} R^{\delta,\xi}_{v,h}$ exists in $C^{2}_{0}(\mathbb{R} \times \mathbb{T}^{k})$ and $L^{2}(\mathbb{R} \times \mathbb{T}^{k})$.  Moreover, $R^{\delta,\xi}_{v}$ is the unique solution of the equation
		\begin{equation*}
			\left\{ \begin{array}{r l}
				\mathcal{L}^{\delta}_{v}R^{\delta,\xi}_{v} = 2 \langle a(x) \xi, \mathcal{D}_{v} V^{\delta}_{v} \rangle + \langle \text{div} \, a, \xi \rangle V^{\delta}_{v} & \text{in} \, \, \mathbb{R} \times \mathbb{T}^{k}, \\
				\int_{\mathbb{R} \times \mathbb{T}^{k}} R^{\delta,\xi}_{v}(0,x) \, dx = 0.
			\end{array} \right.
		\end{equation*}
	\end{prop}

	\begin{proof}  The main technicality in the proof is that we need to work around the kernel of $\mathcal{L}_{v}^{\delta}$.  Since $(R^{\delta,\xi}_{v,h})_{h \in \mathbb{R}} \subseteq L^{2}(\mathbb{R} \times \mathbb{T}^{k})$, we can fix $(Q^{\delta,\xi}_{v,h})_{h \in \mathbb{R}} \subseteq \langle V_{e}\rangle^{\perp}$ and $(c^{\delta}_{h})_{h \in \mathbb{R}}$ such that 
		\begin{equation*}
			R^{\delta,\xi}_{v,h} = c^{\delta}_{h} V^{\delta}_{v} + Q^{\delta,\xi}_{v,h}.
		\end{equation*}
	
	Since $U^{\delta}_{v + h \xi} \to U^{\delta}_{v}$ uniformly as $h \to 0$, we can fix a $\eta > 0$ such that 
		\begin{equation} \label{E: generate_exponential_estimates}
			\|B^{\delta}_{h}\|_{L^{\infty}(\mathbb{R} \times \mathbb{T}^{k})} \leq \frac{\alpha}{2} < W''(1) \wedge W''(-1) \quad \text{if} \, \,  |h| < \eta.
		\end{equation}
Notice that with this choice of $\eta$, the family $(R^{\delta,\xi}_{v,h})_{h \in (-\eta,\eta)}$ satisfies the hypotheses of Proposition \ref{P: eigenfunction_estimate}.  Hence there is a $C(\eta) > 0$ such that, for each $h \in (-\eta,\eta)$,
	\begin{equation} \label{E: exponential_estimate}
		|R^{\delta}_{v,h}(s,x)| \leq C(\eta) \|R^{\delta}_{v,h}\|_{L^{\infty}(\mathbb{R} \times \mathbb{T}^{k})} \exp \left( - C(\eta)^{-1} |s| \right).
	\end{equation}
	
	We claim that $(Q^{\delta,\xi}_{v,h})_{h \in (-\eta,\eta)}$ is pre-compact in $C_{0}(\mathbb{R} \times \mathbb{T}^{k})$ and $(c^{\delta}_{h})_{h \in (-\eta,\eta)}$ is bounded in $\mathbb{R}$.  To see this, we first prove that $\limsup_{h \to 0} \|R^{\delta,\xi}_{v,h}\|_{L^{\infty}(\mathbb{R} \times \mathbb{T}^{k})} < \infty$.  
	
	Assume to the contrary that there is a sequence $(h_{n})_{n \in \mathbb{N}} \subseteq (-\eta,\eta)$ such that $\lim_{n \to \infty} \|R^{\delta,\xi}_{v,h_{n}}\|_{L^{\infty}(\mathbb{R} \times \mathbb{T}^{k})} = \infty$.  
	
	Define $(\tilde{R}_{n})_{n \in \mathbb{N}}$ by $\tilde{R}_{n} = \|R^{\delta,\xi}_{v,h_{n}}\|^{-1}_{L^{\infty}(\mathbb{R} \times \mathbb{T}^{k})} R^{\delta,\xi}_{v,h_{n}}$.  Notice that $\tilde{R}_{n}$ satisfies the PDE 
		\begin{equation} \label{E: contra_eqn}
			\mathcal{L}^{\delta}_{v}\tilde{R}_{n} = B^{\delta}_{h_{n}} \tilde{R}_{n} + \|R^{\delta,\xi}_{v,h_{n}}\|^{-1}_{L^{\infty}(\mathbb{R} \times \mathbb{T}^{k})} K^{\delta}_{h_{n}}.
		\end{equation}
Thus, Schauder estimates (cf.\ \cite[Theorem 6.2]{Gilbarg Trudinger}) imply $(\tilde{R}_{n})_{n \in \mathbb{N}}$ is bounded in $C^{2,\mu}_{0}(\mathbb{R} \times \mathbb{T}^{k})$ for some $\mu \in (0,1)$.  

We claim that $(\tilde{R}_{n})_{n \in \mathbb{N}}$ is pre-compact in $C^{2}_{0}(\mathbb{R} \times \mathbb{T}^{k})$.  	Let us write
	\begin{align*}
		\tilde{R}_{n} &= \tilde{Q}_{n} + \tilde{c}_{n} V_{v}^{\delta}, \\
		\tilde{Q}_{n} &= \|R^{\delta,\xi}_{v,h_{n}}\|^{-1}_{L^{\infty}(\mathbb{R} \times \mathbb{T}^{k})} Q^{\delta,\xi}_{v,h_{n}}, \\
		\tilde{c}_{n} &= \|R^{\delta,\xi}_{v,h_{n}}\|^{-1}_{L^{\infty}(\mathbb{R} \times \mathbb{T}^{k})} c^{\delta}_{h_{n}}.
	\end{align*}
Notice that $\|\tilde{R}_{n}\|_{L^{2}(\mathbb{R} \times \mathbb{T}^{k})}^{2} = \|\tilde{Q}_{n}\|_{L^{2}(\mathbb{R} \times \mathbb{T}^{k})}^{2} + |\tilde{c}_{n}|^{2} \|V_{e}\|_{L^{2}(\mathbb{R} \times \mathbb{T}^{k})}^{2}$.  Moreover, our previous exponential estimate translates to the following one:
	\begin{equation} 
		|\tilde{R}_{n}(s,x)| \leq C(\eta) e^{-C(\eta)^{-1} |s|} \quad \text{for} \, \, (s,x) \in \mathbb{R} \times \mathbb{T}^{k}.
	\end{equation}
Combining this bound with the identity $\|\tilde{R}_{n}\|_{L^{\infty}(\mathbb{R} \times \mathbb{T}^{k})} = 1$, we can invoke Schauder estimates to deduce that $(\tilde{R}_{n})_{n \in \mathbb{N}}$ is pre-compact in both $C^{2}_{0}(\mathbb{R} \times \mathbb{T}^{k})$ and $L^{2}(\mathbb{R} \times \mathbb{T}^{k})$.  An immediate consequence is the boundedness of $(\tilde{c}_{n})_{n \in \mathbb{N}}$.

By compactness, we can assume without loss of generality (i.e.\ by passing to a sub-sequence) that there is an $\tilde{R} \in C^{2}_{0}(\mathbb{R} \times \mathbb{T}^{k})$ and a $\tilde{c} \in \mathbb{R}$ such that $\tilde{R}_{n} \to \tilde{R}$ in $C^{2}_{0}(\mathbb{R} \times \mathbb{T}^{k})$ and $\tilde{c}_{n} \to \tilde{c}$.  Passing to the limit in \eqref{E: contra_eqn} and recalling that $B_{h} \to 0$ uniformly, we find
	\begin{equation*}
		\mathcal{L}_{v}^{\delta}\tilde{R} = 0.
	\end{equation*}   
Thus, $\tilde{R} = \tilde{c} V_{v}^{\delta}$ by Proposition \ref{P: linear_eqn}.  On the other hand,
	\begin{equation*}
		\tilde{c} \int_{\mathbb{T}^{k}} V_{v}^{\delta}(0,x) \, dx = \lim_{n \to \infty}  \int_{\mathbb{T}^{k}} \tilde{R}_{n}(0,x) \, dx = 0.
	\end{equation*} 
Since $V_{v}^{\delta} > 0$ in $\mathbb{R} \times \mathbb{T}^{k}$, this can only mean that $\tilde{c} = 0$, and hence $\tilde{R} = 0$.  This is a contradiction, however, since $\|\tilde{R}\|_{L^{\infty}(\mathbb{R} \times \mathbb{T}^{k})} = \lim_{n \to \infty} \|\tilde{R}_{n}\|_{L^{\infty}(\mathbb{R} \times \mathbb{T}^{k})} = 1$.

	From the preceding discussion, we deduce that $(R^{\delta,\xi}_{v,h})_{h \in (-\eta,\eta)}$ is bounded in $C_{0}(\mathbb{R} \times \mathbb{T}^{k})$.  By Schauder estimates, it is actually bounded in $C^{2,\mu}_{0}(\mathbb{R} \times \mathbb{T}^{k})$.  In view of the estimate \eqref{E: exponential_estimate}, $(R^{\delta,\xi}_{v,h})_{h \in (-\eta,\eta)}$ is pre-compact in both $C^{2}_{0}(\mathbb{R} \times \mathbb{T}^{k})$ and $L^{2}(\mathbb{R} \times \mathbb{T}^{k})$, which implies the real numbers $(c_{h}^{\delta})_{h \in (-\eta,\eta)}$ are also bounded.  Thus, $(Q^{\delta,\xi}_{v,h})_{h \in (-\eta,\eta)}$ is pre-compact in $C^{2}_{0}(\mathbb{R} \times \mathbb{T}^{k})$ and $L^{2}(\mathbb{R} \times \mathbb{T}^{k})$ as well.
	
Pick a sequence $(h_{n})_{n \in \mathbb{N}} \subseteq (0,\infty)$ such that $h_{n} \to 0$ as $n \to \infty$.  Without loss of generality, we can assume there is a $\bar{Q} \in C^{2}_{0}(\mathbb{R} \times \mathbb{T}^{k})$ and a $\bar{c} \in \mathbb{R}$ such that $\bar{Q} = \lim_{n \to \infty} Q^{\delta,\xi}_{v,h_{n}}$ in $C^{2}_{0}(\mathbb{R} \times \mathbb{T}^{k})$ and $L^{2}(\mathbb{R} \times \mathbb{T}^{k})$ and $\bar{c} = \lim_{n \to \infty} c^{\delta}_{h_{n}}$.  Passing to the limit in the equations satisfied by $(Q^{\delta,\xi}_{v,h})_{h \in (-\eta,\eta)}$, we find
	\begin{equation} \label{E: derivative_equation}
		\mathcal{L}_{v}^{\delta} \bar{Q} = \bar{K}^{\delta},
	\end{equation}
where $\bar{K}^{\delta} = \lim_{h \to \infty} K_{h}^{\delta} = 2 \langle \xi, a \mathcal{D}_{v}V_{v}^{\delta} \rangle + \langle \text{div} \, a, \xi \rangle V_{v}^{\delta}$.

Notice that there is at most one solution of \eqref{E: derivative_equation} in $H^{2}(\mathbb{R} \times \mathbb{T}^{k}) \cap \langle V_{e} \rangle^{\perp}$.  Indeed, if $\tilde{Q} \in H^{2}(\mathbb{R} \times \mathbb{T}^{k}) \cap \langle V_{e} \rangle^{\perp}$ is another solution, then $\mathcal{L}_{e}(\tilde{Q} - \bar{Q}) = 0$.  In particular, by Proposition \ref{P: linear_eqn}, $\tilde{Q} - \bar{Q} \in \langle V_{e} \rangle \cap \langle V_{e} \rangle^{\perp} = \{0\}$.  

The previous paragraph shows that the limiting function $\bar{Q}$ did not depend on the sequence $(h_{n})_{n \in \mathbb{N}}$.  Furthermore, integrating on $\{0\} \times \mathbb{T}^{k}$, we obtain
	\begin{equation*}
		0 = \int_{\mathbb{T}^{k}} \bar{Q}(0,x) \, dx + \bar{c} \int_{\mathbb{T}^{k}} V_{v}^{\delta}(0,x) \, dx.
	\end{equation*}
Thus, $\bar{c}$ is also uniquely determined, independently of $(h_{n})_{n \in \mathbb{N}}$.  

Putting it all together, we conclude that there is a unique $Q^{\delta,\xi}_{v} \in C^{2}_{0}(\mathbb{R} \times \mathbb{T}^{k})$ and a unique $c^{\delta} \in \mathbb{R}$ such that $Q^{\delta,\xi}_{v} + c^{\delta} V^{\delta}_{v} = \lim_{h \to 0} R^{\delta,\xi}_{v,h}$ in $C^{2}_{0}(\mathbb{R} \times \mathbb{T}^{k})$ and $L^{2}(\mathbb{R} \times \mathbb{T}^{k})$.  Furthermore, $Q^{\delta,\xi}_{v}$ is the unique solution of $\mathcal{L}_{v}^{\delta}Q^{\delta,\xi}_{v}= 2 \langle \xi, a \mathcal{D}_{v}V^{\delta}_{v} \rangle + \langle \text{div} \, a, \xi \rangle V^{\delta}_{v}$ in $\langle V^{\delta}_{v} \rangle^{\perp}$ and $c$ is determined by the requirement that $\int_{\mathbb{T}^{k}} (Q^{\delta,\xi}_{v}(0,x) + c V_{v}^{\delta}(0,x)) \, dx = 0$. \qed \end{proof}  

\begin{remark} \label{R: regularity in derivative}  Notice that
	\begin{align*}
		\frac{V^{\delta}_{v + h \xi} - V^{\delta}_{v}}{h} &= \partial_{s} R^{\delta,\xi}_{v,h}, \\
		\frac{D_{x}U^{\delta}_{v + h \xi} - D_{x} U^{\delta}_{v}}{h} &= D_{x} R^{\delta,\xi}_{v,h}.
	\end{align*}
Thus, the $C^{2}_{0}(\mathbb{R} \times \mathbb{T}^{k})$ convergence just proved implies that the limits $\partial_{s} R_{v}^{\delta,\xi} = \lim_{h \to 0} \frac{V_{v + h\xi}^{\delta} - V_{v}^{\delta}}{h}$ and $D_{x}R^{\delta,\xi}_{v} = \lim_{h \to 0} \frac{D_{x}U^{\delta}_{v + h \xi} - D_{x} U^{\delta}_{v}}{h}$ both exist in $C_{0}(\mathbb{R} \times \mathbb{T}^{k})$.  Appealing to Schauder estimates and the uniform exponential decay of $(R^{\delta,\xi}_{v,h})_{h \in (-1,1)}$ as $|s| \to \infty$, we can show this convergence also holds in $L^{p}(\mathbb{R} \times \mathbb{T}^{k})$ for any $p \in [1,\infty)$.  \end{remark}

\subsection{Derivatives of $\tilde{\varphi}^{a,\delta}$}  Now we compute the derivatives of $\tilde{\varphi}^{a,\delta}$.  In the case of the gradient, the proof works very generally.

	\begin{prop}  Fix $\delta \geq 0$, $k \in \{1,2,\dots,d\}$, and $K \subseteq \mathbb{R}^{d} \setminus \{0\}$ such that \eqref{E: elliptic regularization assumption} or \eqref{E: laminar assumption} hold.  For each $v \in K$, $\tilde{\varphi}^{a,\delta}$ is differentiable at $v$ and 
		\begin{equation} \label{E: derivative formula}
			D \tilde{\varphi}^{a,\delta}(v) = \int_{\mathbb{R} \times \mathbb{T}^{k}} \partial_{s} U_{v}^{\delta} a(x) \mathcal{D}_{v} U_{v}^{\delta} \, dx \, ds.
		\end{equation}
	\end{prop}
	
	\begin{proof}  Fix $v \in K$.  Recall that we can compute $\tilde{\varphi}^{a,\delta}(v)$ via the formula
		\begin{equation*}
			\tilde{\varphi}^{a,\delta}(v) = \mathscr{T}^{a,\delta}_{v}(U_{v}^{\delta}).
		\end{equation*}
	Since $\mathbb{R}^{d} \setminus (\mathbb{R}^{k} \times \{0\})$ is open for $k < d$, up to replacing $K$ by $K \cup \overline{B(v,\delta)}$ for some small $\delta > 0$, we can assume that $K$ contains a neighborhood of $v$.
	
	In view of Proposition \ref{P: differentiability_of_waves} and Remark \ref{R: regularity in derivative}, we can differentiate under the integral sign to find
		\begin{align*}
			\frac{\tilde{\varphi}^{a,\delta}(v + h \xi) - \tilde{\varphi}^{a,\delta}(v)}{h} &= \int_{\mathbb{R} \times \mathbb{T}^{k}} \left(\delta \partial_{s} U^{\delta}_{v} \partial_{s} R^{\delta,\xi}_{v} + \langle a(x) \mathcal{D}_{v} U^{\delta}_{v}, \mathcal{D}_{v}R^{\delta,\xi}_{v} \rangle \right) \, dx \, ds \\
			&\quad +\int_{\mathbb{R} \times \mathbb{T}^{k}} \left( \langle a(x) \mathcal{D}_{v} U^{\delta}_{v}, \xi \rangle \partial_{s} U^{\delta}_{v} + W'(U^{\delta}_{v}) R^{\delta,\xi}_{v} \right) \, dx \, ds \\
			&\quad + O(h).
		\end{align*}
	Observe that we can write
		\begin{align*}
			\int_{\mathbb{R} \times \mathbb{T}^{k}} &\left(\delta \partial_{s} U^{\delta}_{v} \partial_{s} R^{\delta,\xi}_{v} + \langle a(x) \mathcal{D}_{v} U^{\delta}_{v}, \mathcal{D}_{v}R^{\delta,\xi}_{v} \rangle + W'(U^{\delta}_{v}) R^{\delta,\xi}_{v} \right) \, dx \, ds \\
			&\quad = \int_{\mathbb{R} \times \mathbb{T}^{k}} R^{\delta,\xi}_{v} \left(-\delta \partial_{s}^{2} U^{\delta}_{v} + \mathcal{D}_{v}^{*}(a(x) \mathcal{D}_{v} U^{\delta}_{v}) + W'(U^{\delta}_{v}) \right) \, dx \, ds = 0.
		\end{align*}
	Thus,
		\begin{equation} \label{E: derivative part}
			\lim_{h \to 0} \frac{\tilde{\varphi}^{a,\delta}(v + h \xi) - \tilde{\varphi}^{a,\delta}(v)}{h} = \int_{\mathbb{R} \times \mathbb{T}^{k}} \partial_{s}U^{\delta}_{v} \langle a(x) \mathcal{D}_{v}U^{\delta}_{v}, \xi \rangle \, dx \, ds.
		\end{equation}
	We showed that $\tilde{\varphi}^{a,\delta}$ is Gateaux differentiable.  Since the right-hand side is a continuous function of $v$, it is Fr\'{e}chet differentiable, and then \eqref{E: derivative formula} follows from \eqref{E: derivative part}.  \qed \end{proof}  	
	
Next, we find a formula for $D^{2} \tilde{\varphi}^{a}$.  Here is where we use the full strength of the regularity afforded by laminarity.

	\begin{prop} \label{P: convexity}  $\tilde{\varphi}^{a,\delta} \in C^{2}(\mathbb{R}^{d} \setminus \{0\})$ for $\delta > 0$, and $\tilde{\varphi}^{a} \in C^{2}(\mathbb{R}^{d} \setminus (\mathbb{R}^{k} \times \{0\}))$ if $k < d$.  Further, if \eqref{E: elliptic regularization assumption} or \eqref{E: laminar assumption} hold, then, for each $v \in K$ and $\xi \in \mathbb{R}^{d}$, if $\Psi_{v}^{\delta,\xi}$ is defined in $\mathbb{R} \times \mathbb{T}^{k}$ by $\Psi_{v}^{\delta,\xi} = (V_{v}^{\delta})^{-1} R_{v}^{\delta,\xi}$, then
		\begin{align} 
			\langle D^{2} \tilde{\varphi}^{a,\delta}(v) \xi, \xi \rangle &= \int_{\mathbb{R} \times \mathbb{T}^{k}} \langle a(x) (\xi + \mathcal{D}_{v} \Psi_{v}^{\delta,\xi}), \xi + \mathcal{D}_{v} \Psi_{v}^{\delta,\xi} \rangle  |V_{v}^{\delta}|^{2} \, dx \, ds \label{E: convexity_legendre_transform} \\
			&\quad + \int_{\mathbb{R} \times \mathbb{T}^{k}} \delta |\partial_{s} \Psi_{v}^{\delta,\xi}|^{2} |V^{\delta}_{v}|^{2} \, dx \, ds. \nonumber
		\end{align}
In particular, $\tilde{\varphi}^{a,\delta}$ is convex for every $\delta \geq 0$.  Moreover, the following bound holds:
		\begin{equation} \label{E: upper bound general case}
			\tilde{M}^{a,\delta}(v)^{-1}D^{2}\tilde{\varphi}^{a,\delta}(v) \leq \Lambda \text{Id} \quad \text{in} \, \, K.
		\end{equation}
	\end{prop}  
	
	\begin{proof}  Fix $v \in K$.  Differentiating under the integral sign in \eqref{E: derivative formula} using Proposition \ref{P: differentiability_of_waves} and Remark \ref{R: regularity in derivative}, we obtain
		\begin{align*}
			D^{2}\tilde{\varphi}^{a,\delta}(v) &= \int_{\mathbb{R} \times \mathbb{T}^{k}} \left(|V_{v}^{\delta}|^{2} a(x) + a(x) \mathcal{D}_{v}U_{v}^{\delta} \otimes \partial_{s} R_{v}^{\delta} + V_{v}^{\delta} a(x) \mathcal{D}_{v}R_{v}^{\delta} \right) \, dx \, ds \\
					&= \int_{\mathbb{R} \times \mathbb{T}^{k}} \left(|V^{\delta}_{v}|^{2} a(x) -a(x) \mathcal{D}_{v} V_{v}^{\delta} \otimes R_{v}^{\delta} + \mathcal{D}_{v}^{*}(a(x) V_{v}^{\delta}) R_{v}^{\delta}\right) \, dx \, ds \\
					&=  \int_{\mathbb{R} \times \mathbb{T}^{k}} \left[|V_{v}^{\delta}|^{2} a(x) \, dx \, ds - \left(2 a(x) \mathcal{D}_{v} V^{\delta}_{v} + V^{\delta}_{v} \text{div} \, a \right) \otimes R_{v}^{\delta} \right]  \, dx \, ds \\
					&= \int_{\mathbb{R} \times \mathbb{T}^{k}} |V^{\delta}_{v}|^{2} a(x) \, dx \, ds - \int_{\mathbb{R} \times \mathbb{T}^{k}} \mathcal{L}_{v}^{\delta} R_{v}^{\delta} \otimes R_{v}^{\delta} \, dx \, ds.
		\end{align*}
	Integration-by-parts then gives
		\begin{align*}
			\langle D^{2} &\tilde{\varphi}^{a,\delta}(v) \xi, \xi \rangle =  \int_{\mathbb{R} \times \mathbb{T}^{k}} \langle a(x)\xi, \xi \rangle |V_{v}^{\delta}|^{2} \, dx \, ds \\
			&\quad - \int_{\mathbb{R} \times \mathbb{T}^{k}} \left( \delta |\partial_{s} R^{\delta,\xi}_{v}|^{2} + \langle a(x) \mathcal{D}_{v} R_{v}^{\delta,\xi}, \mathcal{D}_{v} R^{\delta,\xi}_{v} \rangle + W''(U_{e}) |R^{\delta,\xi}_{v}|^{2} \right) \, dx \, ds.
		\end{align*}
	Since $\mathcal{L}_{v}^{\delta}$ is a non-negative operator by Proposition \ref{P: non-negative_operator}, the right-most term is non-positive and we deduce from this that
		\begin{equation*}
			\langle D^{2} \tilde{\varphi}^{a,\delta}(v) \xi, \xi \rangle \leq  \int_{\mathbb{R} \times \mathbb{T}^{k}} \langle a(x)\xi, \xi \rangle |V_{v}^{\delta}|^{2} \, dx \, ds \leq \Lambda \tilde{M}^{a}(e).
		\end{equation*}

Now we substitute $\Psi_{v}^{\delta,\xi}$ for $R_{v}^{\delta,\xi}$ and use the equation satisfied by $R_{v}^{\delta,\xi}$ to obtain
		\begin{align*}
			\langle D^{2} \tilde{\varphi}^{a,\delta}(v) \xi, \xi \rangle &= \int_{\mathbb{R} \times \mathbb{T}^{k}} \langle a(x)(\xi + \mathcal{D}_{v} \Psi^{\delta,\xi}_{v}), \xi + \mathcal{D}_{v} \Psi^{\delta,\xi}_{v} \rangle |V_{v}^{\delta,\xi}|^{2} \, dx \, ds \\
				&\quad + \int_{\mathbb{R} \times \mathbb{T}^{k}} \delta |\partial_{s} \Psi^{\delta,\xi}_{v}|^{2} |V^{\delta}_{v}|^{2} \, dx \, ds.
		\end{align*}
This establishes \eqref{E: convexity_legendre_transform} and it shows that $D^{2} \tilde{\varphi}^{a,\delta} \geq 0$ in $K$.  For $\delta > 0$, this proves $\tilde{\varphi}^{a,\delta}$ is convex in $\mathbb{R}^{d} \setminus \{0\}$.  \qed \end{proof} 

Where the surface tension $\tilde{\varphi}^{a}$ is concerned, the previous computation shows that it is strictly convex away from the lamination.  More precisely, we have

\begin{prop}  In the laminar setting, that is, when $k \in \{1,2,\dots,d-1\}$, $\tilde{\varphi}^{a}$ is strictly convex away from the laminations in the sense that, for each $e \in S^{d-1} \setminus (S^{k-1} \setminus \{0\})$, 
	\begin{equation*}
		\langle D^{2} \tilde{\varphi}^{a}(e) \xi, \xi \rangle > 0 \quad \text{if} \, \, \xi \in \mathbb{R}^{d} \cap \langle e \rangle^{\perp}.
	\end{equation*}
Furthermore, the following upper bound holds:
	\begin{equation*}
		\tilde{M}^{a}(e)^{-1} D^{2} \tilde{\varphi}^{a}(e) \leq \Lambda (\text{Id} - e \otimes e) \quad \text{for} \, \, e \in S^{d-1} \setminus (S^{k-1} \times \{0\}).
	\end{equation*}
\end{prop}   

	\begin{proof}  Suppose $\xi \in \mathbb{R}^{d} \setminus \langle v \rangle$.  If $\langle D^{2} \tilde{\varphi}^{a}(e) \xi, \xi \rangle = 0$, then \eqref{E: convexity_legendre_transform} would imply that $\xi + \mathcal{D}_{e} \Psi^{0,\xi}_{e} = 0$ in $\mathbb{R} \times \mathbb{T}^{k}$.  However, from the definition of $\mathcal{D}_{e}$, this would yield
	\begin{equation*}
		0 = \int_{\mathbb{T}^{k}} (\xi + \mathcal{D}_{e}\Psi_{e}^{0,\xi}(0,x)) \, dx = \xi + e \int_{\mathbb{T}^{k}} \partial_{s} \Psi_{e}^{0,\xi}(0,x) \, dx,
	\end{equation*}
which is impossible unless $\xi \in \langle e \rangle$. 

Finally, we know that $\tilde{\varphi}^{a}$ is positively one-homogeneous so $D^{2}\tilde{\varphi}^{a}(e) e = 0$ for each $e \in S^{d-1} \setminus (S^{k-1} \times \{0\})$.  Thus, the final statement follows from the bound in the previous result.  \qed \end{proof}
		
Finally, we show that the previous computation is consistent with the one in \cite{barles souganidis}.

	\begin{prop}  If $e \in S^{d - 1} \setminus (S^{k - 1} \times \{0\})$, then the matrix $\tilde{\mathcal{S}}^{a}$ of \eqref{E: anisotropic curvature flow} satisfies $\tilde{\mathcal{S}}^{a}(e) = D^{2}\tilde{\varphi}^{a}(e)$.  \end{prop}  
	
	\begin{proof}  Recall that in the very first computation in the previous proof, we obtained
		\begin{equation*}
			D^{2}\tilde{\varphi}^{a}(e) = \int_{\mathbb{R} \times \mathbb{T}^{k}} \left(|V_{e}^{0}|^{2} a(x) + a(x) \mathcal{D}_{e}U^{0}_{e} \otimes \partial_{s} R^{0}_{e} + V^{0}_{e} a(x) \mathcal{D}_{e}R^{0}_{e} \right) \, dx \, ds.
		\end{equation*}
	Writing $\mathcal{D}_{e}U^{0}_{e} = e \partial_{s} U^{0}_{e} + D_{x}U^{0}_{e}$ and integrating both terms by parts, we arrive at
		\begin{align*}
			D^{2}\tilde{\varphi}^{a}(e) &= \int_{\mathbb{R} \times \mathbb{T}^{k}} \big (|V_{e}^{0}|^{2} a(x) + V_{e}^{0} \cdot \text{div} \, a \otimes R^{0}_{e} + 2 V^{0}_{e} a(x) D_{x} R^{0}_{e} \\
			&\qquad \qquad + 2 V^{0}_{e} a(x) e \otimes \partial_{s} R^{0}_{e} \big ) \, dx \, ds.
		\end{align*} 
	
	To finish the proof, first, recall that symmetric matrices are determined by their quadratic forms.  Therefore, it only remains to show that $\langle D^{2} \tilde{\varphi}(e) \xi, \xi \rangle = \langle \tilde{\mathcal{S}}^{a}(e)\xi, \xi\rangle$ independently of the choice of $\xi \in \mathbb{R}^{d}$.  Additionally, recall that if $w,v \in \mathbb{R}^{d}$, then 
		\begin{equation} \label{E: tensor algebra}
			\langle (w \otimes v) \xi, \xi \rangle = \langle \xi, w \rangle \langle v, \xi \rangle = \langle (v \otimes w) \xi, \xi \rangle.
		\end{equation}
		
	Using \eqref{E: tensor algebra}, we find
		\begin{align}
			\langle D^{2} \tilde{\varphi}(e) \xi, \xi \rangle &= \int_{\mathbb{R} \times \mathbb{T}^{k}} V^{0}_{e} \left( V^{0}_{e} \langle a(x) \xi, \xi \rangle + \langle \text{div}\, a, \xi \rangle R_{e}^{0,\xi} 		
			 \right.   \label{E: useful_later}
			\\
			&\qquad \left. + 2 \langle a(x) D_{x} R_{e}^{0,\xi}, \xi \rangle + 2 \langle a(x) e, \xi \rangle \partial_{s} R_{e}^{0,\xi} \right) \, dx \, ds \nonumber \\
				&= \int_{\mathbb{R} \times \mathbb{T}^{k}} V^{0}_{e} \left( \langle (a(x) e \otimes \partial_{s} R^{0}_{e}) \xi, \xi \rangle + \langle (\partial_{s} R^{0}_{e} \otimes a(x)e) \xi, \xi \rangle   \right. \nonumber \\
				&\qquad \left. + V^{0}_{e} \langle a(x) \xi, \xi \rangle + 2 \langle a(x) D_{x} R^{0}_{e}\xi, \xi \rangle + \frac{1}{2} \left(\langle (\text{div}\, a \otimes R^{0}_{e})\xi, \xi \rangle  \right. \right.\nonumber \\
				&\qquad \left. \left. + \langle (R^{0}_{e} \otimes \text{div} \, a) \xi, \xi \rangle \right) \right) \, dx \, ds. \nonumber
		\end{align}
	Thus, by our previous observation and \cite[Equation 6.22]{barles souganidis},
		\begin{align*}
			D^{2}\tilde{\varphi}(e) &= \int_{\mathbb{R} \times \mathbb{T}^{k}} V^{0}_{e} \bigg( a(x) e \otimes \partial_{s} R^{0}_{e} + \partial_{s} R^{0}_{e} \otimes a(x)e + V^{0}_{e} a(x) \\
				&\qquad  + 2 a(x) D_{x} R^{0}_{e} + \frac{1}{2} \left(\text{div}\, a \otimes R^{0}_{e} + R^{0}_{e} \otimes \text{div} \, a \right) \bigg) \, dx \, ds = \tilde{\mathcal{S}}^{a}(e).
		\end{align*} \qed
	\end{proof}  
		
\subsection{Higher regularity and correctors} \label{S: higher_regularity}  The purpose of this section is twofold.  First, we discuss higher regularity of $\tilde{\varphi}^{a}$ that can be deduced when additional differentiability assumptions are imposed on $W$.  Next, with this additional regularity, we construct correctors that will be used in the analysis of the sharp interface limit.  

In this section and Section \ref{S: sharp_interface_limit}, since $\delta = 0$ throughout, we will write $U_{v}$ in place of $U^{0}_{v}$, $\mathcal{L}_{v}$ instead of $\mathcal{L}_{v}^{0}$, $R_{v}^{\xi}$ for $R_{v}^{0,\xi}$, and so on.  

Concerning the differentiability of $\tilde{\varphi}^{a}$, we have
	
	\begin{prop} \label{P: higher_regularity} If in addition to \eqref{A: a_assumption_1}, \eqref{A: W_assumption_1}, \eqref{A: a_assumption_2}, \eqref{A: W_assumption_2} and \eqref{E: laminar}, we also assume that $W \in C^{n,\alpha}([-1,1])$ for some $n \in \mathbb{N} \setminus \{1\}$, then $\tilde{\varphi}^{a} \in C^{n}(\mathbb{R}^{d} \setminus (\mathbb{R}^{k} \times \{0\}))$ and the map $v \mapsto U_{v}$ is $(n - 1)$-times continuously Fr\'{e}chet differentiable in the $BC(\mathbb{R} \times \mathbb{T}^{k})$ topology.  \end{prop}  
	
	\begin{proof}  The proof is by induction on $n \in \mathbb{N} \setminus \{1\}$.  The base case $n = 2$ was treated already in Proposition \ref{P: differentiability_of_waves}.  The rest of the details are left to the interested reader. \qed \end{proof}

Next, we mention auxiliary functions that will be used in the investigation of the sharp interface limit.

	\begin{prop} \label{P: correctors}  If \eqref{A: a_assumption_1}, \eqref{A: W_assumption_1}, \eqref{A: a_assumption_2}, \eqref{A: W_assumption_2}, and \eqref{E: laminar} all hold, then, for each $e \in S^{d-1} \setminus (S^{k-1} \times \{0\})$ and $\xi \in \mathbb{R}^{d}$, there is a unique $P^{\xi}_{e} \in H^{2}(\mathbb{R} \times \mathbb{T}^{d})$ solving the PDE
		\begin{equation*}
			\left\{ \begin{array}{r l}
					\mathcal{L}_{e} P^{\xi}_{e} = F_{\xi,e} & \text{in} \, \, \mathbb{R} \times \mathbb{T}^{d}, \\
				\lim_{|s| \to \infty} P^{\xi}_{e}(s,x) = 0 & \text{uniformly in} \, \, \mathbb{T}^{d}, \\
				\int_{\mathbb{T}^{d}} P^{\xi}_{e}(0,x) \, dx = 0,
				\end{array} \right.
		\end{equation*}
	where $F_{\xi,e} : \mathbb{R} \times \mathbb{T}^{d} \to \mathbb{R}$ is given by 
		\begin{align*}
			F_{\xi,e} &= - \tilde{M}^{a}(e)^{-1} \langle D^{2} \tilde{\varphi}^{a}(e) \xi, \xi \rangle V_{e} + \langle a(x)\xi, \xi \rangle V_{e} + \langle \text{div} \, a, \xi \rangle R^{\xi}_{e}  \\
				&\qquad + 2 \langle a(x) \xi, D_{x} R^{\xi}_{e} \rangle + 2 \langle a(x) e, \xi \rangle \partial_{s} R^{\xi}_{e}.
		\end{align*}
	
	If, in addition, $W \in C^{4,\alpha}([-1,1])$, then, for each fixed $\xi$, the function $P^{\xi} : S^{d-1} \setminus (S^{k-1} \times \{0\}) \to H^{2}(\mathbb{R} \times \mathbb{T}^{d})$ given by $P^{\xi}(e) = P_{e}^{\xi}$ is twice continuously differentiable with respect to both the $H^{2}(\mathbb{R} \times \mathbb{T}^{d})$ and $C_{0}(\mathbb{R} \times \mathbb{T}^{d})$ topologies.  \end{prop}  
		
	\begin{proof}  Concerning existence, notice that $F_{\xi,e} \in \langle V_{e} \rangle^{\perp}$ by \eqref{E: useful_later} and the definition of $\tilde{M}^{a}(e)$.  Thus, Proposition \ref{P: linear_eqn} provides the existence of $P_{e}^{\xi}$.  Arguing as in the analysis of the function $R^{\xi}_{e}$, we see that $P_{e}^{\xi}(s,x) \to 0$ at an exponential rate as $|s| \to \infty$.  
	
	The proof that $e \mapsto P^{\xi}_{e}$ is differentiable in $e$ proceeds exactly as in the corresponding proof for $U_{e}$.  This is where we need a third derivative of $W$.  The second derivative in $e$ is obtained the same way, and explains our assumption that $W$ has four continuous derivatives. \qed \end{proof}  
	
In the proof of Theorem \ref{T: sharp_interface_limit_graphs}, we will use the following extension of the previous construction.  Given $A \in \mathcal{S}_{d}$, if we expand $A$ as $A = \sum_{i = 1}^{d} \lambda_{i} \xi^{A}_{i} \otimes \xi^{A}_{i}$, where $\{\xi^{A}_{1},\dots,\xi^{A}_{d}\}$ is an orthonormal basis, then we can define $P_{e}^{A} \in H^{2}(\mathbb{R} \times \mathbb{T}^{k}) \cap C^{2}_{0}(\mathbb{R} \times \mathbb{T}^{k})$ by 
	\begin{equation} \label{E: decomposition}
		P_{e}^{A} = \sum_{i = 1}^{d} \lambda_{i} P_{e}^{\xi_{i}^{A}}.
	\end{equation}
Notice that $P_{e}^{A}$ satisfies the equation
	\begin{equation} \label{E: corrector equation with matrix}
		\mathcal{L}_{e} P^{A}_{e} = F_{A,e} \quad \text{in} \, \, \mathbb{R} \times \mathbb{T}^{k}, \quad \int_{\mathbb{T}^{k}} P^{A}_{e}(0,x) \, dx = 0,
	\end{equation}
where $F_{A,e}$ is given by 
	\begin{align*}
		F_{A,e} &= - \tilde{M}^{a}(e)^{-1} \text{tr}(D^{2} \tilde{\varphi}^{a}(e)A) V_{e} + \text{tr}(a(x)A) V_{e} + \text{tr}([\text{div} \, a \otimes R_{e}] A)  \\
			&\quad + 2 \text{tr}( [a(x) \otimes D_{x}R_{e}] A) + 2 \text{tr}( [a(x)e \otimes \partial_{s} R_{e}]A).
	\end{align*}
Observe that, by uniqueness of solutions of \eqref{E: corrector equation with matrix}, the map $A \mapsto P_{e}^{A}$ is linear.  Hence as a consequence of Proposition \ref{P: correctors}, we have

	\begin{prop} \label{P: key corrector result} If \eqref{A: a_assumption_1}, \eqref{A: W_assumption_1}, \eqref{A: a_assumption_2}, \eqref{A: W_assumption_2}, \eqref{A: additional_regularity}, and \eqref{E: laminar} all hold, then there is a $C^{2}$ map $P : S^{d-1} \setminus (S^{k-1} \times \{0\}) \to C^{2}_{0}(\mathbb{R} \times \mathbb{T}^{k}; \mathcal{S}_{d})$ such that if $e \in S^{d-1} \setminus (S^{k-1} \times \{0\})$, $A \in \mathcal{S}_{d}$, and $P^{A}_{e}$ is given by \eqref{E: decomposition}, then 
		\begin{equation*}
			P_{e}^{A}(s,x) = \text{tr}\left( P_{e}(s,x) A \right) \quad \text{for each} \, \, (s,x) \in \mathbb{R} \times \mathbb{T}^{k}.
		\end{equation*}
	Furthermore, for each $K \subset \subset S^{d-1} \setminus (S^{k-1} \times \{0\})$, there is a constant $C(K) > 0$ such that 
		\begin{equation*}
			\sqrt{\text{tr}(P_{e}(s,x)^{2})} \leq C(K) \exp \left( - C(K)^{-1} |s| \right) \quad \text{for} \, \, (s,x) \in \mathbb{R} \times \mathbb{T}^{k}.
		\end{equation*}\end{prop}

\subsection{On elliptic regularization} \label{S: elliptic regularization}  The remainder of the section deals with elliptic regularization, the main goal being the proof of Theorem \ref{T: no lower bound}.  Accordingly, we henceforth set $k = d$.  

We start with a simple observation concerning the approximate mobility $\tilde{M}^{a,\delta}$, namely, its monotonicity in the regularization parameter $\delta$.

\begin{prop} \label{P: monotone mobility}  If $a$ and $W$ satisfy \eqref{A: a_assumption_1}, \eqref{A: W_assumption_1}, \eqref{A: a_assumption_2}, and \eqref{A: W_assumption_2}, then, for any $v \in \mathbb{R}^{d}$, the function $\delta \mapsto \tilde{M}^{a,\delta}(v)$ is non-increasing in $(0,\infty)$.  \end{prop} 

	\begin{proof}  If $\delta_{1},\delta_{2} \in (0,\infty)$ and $\delta_{1} < \delta_{2}$, then we use the minimization properties of $U^{\delta_{1}}_{v}$ and $U^{\delta_{2}}_{v}$ to compute:
	\begin{align*}
		\mathscr{T}^{a}_{v,\delta_{1}}(U^{\delta_{1}}_{v}) + \frac{1}{2} (\delta_{2} - \delta_{1}) \tilde{M}^{a,\delta_{1}}(v) &= \mathscr{T}^{a}_{v,\delta_{2}}(U^{\delta_{1}}_{v}) \\
		&\geq \mathscr{T}^{a}_{v,\delta_{2}}(U^{\delta_{2}}_{v}) \\
		&= \mathscr{T}^{a}_{v,\delta_{1}}(U^{\delta_{2}}_{v}) + \frac{1}{2} (\delta_{2} - \delta_{1}) \tilde{M}^{a,\delta_{2}}(v) \\
		&\geq \mathscr{T}^{a}_{v,\delta_{1}}(U^{\delta_{1}}_{v}) + \frac{1}{2} (\delta_{2} - \delta_{1}) \tilde{M}^{a,\delta_{2}}(v).
	\end{align*}
	Hence $\tilde{M}^{a,\delta_{1}}(v) \geq \tilde{M}^{a,\delta_{2}}(v)$. \qed \end{proof}

The monotonicity of $\tilde{M}^{a,\delta}(v)$ in $\delta$ will be used in the next section to prove Theorem \ref{T: no lower bound}.  In the meantime, we use it to show that elliptic regularization provides an approximation of both the surface tension $\tilde{\varphi}^{a}$ and the mobility $\tilde{M}^{a}$ as defined in Section \ref{S: mobility}.   

	\begin{prop} \label{P: regularization}  $\lim_{\delta \to 0^{+}} (\tilde{\varphi}^{a,\delta}, \tilde{M}^{a,\delta}) = (\tilde{\varphi}^{a},\tilde{M}^{a})$ pointwise in $\mathbb{R}^{d} \setminus \{0\}$. \end{prop}  
	
		\begin{proof}  We begin by proving that $\lim_{\delta \to 0^{+}} \tilde{\varphi}^{a,\delta} = \tilde{\varphi}^{a}$.  First, observe that we can use a candidate of the form $V(s,x) = q(s)$ in \eqref{E: approximate surface tension definition} to see that, for any fixed $v$, the set $\{\tilde{\varphi}^{a,\delta}(v) \, \mid \, \delta \in (0,1]\}$ is bounded.
		
		Fix $v \in \mathbb{R}^{d} \setminus \{0\}$.  To prove that $\tilde{\varphi}^{a,\delta}(v) \to \tilde{\varphi}^{a}(v)$ as $\delta \to 0^{+}$, it suffices to show that it is true along any convergent subsequence.  Thus, let us fix a $(\delta_{n})_{n \in \mathbb{N}} \subseteq (0,1]$ such that $\lim_{n \to \infty} \delta_{n} = 0$ and the limit $\lim_{n \to \infty} \tilde{\varphi}^{a,\delta_{n}}(v)$ exists.  
		
		Notice that $\mathscr{T}^{a}_{v}(U^{\delta_{n}}_{v}) \leq \mathscr{T}^{a}_{v,\delta_{n}}(U^{\delta_{n}}_{v}) = \tilde{\varphi}^{a,\delta_{n}}(v)$ for all $n \in \mathbb{N}$.  Therefore, by Proposition \ref{P: compactness} and Remark \ref{R: extension_to_space}, there is a subsequence $(n_{j})_{j \in \mathbb{N}} \subseteq \mathbb{N}$ and a $U \in \mathscr{X}$ such that 
			\begin{align*}
				\lim_{j \to \infty} \delta_{n_{j}} = 0, \quad \lim_{j \to \infty} U^{\delta_{n_{j}}}_{v} = U \quad \text{a.e. in} \, \, \mathbb{R} \times \mathbb{T}^{d}, \\
				 \text{and} \quad \mathscr{T}^{a}_{v}(U) \leq \lim_{j \to \infty} \mathscr{T}^{a}_{v,\delta_{n_{j}}}(U^{\delta_{n_{j}}}_{v}) = \lim_{n \to \infty} \tilde{\varphi}^{a,\delta_{n}}(v).
			\end{align*}  
		Given any $V \in $, we compute
			\begin{align*}
				\mathscr{T}^{a}_{v}(U) \leq \lim_{n \to \infty} \tilde{\varphi}^{a,\delta_{n}}(v) \leq \lim_{n \to \infty} \mathscr{T}^{a}_{v,\delta_{n}}(V) = \mathscr{T}^{a}_{v}(V).
			\end{align*}
		Thus, by Proposition \ref{P: smooth_approximation} and Remark \ref{R: extension_to_space},
			\begin{equation*}
				\mathscr{E}^{a}(v) \leq \mathscr{T}^{a}_{v}(U) \leq \lim_{n \to \infty} \tilde{\varphi}^{a,\delta_{n}}(v) \leq \mathscr{E}^{a}(v)
			\end{equation*} 
		Since $\mathscr{E}^{a} = \tilde{\varphi}^{a}$, we conclude that $\lim_{n \to \infty} \tilde{\varphi}^{a,\delta_{n}}(v) = \tilde{\varphi}^{a}(v)$ as claimed.
		
		Finally, we show that $\lim_{\delta \to 0^{+}} \tilde{M}^{a,\delta} = \tilde{M}^{a}$, interpreting the limit in the sense of extended real numbers if necessary.  (The limit necessarily exists by monotonicity.)  To start with, we will prove that
			\begin{equation} \label{E: useful mobility bound}
				\tilde{M}^{a}(v) \geq \sup \left\{ \tilde{M}^{a,\delta}(v) \, \mid \, \delta > 0\right\}.
			\end{equation}
		This implies that $\tilde{M}^{a}(v) = \lim_{\delta \to 0^{+}} \tilde{M}^{a,\delta}(v)$ if $\lim_{\delta \to 0^{+}} \tilde{M}^{a,\delta}(v) = \infty$.  
		
		To see that \eqref{E: useful mobility bound} holds, suppose that $\tilde{M}^{a}(v) < \infty$; otherwise, the bound is automatic.  Since the mobility is finite, the definition \eqref{E: mobility} implies that we can fix a $U \in \mathscr{X}_{+}$ such that $\mathscr{E}^{a}(v) = \mathscr{T}^{a}_{v}(U)$ and $\partial_{s} U \in L^{2}(\mathbb{R} \times \mathbb{T}^{d})$.  Using the minimization properties of $U$ and $U^{\delta}_{v}$ for arbitrary $\delta > 0$, we find
			\begin{equation*}
				\frac{1}{2} \delta \tilde{M}^{a}(v) + \mathscr{T}^{a}_{v}(U) = \mathscr{T}^{a}_{v,\delta}(U) \geq \mathscr{T}^{a}_{v,\delta}(U^{\delta}) \geq \frac{1}{2} \delta \tilde{M}^{a,\delta}(v) + \mathscr{T}^{a}_{v}(U).
	\end{equation*}
Thus, since this is true for any $U$ as above, \eqref{E: useful mobility bound} holds.

It remains to consider the case when $\lim_{\delta \to 0^{+}} \tilde{M}^{a,\delta}(v) < \infty$.  Recall that we know $\limsup_{\delta \to 0^{+}} \mathscr{T}^{a}_{e}(U^{\delta}) < \infty$ since $\mathscr{T}^{a}_{v,\delta}(U^{\delta}) \to \tilde{\varphi}^{a}(v)$.  Together with the bound on $(\tilde{M}^{a,\delta}(v))_{\delta > 0}$, this implies that $(U^{\delta})_{\delta > 0}$ is bounded in the space $H^{1}((-M,M) \times \mathbb{T}^{d})$ for each $M > 0$.  Thus, let us choose a subsequence $(\delta_{n})_{n \in \mathbb{N}} \subseteq (0,1]$ and a $U \in \mathscr{X}_{+}$ such that
		\begin{gather*}
			\lim_{n \to \infty} \delta_{n} = 0, \quad \lim_{n \to \infty} U^{\delta_{n}} = U \quad \text{a.e. in} \, \, \mathbb{R} \times \mathbb{T}^{d}, \quad \text{and} \\
			\lim_{n \to \infty} \partial_{s} U^{\delta_{n}} = \partial_{s} U \quad \text{weakly in} \, \, L^{2}(\mathbb{R} \times \mathbb{T}^{d}).
		\end{gather*}
Since $\mathscr{T}^{a}_{v}(U) \leq \liminf_{n \to \infty} \mathscr{T}^{a}_{v,\delta_{n}}(U^{\delta_{n}}_{v}) = \tilde{\varphi}^{a}(v)$, we know that $\mathscr{T}^{a}_{v}(U) = \mathscr{E}^{a}(v)$, that is, $U$ is a candidate for the infimum in \eqref{E: mobility}.  Applying this together with the weak convergence of $(\partial_{s} U^{\delta_{n}})_{n \in \mathbb{N}}$, we find
	\begin{equation*}
		\tilde{M}^{a}(v) \leq \int_{\mathbb{R} \times \mathbb{T}^{d}} |\partial_{s} U|^{2} \, dx \, ds \leq \liminf_{n \to \infty} \int_{\mathbb{R} \times \mathbb{T}^{d}} |\partial_{s} U_{v}^{\delta_{n}}|^{2} \, dx \, ds = \lim_{\delta \to 0^{+}} \tilde{M}^{a,\delta}(v).
	\end{equation*}
Combining this with \eqref{E: useful mobility bound}, we conclude that $\tilde{M}^{a}(v) = \lim_{\delta \to 0^{+}} \tilde{M}^{a,\delta}(v)$.  \qed \end{proof}

\subsection{Lower bound implies smoothness} \label{S: no lower bound} In the previous section, we have shown that elliptic regularization provides an approximation of both the surface tension $\tilde{\varphi}^{a}$ and the mobility $\tilde{M}^{a}$.  Thus, it seems natural to interpret \eqref{E: upper bound general case} as an a priori bound on $(\tilde{M}^{a})^{-1} \|D^{2} \tilde{\varphi}^{a}\|$.  In this section, we prove Theorem \ref{T: no lower bound}, which concerns the possibility of a matching lower bound.
	
	In view of some classical counter-examples in Aubry-Mather theory, we expect the conclusions of Theorem \ref{T: no lower bound} to be false for ``generic" coefficients $a$ and $W$.  Put another way, instead of interpreting the theorem as a positive result, it seems more appropriate to view it as an obstruction to bounds like \eqref{E: lower bound}.
	
	\begin{proof}[Proof of Theorem \ref{T: no lower bound}]  Let us assume that there is an $\mathcal{H}^{d-1}$-measurable $A \subseteq E$ such that $\mathcal{H}^{d - 1}(A) > 0$ and $\tilde{M}^{a}(e) = \infty$ for each $e \in A$.  It is convenient to define $\tilde{A} = \{v \in \mathbb{R}^{d} \setminus \{0\} \, \mid \, \frac{v}{\|v\|} \in A\}$.  Notice that $\tilde{A}$ is Lebesgue measurable and $\mathcal{L}^{d}(\tilde{A}) > 0$.  By Proposition \ref{P: regularization}, if $v \in \tilde{A}$, then			\begin{equation*}
			\lim_{\delta \to 0^{+}} \tilde{M}^{a,\delta}(v) = \infty.
		\end{equation*}
		
	Given that $\mathcal{L}^{d}(\tilde{A}) > 0$, we can fix a compact set $K \subseteq \tilde{A}$ with $\mathcal{L}^{d}(K) > 0$.  
	
	By Proposition \ref{P: regularization}, $\tilde{\varphi}^{a,\delta} \to \tilde{\varphi}^{a}$ pointwise in $\mathbb{R}^{d}$.  Since these functions are convex, the convergence is actually better than pointwise --- it is local uniform convergence.  From this, it is immediate that $D^{2} \tilde{\varphi}^{a,\delta} \overset{*}{\rightharpoonup} D^{2} \tilde{\varphi}^{a}$ in $C_{0}(\mathbb{R}^{d})^{*}$.  Thus, if $\xi \in S^{d - 1}$, we find
		\begin{align*}
			\int_{K} \langle D^{2}\tilde{\varphi}^{a}(dv) \xi, \xi \rangle &\geq \limsup_{\delta \to 0^{+}} \int_{K} \langle D^{2}\tilde{\varphi}^{a,\delta}(v) \xi, \xi \rangle \, dv \\
				&\geq c \limsup_{\delta \to 0^{+}} \int_{K} \tilde{M}^{a,\delta}(v) \, dv.
		\end{align*}  
	Recall from Proposition \ref{P: monotone mobility} that the function $\delta \mapsto \tilde{M}^{a,\delta}$ is non-increasing.  Therefore, by the monotone convergence theorem, 
		\begin{equation*}
			\lim_{\delta \to 0^{+}} \int_{K} \int_{\mathbb{R} \times \mathbb{T}^{d}} \partial_{s} U^{\delta}_{v}(s,x)^{2} \, dx \, ds \, dv = \int_{K} \int_{\mathbb{R} \times \mathbb{T}^{d}} \partial_{s} U_{v}(s,x)^{2} \, dx \, ds \, dv = \infty.
		\end{equation*}
	From this, we conclude that
		\begin{equation*}
			\int_{K} \langle D^{2}\tilde{\varphi}^{a}(dv)\xi, \xi \rangle = \infty.
		\end{equation*}
	However, this contradicts the fact that $D^{2}\tilde{\varphi}^{a}$ is a Radon measure.  \qed \end{proof}

\section{An Example in 2D}  \label{S: example_2D}

Given what we have proved in the previous section, it is natural to ask what happens as $\text{dist}(e, S^{k - 1} \times \{0\}) \to 0$.  This section is devoted to the study of a specific class of examples.  We will see that some of the natural regularity properties we might hope for actually break down as the angle between $e$ and the laminations tends to zero.  

\subsection{Class of matrix fields $a$}  Let $\delta, \kappa \in (0,\frac{1}{4})$ be free parameters to be determined below.  Let $a_{1} : \mathbb{T} \to \mathbb{R}$ be a periodic function satisfying
\begin{equation*}
a_{1}(x) = 1 \, \, \text{if} \, \, x \in \left[\kappa, \frac{1}{2} - \kappa\right], \quad a_{1}(x) = \delta \, \, \text{if} \, \, x \in \left[\frac{1}{2} + \kappa, 1 - \kappa\right]
\end{equation*}
and monotone in each interval in between.
We also assume that $a_{1}$ is symmetric with respect to reflections around $\frac{1}{4}$ and $\frac{3}{4}$, that is,
\begin{equation} \label{E: reflection}
a_{1} \left(x \right) = a_{1} \left( \frac{1}{4} + \left(\frac{1}{4} - x \right) \right), \quad a_{1} \left( x \right) = a_{1} \left( \frac{3}{4} + \left( \frac{3}{4} - x \right) \right).
\end{equation}

Let $a_{2} : \mathbb{T} \to (0,\infty)$ be any positive periodic function.  We will assume, for definitness, that, like $a_{1}$, $a_{2}$ satisfies $\delta \leq a_{2}(x) \leq 1$ for each $x \in \mathbb{T}$.  

Finally, define $a : \mathbb{T} \to \mathcal{S}_{2}$ by 
	\begin{equation} \label{E: matrix field}
		a(x) = a_{1}(x) e_{1} \otimes e_{1} + a_{2}(x) e_{2} \otimes e_{2}.
	\end{equation}
	
Lastly, again for definiteness, we will use $W(u) = \frac{1}{4}(1 - u^{2})^{2}$ in this section.
	
We will prove the following:

\begin{theorem} \label{T: counterexample} For each $\kappa \in (0,\frac{1}{4})$, there is a $\bar{\delta} > 0$ such that if $\delta \in (0,\bar{\delta})$, $a$ is given by \eqref{E: matrix field}, and $W(u) = \frac{1}{4}(1 - u^{2})^{2}$, then no minimizer in $\mathcal{M}_{e_{1}}(\mathbb{R} \times \mathbb{T})$ or $\mathcal{M}_{-e_{1}}(\mathbb{R} \times \mathbb{T})$ is continuous.  \end{theorem}

In the language of Aubry-Mather theory, we prove that the sets of plane-like minimizers in the directions $e_{1}$ and $-e_{1}$ have gaps.  Notice that, by our assumptions, we only need to prove this for the direction $e_{1}$ and then the $-e_{1}$ case will follow by symmetry.

For completeness, we show how Theorem \ref{T: counterexample} implies Theorem \ref{T: non_integrable}.

\begin{proof}[Proof of Theorem \ref{T: non_integrable}]  To see that there may not be continuous minimizers in dimension $d = 1$, let $a_{0} : \mathbb{T} \to (0,\infty)$ be given by $a_{0} = a_{1}$, with $a_{1}$ defined as above, and set $a_{2} = a_{1}$.  Notice that the one-dimensional minimizers $\mathcal{M}_{\pm e_{1}}^{a_{0}}(\mathbb{R} \times \mathbb{T})$ for $a_{0}$ correspond to the minimizers $\mathcal{M}_{\pm e_{1}}^{a}(\mathbb{R} \times \mathbb{T})$ for the function $a$ of \eqref{E: matrix field} in two-dimensions  since Proposition \ref{P: laminar symmetry} implies the latter depend only on the first coordinate.  Hence Theorem \ref{T: counterexample} implies that $\mathcal{M}_{e_{1}}^{a_{0}}(\mathbb{R} \times \mathbb{T}) \cup \mathcal{M}_{-e_{1}}^{a_{0}}(\mathbb{R} \times \mathbb{T})$ consists of discontinuous functions.

Next, suppose that $d \in \mathbb{N}$ is arbitrary.  Once again, let $a_{0} = a_{1} = a_{2}$.  Let $a : \mathbb{T}^{d} \to \mathcal{S}_{d}$ be the matrix field $a(x) = a_{0}(x_{1})$.  Observe that, by Proposition \ref{P: laminar symmetry}, if $U \in \mathcal{M}^{a}_{e_{1}}(\mathbb{R} \times \mathbb{T}^{d}) \cup \mathcal{M}^{a}_{-e_{1}}(\mathbb{R} \times \mathbb{T})$, then $U(x) = \tilde{U}(x_{1})$ for some $\tilde{U} \in \mathcal{M}^{a_{0}}_{e_{1}}(\mathbb{R} \times \mathbb{T}) \cup \mathcal{M}^{a_{0}}_{-e_{1}}(\mathbb{R} \times \mathbb{T})$.  By the one-dimensional result of the previous paragraph, $\tilde{U}$ is discontinuous, and hence $U$ is also.     \qed \end{proof} 

\subsection{Gaps} \label{S: gaps}  To prove Theorem \ref{T: counterexample}, we will start by showing that there is no plane-like minimizer of \eqref{E: functional} satisfying $u(\frac{1}{4}) = 0$.  We start by proving a lower bound on the energy of an arbitrary front-like function $u$ with $u \left(\frac{1}{4} \right) =0$.  In particular, we show this is $O(1)$.  We then show that it is possible to find a front-like function with $u \left(\frac{3}{4}\right) = 0$ for which the energy is on the order of $\sqrt{\delta}$ provided $\delta \ll 1$.  

We start with the lower bound:

\begin{prop} \label{P: O_1_estimate}  If $u : \mathbb{R} \to [-1,1]$ satisfies $u \left(\frac{1}{4} \right) = 0$ and $\lim_{x \to \pm \infty} u(x) = \pm 1$, then there is a universal constant $\sigma_{\kappa} > 0$ depending only on $W$ and $\kappa$ such that
\begin{equation*}
\int_{-\infty}^{\infty} \left( \frac{a(x) u'(x)^{2}}{2} + W(u(x)) \right) \, dx \geq \sigma_{\kappa}.
\end{equation*}   \end{prop}  

The proof is inspired by an idea appearing in the lecture notes of Alberti \cite{alberti guide}.  To lighten the notation, we will write
\begin{equation*}
\mathcal{F}^{a}_{1}(u;[a,b]) = \int_{a}^{b} \left( \frac{a(x) u'(x)^{2}}{2} + W(u(x)) \right) \, dx
\end{equation*}
if $u : \mathbb{R} \to [-1,1]$ is any function and $a < b$.  

\begin{proof}  First, we make the following observation:
\begin{equation*}
\mathcal{F}^{a}_{1}(u; \mathbb{R}) \geq 2 \min \left\{ \mathcal{F}^{a}_{1}\left(u; \left(-\infty,\frac{1}{4}\right)\right), \mathcal{F}^{a}_{1}\left(u; \left(\frac{1}{4},\infty\right)\right) \right\}.
\end{equation*}
Assume without loss of generality that $\mathcal{F}^{a}_{1}(u; (-\infty,\frac{1}{4})) \geq \mathcal{F}^{a}_{1}(u; (\frac{1}{4},\infty))$.  It follows that if we define $u_{\text{sym}}$ by 
\begin{equation*}
u_{\text{sym}}(x) = \left\{ \begin{array}{r l}
					u(x), & x \geq \frac{1}{4}, \\
					-u \left(\frac{1}{4} + \left(\frac{1}{4} -x\right) \right), & x \leq \frac{1}{4},
					\end{array} 
					\right.
\end{equation*}
then $u_{\text{sym}}\left(\frac{1}{4}\right) = 0$ and, by the symmetry properties of $a$ and $W$, 
\begin{equation*}
\mathcal{F}^{a}_{1}(u_{\text{sym}}; \mathbb{R}) = 2 \mathcal{F}^{a}_{1} \left(u; \left(\frac{1}{4},\infty\right)\right) \leq \mathcal{F}^{a}_{1}(u; \mathbb{R}).
\end{equation*}  
Thus, we can assume that $u(\frac{1}{4} + (\frac{1}{4}-x)) = -u(x)$ in what follows.  

Let $\zeta \in (0,1)$ be a free parameter.  There are two cases to check: 
\begin{itemize}
\item[(i)] $|u (\bar{x})| \geq 1 - \zeta$ for some $\bar{x} \in \left(\frac{1}{4}, \frac{1}{2} - \kappa\right]$,
\item[(ii)] $|u| < 1 - \zeta$ in $\left(\frac{1}{4}, \frac{1}{2} - \kappa\right]$.
\end{itemize}

Consider case (i) first.  Notice that $|u \left(\frac{1}{4} + (\frac{1}{4} - \bar{x}) \right)| \geq 1 - \zeta$ by symmetry.  We will estimate $\mathcal{F}^{a}_{1}(u; \left[\frac{1}{2} - \bar{x}, \bar{x}\right])$ by extending $u$ to a function on $\mathbb{R}$ in a controlled way and then taking advantage of what we know about the energy when $a \equiv 1$.  

First, assume that $u(\bar{x}) \geq 1 - \zeta$.  Note that $u(\frac{1}{4} + (\frac{1}{4} - \bar{x})) \leq -1 + \zeta$.  Define $\bar{u} : \mathbb{R} \to [-1,1]$ by 
\begin{equation*}
\left\{\begin{array}{l l}
\bar{u}(x) = -1  &\quad \text{if} \, \, x \in (-\infty, \frac{1}{2} - (\bar{x} + \zeta)], \\
\bar{u}(x) = u(x) &\quad \text{if} \, \, x \in \left[\frac{1}{2} - \bar{x}, \bar{x}\right], \\
\bar{u}(x) = 1 &\quad \text{if} \, \, x \in [\bar{x} + \zeta, \infty),
\end{array} \right.
\end{equation*}
and interpolating linearly in between.  (Note that $\frac{1}{2} - y = \frac{1}{4} + ( \frac{1}{4} - y)$.)  If we momentarily replace $a$ by $1$, we have
\begin{align*}
\int_{-\infty}^{\infty} \left(\frac{\bar{u}'(x)^{2}}{2} + W(\bar{u}(x)) \right) \, dx &\leq \mathcal{F}^{a}_{1} \left(u; \left[\frac{1}{2} - \bar{x}, \bar{x}\right]\right) + 2\zeta \left(\left(\frac{1 - u(\bar{x})}{\zeta}\right)^{2} \right. \\
	&\quad  + \max\{W(u) \, \mid \, u(\bar{x}) \leq u \leq 1\} \Bigg) \\
	&\leq \mathcal{F}^{a}_{1} \left(u; \left[\frac{1}{2} - \bar{x}, \bar{x}\right]\right) \\
	&\quad+ 2(1 + \max\{W(u) \, \mid \, 0 \leq u \leq 1\})\zeta.
\end{align*}
On the other hand, we know that the left-hand side is bounded below by classical arguments.  Specifically, we obtain (cf.\ \cite{alberti guide})
\begin{align*}
\int_{-\infty}^{\infty} \left(\frac{\bar{u}'(x)^{2}}{2} + W(\bar{u}(x)) \right) \, dx 
		&\geq \int_{-1}^{1} \sqrt{2W(u)} \, du.
\end{align*}
Putting it all together, we find
\begin{equation} \label{E: energy_case_i}
\mathcal{F}^{a}_{1}\left(u; \left[\kappa, \frac{1}{2} - \kappa\right]\right) \geq \int_{-1}^{1} \sqrt{2W(u)} \, du - C\zeta =: f(\zeta),
\end{equation}
where $C = 2(1 + \max\{W(u) \, \mid \, 0 \leq u \leq 1\})$. 

If instead we had $u(\bar{x}) \leq -1 + \zeta$, then we could repeat the previous computation defining $\bar{u}$ instead by
\begin{equation*}
	\left\{\begin{array}{l l}
		\bar{u}(x) = 1 &\quad \text{if} \, \, x \in (-\infty, \frac{1}{2} - (\bar{x} + \zeta)], \\
		\bar{u}(x) = u(x) &\quad \text{if} \, \, x \in \left[\frac{1}{2} - \bar{x}, a\right], \\
		\bar{u}(x) = -1 &\quad \text{if} \, \, x \in [\bar{x} + \zeta, \infty),
	\end{array} \right.
\end{equation*}
and interpolating linearly in between.  (The fact that $\lim_{x \to \pm \infty} \bar{u}(x) = \mp 1$ is not relevant where the estimation of the energy is concerned.)  Therefore, in case (i), estimate \eqref{E: energy_case_i} holds.

Now consider case (ii).  Since $|u| < 1 - \zeta$ in $\left(\frac{1}{4}, \frac{1}{2} + \kappa\right]$ and $u$ is anti-symmetric about $\frac{1}{4}$, it follows that $|u| < 1 - \zeta$ in $[\frac{1}{2} - \kappa, \frac{1}{2} + \kappa]$.  Thus, we obtain the following trivial bound:
\begin{align*}
\mathcal{F}^{a}_{1}\left(u; \left[\kappa, \frac{1}{2} - \kappa\right]\right) &\geq \min \left\{ W(u) \, \mid \, -(1 - \zeta) < u < 1 + \zeta \right\} \left(\frac{1}{2} - 2 \kappa \right) \\
&=: g(\zeta).
\end{align*}

To conclude, we pick $\zeta_{\kappa} > 0$ so small that $f(\zeta_{\kappa}) > 0$ and then we set 
\begin{equation*}
\sigma_{\kappa} = \min \left\{f(\zeta_{\kappa}),g(\zeta_{\kappa}) \right\}.
\end{equation*}
Finally, we have
$
\mathcal{F}^{a}_{1}(u; \mathbb{R}) \geq \mathcal{F}^{a}_{1}\left(u; \left[\kappa, \frac{1}{2} - \kappa\right]\right) \geq \sigma_{\kappa}$.
 \qed \end{proof}  

Now we show that it is possible to get a better energy than in the previous result.

\begin{prop} \label{P: root_delta}  For each $\kappa \in (0,\frac{1}{4})$, as $\delta \to 0^{+}$, there is a $u^{\delta} : \mathbb{R} \to [-1,1]$ satisfying $\lim_{x \to \pm 1} u(x) = \pm 1$ such that
\begin{equation*}
\mathcal{F}^{a}_{1}(u^{\delta}; \mathbb{R}) \leq \sqrt{\delta} \left(\int_{-1}^{1} \sqrt{W(u)} \, du + o(1)\right).
\end{equation*}
  \end{prop}  

\begin{proof}  Define $u : \mathbb{R} \to \mathbb{R}$ by 
\begin{equation*}
u(x) = \tanh \left(\frac{x - \frac{3}{4}}{\sqrt{\delta}} \right).
\end{equation*}
Since $u$ is a minimal front-like stationary solution for \eqref{E: main} with $a$ replaced by $\delta \text{Id}$, we already have the following result to work with:
\begin{equation*}
\int_{-\infty}^{\infty} \left( \frac{\delta u'(x)^{2}}{2} + W(u(x)) \right) \, dx = \sqrt{\delta} \int_{-1}^{1} \sqrt{2W(u)} \, du.
\end{equation*}
Thus, compensating for the error introduced by changing $a$, we find
\begin{align*}
\mathcal{F}^{a}_{1}(u; \mathbb{R}) &\leq \sqrt{\delta} \int_{-1}^{1} \sqrt{W(u)} \, du + \frac{(1 - \delta)}{2} \int_{-\infty}^{\frac{1}{2} + \kappa} u'(x)^{2} \, dx \\
	&\quad+ \frac{(1 - \delta)}{2} \int_{1 - \kappa}^{\infty} u'(x)^{2} \, dx \\
	&\leq \sqrt{\delta} \int_{-1}^{1} \sqrt{W(u)} \, du + (1 - \delta) \int_{1 - \kappa}^{\infty} e^{- \frac{2x}{\sqrt{\delta}}} \, dx \\
	&= \sqrt{\delta} \int_{-1}^{1} \sqrt{W(u)} \, du + (1 - \delta) \sqrt{\delta} \exp\left(-\frac{2(1 - \kappa)}{\sqrt{\delta}}\right).
\end{align*}
Since $(1 - \delta) \sqrt{\delta} \exp\left(-\frac{2(1 - \kappa)}{\sqrt{\delta}}\right) = o(\sqrt{\delta})$ as $\delta \to 0^{+}$, the result follows.  \qed
\end{proof}  

\begin{remark}  If $W(u) = \frac{1}{4}(1 - u^{2})^{2}$ is replaced by any other potential satisfying \eqref{A: W_assumption_1} and \eqref{A: W_assumption_2}, then Proposition \ref{P: root_delta} still holds, but it is necessary to replace the hyperbolic tangent by the standing wave solution of $-u'' + W'(u) = 0$ in $\mathbb{R}$.  By \eqref{A: W_assumption_2} and Schauder estimates, this function enjoys the same exponential decay properties that were used to establish the $\sqrt{\delta}$ estimate.  \end{remark}  

Putting Propositions \ref{P: root_delta} and \ref{P: O_1_estimate} together, we see that if $\delta$ is sufficiently small (depending on $\kappa$), then there is no minimal front-like stationary solution $u$ of \eqref{E: main} satisfying $u \left(\frac{1}{4} \right) = 0$.

Now we conclude the proof:

\begin{proof}[Proof of Theorem \ref{T: counterexample}]  By the symmetry assumptions on $a$, we know that $U \in \mathcal{M}_{-e_{1}}(\mathbb{R} \times \mathbb{T})$ if and only if the function $\tilde{U}(s,x) = U(s,\frac{1 - 2x}{2})$ satisfies $\tilde{U} \in \mathcal{M}_{e_{1}}(\mathbb{R} \times \mathbb{T})$.  Thus, we only need to study $\mathcal{M}_{e_{1}}(\mathbb{R} \times \mathbb{T})$.  

Suppose $U \in \mathcal{M}_{e_{1}}(\mathbb{R} \times \mathbb{T})$ is in $C(\mathbb{R} \times \mathbb{T})$.  Let $\{u_{\zeta}\}_{\zeta \in \mathbb{R}}$ be the functions generated by $U$.  By assumption, the function $\zeta \mapsto u_{\zeta}(\frac{1}{4})$ is continuous.  Moreover, $\lim_{\zeta \to -\infty} u_{\zeta}(\frac{1}{4}) = 1$ and $\lim_{\zeta \to \infty} u_{\zeta}(\frac{1}{4}) = -1$.  Thus, there is a $\zeta_{*} \in \mathbb{R}$ such that $u_{\zeta_{*}}(\frac{1}{4}) = 0$.     

On the other hand, since $U$ is continuous, Proposition \ref{P: continuity of the minimizers} implies $u_{\zeta_{*}}$ is a plane-like minimizer of $\mathcal{F}^{a}_{1}$.  In particular, if $u^{\delta}$ is the function obtained in Proposition \ref{P: root_delta} and $\sigma_{\kappa} > 0$ is the constant from Proposition \ref{P: O_1_estimate}, then
	\begin{equation*}
		\sigma_{\kappa} \leq \mathcal{F}^{a}_{1}(u_{\zeta_{*}}; \mathbb{R}) \leq \mathcal{F}^{a}_{1}(u^{\delta}; \mathbb{R}).
	\end{equation*}  
This is a contradiction if $\delta > 0$ is chosen small enough.  \qed \end{proof}  

\subsection{Non-differentiability at $\pm e_{1}$} \label{S: not differentiable}  Now we show that $\tilde{\varphi}^{a}$ is not differentiable at $e_{1}$ (nor, by symmetry, at $-e_{1}$).  To start with, it will be useful in what follows to utilize so-called heteroclinic minimizers located inside the gaps of the one-dimensional ones.  Since we are working in a laminar medium in $\mathbb{R}^{2}$, the structure of these heteroclinic solutions is particularly simple.  A much more general treatment can be found in \cite{minimal_laminations}.

In the rest of this section, we will write $(x,y)$ for points in $\mathbb{R}^{2}$ with $\langle (x,y), e_{1} \rangle = x$ and $\langle (x,y), e_{2} \rangle = y$.  Moreover, for $e \in S^{1} \setminus \{e_{1},-e_{1}\}$, we will let $\{u^{e}_{\zeta}\}_{\zeta \in \mathbb{R}}$ denote the family of functions generated by $U_{e}$, where $U_{e}$ is the unique minimizer in $\mathcal{M}_{e}(\mathbb{R} \times \mathbb{T})$ with $\int_{\mathbb{T}} U_{e}(0,x) \, dx = 0$.  

\begin{prop} \label{P: tertiary_foliation}  If $(\nu_{n})_{n \in \mathbb{N}} \subseteq S^{1} \setminus \{e_{1},-e_{1}\}$ and $(\zeta_{n})_{n \in \mathbb{N}} \subseteq \mathbb{R}$ satisfy, for each $n \in \mathbb{N}$,
	\begin{itemize}
		\item[(i)] $\langle \nu_{n},e_{2} \rangle > 0$ (resp.\ $\langle \nu_{n}, e_{2} \rangle < 0$),
		\item[(ii)] $u^{\nu_{n}}_{\zeta_{n}}(\frac{1}{4},0) = 0$,
	\end{itemize}
and if $\lim_{n \to \infty} \nu_{n} = e_{1}$, then there is a subsequence $(n_{j})_{j \in \mathbb{N}}$ and a Class A minimizer $u$ of $\mathcal{F}^{a}_{1}$ such that $u = \lim_{j \to \infty} u^{\nu_{n_{j}}}_{\zeta_{n_{j}}}$ locally uniformly in $\mathbb{R}^{2}$ and
	\begin{align*}
		u(x + ke_{1},y) &\geq u (x,y) \quad \text{(resp.} \, \, u(x+ ke_{1},y) \leq u(x,y)\text{)} \quad \text{if} \, \, k \in \mathbb{N}, \\
		u(x,y + \delta) &> u(x,y) \quad \text{(resp.} \, \, u(x,y + \delta) < u(x,y)\text{)} \quad \text{if} \, \, \delta > 0, \\
		\lim_{x \to \pm \infty} u(x,y) &= \pm 1 \quad \text{(resp.} \lim_{x \to \pm \infty} u(x,y) = \mp 1\text{)}.
	\end{align*}  
\end{prop}  

	\begin{proof}  To start with, assume that $\langle \nu_{n},e_{2} \rangle > 0$ independently of $n$.  By the Arzel\`{a}-Ascoli Theorem and elliptic regularity, if $(n_{j})_{j \in \mathbb{N}} \subseteq \mathbb{N}$ is any subsequence, then there is a further subsequence $(n_{j_{k}})_{k \in \mathbb{N}}$ and a Class A minimizer $u$ of $\mathcal{F}^{a}_{1}$ such that $u = \lim_{k \to \infty} u^{\nu_{n_{j_{k}}}}_{\zeta_{n_{j_{k}}}}$ locally uniformly.  
	
	Note, on the other hand, that, passing to another sub-sequence if necessary, $U_{e_{n}} \to U$ pointwise as $n \to \infty$, where $U \in \mathcal{M}_{e_{1}}(\mathbb{R} \times \mathbb{T})$.  (This follows from the compactness result proved in Section \ref{S: existence}.)  Using this and the monotonicity of the families $\{u^{\nu_{n}}_{\zeta}\}_{\zeta \in \mathbb{R}}$, it is not hard to show that $\lim_{\langle x,e_{1} \rangle \to \pm \infty} u(x) = \pm 1$ uniformly in $\langle e_{1} \rangle^{\perp}$.  
	
	If $k \in \mathbb{N}$, then $\langle ke_{1}, \nu_{n} \rangle > 0$ for large enough $n$.  Thus, $u(x + ke_{1},y) \geq u(x,y)$.  Similarly, if $k \in - \mathbb{N}$, then $u(x + k,y) \leq u(x,y)$.  In view of what was proved in the previous paragraph, the inequality is strict if $|k| > 0$.  
	
	Now recall that we can write
		\begin{equation*}
			u^{\nu_{n}}_{\zeta_{n}}(x,y) = U_{\nu_{n}}(x \langle \nu_{n}, e_{1} \rangle + y \langle \nu_{n}, e_{2} \rangle - \zeta, x).
		\end{equation*}
	Thus, $(u^{\nu_{n}}_{\zeta_{n}})_{y} = \langle \nu_{n},e_{2} \rangle \partial_{s} U_{\nu_{n}} > 0$ in $\mathbb{R}^{2}$.  Therefore, since $u^{\nu_{n}}_{\zeta_{n}} \to u$ locally uniformly, it follows that $u_{y} \geq 0$.  Finally, observe that $v = u_{y}$ satisfies $- \text{div}(a(x) Dv) + W''(u) v = 0$ in $\mathbb{R}^{2}$.  Thus, by the strong maximum principle (cf.\ \cite[Corollary A.3]{valdinoci de la llave}), either $v > 0$ or $v \equiv 0$.
	
	If $u_{y} \equiv 0$, then $u = u(x)$ and then the fact that $u(\frac{1}{4}) = 0$ and $u$ is a Class A minimizer heteroclinic between $1$ and $-1$ would contradict Theorem \ref{T: counterexample}.  Therefore, $u_{y} > 0$ in $\mathbb{R}^{2}$. \qed \end{proof}  
	    
At this point, we will want to dig deeper into the properties of the minimizers $\{u^{e}_{\zeta}\}_{\zeta \in \mathbb{R}}$ generated by $U_{e}$ with $e \in S^{1} \setminus \{e_{1},-e_{1}\}$.  

Notice that $u^{e}_{\zeta}(x,y) = U_{e}(x \langle e, e_{1} \rangle + y \langle e,e_{2} \rangle - \zeta, x)$.  From this, we see that 
	\begin{equation*}
		u^{e}_{\zeta}(x,y) = u^{e}_{0}(x,y -\langle e,e_{2} \rangle^{-1}\zeta).
	\end{equation*} 
In particular, the functions $\{u^{e}_{\zeta}\}_{\zeta \in \mathbb{R}}$ are generated by translation in the $y$ variable.  

Next, in connection with Remark \ref{R: symmetry birkhoff}, observe that $\{u^{e}_{\zeta}\}_{\zeta \in \mathbb{R}}$ is periodic with respect to a finer lattice than the module $M_{e}$ defined in Section \ref{S: preliminary_analysis}.  In this simple, two-dimensional setting, we can simply observe that
	\begin{align*}
		u^{e}_{\zeta} \left(x + 1, y - \frac{\langle e, e_{1} \rangle}{\langle e, e_{2} \rangle}\right)
		&= U_{e}(x \langle e, e_{1} \rangle + y \langle e - \zeta, e_{2} \rangle, x) = u^{e}_{\zeta}(x,y).
	\end{align*}
From this, it is convenient to define $I_{e}$ (analogous to $Q_{e}$ in Section \ref{S: preliminary_analysis}) by
	\begin{equation}
		I_{e} = \left\{s \left(1,- \langle e, e_{2} \rangle ^{-1} \langle e, e_{1} \rangle \right) \, \mid \, s \in [0,1]\right\}.
	\end{equation} 
	
Now we show that $\tilde{\varphi}^{a}$ is not differentiable at $\{e_{1},-e_{1}\}$ in this set-up:
	
\begin{prop} \label{P: non-differentiability} If we define $e_{\theta} = \cos(\theta) e_{1} + \sin(\theta)e_{2}$, then 
	\begin{equation*}
		\lim_{\theta \to 0^{+}} \langle D \tilde{\varphi}^{a}(e_{\theta}),e_{2} \rangle > 0, \quad \lim_{\theta \to 0^{-}} \langle D \tilde{\varphi}^{a}(e_{\theta}), e_{2} \rangle < 0.
	\end{equation*}  
In particular, $\tilde{\varphi}^{a}$ is not differentiable at $e_{1}$ or $-e_{1}$.  \end{prop}  

	\begin{proof}  Suppose $\theta \in (-\frac{\pi}{2},\frac{\pi}{2}) \setminus \{0\}$ and fix $\zeta_{\theta} \in \mathbb{R}$ such that $u_{\zeta_{\theta}}^{e_{\theta}}$ satisfies $u_{\zeta_{\theta}}^{e_{\theta}}(\frac{1}{4},0) = 0$.  Recall from Section \ref{S: einstein relation laminar media} that $D \tilde{\varphi}^{a}(e_{\theta})$ is given by
		\begin{align*}
			D \tilde{\varphi}^{a}(e_{\theta}) &= \int_{\mathbb{R} \times \mathbb{T}^{d}} a(x) \mathcal{D}_{e_{\theta}} U_{e_{\theta}} \partial_{s}U_{e_{\theta}} \, dx \, ds \\
				&= \langle e_{\theta}, e_{2} \rangle^{-1} \mathcal{H}^{1}(I_{e_{\theta}})^{-1} \int_{I_{e_{\theta}} \oplus_{e_{\theta}} \mathbb{R}} a(x) D u^{e_{\theta}}_{\zeta_{\theta}}(x,y) \partial_{y} u^{e_{\theta}}_{\zeta_{\theta}}(x,y) \, dx \, dy \\
				&=  \text{sgn}(\langle e_{\theta},e_{2} \rangle)\int_{I_{e_{\theta}} \oplus_{e_{\theta}} \mathbb{R}} a(x) D u^{e_{\theta}}_{\zeta_{\theta}}(x,y) \partial_{y} u^{e_{\theta}}_{\zeta_{\theta}}(x,y) \, dx \, dy.
		\end{align*}
	In particular, since $a$ takes values in the diagonal matrices,
		\begin{equation*}
			\langle D \tilde{\varphi}^{a}(e_{\theta}), e_{2} \rangle = \text{sgn}(\langle e_{\theta},e_{2} \rangle) \int_{I_{e_{\theta}} \oplus_{e_{\theta}} \mathbb{R}} a_{2}(x) \partial_{y} u_{\zeta_{\theta}}^{e_{\theta}}(x,y)^{2} \, dx \, dy.
		\end{equation*}
	This shows that $\langle D \tilde{\varphi}^{a}(e_{\theta}), e_{2} \rangle > 0$ if $\langle e_{\theta},e_{2} \rangle > 0$ and $\langle D \tilde{\varphi}^{a}(e_{\theta}), e_{2} \rangle < 0$ if $\langle e_{\theta},e_{2} \rangle < 0$.  
	
	Since we are working in dimension two, note that $\partial \tilde{\varphi}^{a}(e_{1})$ is either a singleton or a line segment.  Thus, $\lim_{\theta \to 0^{+}} D\tilde{\varphi}^{a}(e_{\theta})$ and $\lim_{\theta \to 0^{-}} D \tilde{\varphi}^{a}$ both exist and converge to either $D \tilde{\varphi}^{a}(e_{1})$ or the (distinct) boundary endpoints of $\partial \tilde{\varphi}^{a}(e_{1})$.  Thus, from the previous paragraph, we see that $D\tilde{\varphi}^{a}(e_{1})$ exists only if $\lim_{\theta \to 0} \langle D \tilde{\varphi}^{a}(e_{\theta}),e_{2} \rangle = 0$.  That means that to prove non-differentiability, we only need to show that 
		\begin{equation*}
			\liminf_{\theta \to 0} |\langle D \tilde{\varphi}(e_{\theta}), e_{2} \rangle| > 0.
		\end{equation*}  
	
	We will proceed arguing by contradiction.  Suppose that, in contrast to what we wish to prove, $\lim_{\theta \to 0^{+}} \langle D \tilde{\varphi}(e_{\theta}),e_{2} \rangle = 0$.  Appealing to our previous computations and the positivity of $a_{2}$, we find
		\begin{equation} \label{E: integral}
			 \lim_{\theta \to 0^{+}} \int_{I_{e_{\theta}} \oplus \mathbb{R}} \partial_{y} u_{\zeta_{\theta}}^{e_{\theta}}(x,y)^{2} \, dx \, dy = 0.
		\end{equation}
	We claim this is impossible.
	
	Indeed, if \eqref{E: integral} were true, then we would deduce that $(u_{\zeta_{\theta}}^{e_{\theta}})_{y} \to 0$ in $L^{2}_{\text{loc}}(\mathbb{R}^{2})$.  Passing to a sub-sequence $\theta_{n} \to 0^{+}$, we can assume that there is a Class A minimizer $u$ satisfying the conclusions of Proposition \ref{P: tertiary_foliation} such that $u^{e_{\theta_{n}}}_{\zeta_{\theta_{n}}} \to u$ locally uniformly.  From the local uniform convergence, we deduce that $(u^{e_{\theta_{n}}}_{\zeta_{\theta_{n}}})_{y} \rightharpoonup u_{y}$ in $L^{2}_{\text{loc}}(\mathbb{R}^{2})$.  We are left to conclude that $u_{y} \equiv 0$, which contradicts the fact that $u$ is strictly increasing in the $y$ variable.  
	
	Thus, we conclude that $\tilde{\varphi}^{a}$ is not differentiable at $e_{1}$.  That $\tilde{\varphi}^{a}$ is not differentiable at $-e_{1}$ follows by symmetry. \qed  
\end{proof}  

From the previous result, we deduce

\begin{prop}  In the notation of Proposition \ref{P: non-differentiability}, there are positive constants $C_{+},C_{-} > 0$ such that 
	\begin{equation*}
		C_{\pm} = \lim_{\theta \to 0^{\pm}} |\langle e_{\theta},e_{2} \rangle| \int_{\mathbb{R} \times \mathbb{T}} a_{2}(x) \partial_{s}U_{e_{\theta}}(s,x)^{2} \, ds \, dx.
	\end{equation*}
\end{prop}  

	\begin{proof}  In the previous proof, we showed that $\lim_{\theta \to 0^{\pm}} |\langle D\tilde{\varphi}^{a}(e_{\theta}), e_{2} \rangle| > 0$.
	
	On the other hand, using the identity $\langle \mathcal{D}_{e_{\theta}} U_{e_{\theta}}, e_{2} \rangle = \langle e_{\theta},e_{2} \rangle \partial_{s} U_{e_{\theta}}$, we find
		\begin{align*}
			|\langle D\tilde{\varphi}^{a}(e_{\theta}), e_{2} \rangle| &= \left| \int_{\mathbb{R} \times \mathbb{T}}  \langle a(x)\mathcal{D}_{e_{\theta}}U_{e_{\theta}}(s,x), e_{2} \rangle \partial_{s} U_{e_{\theta}}(s,x) \, dx \, ds\right| \\
				&= \left|\langle e_{\theta},e_{2} \rangle \right| \int_{\mathbb{R} \times \mathbb{T}} a_{2}(x) \partial_{s} U_{e_{\theta}}(s,x)^{2} \, dx \, ds. 
		\end{align*} \qed
	\end{proof}  
	
Finally, we conclude with the

\begin{proof}[Proof of Theorem \ref{T: dim_2_stuff}]  Notice that the previous corollary and \eqref{A: a_assumption_1} shows that (ii) holds.  

We already proved (i) in Proposition \ref{P: non-differentiability}.   Now we prove (iii).  We proceed by appealing to the fact that $D^{2}\tilde{\varphi}^{a}$ is a Radon measure in $\mathbb{R}^{d}$.  

From (ii), we know there is a $C > 0$ such that
	\begin{equation*}
		C^{-1} |\langle e, e_{2} \rangle|^{-1} \leq \tilde{M}^{a}(e) \leq C |\langle e, e_{2} \rangle|^{-1}
	\end{equation*}
For convenience, extend $\tilde{M}^{a}$ to $\mathbb{R}^{d} \setminus \{0\}$ by $\tilde{M}^{a}(v) = \tilde{M}^{a}(\|v\|^{-1}v)$.  From this, we see that
		\begin{align*}
			C^{-1} \int_{\{\frac{1}{2} \leq \|v\| \leq 2\}} &\left(\frac{\|D^{2}\tilde{\varphi}^{a}(v)\|}{\tilde{M}^{a}(v)}\right) |\langle v, e_{2} \rangle|^{-1} \, dv \\
			&\qquad \leq \int_{\{\frac{1}{2} \leq \|v\| \leq 2\}} \left(\frac{\|D^{2} \tilde{\varphi}^{a}(v)\|}{\tilde{M}^{a}(v)}\right) \tilde{M}^{a}(v) \, dv \\
				&\qquad \leq \| D^{2} \tilde{\varphi}^{a}\| \left(\left\{\frac{1}{2} \leq \|v\| \leq 2\right\}\right) < \infty.
		\end{align*}
	Since $e \mapsto |\langle e, e_{2} \rangle|^{-1}$ is not integrable in any arc of $S^{1}$ containing $e_{1}$ or $-e_{1}$ and $v \mapsto \tilde{\varphi}^{a}(v)$ is $1$-homogeneous, we conclude $\liminf_{e \to \pm e_{1}} \frac{\|D^{2}\tilde{\varphi}^{a}(e)\|}{\tilde{M}^{a}(e)} = 0$. 
	
Finally, note that we can take $a$ to have the form $a = a_{0} \text{Id}$ by setting $a_{2} = a_{1}$ in \eqref{E: matrix field}. \qed  \end{proof}
	
\section{Sharp Interface Limit for Graphs}  \label{S: sharp_interface_limit}

In this section, we prove Theorem \ref{T: sharp_interface_limit_graphs}.  Concerning $a$ and $W$, assumptions \eqref{A: a_assumption_1}, \eqref{A: W_assumption_1}, \eqref{A: a_assumption_2}, and \eqref{A: W_assumption_2} are all in effect, $a$ is $\mathbb{Z}^{k} \times \mathbb{R}^{d-k}$-periodic, and, in addition, we assume \eqref{A: additional_regularity} and \eqref{A: sign}.

The additional regularity of $W$ in \eqref{A: additional_regularity} allows us to invoke Propositions \ref{P: higher_regularity} and \ref{P: key corrector result} concerning the regularity of $\tilde{\varphi}^{a}$, $e\mapsto U_{e}$, and $e \mapsto P_{e}$.

The condition on the sign of $W'$ in \eqref{A: sign} frequently appears in the literature on reaction diffusion equations.  It is only needed in the proof of  Lemma \ref{L: initialization} below.

In what follows, we fix $e \in S^{d - 1} \setminus (\mathbb{R}^{k} \times \{0\})$.
	
We let $\pi : \mathbb{R}^{d} \to \mathbb{R}^{d -1}$ be a linear isometry annihilating $\langle e \rangle$ and we will use the identification $x = (x_{e},x')$, where $x_{e} = \langle x, e \rangle$ and $x' = \pi(x)$.  

Finally, since we are interested in Theorem \ref{T: sharp_interface_limit_graphs}, it will be convenient to fix an orientation.  More precisely, notice that assumption (i) of the theorem actually implies (via the intermediate value theorem) that 
	\begin{align*}
		&\text{either} \quad \{u_{0} > 0\} = \left\{x \in \mathbb{R}^{d} \, \mid \, x_{e} > \mathcal{U}(x') \right\} \\
		&\text{or} \quad \{u_{0} > 0\} = \left\{ x \in \mathbb{R}^{d} \, \mid \, x_{e} < \mathcal{U}(x') \right\}.
	\end{align*}
Notice that if the former is true, then we can switch to the latter by replacing $e$ by $-e$.  Therefore, in order to avoid disorienting case analyses, we will assume henceforth that the former always holds, that is,
	\begin{equation} \label{E: orientation}
		\{u_{0} > 0\} = \{(x_{e},x') \in \mathbb{R}^{d} \, \mid \, x_{e} > \mathcal{U}(x') \}.
	\end{equation}

\subsection{Graph Equation}  To make sense of the anisotropic curvature flow \eqref{E: anisotropic curvature flow}, we use the level set formulation.  Recall that this encodes the geometric flow via the nonlinear parabolic PDE
	\begin{equation} \label{E: level set}
		\left\{ \begin{array}{r l}
			u_{t} - \tilde{M}^{a}(\widehat{Du})^{-1} \text{tr} \left( D^{2} \tilde{\varphi}^{a}(\widehat{Du}) D^{2}u \right) = 0 & \text{in} \, \, \mathbb{R}^{d} \times (0,T), \\
			u = u_{0} & \text{on} \, \, \mathbb{R}^{d} \times \{0\}.
		\end{array} \right.
	\end{equation}
The idea of the level set approach is that the solution $(E_{t})_{t \geq 0}$ of \eqref{E: anisotropic curvature flow} for a given initial open set $E_{0}$ can be obtained by choosing the initial datum $u_{0}$ so that $\{u_{0} > 0\} = E_{0}$ and defining $E_{t} = \{u(\cdot,t) > 0\}$.  See \cite{barles souganidis} for a complete account of the level set formulation. 

Note that \eqref{E: level set} is ambiguous at this stage since the ratio $(\tilde{M}^{a})^{-1} D^{2}\tilde{\varphi}^{a}$ is not necessarily well-defined in the directions of the laminations.  Nonetheless, as in the case of mean curvature flow, \eqref{E: anisotropic curvature flow} has a graph formulation, which is unambiguous when the graphs cross the laminations.  More precisely, define the matrix field $\tilde{\mathcal{G}} : \mathbb{R}^{d -1} \to \mathcal{S}_{d-1}$ by 
	\begin{equation*}
		\tilde{\mathcal{G}}(q) = \tilde{M}^{a}(q_{g})^{-1} \pi D^{2}\tilde{\varphi}^{a}(q_{g}) \pi^{*},
	\end{equation*}
where $q_{g} = (1 + \|q\|^{2})^{-\frac{1}{2}} (1,-q)$ and $\pi$ is the linear map defined above.  This is well-defined by our assumptions on $e$.

One can prove the following (e.g.\ using inf- and sup-convolutions):

	\begin{prop} \label{P: graph_equation_fact} Suppose $h : \mathbb{R}^{d -1} \times [0,T] \to \mathbb{R}$ is upper (resp.\ lower) semi-continuous and $h_{0} \in C(\mathbb{R}^{d- 1})$.  The function $u: \mathbb{R}^{d} \times [0,T] \to \mathbb{R}$ given by $u(x,t) = x_{e} - h(x',t)$ is a viscosity sub-solution (resp.\ super-solution) of \eqref{E: level set} with initial datum $u_{0}(x) = x_{e} - h_{0}(x')$ if and only if $h$ is a viscosity super-solution (resp.\ sub-solution) of 
		\begin{equation} \label{E: graph_equation}
			\left\{ \begin{array}{r l}
				h_{t} - \text{tr}(\tilde{\mathcal{G}}(Dh)D^{2}h) = 0 & \text{in} \, \, \mathbb{R}^{d -1} \times (0,T), \\
				h(\cdot,0) = h_{0} & \text{on} \, \, \mathbb{R}^{d -1}.
			\end{array} \right.
		\end{equation}
	\end{prop}  
	
In terms of the geometric flow, Proposition \ref{P: graph_equation_fact} says that if $E_{0} = \{x_{e} > h_{0}(x')\}$ and $h$ is a solution of \eqref{E: graph_equation}, we can generate a graphical solution $(E_{t})_{t \geq 0}$ of \eqref{E: anisotropic curvature flow} with initial datum $E_{0}$ by setting $E_{t} = \{x_{e} > h(x',t)\}$.  This is precisely the approach we will take in Theorem \ref{T: sharp_interface_limit_graphs}.
	
In studying \eqref{E: graph_equation}, we will restrict to the following families of test functions $\mathscr{P}^{+}(\bar{M})$ and $\mathscr{P}^{-}(\bar{M})$: first, we say that $\varphi \in \mathscr{P}_{0}^{\pm}(\bar{M})$ if there is a smooth function $\varphi_{0} : \mathbb{R}^{d-1} \to \mathbb{R}$, an $x_{0}' \in \mathbb{R}^{d-1}$, constants $M, R > 0$ with $M \geq \bar{M}$, a $C \in \mathbb{R}$, and a cut-off function $\rho \in C^{\infty}_{c}(\mathbb{R}^{d-1})$ satisfying $\rho(x') = 1$ for $\|x' - x_{0}'\| \leq R$ such that
	\begin{equation*}
		\varphi(x') = \varphi_{0}(x') \rho(x') \pm (1 - \rho(x')) M\|x' - x_{0}'\|+ C.
	\end{equation*}
Finally, we say that $\varphi \in \mathscr{P}^{\pm}(\bar{M})$ if there is a $\varphi_{1} \in \mathscr{P}_{0}^{\pm}(\bar{M})$, a $t_{0} \in [0,T]$, and constants $a,b \in \mathbb{R}$ such that
	\begin{equation*}
		\varphi(x',t) = \varphi_{1}(x') + a(t - t_{0}) + \frac{b(t-t_{0})^{2}}{2}.
	\end{equation*}

The idea is we want test functions whose graphs (at any fixed time) have nice tubular neighborhoods and normal vectors bounded a positive distance away from $S^{k-1} \times \{0\}$.  Since the graph of any function in $\mathscr{P}^{+}(\bar{M})$ or $\mathscr{P}^{-}(\bar{M})$ is a compact perturbation of a cone, these functions are well suited to the purpose.
	
Here is how we define sub- and super-solutions of \eqref{E: graph_equation} using $\mathscr{P}^{+}$ and $\mathscr{P}^{-}$:

	\begin{definition}  Given an $\bar{M} > 0$, we say that an upper semi-continuous function $h : \mathbb{R}^{d -1} \times [0,T] \to \mathbb{R}$ is a $\mathscr{P}(\bar{M})$-sub-solution of \eqref{E: graph_equation} if, for each $\varphi \in \mathscr{P}^{+}(\bar{M})$, if $h - \varphi$ has a strict global maximum at $(x_{0}',t_{0}) \in \mathbb{R}^{d-1} \times (0,T)$, then 
		\begin{equation*}
			\varphi_{t}(x_{0}',t_{0}) - \text{tr}(\tilde{\mathcal{G}}(D\varphi(x_{0}',t_{0}))D^{2}\varphi(x_{0}',t_{0})) \leq 0.
		\end{equation*} 
		
	We say that a lower semi-continuous function $g : \mathbb{R}^{d -1} \times [0,T] \to \mathbb{R}$ is a $\mathscr{P}(\bar{M})$-super-solution of \eqref{E: graph_equation} if, for each $\varphi \in \mathscr{P}^{-}(\bar{M})$, if $g - \varphi$ has a strict global minimum at $(x_{0}',t_{0}) \in \mathbb{R}^{d-1} \times (0,T)$, then 
		\begin{equation*}
			\varphi_{t}(x_{0}',t_{0}) - \text{tr}(\tilde{\mathcal{G}}(D\varphi(x_{0}',t_{0}))D^{2}\varphi(x_{0}',t_{0})) \geq 0.
		\end{equation*}
	\end{definition} 

Since we are working with sub- and super-solutions that remain a bounded distance from a hyperplane, the next result is not hard to prove.

	\begin{prop} \label{P: comparison_principle}  Fix a $q \in \mathbb{R}^{d-1}$.  If $h : \mathbb{R}^{d -1} \times [0,T] \to \mathbb{R}$ is an upper semi-continuous $\mathscr{P}(\|q\| + 1)$-sub-solution of \eqref{E: graph_equation} and $g : \mathbb{R}^{d-1} \times [0,T] \to \mathbb{R}$ is a lower semi-continuous $\mathscr{P}(\|q\|+1)$-super-solution, and if $h$ and $g$ satisfy		
		\begin{align*}
			&\sup \left\{ |h(x',t) - \langle q,x' \rangle| + |g(x',t) - \langle q, x' \rangle| \, \mid \, (x',t) \in \mathbb{R}^{d-1} \times [0,T] \right\} < \infty, \\
		&\lim_{\delta \to 0^{+}} \sup \left\{h(x',0) - g(y',0) \, \mid \, x',y' \in \mathbb{R}^{d-1}, \, \, \|x' - y'\| \leq \delta \right\} = 0,
	\end{align*}
then $h \leq g$ in $\mathbb{R}^{d-1} \times [0,T]$.  \end{prop}  
	
The reason the proposition is true is that a $\mathscr{P}(\|q\| + 1)$-sub-solution is actually a sub-solution in the usual sense, and similarly for super-solutions.  We prove this in Appendix \ref{A: comparison} below.  The comparison result then follows by standard arguments.

Since many of the results to come do not depend on the particular value of $\bar{M}$, we will write
	\begin{equation*}
		\mathscr{P}^{\pm}_{0} = \bigcup_{\bar{M} > 0} \mathscr{P}^{\pm}_{0}(\bar{M}).
	\end{equation*}
	
\subsection{Subgraphs and Supergraphs}

To show that the limiting macroscopic behavior of the boundary of the sets $\{u^{\epsilon} \approx 1\}$ and $\{u^{\epsilon} \approx -1\}$ is described by a graph evolving by \eqref{E: graph_equation}, we will use the minimal supergraphs and maximal subgraphs defined next.  Given $t > 0$, we define sets $\Omega^{1}_{t}$ and $\Omega_{t}^{2}$ by 
	\begin{align*}
		\Omega^{1}_{t} &= \left\{x \in \mathbb{R}^{d} \, \mid \, \liminf \nolimits_{*} u^{\epsilon}(x,t) = 1 \right\}, \\
		\Omega^{2}_{t} &= \left\{x \in \mathbb{R}^{d} \, \mid \, \limsup \nolimits^{*} u^{\epsilon}(x,t) = -1 \right\}.
	\end{align*}
The minimal supergraph and maximal subgraph of $\Omega^{1}_{t}$ and $\Omega^{2}_{t}$, respectively, are described by the functions
	\begin{align}
		\eta(x',t) &= \inf \left\{ y \in \mathbb{R} \, \mid \, (y,\infty) \times \{x'\} \subseteq \Omega^{1}_{t} \right\}, \label{E: definition_of_eta} \\
		\nu(x',t) &= \sup \left\{ y \in \mathbb{R} \, \mid \, (-\infty,y) \times \{x'\} \subseteq \Omega^{2}_{t} \right\}. \label{E: definition_of_nu}
	\end{align}
In the next section, we will see that $\eta = \nu$, this function solves \eqref{E: graph_equation} and is continuous, and $\Omega^{1}_{t} = \{x_{e} > \eta(x',t)\}$ and $\Omega_{t}^{2} = \{x_{e} < \eta(x',t)\}$.  
	
The next result, which follows from assumptions (i) and (ii) of Theorem \ref{T: sharp_interface_limit_graphs} and ideas presented in \cite{barles souganidis}, ensures that $\eta$ and $\nu$ are a bounded distance from the function $x' \mapsto \langle q, x' \rangle$ for all time.  In the statement, the direction $e_{0} \in S^{d-1}$ is chosen so that $\{x_{e} \geq \langle q,x'\rangle\} = \{\langle x,e_{0} \rangle \geq 0\}$.
	
		\begin{prop} \label{P: bounded_width}  Assume that $u_{0} : \mathbb{R}^{d} \to [-1,1]$ is a uniformly continuous function satisfying hypothesis (i) and (ii) of Theorem \ref{T: sharp_interface_limit_graphs} and let $e_{0} = \frac{e - \pi^{-1}(q)}{\|e - \pi^{-1}(q)\|}$.  
	If $(u^{\epsilon})_{\epsilon > 0}$ are the solutions of 
		\begin{equation*}
			\left\{ \begin{array}{r l}
				u^{\epsilon}_{t} - \text{div}(a(\frac{x}{\epsilon})Du^{\epsilon}) + \epsilon^{-2} W'(u^{\epsilon}) = 0 & \text{in} \, \, \mathbb{R}^{d} \times (0,\infty), \\
				u^{\epsilon}(\cdot,0) = u_{0} & \text{on} \, \, \mathbb{R}^{d},
			\end{array} \right.
		\end{equation*}
	and $\limsup^{*} u^{\epsilon}$ and $\liminf_{*} u^{\epsilon}$ are the upper and lower half-relaxed limits defined by \eqref{E: upper half relaxed} and \eqref{E: lower half relaxed}, then for each $s,t \in (0,\infty)$ satisfying $s < t$ and each $\delta \in (\delta_{0},1)$, there is an $M' > 0$ and an $\epsilon_{0} > 0$ such that, for each $\epsilon \in (0,\epsilon_{0})$,
		\begin{align*}
			u^{\epsilon} &\geq 1 - \delta \quad \text{in} \, \, \{x \in \mathbb{R}^{d} \, \mid \, \langle x, e_{0} \rangle \geq M' \} \times \left[s,t\right], \\
			u^{\epsilon} &\leq - 1 + \delta \quad \text{in} \, \, \{x \in \mathbb{R}^{d} \, \mid \, \langle x,e_{0} \rangle \leq -M'\} \times \left[s, t\right].
		\end{align*}
	In particular,
		\begin{align*}
			\liminf \nolimits_{*} u^{\epsilon} &= 1 \quad \text{in} \, \, \{x \in \mathbb{R}^{d} \, \mid \, \langle x, e_{0} \rangle \geq M' \} \times \left(0,t\right], \\
			\limsup \nolimits^{*} u^{\epsilon} &= - 1 \quad \text{in} \, \, \{x \in \mathbb{R}^{d} \, \mid \, \langle x,e_{0} \rangle \leq -M'\} \times \left(0,t\right].
		\end{align*}
	\end{prop}
	
The proof of the proposition is deferred until Section \ref{S: bounded_width}.  Next, we make our previous comments concerning $\eta$ and $\nu$ precise:

	\begin{prop} \label{P: bounded_from_linear}  Fix $T > 0$.  There is an $M_{0} > 0$ such that, for all $t \in (0,T)$,
		\begin{equation*}
			\langle q, x' \rangle - M_{0} \leq \nu(x',t) \leq \eta(x',t) \leq \langle q, x' \rangle + M_{0}.
		\end{equation*}
	\end{prop}  
	
		\begin{proof}  First, notice that the inequality $\limsup^{*} u^{\epsilon}(x,t) \geq \liminf_{*} u^{\epsilon}(x,t)$ implies that $\Omega^{1}_{t} \cap \Omega^{2}_{t} = \emptyset$ for all $t > 0$.  From this, we find that $\nu \leq \eta$ pointwise.
		
		By Proposition \ref{P: bounded_width}, there is an $M' > 0$ such that, for each $t \in (0,T]$,
			\begin{equation*}
				\Omega_{t}^{1} \supseteq \{x \in \mathbb{R}^{d} \, \mid \, \langle x,e_{0} \rangle \geq M'\}, \quad 
				\Omega_{t}^{2} \supseteq \{x \in \mathbb{R}^{d} \, \mid \, \langle x,e_{0} \rangle \leq -M' \}.
			\end{equation*}
		Recall that here $e_{0} = \frac{e - \pi^{-1}(q)}{\|e - \pi^{-1}(q)\|}$.  Unraveling the definitions, we find that
			\begin{equation*}
				\eta(x',t) \leq \langle q, x' \rangle + M'\|e - \pi^{-1}(q)\| \quad \text{if} \, \, x' \in \mathbb{R}^{d-1}.
			\end{equation*}
		The same reasoning shows that $\nu(x',t) \geq \langle q, x' \rangle - M'\|e - \pi^{-1}(q)\|$. \qed		\end{proof}

	Notice that we have not said anything about the structure of the macroscopic phases when $t = 0$.  As we shall see, this is somewhat natural and will not present any difficulties later.

\subsection{Proof of the Sharp Interface Limit}  We will prove Theorem \ref{T: sharp_interface_limit_graphs} by showing that the functions $\eta$ and $\nu$ of the previous section are respectively sub- and super-solutions of \eqref{E: graph_equation}.  To do this, we will construct mesoscopic sub- and super-solutions as in \cite{barles souganidis}.  It is possible to do this in spite of the possible irregularity of the pulsating standing waves precisely because of the graph assumption, the point being we can work with smooth sub- and super-solutions of \eqref{E: graph_equation} whose normal vectors avoid the set $S^{k - 1} \times \{0\}$.  

In what follows, for $t \geq 0$, we define upper and lower semi-continuous envelopes $\eta^{*}$ and $\nu_{*}$ by
	\begin{align*}
		\eta^{*}(x',t) &= \lim_{\delta \to 0^{+}} \sup \left\{ \eta(\tilde{x}',\tilde{t}) \, \mid \, \|\tilde{x}'-x'\| + |\tilde{t} - t| < \delta, \, \, 0 < \tilde{t} \leq T\right\}, \\
		\nu_{*}(x',t) &= \lim_{\delta \to 0^{+}} \inf \left\{ \nu(\tilde{x}',\tilde{t}) \, \mid \, \|\tilde{x}' - x'\| + |\tilde{t} - t| < \delta, \, \, 0 < \tilde{t} \leq T \right\}.
	\end{align*}	

In view of the comparison principle, all we need to prove is the following:

\begin{prop} \label{P: hard_part} Fix $T > 0$.  If $\eta$ is the function defined by \eqref{E: definition_of_eta}, then $\eta^{*}$ is a $\mathscr{P}(\|q\| + 1)$-sub-solution of \eqref{E: graph_equation} and $\eta^{*}(\cdot,0) \leq \mathcal{U}$ in $\mathbb{R}^{d-1}$.
	
	Similarly, if $\nu$ is the function defined by \eqref{E: definition_of_nu}, then $\nu_{*}$ is a $\mathscr{P}(\|q\| + 1)$-super-solution of \eqref{E: graph_equation} and $\mathcal{U} \leq \nu_{*}(\cdot,0)$ in $\mathbb{R}^{d-1}$.  \end{prop}
	
Taking the proposition for granted for now, we can proceed with the

\begin{proof}[Proof of Theorem \ref{T: sharp_interface_limit_graphs}]  Fix $T > 0$.  By definition, since $\eta(x',t) \geq \nu(x',t)$ for all $(x',t) \in \mathbb{R}^{d-1} \times (0,T]$, it follows that $\eta^{*} \geq \eta \geq \nu \geq \nu_{*}$ in $\mathbb{R}^{d-1} \times (0,T]$.  

By Proposition \ref{P: hard_part}, $\eta^{*}$ is a $\mathscr{P}(\|q\|+1)$-sub-solution of \eqref{E: graph_equation} and $\nu_{*}$, a $\mathscr{P}(\|q\|+1)$-super-solution, and $\eta^{*}(\cdot,0) \leq \mathcal{U} \leq \nu_{*}(\cdot,0)$.  Note that the latter inequality and the uniform continuity of $\mathcal{U}$ yields
	\begin{align*}
		&\limsup_{\delta \to 0^{+}} \sup \left\{ \eta^{*}(x',0) - \nu_{*}(y',0) \, \mid \, \|x' - y'\| \leq \delta \right\} \\
		&\quad\leq \lim_{\delta \to 0^{+}} \sup \left\{ \mathcal{U}(x') - \mathcal{U}(y') \, \mid \, \|x' - y'\| \leq \delta\right\} = 0.
	\end{align*}
From this and Proposition \ref{P: bounded_from_linear}, Proposition \ref{P: comparison_principle} implies that $\eta^{*} \leq \nu_{*}$ in $\mathbb{R}^{d-1} \times [0,T]$.  

We showed that $\eta^{*} \leq \nu_{*} \leq \nu \leq \eta \leq \eta^{*}$.  Therefore, $\eta = \nu$ in $\mathbb{R}^{d-1} \times (0,T]$ and, by the definition of $\eta$ and $\nu$, 
	\begin{equation*}
		u^{\epsilon} \to \left\{ \begin{array}{r l}
									1, & \text{locally uniformly in} \, \, \{(x,t) \, \mid \, x_{e} > \eta(x',t)\}, \\
									-1, & \text{locally uniformly in} \, \, \{(x,t) \, \mid \, x_{e} < \eta(x',t)\}.
						\end{array} \right.
	\end{equation*}
Since $\eta$ solves \eqref{E: graph_equation} and $T$ was arbitrary, we complete the proof by invoking Proposition \ref{P: graph_equation_fact}. \qed \end{proof}  
	
To prove Proposition \ref{P: hard_part}, of course, we need a link between the macroscopic problem and the mesoscopic one.  Here we follow the construction of \cite{barles souganidis}, making the necessary alterations so that we we can use graphs as ``test surfaces" instead of compact hypersurfaces.
	
	In what follows, we define the non-linear semi-group $T^{\epsilon}$ so that
		\begin{equation*}
			[T^{\epsilon}(t)w](x) = u(x,t),
		\end{equation*} 
	where $u$ solves \eqref{E: main} with $u(\cdot,0) = w$.  In this section, we will frequently use the fact that this semi-group is monotone: that is, if $u_{1},u_{2} \in BUC(\mathbb{R}^{d})$ satisfy $u_{1} \leq u_{2}$ in $\mathbb{R}^{d}$, then the comparison principle associated with \eqref{E: main} (cf.\ \cite[Proposition 2.1]{aronson weinberger}) implies
		\begin{equation*}
			T^{\epsilon}(t)u_{1} \leq T^{\epsilon}(t)u_{2} \quad \text{in} \, \, \mathbb{R}^{d} \, \, \text{for each} \, \, t > 0.
		\end{equation*} 
	
	Here is the main result we will need:

		\begin{prop} \label{P: main_construction}  For each $\alpha > 0$, $\delta \in (0,1)$, and $\varphi \in \mathscr{P}_{0}^{+}$ (resp.\ $\varphi \in \mathscr{P}_{0}^{-}$), there is an $h_{0} > 0$ depending on $\varphi$ only through bounds on $\|D\varphi\|_{L^{\infty}(\mathbb{R}^{d-1})}$, $\|D^{2}\varphi\|_{L^{\infty}(\mathbb{R}^{d-1})}$, $\|D^{3}\varphi\|_{L^{\infty}(\mathbb{R}^{d-1})}$, and $\|D^{4}\varphi\|_{L^{\infty}(\mathbb{R}^{d-1})}$ such that, for each $h \in (0,h_{0}]$, if $x = (x_{e},x')$ satisfies
		\begin{align*}
			x_{e} &> \varphi(x') + h(\text{tr}(\mathcal{G}(D\varphi(x')) D^{2}\varphi(x')) + \alpha) \\
			(resp. \, \, x_{e} &< \varphi(x') + h(\text{tr}(\mathcal{G}(D\varphi(x'))D^{2} \varphi(x')) - \alpha)),
		\end{align*}
	then, employing the half-relaxed limit notation \eqref{E: upper half relaxed} and \eqref{E: lower half relaxed}, we have
		\begin{align*}
			&\liminf \nolimits_{*} T^{\epsilon}(h)[(1 - \delta) \chi_{\{x_{e} \geq \varphi(x')\}} - \chi_{\{x_{e} < \varphi(x')\}}](x) = 1 \\
			(resp.\ &\limsup \nolimits^{*} T^{\epsilon}(h)[\chi_{\{x_{e} > \varphi(x')\}} + (-1 + \delta) \chi_{\{x_{e} \leq \varphi(x')\}}](x) = -1).
		\end{align*}
	\end{prop}  

As in \cite{barles souganidis}, the proposition is proved in two steps.  First, there is the initialization step.

	\begin{lemma} \label{L: initialization} If $\varphi \in \mathscr{P}_{0}^{+} \cup \mathscr{P}_{0}^{-}$, then, for any $\beta > 0$ sufficiently small and $\delta \in (0,1)$, there are constants $\tau(\beta), \epsilon(\beta) > 0$ such that if $t_{\epsilon} = \tau(\beta) \epsilon^{2} |\log(\epsilon)|$ and $d_{\varphi}$ is the signed distance to the graph $\{x_{e} = \varphi(x')\}$, positive in $\{x_{e} > \varphi(x')\}$, then, for each $\epsilon \in (0,\epsilon(\beta))$, 
		\begin{align*}
			T^{\epsilon}(t_{\epsilon})[(1- \delta) \chi_{\{x_{e} \geq \varphi(x')\}} - \chi_{\{x_{e} < \varphi(x')\}}] \geq (1 - \beta \epsilon) \chi_{\{d_{\varphi}(x)\geq \beta\}} - \chi_{\{d_{\varphi}(x) < \beta\}}, \\
			T^{\epsilon}(t_{\epsilon})[\chi_{\{x_{e} > \varphi(x')\}} + (-1 + \delta) \chi_{\{x_{e} \leq \varphi(x')\}}] \leq \chi_{\{d_{\varphi} > -\beta\}} + (-1 + \beta \epsilon) \chi_{\{d_{\varphi} \leq - \beta\}}.
		\end{align*} 
	\end{lemma}
	
	Next, the propagation step:
	
	\begin{lemma} \label{L: propagation}  Assume that $\varphi \in \mathscr{P}_{0}^{+} \cup \mathscr{P}_{0}^{-}$.  For any $\alpha > 0$, there is an $h_{0} > 0$ depending on $\varphi$  only through $\max\{\|D^{i}\varphi\|_{L^{\infty}(\mathbb{R}^{d-1})} \, \mid \, 1 \leq i \leq 4\}$ such that if $0 < \beta \leq \bar{\beta}(\alpha,\varphi)$, $0 < \epsilon \leq \bar{\epsilon}(\alpha,\beta,\varphi)$, and $d_{\varphi}$ is the signed distance function to the set $\{x_{e} = \varphi(x')\}$, positive in $\{x_{e} > \varphi(x')\}$, then there is a sub-solution (resp.\ super-solution) $w^{\epsilon}$ in $\mathbb{R}^{d} \times (0,h_{0}]$ such that
		\begin{align*}
			&w^{\epsilon}(\cdot,0) \leq (1 - \beta \epsilon) \chi_{\{d_{\varphi}(x) \geq \beta\}} - \chi_{\{d_{\varphi}(x) < \beta\}} \quad \text{in} \, \, \mathbb{R}^{d} \\
		(\text{resp.} \, \, &w^{\epsilon}(\cdot,0) \geq \chi_{\{d(x) > - \beta\}} + (-1 + \beta \epsilon) \chi_{\{d_{\varphi}(x) \leq - \beta\}} \quad \text{in} \, \, \mathbb{R}^{d}).
		\end{align*}
	Furthermore, if $\tilde{\Phi}_{\alpha}^{-} : \mathbb{R}^{d-1} \times [0,h_{0}] \to \mathbb{R}$ is given by
		\begin{equation*}
			\tilde{\Phi}_{\alpha}^{-}(x',t) = \varphi(x') + t(\tilde{\mathcal{G}}(D\varphi(x'),D^{2}\varphi(x')) + \alpha);
		\end{equation*} 
	$d_{\Phi}(\cdot,t)$ is the signed distance to the graph $\{x_{e} = \tilde{\Phi}_{\alpha}^{-}(x',t)\}$, positive in the set $\{x_{e} > \tilde{\Phi}_{\alpha}^{-}(x',t)\}$; and $(x,t) \in \mathbb{R}^{d} \times (0,h_{0}]$ is such that $d_{\Phi}(x,t) > 2 \beta$, then 
		\begin{equation*}
			\liminf \nolimits_{*} w^{\epsilon}(x,t) = 1
		\end{equation*}
	(resp.\ if $\tilde{\Phi}_{\alpha}^{+}(x',t) = \varphi(x') + t (\tilde{\mathcal{G}}(D\varphi(x'),D^{2}\varphi(x')) - \alpha)$; $d_{\Phi}(\cdot,t)$ is the signed distance to the graph $\{x_{e} = \tilde{\Phi}_{\alpha}^{+}(x',t)\}$, positive in the set $\{x_{e} > \tilde{\Phi}_{\alpha}^{+}(x',t)\}$; and $(x,t) \in \mathbb{R}^{d} \times (0,h_{0}]$ is such that $d_{\Phi}(x,t) < - 2 \beta$, then
		\begin{equation*}
			\limsup \nolimits^{*} w^{\epsilon}(x,t) = - 1).
		\end{equation*}
	\end{lemma}  

The proofs of Lemmas \ref{L: initialization} and \ref{L: propagation} are presented in Sections \ref{S: initialization} and \ref{S: propagation}, respectively.  
			
Deferring the proof of Proposition \ref{P: main_construction} until the end of this section, we now use it to show that $\eta^{*}$ and $\nu_{*}$ have the desired properties:

\begin{proof}[Proof of Proposition \ref{P: hard_part}]  

\textbf{Step 1: Sub- and super-solution properties}

We will show that $\eta^{*}$ is a $\mathscr{P}(\|q\| + 1)$-sub-solution; the proof that $\nu_{*}$ is a $\mathscr{P}(\|q\| + 1)$-super-solution is similar.  Suppose that $\varphi \in \mathscr{P}^{+}(\|q\| + 1)$ and $\eta^{*} - \varphi$ has a strict global maximum in $\mathbb{R}^{d-1} \times [0,T]$ at $(x_{0}',t_{0})$ for some $t_{0} > 0$.  Without loss of generality (i.e.\ subtracting a constant from $\varphi$ if necessary), we can assume that $\eta^{*}(x_{0}',t_{0}) = \varphi(x_{0}',t_{0})$.  
	
	We argue by contradiction.  That is, let us assume that there is an $\alpha > 0$ such that 
		\begin{equation*}
			\varphi_{t}(x_{0}',t_{0}) - \text{tr}(\tilde{\mathcal{G}}(D\varphi(x_{0}',t_{0})) D^{2}\varphi(x_{0}',t_{0})) \geq 4 \alpha.
		\end{equation*}
		
	In what follows, it's convenient to define $x_{0,e} = \varphi(x_{0}',t_{0})$. 
	
	Since $(x_{0}',t_{0})$ is a strict global maximum of $\eta^{*} - \varphi$ and $\eta^{*}(x_{0}',t_{0}) = \varphi(x_{0}',t_{0})$, it follows that 
		\begin{equation*}
			\{(x_{e},x') \, \mid \, x_{e} > \varphi(x',t_{0} - h) \} \subseteq \Omega_{t_{0} - h}^{1} \quad \text{if} \, \, h \in (0,t_{0}).
		\end{equation*}
	In particular, since $u^{\epsilon} \to 1$ locally uniformly in $\bigcup_{0 < t \leq T} \Omega_{t}^{1} \times \{t\}$, we can invoke Proposition \ref{P: bounded_width} and the fact that $\varphi \in \mathscr{P}^{+}(\|q\| + 1)$ to find an $\epsilon_{0} > 0$ such that if $\epsilon \in (0,\epsilon_{0})$, then
		\begin{equation*}
			u^{\epsilon} \geq 1 - \delta \quad \text{in} \, \, \bigcup_{\frac{t_{0}}{2} \leq h \leq t_{0}}\{x_{e} > \varphi(x',t_{0} - h)\} \times \{t_{0} - h\}.
		\end{equation*}
	
	Let $\tilde{\Phi}_{\alpha}^{h,-}(x',t) = \varphi(x',t_{0} - h) + t [\text{tr}(\tilde{\mathcal{G}}(D\varphi(x',t_{0} -h))D^{2}\varphi(x',t_{0}-h)) + \alpha]$.  By Proposition \ref{P: main_construction}, there is an $h_{0} \in (0,T - t_{0})$ such that, for each $s \in (0,h_{0})$ and each $h \in [\frac{t_{0}}{2},t_{0}]$,
		\begin{equation*}
			\tilde{\Phi}_{\alpha}^{h,-}(x',s) < x_{e} \quad \implies \quad (x_{e},x') \in \Omega_{t_{0} - h + s}^{1}.
		\end{equation*}
	Now, as $h \to 0^{+}$, we have
		\begin{align*}
			x_{0,e} &= \varphi(x_{0}',t_{0}) > \varphi(x_{0}',t_{0}-h) + \varphi_{t}(x_{0}',t_{0}) h - 3 \alpha h+ o(h) \\
				&= \varphi(x_{0}',t_{0} - h) + h(\text{tr}(\tilde{\mathcal{G}}(D\varphi(x_{0}',t_{0})) D^{2} \varphi(x_{0}',t_{0})) + \alpha) \\
				&\quad + h(\varphi_{t}(x_{0}',t_{0}) - \text{tr}(\tilde{\mathcal{G}}(D\varphi(x_{0}',t_{0})) D^{2} \varphi(x_{0}',t_{0})) - \alpha) - 3 \alpha h + o(h) \\
				&\geq  \varphi(x_{0}',t_{0} - h) + h(\text{tr}(\tilde{\mathcal{G}}(D\varphi(x_{0}',t_{0}-h)) D^{2} \varphi(x_{0}',t_{0}-h)) + \alpha) \\
				&\quad + \alpha h+ o(h) \\
				&= \tilde{\Phi}^{h,-}_{\alpha}(x_{0}',h) + \alpha h + o(h).
		\end{align*}
	Thus, there is an $h' \in (0,h_{0})$ such that
		\begin{equation*}
			x_{0,e} > \tilde{\Phi}^{h',-}_{\alpha}(x_{0}',h').
		\end{equation*}
	By the continuity of $(x_{e},x',t) \mapsto x_{e} - \tilde{\Phi}^{h',-}_{\alpha}(x',t)$, we deduce that there is an $r \in (0,h_{0} - h')$ such that $|x_{e} - x_{e,0}| + \|x' - x_{0}'\| + |t-h'| < r$ implies
		\begin{equation*}
			x_{e} > \tilde{\Phi}^{h',-}_{\alpha}(x',t).
		\end{equation*}
	In other words, $(x_{e},x') \in \Omega_{t_{0} - h' + t}$ for all such triples $(x_{e},x',t)$.  In particular, by taking $x_{e} = x_{0,e} - \frac{r}{2}$, $\|x' - x_{0}'\| < \frac{r}{2}$, and $|t - h'| < \frac{r}{2}$, we find
		\begin{equation*}
			\eta(x',t_{0} - h' + t) \leq x_{0,e} - \frac{r}{2} = \varphi(x_{0}',t_{0}) - \frac{r}{2}.
		\end{equation*}
	Now this contradicts our assumption that $\eta^{*}(x_{0}',t_{0}) = \varphi(x_{0}',t_{0})$ since
		\begin{equation*}
			\eta^{*}(x_{0}',t_{0}) \leq \sup \left\{ \eta(x',t_{0} +s) \, \mid \, \|x' - x_{0}'\| < \frac{r}{2}, \, \, |s| < \frac{r}{2} \right\} \leq \varphi(x_{0}',t_{0}) - \frac{r}{2}.
		\end{equation*}

	Since $\varphi \in \mathscr{P}^{+}(\|q\| + 1)$ and $\alpha > 0$ were arbitrary, we deduce that $\eta^{*}$ is a $\mathscr{P}(\|q\| + 1)$-sub-solution.  
	
\textbf{Step 2: Initial condition}
	
	It remains to show that $\eta^{*}(\cdot,0) \leq \mathcal{U} \leq \nu_{*}(\cdot,0)$.  We only treat the inequality involving $\eta^{*}(\cdot,0)$ since the other one can be handled similarly.
	
	First, observe that it suffices to prove that if $x_{0} \in \mathbb{R}^{d}$ satisfies $x_{0,e} > \mathcal{U}(x_{0}')$, then there is a $\delta > 0$ and a $T' > 0$ such that $B(x_{0},\delta) \subseteq \bigcap_{0 < t < T'} \{x_{e} > \eta(x',t)\}$.  Indeed, once we have proved this, it will follow that
		\begin{equation*}
			\eta(x', t) \leq x_{0,e} - \frac{\delta}{2} \quad \text{if} \, \, \|x'- x_{0}'\| < \frac{\delta}{2}, \, \, 0 < t < T
		\end{equation*}
	and, thus,
		\begin{equation*}
			\eta^{*}(x_{0}',0) \leq \sup \left\{ \eta(x',t) \, \mid \, \|x' - x_{0}'\| < \frac{\delta}{2}, \, \, 0 < t < T\right\} \leq x_{0,e} - \frac{\delta}{2} \leq x_{0,e}.
		\end{equation*}
	Sending $x_{0,e} \to \mathcal{U}(x_{0}')$, we deduce that
		\begin{equation*}
		\eta^{*}(x_{0}',0) \leq \mathcal{U}(x_{0}').
		\end{equation*}
	Since $x_{0}'$ was arbitrary, we can then conclude that $\eta^{*}(\cdot,0) \leq \mathcal{U}$ in $\mathbb{R}^{d-1}$. 
	
	It only remains to prove the claim.  Fix $x_{0} \in \mathbb{R}^{d}$ with $x_{0,e} > \mathcal{U}(x_{0}')$.  A quick contradiction argument involving assumption (ii) shows that there are constants $c > 0$ and $\mu \in (\delta_{0},1)$ depending on $x_{0}$ such that
	\begin{equation} \label{E: key for case 2}
		\bigcup_{r > 0} re + B(x_{0},c) \subseteq \{u_{0} > 1 - \mu\}.
	\end{equation}
Indeed, were this not the case, we could find a sequence $(x_{n})_{n \in \mathbb{N}} \subseteq \mathbb{R}^{d}$ such that 
	\begin{align*}
		&\liminf_{n \to \infty}x_{n,e} \geq x_{0,e} > \mathcal{U}(x_{0}') \\
		&\limsup_{n \to \infty} [|x_{n}' - x_{0}'| + u_{0}(x_{n})] \leq 0
	\end{align*}  
Notice that the last inequality implies $u_{0}(x_{n}) \to 0^{+}$ according to \eqref{E: orientation}.  By hypotheses (i) and (ii), $\limsup_{n \to \infty} x_{n,e} < \infty$.  Thus, we can assume without loss of generality that $x_{n} \to x_{*}$ in $\mathbb{R}^{d}$ and this implies $u_{0}(x_{*}) = 0$ even though $x_{*,e} \geq x_{0,e} > \mathcal{U}(x_{0}') = \mathcal{U}(x_{*}')$.  That contradicts hypothesis (i).  
	
	Recall that, by hypotheses (i) and (ii) and \eqref{E: orientation}, we can choose a large $M > 0$ such that 
		\begin{equation} \label{E: contain plane}
			\{x \in \mathbb{R}^{d} \, \mid \, x_{e} \geq \langle q, x' \rangle + M\} \subseteq \{u_{0} \geq 1 - \delta_{0}\} \subseteq \{u_{0} > 1 - \mu\}.
		\end{equation}

	Now using \eqref{E: key for case 2} and \eqref{E: contain plane}, we let $\varphi \in \mathscr{P}^{+}_{0}$ be a smooth function such that $\{x \in \mathbb{R}^{d} \, \mid \, x_{e} \geq \varphi(x')\} \subseteq \{u_{0} > 1 - \mu\}$, $B(x_{0},\delta) \subseteq \{x_{e} > \varphi(x')\}$ for some $\delta \in (0,c)$, and $\varphi(x') = \bar{M}\|x' - x_{0}'\|$ whenever $\|x' - x_{0}'\|$ is sufficiently large for some $\bar{M} > \|q\| + 1$.  
	
	 Applying Proposition \ref{P: main_construction}, we obtain an $h_{0} > 0$ such that if $x = (x_{e},x')$ and $h \in (0,h_{0}]$ satisfy
	 	\begin{equation} \label{E: subtle_trick}
		x_{e} > \varphi(x') + h(\text{tr}(\tilde{\mathcal{G}}(D\varphi(x'))D^{2}\varphi(x')) + 1),
		\end{equation}
	then 
	 	\begin{equation*}
			\liminf \nolimits_{*} T^{\epsilon}(h)[(1 - \mu) \chi_{\{x_{e} > \varphi(x')\}} - \chi_{\{x_{e} < \varphi(x')\}}] = 1.
		\end{equation*} 
	Thus, since $u^{\epsilon}(\cdot,0) = u_{0}$ and $\{x_{e} > \varphi(x')\} \subseteq \{u_{0} > 1 - \mu\}$, the comparison principle implies
		\begin{equation*}
			\liminf \nolimits_{*} u^{\epsilon}(x,h) = 1.
		\end{equation*}  
Given that $\varphi$ is smooth, $0 \leq \tilde{\mathcal{G}} \leq \Lambda \text{Id}$, and $x_{0,e} > \varphi(x'_{0})$, we conclude there is a $\delta > 0$ and a $T' \leq h_{0}$ such that 
	\begin{equation*}
		\eqref{E: subtle_trick} \, \, \text{holds} \quad \text{if} \quad (x_{e},x') \in B(x_{0},\delta), \, \, h \in (0,T'].
	\end{equation*}
Since \eqref{E: subtle_trick} remains true if $x_{e}$ is increased, we conclude that
	\begin{equation*}
		\eqref{E: subtle_trick} \, \, \text{holds} \quad \text{if} \quad (x_{e},x') \in \bigcup_{r > 0} re + B(x_{0}, \delta), \, \, h \in (0,T']
	\end{equation*}
and, thus,
	\begin{equation*}
		\bigcup_{r > 0} re + B(x_{0}, \delta) \subseteq \Omega_{h}^{1} \quad \text{if} \, \, h \in (0,T'].
	\end{equation*}  
Recalling the definition of the function $\eta$, we obtain $B(x_{0},\delta) \subseteq \bigcap_{0 < t \leq T'} \{x_{e} > \eta(x',t)\}$ as claimed.   \qed \end{proof}  

Finally, for completeness, we prove Proposition \ref{P: main_construction} exactly as in \cite{barles souganidis}.

\begin{proof}[Proof of Proposition \ref{P: main_construction}]  We only prove the one half of the lemma since the other half follows by similar arguments.

Fix $\alpha > 0$, $\delta \in (0,1)$, and $\varphi \in \mathscr{P}^{+}_{0}$.  Let $h_{0} > 0$ be the corresponding constant from Lemma \ref{L: propagation}.

Suppose that $(x,h) \in \mathbb{R}^{d} \times (0,h_{0}]$ is such that 
	\begin{equation*}
		x_{e} > \varphi(x') + h \left( \text{tr} \left( \mathcal{G}(D\varphi(x')) D^{2} \varphi(x') \right) + \alpha \right) =: \tilde{\Phi}^{-}_{\alpha}(x',t).
	\end{equation*}
For $t \in [0,h_{0}]$, let $d_{\Phi}$ be the signed distance to the graph $\{x_{e} = \tilde{\Phi}^{-}_{\alpha}(x',t)\}$, positive in $\{x_{e} > \tilde{\Phi}^{-}_{\alpha}(x',t)\}$, as in the statement of Lemma \ref{L: propagation}; and let $d_{\varphi} = d_{\Phi}(\cdot,0)$, which, in view of the definition of $\tilde{\Phi}^{-}_{\alpha}$, equals the signed distance to the graph $\{x_{e} = \varphi(x')\}$.  It will be convenient to let $\beta_{0} > 0$ be given by 
	\begin{equation*}
		2 \beta_{0} = d_{\Phi}(x,h).
	\end{equation*}

By Lemma \ref{L: propagation} and the choice of $\beta_{0}$, there is a $\bar{\beta}(\alpha,\varphi) > 0$ and, for each $\beta \in \bar{\beta}(\alpha,\varphi)$, an $\bar{\epsilon}(\alpha,\beta,\varphi) > 0$ such that if $\epsilon \in (0,\bar{\epsilon}(\alpha,\beta,\varphi))$ and $\beta \in (0,\bar{\beta}(\alpha,\varphi) \wedge \beta_{0})$, then we can fix a subsolution $w^{\epsilon,\beta}$ of \eqref{E: main} in $\mathbb{R}^{d} \times (0,h_{0})$ such that 
	\begin{equation} \label{E: maddening inequality 1}
		w^{\epsilon,\beta}(\cdot,0) \leq (1 - \beta \epsilon) \chi_{\{d_{\varphi} \geq \beta\}} - \chi_{\{d_{\varphi} < \beta\}} \quad \text{in} \, \, \mathbb{R}^{d}
	\end{equation}
and, in the limit $\epsilon \to 0^{+}$, we have
	\begin{equation*}
		\liminf \nolimits_{*} w^{\epsilon,\beta}(x,h) = 1.
	\end{equation*}

Next, define functions $(v^{\epsilon})_{\epsilon > 0}$ in $\mathbb{R}^{d} \times [0,\infty)$ by setting
	\begin{equation*}
		v^{\epsilon}(\cdot,t) = T^{\epsilon}(t)[(1 - \delta) \chi_{\{x_{e} \geq \varphi(x')\}} - \chi_{\{x_{e} < \varphi(x')\}}] \quad \text{in} \, \, \mathbb{R}^{d}.
	\end{equation*}
By Lemma \ref{L: initialization}, there is a $\beta \in (0,\bar{\beta}(\alpha,\varphi) \wedge \beta_{0})$, which we fix for the rest of the proof, and constants $\epsilon(\beta), \tau(\beta) > 0$ such that if $\epsilon \in (0,\epsilon(\beta))$, then 
	\begin{equation} \label{E: maddening inequality 2}
		v^{\epsilon}(\cdot,\tau(\beta) \epsilon^{2} |\log(\epsilon)|) \geq (1 - \beta \epsilon) \chi_{\{d_{\varphi} \geq \beta\}} - \chi_{\{d_{\varphi} < \beta\}} \quad \text{in} \, \, \mathbb{R}^{d}.
	\end{equation}

Combining \eqref{E: maddening inequality 1} and \eqref{E: maddening inequality 2} and invoking the comparison principle, we deduce that, for each $t \in (0,h_{0}]$,
	\begin{equation*}
		v^{\epsilon}(\cdot,t + \tau(\beta) \epsilon^{2} |\log(\epsilon)|) \geq w^{\epsilon,\beta}(\cdot,t) \quad \text{in} \, \, \mathbb{R}^{d}.
	\end{equation*}
Therefore, in the limit $\epsilon \to 0^{+}$, we recover
	\begin{equation*}
		\liminf \nolimits_{*} v^{\epsilon}(x,h) \geq \liminf \nolimits_{*} w^{\epsilon,\beta}(x,h) = 1.
	\end{equation*}
\qed \end{proof}

\subsection{Proof of Lemma \ref{L: propagation}} \label{S: propagation}

This section treats the propagation step in the construction of sub- and supersolutions, or Lemma \ref{L: propagation}.  As mentioned above, we follow the arguments in \cite{barles souganidis}, making minor changes where necessary and employing the results of Appendix \ref{A: tubular_neighborhoods} to replace compact hypersurfaces by graphs.  

Since the two halves of the lemma are symmetrical, we only prove the first half.  Further, in the proofs, we will write $C$ to denote a generic positive constant whose value can change from line to line.

	To start the proof, we fix an arbitrary $\alpha > 0$.

	Next, we replace $\tilde{\Phi}_{\alpha}^{-}$ by a suitable smooth approximation.  Fix a smooth function $\tilde{\mathcal{G}}_{\alpha} : \mathbb{R}^{d} \to \mathcal{S}_{d-1}$ such that if $\|p\| \leq \|D^{2}\varphi\|_{L^{\infty}(\mathbb{R}^{d-1})}$ and $\|X\| \leq \|D^{2}\varphi\|_{L^{\infty}(\mathbb{R}^{d-1})}$, then
		\begin{equation} \label{E: approximate nonlinearity}
			\left| \text{tr} \left( \tilde{\mathcal{G}}_{\alpha}(p) X \right) -  \text{tr}\left(\tilde{\mathcal{G}}(p)X\right) \right| \leq \frac{\alpha}{4}.
		\end{equation}
	Define $\tilde{\Gamma}_{\alpha}^{-}$ by 
	\begin{equation*}
		\tilde{\Gamma}_{\alpha}^{-}(x',t) = \varphi(x') +t \left(\text{tr} (\tilde{\mathcal{G}}_{\alpha}(D\varphi(x')) D^{2}\varphi(x')) + \frac{3 \alpha}{4}\right).
	\end{equation*}
Notice that the first derivative of $\tilde{\Gamma}_{\alpha}^{-}$ is bounded, that is,
	\begin{equation}
		C_{\Gamma} := \sup \left\{ \|D \tilde{\Gamma}_{\alpha}^{-}(x',t)\| \, \mid \, (x',t) \in \mathbb{R}^{d-1} \times [0,1] \right\} < \infty. \label{E: deriv bounded graph part}
	\end{equation}
	
Since $\varphi \in \mathscr{P}_{0}^{+}$ and $\tilde{\mathcal{G}}_{\alpha}$ is smooth, $\tilde{\Gamma}_{\alpha}^{-}$ is smooth in both variables and there is an $h_{0} \in (0,1)$ depending only on $\alpha$ and $\max\{\|D^{i} \varphi\|_{L^{\infty}(\mathbb{R}^{d-1})} \, \mid \, i \in \{1,2,3,4\}\}$ such that
	\begin{equation}
		\tilde{\Gamma}_{\alpha,t}^{-} - \text{tr}(\tilde{\mathcal{G}}(D\tilde{\Gamma}_{\alpha}^{-}) D^{2}\tilde{\Gamma}_{\alpha}^{-}) \geq \frac{5 \alpha}{8} \quad \text{in} \, \, \mathbb{R}^{d - 1} \times (0,h_{0}]. \label{E: graph equation needed}
	\end{equation}
Define open sets $(E_{t})_{t \geq 0}$ by 
	\begin{equation*}
		E_{t} = \left\{x \in \mathbb{R}^{d} \, \mid \, x_{e} > \tilde{\Gamma}^{-}_{\alpha}(x',t) \right\}.
	\end{equation*}
Note that \eqref{E: graph equation needed} implies that $(E_{t})_{t \geq 0}$ is a sub-flow of \eqref{E: anisotropic curvature flow}.  

For each $t \in [0,h_{0}]$, let $d_{\Gamma}(\cdot,t) : \mathbb{R}^{d} \to \mathbb{R}$ be the signed distance to the graph $\partial E_{t}$, positive in $E_{t}$.  By Proposition \ref{P: surface_fact} and \eqref{E: graph equation needed}, we can fix a $\gamma > 0$ such that $d_{\Gamma} \in BUC^{4}(\{|d_{\Gamma}| < \gamma\})$ and, in that domain, $d_{\Gamma}$ satisfies the differential inequality
	\begin{equation} \label{E: distance function subsolution}
		(d_{\Gamma})_{t} - \text{tr}(\tilde{\mathcal{S}}^{a}(Dd_{\Gamma}) D^{2}d_{\Gamma}) \leq - \frac{\alpha}{4 \sqrt{1 + C_{\Gamma}^{2}}}.
	\end{equation}
Henceforth, let us write $d = d_{\Gamma}$ to declutter the notation.

Now we start the construction of the mesoscopic sub-solution.  As in \cite{barles souganidis}, we define a function $v^{\epsilon}$ in $\{|d| < \gamma\}$ by  
	\begin{equation*}
		v^{\epsilon}(x,t) = U_{Dd(x,t)} \left(\frac{d(x,t)-2 \beta}{\epsilon},\frac{x}{\epsilon} \right) + \epsilon \left( P^{D^{2}d(x,t)}_{Dd(x,t)} \left( \frac{d(x,t) - 2 \beta}{\epsilon}, \frac{x}{\epsilon}\right) - 2 \beta \right),
	\end{equation*}
where $P^{D^{2}d(x,t)}_{Dd(x,t)}$ is the corrector defined in Section \ref{S: higher_regularity} and $\beta > 0$ is a free variable to be determined below.

At this stage, it is convenient to define $K_{\Gamma} \subset \subset S^{d-1}$ and $\mathcal{S}_{d}(\Gamma) \subseteq \mathcal{S}_{d}$ by
	\begin{align*}
		K_{\Gamma} &= \overline{\{Dd(x,t) \, \mid \, (x,t) \in \mathbb{R}^{d} \times [0,h_{0}], \, \, |d_{\Gamma}(x,t)| < \gamma \}}, \\
		\mathcal{S}_{d}(\Gamma) &= \overline{\{D^{2}d(x,t) \, \mid \, (x,t) \in \mathbb{R}^{d} \times [0,h_{0}], \, \, |d_{\Gamma}(x,t)| < \gamma \}}.
	\end{align*}
Note that $\mathcal{S}_{d}(\Gamma)$ is compact since $d$ has bounded second derivatives in the domain $\{|d| < \gamma\}$.

\begin{prop}  $K_{\Gamma} \subset \subset S^{d-1} \setminus (\mathbb{R}^{k} \times \{0\})$.  \end{prop}  

	\begin{proof}  First of all, since $d$ is a signed distance function and it is $C^{4}$ in $\{|d| < \gamma\}$, we know that $K_{\Gamma} \subseteq S^{d-1}$.  
	
	To prove that $K_{\Gamma} \subset \subset \mathbb{R}^{d} \setminus (\mathbb{R}^{k} \times \{0\})$, we will show that 
		\begin{equation} \label{E: distance compact set thing}
			\text{dist}(K_{\Gamma}, \mathbb{R}^{k} \times \{0\}) \geq \frac{1}{\sqrt{1 + C_{\Gamma}^{2}}},
		\end{equation}
	where $C_{\Gamma}$ is given by \eqref{E: deriv bounded graph part}.
	To see this, notice that if $(x,t) \in \mathbb{R}^{d} \times [0,h_{0}]$ and $|d(x,t)| < \gamma$, then there is a unique $y \in \{d(\cdot,t) = 0\}$ such that $|d(x,t)| = \| y - x\|$ and 
		\begin{equation*}
			Dd(x,t) = \frac{e - \pi^{*}(D\tilde{\Gamma}^{-}_{\alpha}(y',t))}{\sqrt{1 + \|D\tilde{\Gamma}^{-}_{\alpha}(y',t)\|^{2}}}.
		\end{equation*}
	In particular,
		\begin{equation*}
			\text{dist}(Dd(x,t), \mathbb{R}^{k} \times \{0\}) = \frac{1}{\sqrt{1 + \|D\tilde{\Gamma}^{-}_{\alpha}(y',t)\|^{2}}} \geq \frac{1}{\sqrt{1 + C_{\Gamma}^{2}}}.
		\end{equation*}
	  Since $(x,t)$ was arbitrary, this proves \eqref{E: distance compact set thing}. \qed \end{proof}  

\begin{prop} \label{P: subsolution first part} There are constants $\nu, \beta_{0} > 0$ depending only on $\alpha$ and the constant $C$ of \eqref{E: deriv bounded graph part} such that if $\beta \in (0,\beta_{0})$, then 
	\begin{align*}
		v^{\epsilon}_{t} - \text{div}(a(\epsilon^{-1}x) Dv^{\epsilon}) &+ \epsilon^{-2} W'(v^{\epsilon}) \leq -\epsilon^{-1} \nu + O(1) \\
		&\qquad\text{in} \, \, \{(x,t) \in \mathbb{R}^{d} \times (0,h_{0}) \, \mid \, |d_{\gamma}(x,t)| \leq \gamma\}.
	\end{align*}
\end{prop}  

	\begin{proof}  We have assumed that $k < d$.  Together with the previous result, this implies that \eqref{E: laminar assumption} is in effect. Thus, by Proposition \ref{P: key corrector result}, the map $(s,x,v,A) \mapsto P^{A}_{v}(s,x)$ is twice continuously differentiable in a neighborhood of $\mathbb{R} \times \mathbb{T}^{d} \times K_{\Gamma} \times \mathcal{S}_{d}(\Gamma)$ and all its derivatives (up to second order) are bounded in that neighborhood.  Thus, plugging $v^{\epsilon}$ into the equation and expanding the terms involved, we find
		\begin{align*}
			v^{\epsilon}_{t} - \text{div}(a(\epsilon^{-1}x) Dv^{\epsilon}) + &\epsilon^{-2} W'(v^{\epsilon}) = \epsilon^{-2} \left(\mathcal{D}_{Dd}^{*}(a(\epsilon^{-1}x) \mathcal{D}_{Dd} U_{Dd}) + W'(U_{Dd}) \right) \\
		& + \epsilon^{-1} \left( \mathcal{D}_{Dd}^{*}(a(\epsilon^{-1}x) \mathcal{D}_{Dd}P^{D^{2}d}_{Dd}) + W''(U_{Dd}) P^{D^{2}d}_{Dd} \right) \\
		& + \epsilon^{-1} \partial_{s} U_{Dd} \left[ d_{t}(x,t) - \text{tr}\left(a(\epsilon^{-1} x)D^{2} d(x,t)\right) \right] \\
		& - \epsilon^{-1} \left[ 2\left \langle a(\epsilon^{-1}x) Dd(x,t), D^{2}d(x,t) \partial_{s} D_{e}U_{Dd} \right \rangle \right.\\
		& \left.+ 2 \text{tr} \left( a(\epsilon^{-1} x) D^{2}d(x,t) D_{x}D_{e}U_{Dd} \right) - 2 \beta W''(U_{Dd(x,t)})\right. \\
		& \left. + \left \langle \text{div} \, a(\epsilon^{-1} x), D^{2} d(x,t) D_{e}U_{Dd} \right \rangle \right] + O(1).
		\end{align*}
Invoking the PDE solved by the pulsating standing wave $(s,x) \mapsto U_{v}(s,x)$ and the corrector $(s,x) \mapsto P^{A}_{v}(s,x)$ for arbitrary $(v,A) \in K_{\Gamma} \times \mathcal{S}_{d}(\Gamma)$, this becomes 
	\begin{align*}
		v^{\epsilon}_{t} - \text{div}(a(\epsilon^{-1}x) Dv^{\epsilon}) &+ \epsilon^{-2} W'(v^{\epsilon})  \\
			&\quad = \epsilon^{-1} \partial_{s} U_{Dd(x,t)} [d_{t}(x,t) - \text{tr}(\tilde{\mathcal{S}}^{a}(Dd(x,t)) D^{2}d(x,t))] \\
			&\quad \left. - 2 \beta \epsilon^{-1} W''(U_{Dd(x,t)}) \right) + O(1) \\
			&\quad \leq \epsilon^{-1} \left(-\frac{\alpha \partial_{s}U_{Dd(x,t)}}{4\sqrt{1 + C^{2}}} - 2 \beta W''(U_{Dd(x,t)})\right) + O(1).
	\end{align*}
In the last line, we used the estimate \eqref{E: distance function subsolution}.  To complete the proof, it only remains to show that the $\epsilon^{-1}$ order term is bounded away from zero.

Here we follow \cite[Lemma 4.3]{barles souganidis}.  Applying Proposition \ref{P: exponential decay pulsating} with $K = K_{\Gamma}$ and recalling \eqref{A: W_assumption_2}, we find that
	\begin{gather*}
		\lim_{M \to \infty} \inf \left\{ W''(U_{v}(s,x)) \, \mid \, (s,x) \in (\mathbb{R} \setminus [-M,M]) \times \mathbb{T}^{k}, \, \, v \in K_{\Gamma} \right\} \geq \alpha > 0.
	\end{gather*}
Thus, in particular, we can choose an $M_{1} > 0$ such that 
	\begin{align*}
		\eta := \inf \left\{ W''(U_{v}(s,x)) \, \mid \, (s,x) \in (\mathbb{R} \setminus [-M_{1},M_{1}]) \times \mathbb{T}^{k}, \, \, v \in K_{\Gamma} \right\} > 0.
	\end{align*}
At the same time, the map $(s,x,v) \mapsto \partial_{s}U_{v}(s,x)$ is positive and continuous in the compact set $[-M_{1},M_{1}] \times \mathbb{T}^{k} \times K_{\Gamma}$ by Proposition \ref{P: eigenfunction_estimate} and Remark \ref{R: regularity in derivative}.  Hence it is bounded away from zero in that set and we can choose a small $\beta_{0} > 0$ depending only on $\alpha$ and $C$ such that
	\begin{align*}
		\nu_{1} &:= \inf \left\{ \frac{\alpha \partial_{s}U_{v}(s,x)}{4\sqrt{1 + C^{2}}} - 2 \beta_{0} |W''(U_{v}(s,x))| \, \mid \, (s,x) \in [-M_{1},M_{1}] \times \mathbb{T}^{k}, \right. \\
		&\qquad  \left. v \in K_{\Gamma} \right\} > 0.
	\end{align*}
Putting it all together, we find that, for each $\beta \in (0,\beta_{0})$,
	\begin{align*}
		\nu &:= \inf \left\{ \frac{\alpha \partial_{s}U_{v}(s,x)}{4\sqrt{1 + C^{2}}} + 2 \beta W''(U_{v}(s,x)) \, \mid \, (s,x) \in \mathbb{R} \times \mathbb{T}^{k}, \, \, v \in K_{\Gamma} \right\} \\
			&\geq \min\left\{ \beta \eta, \nu_{1} \right\} > 0.
	\end{align*}
\qed \end{proof}  

It remains to extend the subsolution $v^{\epsilon}$ to the entire space $\mathbb{R}^{d} \times [0,h_{0}]$.  To start with, we define the extension $\bar{v}^{\epsilon} : \{d < \gamma\} \to \mathbb{R}$ by 
	\begin{equation} \label{E: subsolution construction step two}
		\bar{v}^{\epsilon}(x,t) = \left\{ \begin{array}{r l}
								\max\{v^{\epsilon}(x,t),-1\}, & \text{if} \, \, -\frac{\gamma}{2} < d(x,t) < \gamma, \\
								-1, & \text{if} \, \, d(x,t) \leq - \frac{\gamma}{2}
							\end{array} \right.
	\end{equation}

\begin{prop} \label{P: subsolution second part} Let $\beta_{0}$ be the constant determined in Proposition \ref{P: subsolution first part}.  There is an $\epsilon_{0} > 0$ such that if $\beta \in (0,\beta_{0})$ and $\epsilon \in (0,\epsilon_{0})$, then 
	\begin{align*}
		\bar{v}^{\epsilon}_{t} - \text{div}(a(\epsilon^{-1}x) D\bar{v}^{\epsilon}) &+ \epsilon^{-2} W'(\bar{v}^{\epsilon}) \leq 0 \\
			& \text{in} \, \, \{(x,t) \in \mathbb{R}^{d} \times (0,h_{0}) \, \mid \, d_{\Gamma}(x,t) < \gamma\}.
	\end{align*}
\end{prop}

	\begin{proof}  By the previous proposition, we know that $v^{\epsilon}$ is a sub-solution of \eqref{E: main} in $\{-\gamma/2 < d < \gamma\}$ if $\beta \in (0,\beta_{0})$ and $\epsilon$ is small enough.  Therefore, being the maximum of two sub-solutions in this domain, $\bar{v}^{\epsilon}$ is itself a sub-solution there.  
	
	To conclude, it suffices to show that $\bar{v}^{\epsilon}$ is a sub-solution in the domain $\{d < - \gamma/4\}$ if $\beta \in (0,\beta_{0})$ and $\epsilon$ is small enough.  In fact, if $(x,t) \in \mathbb{R}^{d} \times (0,h_{0})$ and $-\gamma < d(x,t) < -\frac{\gamma}{4}$, then we can invoke the exponential decay estimates in Propositions \ref{P: exponential decay pulsating} and \ref{P: key corrector result} to find
		\begin{align*}
			v^{\epsilon}(x,t) &\leq U_{Dd(x,t)}\left(-\frac{\gamma}{4\epsilon},\frac{x}{\epsilon}\right) + \epsilon \left( P_{Dd(x,t)}^{D^{2}d(x,t)}\left(\frac{d(x,t) - 2 \beta}{\epsilon},\frac{x}{\epsilon}\right) - 2 \beta \right) \\
			&\leq -1 + (1 + \epsilon) C \exp\left(- \frac{\gamma}{4 C \epsilon} \right) - 2 \beta \epsilon.
		\end{align*}
	Hence $v^{\epsilon} < -1 - \beta \epsilon$ in $\{-\gamma < d_{\Gamma} < -\frac{\gamma}{4}\}$ as soon as $\epsilon$ is small enough.  In particular, for $\beta \in (0,\beta_{0})$ and $\epsilon$ small enough,
		\begin{equation*}
			\bar{v}^{\epsilon}(x,t) = -1 \quad \text{if} \, \, (x,t) \in \mathbb{R}^{d} \times (0,h_{0}) \, \, \text{and} \, \,  d(x,t) < -\frac{\gamma}{4}.
		\end{equation*}
	   Since $-1$ is a sub-solution, it follows that $\bar{v}^{\epsilon}$ is a sub-solution in $\{d < - \gamma/4\}$. \qed \end{proof}  
	   
Finally, we extend $\bar{v}^{\epsilon}$ to a sub-solution in $\mathbb{R}^{d} \times [0,h_{0}]$.  Toward this end, given a $\beta \in (0,\beta_{0})$, we fix a smooth, non-decreasing function $\psi_{\beta} : \mathbb{R} \to [0,1]$ such that 
	\begin{equation*}
		\psi_{\beta}(r) = 1 \quad \text{if} \, \, r \geq \frac{7 \gamma}{8} + 2 \beta, \quad \psi_{\beta}(r) = 0 \quad \text{if} \, \, r \leq \frac{3 \gamma}{4} + 2 \beta.
	\end{equation*}  
Let $w^{\epsilon} : \mathbb{R}^{d} \times [0,h_{0}] \to \mathbb{R}$ be given by 
	\begin{equation*}
		w^{\epsilon}(x,t) = \{1 - \psi_{\beta}(d(x,t))\} \bar{v}^{\epsilon}(x,t) + \psi_{\beta}(x,t) (1 - \beta \epsilon).
	\end{equation*}
	
\begin{prop}  Let $\beta_{0}$ and $\epsilon_{0}$ be the constants of Propositions \ref{P: subsolution first part} and \ref{P: subsolution second part}, respectively.  There is an $\epsilon_{\bigstar} \in (0,\epsilon_{0})$ such that if $\beta \in (0,\beta_{0} \wedge 1)$ and $\epsilon \in (0,\epsilon_{\bigstar})$, then 
	\begin{align*}
		w^{\epsilon}_{t} - \text{div}(a(\epsilon^{-1} x) Dw^{\epsilon}) + \epsilon^{-2} W'(w^{\epsilon}) \leq 0 \quad \text{in} \, \, \mathbb{R}^{d} \times (0,h_{0}), \\
		w^{\epsilon}(\cdot,0) \leq (1 - \beta \epsilon) \chi_{\{d(\cdot,0) \geq \beta\}} - \chi_{\{d(\cdot,0) < \beta\}} \quad \text{on} \, \, \mathbb{R}^{d} \times \{0\}.
	\end{align*}
Furthermore, for $\beta \in (0,\beta_{0})$, we have
	\begin{equation*}
		\liminf\nolimits_{*} w^{\epsilon}(x,t) = 1 \quad \text{if} \, \, d(x,t) > 2 \beta.
	\end{equation*}
\end{prop}  

In the proof, we will use the following consequence of Schauder estimates.  Observe that if $v \in K_{\Gamma}$, then the function $U_{v} - 1$ is a solution of the uniformly elliptic PDE
	\begin{equation*}
		\mathcal{D}_{v}^{*}(a(x) \mathcal{D}_{v}(U_{v} - 1)) = -W'(U_{v}) \quad \text{in} \, \, \mathbb{R} \times \mathbb{T}^{k}.
	\end{equation*}
Hence, by Schauder estimates, \eqref{E: exponential pulsating estimate 1}, and \eqref{E: exponential pulsating estimate 2}, there is a $C > 0$ such that 
	\begin{equation*}
		\|\mathcal{D}_{v}U_{v}(s,x)\| \leq C \exp(-C^{-1}|s|) + C |W'(U_{v}(s,x))|.
	\end{equation*}
At the same time, by \eqref{A: W_assumption_2}, there is an $M_{\bigstar} > 0$ such that 
	\begin{equation*}
		|W'(U_{v}(s,x)| \leq 2\alpha |U_{v}(s,x) - 1| \quad \text{if} \, \, (s,x) \in (\mathbb{R} \setminus [-M_{\bigstar}, M_{\bigstar}]) \times \mathbb{T}^{k}.
	\end{equation*}
Hence, up to making $C$ larger, we have
	\begin{equation} \label{E: derivative estimate}
		\|\mathcal{D}_{v}U_{v}(s,x)\| \leq C \exp(-C^{-1}|s|) \quad \text{for} \, \, (s,x) \in \mathbb{R} \times \mathbb{T}^{k}.
	\end{equation}

	\begin{proof}  We start by proving that $w^{\epsilon}$ is a sub-solution provided $\epsilon$ is sufficiently small; then we prove the bound on $w^{\epsilon}(\cdot,0)$; and, finally, we prove the claim concerning $\liminf_{*} w^{\epsilon}$.  Henceforth fix $\beta \in (0,\beta_{0})$.
	
	\textbf{Step 1: $w^{\epsilon}$ is a sub-solution}
	
	To start with, we will show that $w^{\epsilon}$ is a sub-solution in $\{d > \frac{\gamma}{2} + 2 \beta\}$.  Observe that if $(x,t) \in \mathbb{R}^{d} \times (0,h_{0})$ and $d(x,t) \geq \frac{\gamma}{2} + 2 \beta$, then we can use the exponential estimates in Proposition \ref{P: exponential decay pulsating} to find
		\begin{align}
			v^{\epsilon}(x,t) &= U_{Dd(x,t)}\left(\frac{d(x,t) - 2 \beta}{\epsilon}, \frac{x}{\epsilon} \right) + \epsilon \left( P^{D^{2}d(x,t)}_{Dd(x,t)} \left(\frac{d(x,t) - 2 \beta}{\epsilon}, \frac{x}{\epsilon} \right) - 2 \beta \right) \nonumber \\
				&= 1 + C \exp \left(-\frac{\gamma}{2 C \epsilon} \right) + O(\epsilon). \label{E: silly v estimate we need}
		\end{align}
	Hence, if $\epsilon$ is small enough, we have
		\begin{equation*}
			\bar{v}^{\epsilon} = v^{\epsilon} \quad \text{in} \, \, \left\{(x,t) \in \mathbb{R}^{d} \times (0,h_{0}) \, \mid \, d(x,t) > \frac{\gamma}{2} + 2 \beta \right\}.
		\end{equation*}
	
	Since $\bar{v}^{\epsilon} = v^{\epsilon}$ in $\{d > \frac{\gamma}{2} + 2 \beta\}$, $w^{\epsilon}$ is smooth enough in that domain for us to evaluate its derivatives directly.  Doing so and applying Proposition \ref{P: subsolution first part}, we find
		\begin{align*}
			w^{\epsilon}_{t} - \text{div}(a(\epsilon^{-1}x) Dw^{\epsilon}) + \epsilon^{-2} W'(w^{\epsilon}) &\leq (1 - \psi_{\beta}(d))(- \nu \epsilon^{-1} + O(1)) \\
			&\quad + \epsilon^{-2} \psi_{\beta}(d) W'(1 - \beta \epsilon) + \mathcal{I}(\epsilon,\beta),
		\end{align*}
	where the error term $\mathcal{I}(\epsilon,\beta)$ is given by 
		\begin{align*}
			\mathcal{I}(\epsilon,\beta) &= \epsilon^{-2} \left[ W'(w^{\epsilon}) - (1 - \psi_{\beta}(d)) W'(v^{\epsilon}) - \psi_{\beta}(d) W'(1 - \beta \epsilon) \right] \\
			&\quad + (1 - \beta \epsilon - v^{\epsilon}) \psi_{\beta}'(d) d_{t} - 2 \psi_{\beta}'(d) \langle a(\epsilon^{-1}x) Dv^{\epsilon}, Dd \rangle \\
			&\quad - (1 - \beta \epsilon - v^{\epsilon}) \left[ \text{tr}(a(\epsilon^{-1}x) D^{2}d) + \langle a(\epsilon^{-1}x) Dd, Dd \rangle \right] \\
			&\quad - \epsilon^{-1} (1 - \beta \epsilon - v^{\epsilon}) \psi_{\beta}'(d) \langle (\text{div} \, a)(\epsilon^{-1}x), Dd \rangle.
		\end{align*}
Notice that $W'(1 - \beta \epsilon) = W'(1) - \beta \epsilon W''(1) + O(\epsilon^{2})$.  Thus, we have
	\begin{align*}
		w^{\epsilon}_{t} - \text{div}(a(\epsilon^{-1}x) Dw^{\epsilon}) &+ \epsilon^{-2} W'(w^{\epsilon}) \leq - \min\{\nu, W''(1)\} \epsilon^{-1} + O(1) + \mathcal{I}(\epsilon,\beta) \\
			&\qquad \text{in} \, \, \left\{(x,t) \in \mathbb{R}^{d} \times (0,h_{0}) \, \mid \, d(x,t) > \frac{\gamma}{2} + 2 \beta \right\}.
	\end{align*}
To conclude that $w^{\epsilon}$ is a sub-solution in $\{d > \frac{\gamma}{2} + 2 \beta\}$, we will show that $\mathcal{I}(\epsilon,\beta) = O(1)$ as $\epsilon \to 0^{+}$.  

Indeed, we can write
	\begin{align*}
		\mathcal{I}(\epsilon,\beta) &= \mathcal{I}_{1}(\epsilon,\beta) + \mathcal{I}_{2}(\epsilon,\beta) + \mathcal{I}_{3}(\epsilon,\beta), \\
		\mathcal{I}_{1}(\epsilon,\beta) &:= \epsilon^{-2} \left[ W'(w^{\epsilon}) - (1 - \psi_{\beta}(d)) W'(v^{\epsilon}) - \psi_{\beta}(d) W'(1 - \beta \epsilon) \right], \\
		\mathcal{I}_{2}(\epsilon,\beta) &:= (1 - \beta \epsilon - v^{\epsilon}) \psi_{\beta}'(d) d_{t} - 2 \psi_{\beta}'(d) \langle a(\epsilon^{-1}x) Dv^{\epsilon}, Dd \rangle \\
					&\quad - \epsilon^{-1} (1 - \beta \epsilon - v^{\epsilon}) \psi_{\beta}'(d) \langle (\text{div} \, a)(\epsilon^{-1}x), Dd \rangle, \\
		\mathcal{I}_{3}(\epsilon,\beta) &:= - (1 - \beta \epsilon - v^{\epsilon}) \left[ \text{tr}(a(\epsilon^{-1}x) D^{2}d) + \langle a(\epsilon^{-1}x) Dd, Dd \rangle \right].
	\end{align*}
In view of \eqref{E: silly v estimate we need}, we have
	\begin{equation*}
		\mathcal{I}_{3}(\epsilon,\beta) = O(1) (1 - \beta \epsilon - v^{\epsilon}) = O(\epsilon).
	\end{equation*}
Similarly, notice that we can write
	\begin{align*}
		Dv^{\epsilon}(x,t) &= \epsilon^{-1} \mathcal{D}_{Dd(x,t)} U_{Dd(x,t)}\left(\frac{d(x,t) - 2 \beta}{\epsilon}, \frac{x}{\epsilon}\right) + O(1)
	\end{align*}
and, thus, by \eqref{E: derivative estimate},
	\begin{equation*}
		\|Dv^{\epsilon}(x,t)\| \leq \epsilon^{-1} C \exp \left( - \frac{\gamma}{2 C \epsilon} \right) + O(1) = O(1).
	\end{equation*}
Combining this with \eqref{E: silly v estimate we need}, we obtain
	\begin{equation*}
		\mathcal{I}_{2}(\epsilon,\beta) = O(\epsilon) + O(1) - \epsilon^{-1}(1 - \beta \epsilon - v^{\epsilon}) \psi_{\beta}'(d) \langle (\text{div} \, a)(\epsilon^{-1}x), Dd \rangle = O(1).
	\end{equation*}
Finally, Taylor expanding the $W'$ terms in $\mathcal{I}_{1}(\epsilon,\beta)$ around $1$ and invoking \eqref{E: silly v estimate we need} once more, the linear terms cancel and we are left with
	\begin{equation*}
		\mathcal{I}_{1}(\epsilon,\beta) = \epsilon^{-2} \left\{ O([v^{\epsilon} - 1]^{2}) + O(\epsilon^{2}) + O([w^{\epsilon} - 1]^{2}) \right\} = O(1).
	\end{equation*}

Combining the results of the previous two paragraphs, we conclude that 
	\begin{align*}
		w^{\epsilon}_{t} - \text{div}(a(\epsilon^{-1}x) Dw^{\epsilon}) &+ \epsilon^{-2} W'(w^{\epsilon}) \leq - \min\{\nu, W''(1)\} \epsilon^{-1} + O(1)\\
			&\qquad \text{in} \, \, \left\{(x,t) \in \mathbb{R}^{d} \times (0,h_{0}) \, \mid \, d(x,t) > \frac{\gamma}{2} + 2 \beta \right\}.
	\end{align*}
Hence as soon as $\epsilon$ is small enough, $w^{\epsilon}$ is a sub-solution in $\{d > \frac{\gamma}{2} + 2 \beta\}$.  

Next, notice that $w^{\epsilon} = \bar{v}^{\epsilon}$ in $\{d < \frac{3 \gamma}{4} + 2 \beta\}$ by construction.  Therefore, Proposition \ref{P: subsolution second part} implies it is a sub-solution in that domain as well provided $\epsilon < \epsilon_{0}$.  Since the two domains cover $\mathbb{R}^{d} \times (0,h_{0})$, we conclude that $w^{\epsilon}$ is a sub-solution.

\textbf{Step 2: Bound on $w^{\epsilon}(\cdot,0)$}

First, observe that if $d(x,0) < \beta$, then, arguing as in \eqref{E: silly v estimate we need}, we find
	\begin{equation*}
		v^{\epsilon}(x,t) \leq -1 + C \exp \left( -\frac{\beta}{C \epsilon} \right) - 2 \beta \epsilon.
	\end{equation*}
Thus, as soon as $\epsilon$ is sufficiently small, we have
	\begin{equation*}
		w^{\epsilon}(x,t) = \bar{v}^{\epsilon}(x,t) = -1.
	\end{equation*}
This proves $w^{\epsilon}(x,0) \leq -1$ if $d(x,0) < \beta$.
	
Next, we need to show that $w^{\epsilon}(x,0) \leq 1 - \beta \epsilon$ if $d(x,0) \geq \beta$.  Let us begin by fixing $c > 0$ such that 
	\begin{equation*}
		\sup \left\{ |P_{v}^{A}(s,x)| \, \mid \, (s,x) \in [c,\infty) \times \mathbb{T}^{k}, v \in K_{\Gamma}, \, \, A \in \mathcal{S}_{d}(\Gamma) \right\} \leq \beta.
	\end{equation*}
This is possible by Proposition \ref{P: key corrector result}.  With this choice of $c$, observe that if $2\beta + c \epsilon \leq d(x,0)$, then 
	\begin{equation*}
		v^{\epsilon}(x,t) \leq 1 + \epsilon \left(P^{D^{2}d(x,t)}_{Dd(x,t)}\left(\frac{d(x,0) - 2 \beta}{\epsilon},\frac{x}{\epsilon}\right) - 2 \beta\right) \leq -1 - \beta \epsilon.
	\end{equation*}
On the other hand, since $v \mapsto 1 - U_{v}(s,x)$ is positive and continuous, there is a $\delta > 0$ such that
	\begin{equation*}
		1 - \delta := \sup \left\{ U_{v}(s,x) \, \mid \, (s,x) \in (-\infty,c] \times \mathbb{T}^{k}, \, \,  v \in K_{\Gamma} \right\}.
	\end{equation*}
Hence the inequality $\beta \leq d(x,0) \leq 2 \beta + c \epsilon$ implies that
	\begin{equation*}
		v^{\epsilon}(x,t) \leq 1 - \delta - 2 \beta \epsilon + O(\epsilon).
	\end{equation*}
Thus, for $\epsilon$ small enough, we have $w^{\epsilon}(x,0) \leq 1 - \beta \epsilon$ whenever $d(x,0) \geq \beta$.

\textbf{Step 3: Behavior of $\liminf_{*} w^{\epsilon}$} 

Finally, notice that the inequality $d(x,t) > 2 \beta$ implies the existence of an $\eta > 0$ and a neighborhood $\mathcal{O}$ of $(x,t)$ such that $d(y,s) \geq 2 \beta + \eta$ for $(y,s) \in \mathcal{O}$.  Arguing as in \eqref{E: silly v estimate we need}, we deduce that
	\begin{equation*}
		\inf \left\{ w^{\epsilon}(y,s) \, \mid \, (y,s) \in \mathcal{O} \right\} \geq 1 + C \exp \left( - \frac{\eta}{C \epsilon} \right) + O(\epsilon)
	\end{equation*}
and, thus,
	\begin{equation*}
		\liminf\nolimits_{*} w^{\epsilon}(x,t) = 1.
	\end{equation*}
\qed \end{proof}  

It only remains to use the previous result to verify the lemma:

\begin{proof}[Proof of Lemma \ref{L: propagation}]  Notice that $d(\cdot,0) = d_{\varphi}$.  Hence the previous proposition implies that $w^{\epsilon}(\cdot,0) \leq (1 - \beta \epsilon) \chi_{\{d_{\varphi} \geq \beta\}} - \chi_{\{d_{\varphi} < \beta\}}$ provided $\epsilon$ is sufficiently small.  

Next, we will show that if $d_{\Phi}(x,t) > 2 \beta$ for some $(x,t) \in \mathbb{R}^{d} \times (0,h_{0})$, then $\liminf_{*} w^{\epsilon}(x,t) = 1$.  In view of the previous proposition, it suffices to show that $d(x,t) > 2 \beta$ in this case.  In fact, we claim that $d \geq d_{\varphi}$.  

Indeed, since we are working with signed distances, the inequality $d \geq d_{\varphi}$ holds if and only if $\{d_{\varphi} > 0\} \subseteq \{d > 0\}$.  To see that the latter condition holds, observe that $d(x,t) > 0$ if and only if $x_{e} > \tilde{\Phi}_{\alpha}^{-}(x',t)$.  At the same time, if $x_{e} > \tilde{\Phi}_{\alpha}^{-}(x',t)$, then, using \eqref{E: approximate nonlinearity}, we compute
	\begin{align*}
		\tilde{\Gamma}^{-}_{\alpha}(x',t) &= \varphi(x') + t \left[ \text{tr} \left( \mathcal{G}_{\alpha}(D\varphi(x')) D^{2}\varphi(x') \right) + \frac{3 \alpha}{4} \right] \\
			&\leq \varphi(x') + t \left[ \text{tr} \left( \mathcal{G}(D\varphi(x')) D^{2} \varphi(x') \right) + \alpha \right] \\
			&\quad + t \left| \text{tr} \left( [\mathcal{G}_{\alpha}(D\varphi(x')) - \mathcal{G}(D\varphi(x'))] D^{2}\varphi(x') \right) \right| - \frac{t \alpha}{4} \\
			&\leq \tilde{\Phi}_{\alpha}^{-}(x',t) < x_{e}.
	\end{align*}
Hence the inequality $x_{e} > \tilde{\Phi}_{\alpha}^{-}(x',t)$ implies $x_{e} > \tilde{\Gamma}_{\alpha}^{-}(x',t)$, which implies that $d(x,t) > 0$.  This proves $\{d_{\varphi} > 0\} \subseteq \{d > 0\}$ and, thus, also the bound $d \geq d_{\varphi}$.  

In view of the previous paragraph, if $d_{\Phi}(x,t) > 2 \beta$, then we must also have $d(x,t) > 2 \beta$.  Therefore, by the previous proposition, if $d_{\Phi}(x,t) > 2 \beta$, then $\liminf_{*} w^{\epsilon}(x,t) = 1$. \qed \end{proof}  

\subsection{Proof of Lemma \ref{L: initialization}} \label{S: initialization}

We proceed next to the analysis of the initialization step, or Lemma \ref{L: initialization}.  As in the last section, we only prove one half of the lemma.  

The basic result we need here is from \cite{chen}.  

\begin{lemma}[\cite{chen}, Section 3] \label{L: chen lemma} If $W$ satisfies \eqref{A: W_assumption_1}, \eqref{A: W_assumption_2}, \eqref{A: additional_regularity}, and \eqref{A: sign} , then, for each $\epsilon > 0$, there is a $C^{1}$ function $\overline{W}_{\epsilon} : \mathbb{R} \to \mathbb{R}$ such that
	\begin{equation} \label{E: comparing chen thing}
		\overline{W}_{\epsilon}'(u) \geq W'(u) \quad \text{if} \, \, u \in [-1,1]
	\end{equation}
and a $C^{2}$ function $\chi^{\epsilon} : \mathbb{R} \times [0,\infty) \to \mathbb{R}$ satisfying the ODE
	\begin{equation*}
		\left\{ \begin{array}{r l}
			\chi^{\epsilon}_{t}(s,t) = -\overline{W}'_{\epsilon}(\chi^{\epsilon}(s,t)) & \text{for} \, \, (s,t) \in \mathbb{R} \times [0,\infty), \\
			\chi^{\epsilon}(s,0) = s & \text{if} \, \, s \in \mathbb{R}
		\end{array} \right.
	\end{equation*}
and for which the following statements all hold:
	\begin{itemize}
		\item[(i)] $\chi^{\epsilon}(s,t) < \chi^{\epsilon}(s',t)$ if $s < s'$.
		\item[(ii)] For each $\beta > 0$, there are $\tau(\beta), \epsilon(\beta) > 0$ such that, for each $\epsilon \in (0,\epsilon(\beta))$,
			\begin{align*}
				\chi^{\epsilon}(s,t) &\geq 1 - \beta \epsilon \quad \text{if} \, \, s \geq 3 \epsilon |\log(\epsilon)|, \, \, t \geq \tau(\beta) |\log(\epsilon)|.
			\end{align*}
		\item[(iii)] Given any $a > 0$, there are constants $M(a), \epsilon'(a) > 0$ such that, for each $\epsilon \in (0,\epsilon(a))$,
			\begin{equation*}
				\frac{|\chi_{ss}(s,t)|}{\chi_{s}(s,t)} \leq \frac{M(a)}{\epsilon} \quad \text{for} \, \, (s,t) \in \mathbb{R} \times [0,a |\log(\epsilon)|].
			\end{equation*}
	\end{itemize}
\end{lemma} 

	\begin{proof}[Proof of Lemma \ref{L: initialization}]  To start with, define $\psi : \mathbb{R}^{d} \to \mathbb{R}$ by $\psi(x) = \rho(d_{\varphi}(x)$, where $\rho : \mathbb{R} \to \mathbb{R}$ is a smooth, non-decreasing function such that 
		\begin{equation*}
			\rho(s) = \left\{ \begin{array}{r l}
							-1, & \text{if} \, \, s \leq 0, \\
							1- \delta, & \text{if} \, \, s \geq \beta
						\end{array} \right.
		\end{equation*}
	
	Let $\epsilon(\beta), \tau(\beta) > 0$ and $\chi^{\epsilon}$ be as in Lemma \ref{L: chen lemma}, and let $a = \tau(\beta)$.  Henceforth assume that $\epsilon \in (0,\epsilon(\beta) \wedge \epsilon'(a))$.  
	
	Letting $M(a)$ be the constant from Lemma \ref{L: chen lemma}, (iii), we define $K > 0$ by
		\begin{equation*}
			K = 1 + M(a) + \|\text{div}\, a\|_{L^{\infty}(\mathbb{T}^{d})} \|D\psi\|_{L^{\infty}(\mathbb{R}^{d})}.
		\end{equation*}Define $\underline{u}^{\epsilon} : \mathbb{R}^{d} \times [0,\infty) \to \mathbb{R}$ by 
		\begin{equation*}
			\underline{u}^{\epsilon}(x,t) = \chi^{\epsilon}(\psi(x) - \epsilon^{-1}Kt, \epsilon^{-2}t).
		\end{equation*}
	Observe that, in the domain $\mathbb{R}^{d} \times [0,\tau(\beta)\epsilon^{2} |\log(\epsilon)|]$, we can compute
		\begin{align*}
			\underline{u}^{\epsilon}_{t} - \text{div}(a(\epsilon^{-1}x) D\underline{u}^{\epsilon}) &+ \epsilon^{-2} \overline{W}_{\epsilon}'(\underline{u}^{\epsilon}) \\
			&= \epsilon^{-1} \chi^{\epsilon}_{s} \left( - K - \langle a(\epsilon^{-1}x) D\psi(x), D\psi(x) \rangle \frac{\chi_{ss}^{\epsilon}}{\chi_{s}^{\epsilon}} \right. \\
			&\quad \left. - \epsilon \text{tr} (a(\epsilon^{-1}x) D^{2}\psi(x)) - \langle (\text{div} \, a)(\epsilon^{-1}x), D\psi(x) \rangle \right) \\
			&\leq \epsilon^{-1} \chi^{\epsilon}_{s} \left( -1 + \epsilon \text{tr}(a(\epsilon^{-1} x) D^{2}\psi(x)) \right).
		\end{align*}
	Hence $\underline{u}^{\epsilon}_{t} - \text{div}(a(\epsilon^{-1}x) D\underline{u}^{\epsilon}) + \epsilon^{-2} \overline{W}'_{\epsilon}(\underline{u}^{\epsilon}) \leq -\frac{1}{2}\epsilon^{-1}$ in $\mathbb{R}^{d} \times [0,\tau(\beta)\epsilon^{2}|\log(\epsilon)|]$ as soon as $\epsilon$ is small enough.
	
At the same time, if we let $\overline{u}^{\epsilon} : \mathbb{R}^{d} \times [0,\infty) \to \mathbb{R}$ be the solution of the PDE
	\begin{equation*}
		\left\{ \begin{array}{r l}
				\overline{u}^{\epsilon}_{t} - \text{div}(a(\epsilon^{-1}x) D\overline{u}) + \epsilon^{-2} W'(\overline{u}^{\epsilon}) = 0 & \text{in} \, \, \mathbb{R}^{d} \times [0,\infty), \\
				\overline{u}^{\epsilon} = (1 - \delta) \chi_{\{d_{\varphi} \geq 0\}} - \chi_{\{d \varphi < 0\}} & \text{on} \, \, \mathbb{R}^{d} \times \{0\}.
			\end{array} \right.
	\end{equation*}
then \eqref{E: comparing chen thing} implies that $\overline{u}^{\epsilon}_{t} - \text{div}(a(\epsilon^{-1}x) D\overline{u}) + \epsilon^{-2} \overline{W}'_{\epsilon}(\overline{u}^{\epsilon}) \geq 0$ in $\mathbb{R}^{d} \times [0,\infty)$.  Furthermore, notice that
	\begin{equation*}
		v^{\epsilon}(x,0) = \chi^{\epsilon}(\psi(x),0) = \psi(x) \geq \overline{u}^{\epsilon}(x,0) \quad \text{if} \, \, x \in \mathbb{R}^{d}.
	\end{equation*}
Accordingly, the comparison principle implies that
	\begin{equation} \label{E: working it through initialization}
		\overline{u}^{\epsilon} \geq \underline{u}^{\epsilon} \quad \text{in} \, \, \mathbb{R}^{d} \times [0, \tau(\beta) \epsilon^{2} |\log(\epsilon)|].
	\end{equation}

Finally, we show that 
	\begin{equation*}
		\overline{u}^{\epsilon}(\cdot,\tau(\beta) \epsilon^{2} |\log(\epsilon)|) \geq (1 - \beta \epsilon) \chi_{\{d_{\varphi} \geq \beta\}} - \chi_{\{d_{\varphi} < \beta\}} \quad \text{in} \, \, \mathbb{R}^{d}.
	\end{equation*}  
Notice that the inequality in $\{d_{\varphi} < \beta\}$ is automatic since $|\overline{u}^{\epsilon}| \leq 1$.  Suppose that  $d_{\varphi}(x) \geq \beta$.  Observe that if $\epsilon$ is small enough, then
	\begin{equation*}
		1 - \delta - K \tau(\beta) \epsilon |\log(\epsilon)| \geq 3 \epsilon |\log(\epsilon)|
	\end{equation*}
Thus, using \eqref{E: working it through initialization} and (i) and (ii) of Lemma \ref{L: chen lemma}, we find
	\begin{equation*}
		\overline{u}^{\epsilon}(x,\tau(\beta)\epsilon^{2}|\log(\epsilon)|) \geq \chi^{\epsilon}(1 - \delta - K \tau(\beta) \epsilon |\log(\epsilon)|,\tau(\beta) |\log(\epsilon)|) \geq 1 - \beta \epsilon.
	\end{equation*}
\qed \end{proof}   
	
\subsection{Proof of Proposition \ref{P: bounded_width}} \label{S: bounded_width}  Here we give the
	
\begin{proof}[Proof of Proposition \ref{P: bounded_width}]  Let $\beta, M > 0$ be free parameters.  Set $\varphi(x) = \langle x,e_{0} \rangle - M$ and $\tilde{\Phi}(x,t) = \langle x, e_{0} \rangle - M - t$.  
		
		By assumptions (i) and (ii) of Theorem \ref{T: sharp_interface_limit_graphs} and \eqref{E: orientation}, if we choose $M > 0$ large enough, then $\{u_{0} \geq 1 - \delta\} \supseteq \{u_{0} \geq 1 - \delta_{0}\} \supseteq \{x \in \mathbb{R}^{d} \, \mid \, \langle x,e_{0} \rangle \geq M\}$.  Hence 
			\begin{equation*}
				u^{\epsilon}(\cdot,0) = u_{0} \geq (1 - \delta) \chi_{\{\varphi \geq 0\}} -\chi_{\{\varphi < 0\}}.
			\end{equation*}
		Now Lemma 4.1 in \cite{barles souganidis} implies
			\begin{equation*}
				u^{\epsilon}(\cdot, t_{\epsilon}) \geq T^{\epsilon}(t_{\epsilon})[(1 - \delta) \chi_{\{\varphi \geq 0\}} - \chi_{\{\varphi < 0\}}] \geq (1 - \beta \epsilon) \chi_{\{\varphi \geq \beta\}} - \chi_{\{\varphi < \beta\}}
			\end{equation*} 
		for some $t_{\epsilon} \to 0^{+}$ depending on $\beta$ and $\delta$.  
			
		Now we will show that we can extend this estimate up to time $t$.  Let $V_{e_{0}} = \partial_{s} U_{e_{0}}$ and define $v : \mathbb{R}^{d} \times [0,T] \to \mathbb{R}$ by 
			\begin{equation*}
				v(x,t) = U_{e_{0}} \left(\frac{\tilde{\Phi}(x,t) - 2 \beta}{\epsilon}, \frac{x}{\epsilon}\right) + \epsilon \left(V_{e_{0}}\left(\frac{\tilde{\Phi}(x,t) - 2 \beta}{\epsilon}, \frac{x}{\epsilon}\right) - 2 \beta\right).
			\end{equation*} 	
		Plugging $v$ into the equation, we find
			\begin{align*}
				v_{t} - \text{div} \left(a\left(\frac{x}{\epsilon}\right) Dv\right) + \epsilon^{-2} W'(v) &= - \epsilon^{-1} (V_{e} + 2W''(U_{e}) \beta) + O(1).
			\end{align*}
		Here is where we choose $\beta$.  Since $W''(U_{e_{0}}) > 0$ holds when $|\langle x,e_{0}\rangle|$ is large enough, we only need to choose a $\beta$ small enough that we have
		\begin{equation*}
			V_{e_{0}} \geq 2 \beta (\|W''\|_{L^{\infty}([-1,1])} + 1)
		\end{equation*} 
	for all $x \in \mathbb{R}^{d}$ for which $|\langle x,e_{0}\rangle - M| + t$ is in some bounded interval.  With this choice, $v$ is a sub-solution if $\epsilon$ is sufficiently small.  
		
		Next, set $w(x,t) = \max\{v(x,t),-1\}$ and note that $w$ is also a sub-solution.  Observe that the exponential estimates on $V_{e_{0}}$ as $s \to \infty$ (i.e.\ Proposition \ref{P: eigenfunction_estimate}) yield the existence of an $\epsilon_{0} > 0$ such that if $\epsilon \in (0,\epsilon_{0})$, then
			\begin{align*}
				w(x,0) &= \max \left\{U_{e_{0}} \left(\frac{\phi(x) - 2 \beta}{\epsilon}, \frac{x}{\epsilon} \right)
				+ \epsilon \left(V_{e_{0}} \left(\frac{\phi(x) - 2 \beta}{\epsilon}, \frac{x}{\epsilon}\right) - 2 \beta\right), - 1 \right\} \\
					&\leq \left(1- \beta \epsilon\right) \chi_{\{\phi \geq \beta\}} - \chi_{\{\phi < \beta\}}.
			\end{align*}
		Finally, observe that if $\langle x,e_{0} \rangle \geq M'$ for some $M' > 1 + M + t + 2 \beta$, then, by making $\epsilon_{0}$ smaller if necessary, we obtain, for $\epsilon \in (0,\epsilon_{0})$ and $s \in \left[0,t\right]$,
			\begin{align*}
				w(x,s) &= U_{e_{0}} \left( \frac{\langle x, e_{0}\rangle - M - s - 2 \beta}{\epsilon}, \frac{x}{\epsilon} \right) \\
				&\qquad+ \epsilon \left( V_{e_{0}} \left( \frac{\langle x, e_{0} \rangle - M - s - 2 \beta}{\epsilon}, \frac{x}{\epsilon} \right) - 2 \beta \right) \\
					&\geq 1 - C e^{ -(C \epsilon)^{-1}(M' - M - t - 2 \beta)} - 2 \beta \epsilon \\
					&\geq 1 - \delta.
			\end{align*} 
Putting it all together, we deduce that if $r \in \left[s,t\right]$, $\langle x,e \rangle \geq M'$, and $\epsilon$ is sufficiently small, then
	\begin{equation*}
		u^{\epsilon}(x,r) \geq w(x,r - t_{\epsilon}) \geq 1 - \delta.
	\end{equation*}
	
The lower bound is obtained similarly. \qed
		\end{proof}

\appendix

\section{Technical Lemmata} \label{A: elementary_properties}

\subsection{Approximation results} \label{A: approximation}  The main goal of this section is to prove

\begin{prop} \label{P: smooth_approximation} The following two identities hold:
	\begin{align*}
	\mathscr{E}^{a}(e) &= \inf \left\{ \mathscr{T}^{a}_{e}(U) \, \mid \, U \in \mathscr{X} \cap C^{\infty}_{\text{sgn}}(\mathbb{R} \times \mathbb{T}^{d}) \right\}, \\
		\inf \left\{ \mathscr{T}^{a}_{e}(U) \, \mid \, U \in \mathscr{X}_{+} \right\} &= \inf \left\{ \mathscr{T}^{a}_{e}(U) \, \mid \, U \in \mathscr{X}_{+} \cap C^{\infty}_{\text{sgn}}(\mathbb{R} \times \mathbb{T}^{d}) \right\}.
	\end{align*}
\end{prop}  

First, we observe that any $U \in \mathscr{X}$ can be well approximated by a function in $\mathscr{X}$ that equals $\text{sgn}(s)$ outside of a compact subset of $\mathbb{R} \times \mathbb{T}^{d}$:

\begin{prop} \label{P: cut_off}  If $U$ is a measurable function on $\mathbb{R} \times \mathbb{T}^{d}$ such that $|U| \leq 1$ a.e., $\mathscr{T}^{a}_{e}(U) < \infty$, and \eqref{E: asymptotic} holds, then, for each $\epsilon > 0$, there is a measurable function $\tilde{U}_{\epsilon}$ in $\mathbb{R} \times \mathbb{T}^{d}$ and an $M_{\epsilon} > 0$ such that
	\begin{itemize}
		\item[(i)] $|\tilde{U}_{\epsilon}| \leq 1$ a.e.,
		\item[(ii)] $\mathscr{T}^{a}_{e}(\tilde{U}_{\epsilon}) < \infty$,
		\item[(iii)] $\text{sgn}(s) \tilde{U}(s,x) = 1$ if $|s| \geq M_{\epsilon}$,
		\item[(iv)] $|\mathscr{T}^{a}_{e}(\tilde{U}_{\epsilon}) - \mathscr{T}^{a}_{\epsilon}(U)| < \epsilon$.
	\end{itemize}
If, in addition, $\partial_{s} U \geq 0$, then $\tilde{U}_{\epsilon}$ can be chosen in such a way that $\partial_{s} \tilde{U}_{\epsilon} \geq 0$.  
\end{prop}  

\begin{proof}  Let $\varphi_{N}$ be a smooth cut-off function supported in $\mathbb{R} \times [-N,N]$ and set $U_{N} = \varphi_{N} U + (1 - \varphi_{N}) \text{sgn}(s)$.  Notice that the monotonicity property of $U$ is preserved provided $\text{sgn}(s) \partial_{s} \varphi_{N} \leq 0$.  Suitably choosing $N$ and the cut-off function, it is possible to ensure that (iv) holds (cf.\ \cite[Proof of Proposition 6]{previous_paper}).  \qed \end{proof}  

Next, we show that we can smooth the function obtained in the previous result without affecting its energy too much:

\begin{prop} \label{P: smooth_approx} If $U$ satisfies the hypotheses of Proposition \ref{P: cut_off}, then, for each $\epsilon > 0$, there is a $\hat{U}_{\epsilon} \in C^{\infty}_{\text{sgn}}(\mathbb{R} \times \mathbb{T}^{d})$ and an $\hat{M}_{\epsilon} > 0$ such that 
	\begin{itemize}
		\item[(i)] $\text{sgn}(s)\hat{U}_{\epsilon}(s,x) = 1$ if $|s| \geq \hat{M}_{\epsilon}$,
		\item[(ii)] $|\hat{U}_{\epsilon}| \leq 1$ in $\mathbb{R} \times \mathbb{T}^{d}$,
		\item[(iii)] $\mathscr{T}^{a}_{e}(\hat{U}_{\epsilon}) < \infty$,
		\item[(iv)] $|\mathscr{T}^{a}_{e}(U) - \mathscr{T}^{a}_{e}(\hat{U}_{\epsilon})| < \epsilon$.
	\end{itemize}
Furthermore, we may assume that $\partial_{s} \hat{U}_{\epsilon} \geq 0$ if $\partial_{s} U \geq 0$ and that $U = \lim_{\epsilon \to 0^{+}} \hat{U}_{\epsilon}$ a.e.\
\end{prop}  

\begin{proof}  Use a mollifier in $\mathbb{R} \times \mathbb{T}^{d}$ to smooth the function obtained in Proposition \ref{P: cut_off}.  Property (ii) is immediate, monotonicity in $s$ is preserved, and (i) holds with $\hat{M}_{\epsilon} \geq M_{\epsilon}$.  (iii) follows from the fact that mollification commutes with $\mathcal{D}_{e}$ and (iv) holds as soon as the mollification parameter is small enough. \qed \end{proof}

Now we proceed with the

\begin{proof}[Proof of Proposition \ref{P: smooth_approximation}]  If $U \in \mathscr{X}$, then Proposition \ref{P: smooth_approx} implies there is a family $(\hat{U}_{\epsilon})_{\epsilon > 0} \in \mathscr{X} \cap C^{\infty}_{\text{sgn}}(\mathbb{R} \times \mathbb{T}^{d})$ such that $\lim_{\epsilon \to 0^{+}} \mathscr{T}^{a}_{e}(\hat{U}_{\epsilon}) = \mathscr{T}^{a}_{e}(U)$.  Thus,
	\begin{equation*}
		\inf \left\{ \mathscr{T}^{a}_{e}(U) \, \mid \, U \in \mathscr{X} \right\} \geq \inf \left\{ \mathscr{T}^{a}_{e}(U) \, \mid \, U \in \mathscr{X} \cap C^{\infty}_{\text{sgn}}(\mathbb{R} \times \mathbb{T}^{d}) \right\}.
	\end{equation*}
The inclusion $\mathscr{X} \cap C^{\infty}_{\text{sgn}}(\mathbb{R} \times \mathbb{T}^{d}) \subseteq \mathscr{X}$ provides the complementary inequality.  

Since for each $U \in \mathscr{X}_{+}$, there is a family $(\hat{U}_{\epsilon})_{\epsilon > 0} \subseteq C^{\infty}_{c}(\mathbb{R} \times \mathbb{T}^{d})$ as above with $\partial_{s} \hat{U}_{\epsilon} \geq 0$, the other identity follows similarly. \qed \end{proof}  

\subsection{On a Class of Uniformly Elliptic Allen-Cahn Functionals in Cylinders}  \label{A: annoying uniformly elliptic part}

This section provides a proof of Proposition \ref{P: exponential decay pulsating}.  As was already mentioned in Section \ref{S: einstein relation laminar media}, the proposition essentially follows from \cite[Theorem 2.3]{valdinoci de la llave}.  Since the fact that the constants in \eqref{E: exponential pulsating estimate 1} and \eqref{E: exponential pulsating estimate 2} do not depend on $v$ is important elsewhere in the paper, but this is not made precise in \cite{valdinoci de la llave}, we provide a complete proof here for the reader's convenience.

\begin{proof}[Proof of Proposition \ref{P: exponential decay pulsating}]  We need to prove existence and uniqueness of $U^{\delta}_{v}$, continuity with respect to $v$, and the estimates \eqref{E: exponential pulsating estimate 2}, \eqref{E: exponential pulsating estimate 2}, and \eqref{E: holder estimate}.  The work is really in the estimates \eqref{E: exponential pulsating estimate 1} and \eqref{E: exponential pulsating estimate 2}, which follow from \cite{valdinoci}.  Since \cite{valdinoci} proves existence at the same time as it establishes the necessary ingredients for these estimates, we discuss existence, \eqref{E: exponential pulsating estimate 1}, and \eqref{E: exponential pulsating estimate 2} simultaneously in the final step of the proof.  We start with the easier issues, namely, uniqueness and \eqref{E: holder estimate}.

\textbf{Step 1: Uniqueness and continuity}  

To establish uniqueness of $U^{\delta}_{v}$, we need to show that there is at most one function $U \in \mathscr{X}^{(k)}$ satisfying the following conditions:
	\begin{gather*}
		\mathscr{T}^{a,\delta}_{v}(U) = \tilde{\varphi}^{a,\delta}(v), \quad \int_{\mathbb{T}^{k}} U(0,x) \, dx = 0, \quad \text{and} \quad \lim_{s \to \pm \infty} U(s,x) = \pm 1.
	\end{gather*}
Suppose, then, that $U$ and $\tilde{U}$ both have these properties.  We will show that $U = \tilde{U}$.  

To start with, observe that the identity $\mathscr{T}^{a,\delta}_{v}(U) = \mathscr{T}^{a,\delta}_{v}(\tilde{U}) = \tilde{\varphi}^{a,\delta}(v)$ imply that both are weak solutions of the Euler-Lagrange equation \eqref{E: pulsating standing wave with regularization}.  By \eqref{E: extra ellipticity}, this equation is uniformly elliptic, and, thus, the boundedness of $U$ and $\tilde{U}$ can be bootstrapped to H\"{o}lder continuity in $\mathbb{R} \times \mathbb{T}^{k}$.

Next, notice that to prove $U = \tilde{U}$, it suffices to show that $U_{m} = U_{M}$, where $U_{m} = U \wedge \tilde{U}$ and $U_{M} = U \vee \tilde{U}$.  

Finally, observe that $\mathscr{T}^{a,\delta}_{v}(U_{m}) + \mathscr{T}^{a,\delta}_{v}(U_{M}) \leq \mathscr{T}^{a,\delta}_{v}(U) + \mathscr{T}^{a,\delta}_{v}(\tilde{U})$, and, thus, we have $\mathscr{T}^{a,\delta}_{v}(U) = \mathscr{T}^{a,\delta}_{v}(\tilde{U}) = \tilde{\varphi}^{a,\delta}(v)$.  Hence $U_{m}$ and $U_{M}$ are both solutions of the Euler-Lagrange equation \eqref{E: pulsating standing wave with regularization}.  Since $U_{m} \leq U_{M}$, the strong maximum principle implies that either $U_{m} = U_{M}$, in which case we are done, or else the inequality $U_{m} < U_{M}$ holds pointwise in $\mathbb{R} \times \mathbb{T}^{k}$.  This last possibility cannot occur, however.  Indeed, if it did, then $U < \tilde{U}$ or $\tilde{U} < U$ would follow by continuity, and then we can compute $|\int_{\mathbb{R} \times \mathbb{T}^{k}} (U(0,x) - \tilde{U}(0,x)) \, dx| = \int_{\mathbb{R} \times \mathbb{T}^{k}} |U(0,x) - \tilde{U}(0,x)| \, dx > 0$, a contradiction.  

Now that uniqueness is clear, continuity of $v \mapsto U^{\delta}_{v}$ follows by a straightforward compactness argument we omit.

\textbf{Step 2: H\"{o}lder estimate}

As noted in Step 1, the function $U^{\delta}_{v}$, if it exists, is a bounded solution of the Euler-Lagrange equation \eqref{E: pulsating standing wave with regularization}.  Hence, by elliptic regularity (see, e.g., \cite[Chapter 4]{Han Lin} or \cite[Chapter 8]{Gilbarg Trudinger}), it is uniformly $\alpha$-H\"{o}lder continuous in $\mathbb{R} \times \mathbb{T}^{k}$ with a modulus that only depends on the ellipticity constants in \eqref{E: extra ellipticity} and $\|W'\|_{L^{\infty}([-1,1])}$.  

\textbf{Step 3: Existence and uniform exponential decay}

Finally, we prove the existence of $U^{\delta}_{v}$ and its exponential decay.  In the notation of \cite{valdinoci}, we are interested in $n = k + 1$, $\omega = (-1,0,0,\dots,0)$, and $\tilde{\mathbb{R}}^{n} = \mathbb{R} \times \mathbb{T}^{k}$.  Let us write the functional $\mathscr{T}^{a,\delta}_{v}$ in the suggestive form
	\begin{equation*}
		\mathscr{T}^{a,\delta}_{v}(V) = \int_{\mathbb{R} \times \mathbb{T}^{k}} \left(\frac{1}{2} \langle A_{v}(s,x) (\partial_{s},D_{x})V, (\partial_{s},D_{x})V \rangle + W(V) \right) \, dx \, ds,
	\end{equation*}
which, due to \eqref{E: extra ellipticity} and the assumptions on $a$ and $W$, fits the hypotheses of \cite{valdinoci}.  
(Of course, $\partial_{s} A_{v} \equiv 0$, but that is irrelevant for us here.)  Since \eqref{A: W_assumption_2} holds, we can fix an $\eta \in (0,1)$ such that
	\begin{equation*}
		W''(u) \geq \frac{\alpha}{2} \quad \text{if} \, \, u \in [-1,-1 + \eta] \cup [1-\eta,1].
	\end{equation*}
For us, $\eta$ will play the role played by $99/100$ in \cite{valdinoci}. 

By Theorem 8.1 in \cite{valdinoci} and its proof applied with the choice of $n$ and $\omega$ above and $\eta$ replacing $99/100$, there is a $U \in H^{1}_{\text{loc}}(\mathbb{R} \times \mathbb{T}^{k}; [-1,1])$ such that 
	\begin{itemize}
		\item[(i)] $\mathscr{T}^{a,\delta}_{v}(U) < \infty$,
		\item[(ii)] $\lim_{s \to \pm \infty} U(s,x) = \pm 1$ for each $(s,x) \in \mathbb{R} \times \mathbb{T}^{k}$,
		\item[(iii)] For any $M > 0$ and any $V \in H^{1}_{\text{loc}}(\mathbb{R} \times \mathbb{T}^{k};[-1,1])$ such that $\text{sgn}(s,x) V(s,x) = 1$ in $[-M,M] \times \mathbb{T}^{k}$, we have
			\begin{equation*}
				\mathscr{T}^{a,\delta}_{v}(U) \leq \mathscr{T}^{a,\delta}_{v}(V).
			\end{equation*}
		\item[(iv)] There is an $\bar{M} > 0$ depending only on the ellipticity constants $\mu(\delta,K)$ and $M(\delta,K)$ of $A_{v}$ (see \eqref{E: extra ellipticity}), $\|\text{div}_{s,x} \, A_{v}\|_{L^{\infty}(\mathbb{R} \times \mathbb{T}^{k})}$, and $W$ such that 
			\begin{equation*}
				\{(s,x) \in \mathbb{R} \times \mathbb{T}^{k} \, \mid \, -1 + \eta \leq U(s,x) \leq 1 - \eta\} \subseteq [-\bar{M},0] \times \mathbb{T}^{k}.
			\end{equation*}  
	\end{itemize}
Notice that, by an approximation argument (cf.\ Proposition \ref{P: cut_off}), (iii) implies that 
	\begin{equation*}
		\mathscr{T}^{a,\delta}_{v}(U) \leq \mathscr{T}^{a,\delta}_{v}(V) \quad \text{if} \, \, V \in \mathscr{X}^{(k)}.
	\end{equation*}
Hence $\mathscr{T}^{a,\delta}_{v}(U) = \tilde{\varphi}^{a,\delta}(v)$.  

Next, notice that there is a dimensional constant $L > 0$ such that 
	\begin{equation*}
		\|\text{div}_{s,x}A_{v}\|_{L^{\infty}(\mathbb{R} \times \mathbb{T}^{k})} \leq L(1 + C_{K}^{2}) \|Da\|_{L^{\infty}(\mathbb{T}^{k})},
	\end{equation*} 
where $C_{K} = \max\{\|v\| \, \mid \, \|v\| \in K\}$.  Thus, (iii) implies that $\bar{M}$ only depends on $\delta$, $K$, $a$, and $W$.  

To get an exponential estimate on $U$, observe that, by the choice of $\eta$, $U$ satisfies the differential inequality 
	\begin{equation*}
		- \delta \partial_{s}^{2} U + \mathcal{D}_{v}^{*}(a(x) \mathcal{D}_{v}U) + \frac{\alpha}{2} (U(s,x) - 1) \geq 0 \quad \text{in} \, \, [0,\infty) \times \mathbb{T}^{k}.  
	\end{equation*}
Thus, applying the maximum principle to $1 - U(s,x) - B \exp(-B^{-1}s)$ for universal large $B$ (cf.\ the proof of Proposition \ref{P: annoying exponential convergence trick}) we find
	\begin{equation*}
		1 - U(s,x) \leq B \exp(-B^{-1}s) \quad \text{if} \, \, s \geq 0.
	\end{equation*}
Since $|U| \leq 1$, by making $B$ larger if necessary, we can assume that $1 - U \leq B \exp(-B^{-1}s)$ in $\mathbb{R} \times \mathbb{T}^{k}$.  

A similar argument using the equation satisfied by $U + 1$ in $(-\infty,-\bar{M}] \times \mathbb{T}^{k}$ shows that $U + 1 \leq B^{-1} \exp(B^{-1}s)$, again with a possibly larger $B$.  Here the choice of $B$ depends on $\bar{M}$, hence it only depends on $\delta$, $K$, $a$, and $W$.

Finally, since $\lim_{s \to \pm \infty} U(s,x) = \pm 1$ for each $(s,x) \in \mathbb{R} \times \mathbb{T}^{k}$ and $U$ is uniformly continuous in $\mathbb{R} \times \mathbb{T}^{k}$ by elliptic regularity, we can choose a $\bar{s} \in \mathbb{R}$ such that
	\begin{equation*}
		\int_{\mathbb{T}^{k}} U(\bar{s},x) \, dx = 0.
	\end{equation*}
From the bounds on $U$, we deduce that $- B \log(B) \leq \bar{s} \leq B \log(B)$.  

Finally, if we define $U^{\delta}_{v}(s,x) = U(s + \bar{x},s)$, then $\mathscr{T}^{a,\delta}_{v}(U^{\delta}_{v}) = \mathscr{T}^{a,\delta}_{v}(U) = \tilde{\varphi}^{a,\delta}(v)$ by translation invariance and
	\begin{equation*}
		1 - U^{\delta}_{v}(s,x) \leq (B e^{-B^{-1}\bar{s}})\exp(-B^{-1}s), \quad U^{\delta}_{v}(s,x) + 1 \leq (B e^{B^{-1}\bar{s}}) \exp(B^{-1}s).
	\end{equation*}
Thus, since we know $|\bar{s}| \leq B \log(B)$, we conclude that $U^{\delta}_{v}$ satisfies the estimates \eqref{E: exponential pulsating estimate 1} and \eqref{E: exponential pulsating estimate 2} for some $C_{0}$ with the desired dependence on the data.
\qed \end{proof}

\subsection{Comparison Principle for $\mathscr{P}$-sub- and super-solutions} \label{A: comparison}  Here we give the 

\begin{proof}[Proof of Proposition \ref{P: comparison_principle}]  Since a comparison principle for ordinary viscosity sub- and super-solutions of \eqref{E: graph_equation} is already known, it suffices to prove that a $\mathscr{P}(\bar{M})$-sub-solution (resp.\ $\mathscr{P}(\bar{M})$-super-solution) is an ordinary sub-solution (resp.\ super-solution).  We will only prove the former statement since the latter follows from analogous arguments.

Suppose then that $h$ is a $\mathscr{P}(\bar{M})$-sub-solution for some fixed $\bar{M} > 0$.  To see that it is an ordinary viscosity sub-solution, it suffices, through the usual reductions, to show that if there is a $(x_{0},t_{0}) \in \mathbb{R}^{d} \times [0,T]$ such that
	\begin{equation} \label{E: strict_global_max}
		h(x,t) \leq h(x_{0},t_{0}) + \langle p, x - x_{0} \rangle + \frac{1}{2} \langle A(x -x_{0}), x - x_{0} \rangle + a (t - t_{0}) + \frac{b(t - t_{0})^{2}}{2},	
	\end{equation}
where $p \in \mathbb{R}^{d-1}$, $A \in \mathcal{S}_{d}$, and $a,b \in \mathbb{R}$, then
	\begin{equation*}
		a - \text{tr}(\tilde{\mathcal{G}}(p) A) \leq 0.
	\end{equation*}
Notice that, by perturbing $A$ and $b$ if necessary, we can assume that $(x_{0},t_{0})$ is the only point in $\mathbb{R}^{d-1} \times [0,T]$ where equality holds in \eqref{E: strict_global_max}.  
	
Since $h$ is bounded, there is an $R > 0$ such that
	\begin{equation*}
		h(x,t) + |a| T +\frac{|b| T}{2}+ 1 \leq \|x\|
	\end{equation*}
if $\|x\| \geq R$ and $t \in [0,T]$.  Without loss of generality, we can assume that $\|x_{0}\| < R$.  

Let $\rho \in C^{\infty}_{c}(\mathbb{R}^{d-1}; [0,1])$ satisfy $\rho(x) = 1$ if $\|x\| \leq R$.  

A straightforward computation now shows that
	\begin{align*}
		h(x,t) &\leq \left(h(x_{0},t_{0})+  \langle p, x - x_{0} \rangle + \frac{1}{2} \langle A(x- x_{0}), x - x_{0} \rangle\right) \rho(x) \\
			&\qquad + (1 - \rho(x)) (\bar{M} + 1)\|x\| + a(t - t_{0}) + \frac{b(t- t_{0})^{2}}{2}.
	\end{align*}
Thus, since $h$ is a $\mathscr{P}(\bar{M})$-sub-solution and $\rho = 1$ in a neighborhood of $(x_{0},t_{0})$, we find
	\begin{equation*}
		a - \text{tr}(\tilde{\mathcal{G}}(p) A) \leq 0.
	\end{equation*} \qed
\end{proof}

\section{Transformation Properties of $\mathcal{L}^{d+1}$ in $\mathbb{R} \times \mathbb{T}^{d}$} \label{A: transformation properties} 

Here we provide the proof of Theorem \ref{T: ergodic_lemma} on the transformation properties of $\mathcal{L}^{d+1}$ under the map $(x,\zeta) \mapsto (\langle x,e \rangle - \zeta,x)$.  We begin with the more demanding irrational case and then sketch how to carry the arguments over to the rational one.

\subsection{Irrational directions --- Preliminaries}  

In the proof of Theorem \ref{T: ergodic_lemma}, we will be interested in averages of periodic funcitons over cubes in $\langle e \rangle^{\perp}$.  When $e$ is irrational, it turns out that such averages are readily analyzed.  To explain this, we digress into some ergodic theory.  

Recall that $M_{e}$ is the module of integers orthogonal to $e$.  Being a submodule of $\mathbb{Z}^{d}$, we can fix a $\mathbb{Z}$-basis $\{k_{1},\dots,k_{r}\}$ of $M_{e}$.  Notice that these are necessarily linearly independent over $\mathbb{R}$.  The first result we will use says we can construct from these an orthogonal $\mathbb{Q}$-basis of $\text{span}_{\mathbb{Q}} M_{e}$:

\begin{lemma}  There is an orthogonal set of vectors $\{k'_{1},\dots,k'_{r}\} \subseteq M_{e}$ such that 
	\begin{equation*}
		\text{span}_{\mathbb{Q}}\{k'_{1},\dots,k'_{r}\} = \text{span}_{\mathbb{Q}} M_{e}
	\end{equation*}
\end{lemma}

\begin{proof}  Let $\{k_{1},\dots,k_{r}\}$ be a $\mathbb{Z}$-basis of $M_{e}$.  Let $k'_{1} = k_{1}$.  Suppose for some $\ell < r$ we have chosen an orthogonal set $\{k_{1}',\dots,k_{\ell}'\} \subseteq M_{e}$ in such a way that $\text{span}_{\mathbb{Q}} \{k_{1}',\dots,k_{\ell}'\} = \text{span}_{\mathbb{Q}}\{k_{1},\dots,k_{\ell}\}$.  We define $k_{\ell + 1}'$ as follows:
\begin{equation*}
k_{\ell + 1}' = k_{\ell + 1} - \sum_{i = 1}^{\ell} \langle k_{\ell + 1}, k_{i}' \rangle k_{i}'.
\end{equation*} 
Clearly, $\text{span}_{\mathbb{Q}}\{k_{1}',\dots,k_{\ell}',k_{\ell + 1}'\} = \text{span}_{\mathbb{Q}}\{k_{1},\dots,k_{\ell},k_{\ell + 1}\}$.  Moreover, $\{k_{1}',\dots,k_{\ell + 1}'\}$ is orthogonal and contained in $M_{e}$.  We continue until we reach $\ell = r$. \qed \end{proof}     

Henceforth, let $\{k_{1},\dots,k_{r}\} \subseteq M_{e}$ be an orthogonal $\mathbb{Q}$-basis of $M_{e}$ and fix an orthonormal basis $\{e_{r + 1},\dots,e_{d - 1}\}$ of $\langle k_{1},\dots,k_{r},e\rangle^{\perp}$ in $\mathbb{R}^{d}$.  Notice that $\{k_{1},\dots,k_{r},e_{r + 1},\dots,e_{d - 1}\}$ spans $\langle e \rangle^{\perp}$.   

As we are interested in averaging $\mathbb{Z}^{d}$-periodic functions over cubes orthogonal to $e$, it is natural to introduce the following group of transformations: given $y \in \langle e \rangle^{\perp}$, we define $\tilde{T}_{y} : \mathbb{T}^{d} \to \mathbb{T}^{d}$ by $\tilde{T}_{y}(x) = x + y$.  Clearly, $\{\tilde{T}_{y}\}_{y \in \langle e \rangle^{\perp}}$ forms a group under composition in the natural way, that is,
\begin{equation*}
\tilde{T}_{0} = \text{Id}, \quad \tilde{T}_{x + y} = \tilde{T}_{x} \circ \tilde{T}_{y}.
\end{equation*}
Moreover, each element of the group preserves the Lebesgue measure $\mathcal{L}^{d}$ on $\mathbb{T}^{d}$.  In fact, we can say more:

\begin{theorem} \label{T: unique_invariant}  $\mathcal{L}^{d}$ is the unique Borel probability measure invariant under $\{\tilde{T}_{y}\}_{y \in \langle e \rangle^{\perp}}$.  In particular, $\mathcal{L}^{d}$ is ergodic.  \end{theorem}  

To prove this, it is convenient to start with an auxiliary lemma:

\begin{lemma}  If $k \in \mathbb{Z}^{d}$ and $\langle k, e_{i} \rangle = 0$ independently of $i \in \{r + 1,\dots,d - 1\}$, then $k \in M_{e}$.  \end{lemma}

\begin{proof}  Since $\{k_{1},\dots,k_{r},e_{r + 1},\dots,e_{d-1},e\}$ is an orthogonal basis of $\mathbb{R}^{d}$, any such $k$ can be written as
	\begin{equation*}
		k = \sum_{i = 1}^{r} \frac{\langle k, k_{i} \rangle}{\|k_{i}\|^{2}} k_{i} + \kappa e
	\end{equation*}
for some $\kappa \in \mathbb{R}$.  To see that $k \in M_{e}$, we only need to show that $\kappa = 0$.  

Now notice that, by our choice of $\{k_{1},\dots,k_{r}\}$, $\frac{\langle k, k_{i} \rangle}{\|k_{i}\|^{2}} \in \mathbb{Q}$ for each $i$.  Thus, $\kappa e = k - \sum_{i = 1}^{r} \frac{\langle k, k_{i} \rangle}{\|k_{i}\|^{2}} k_{i} \in \mathbb{Q}^{d}$.   From the fact that $e \notin \mathbb{R} \mathbb{Z}^{d}$, we conclude $\langle k,e \rangle = \kappa = 0$.  \qed \end{proof}   

Using the lemma, the theorem follows easily:

\begin{proof}[Proof of Theorem \ref{T: unique_invariant}]  Suppose $\mu$ is a Borel probability measure on $\mathbb{T}^{d}$ that is invariant under $\{\tilde{T}_{y}\}_{y \in \langle e \rangle^{\perp}}$.  We will show that $\mu$ equals $\mathcal{L}^{d}$ by computing its Fourier series.  Specifically, we only need to show that $\hat{\mu}(k) = \delta_{0k}$ independently of $k \in \mathbb{Z}^{d}$.  

Since $\mu$ is a probability measure, we find $\hat{\mu}(0) = 1 = \delta_{00}$ by definition.  

Now assume $k \in \mathbb{Z}^{d} \setminus \{0\}$.  We claim that $\hat{\mu}(k) = 0$.  Indeed, since $\mu$ is preserved by $\{\tilde{T}_{y}\}_{y \in \langle e \rangle^{\perp}}$, we can write
\begin{equation*}
\hat{\mu}(k) = \int_{\mathbb{T}^{d}} e^{-i 2 \pi \langle k, x + y \rangle} \, \mu(dx) = e^{-i 2 \pi \langle k, y \rangle} \hat{\mu}(k) \quad \text{if} \, \, y \in \langle  e \rangle^{\perp}.
\end{equation*}
To conclude, we only need to show that there is a $y \in \langle e \rangle^{\perp}$ such that $e^{i 2 \pi \langle k, y \rangle} = -1$.  

Now we use the lemma.  The linear functional $y \mapsto \langle k, y \rangle$ either vanishes on $\langle e \rangle^{\perp}$ or its range equals $\mathbb{R}$.  In view of the previous lemma and the assumption that $k \neq 0$, the second case is the only possibility.  Thus, we can fix a $y_{0} \in \langle e \rangle^{\perp}$ such that $\langle k, y_{0} \rangle = \frac{1}{2}$.  In particular, $e^{-i 2 \pi \langle k, y_{0} \rangle} = -1$.  

The uniqueness of the invariant measure implies ergodicity (cf.\ \cite[Section 4.7]{brin stuck}). \qed 
\end{proof}   

The ergodicity of $\mathcal{L}^{d}$ implies the following result concerning averaging:

\begin{prop}  If $f \in L^{1}(\mathbb{T}^{d})$, then, for a.e.\ $s \in \mathbb{R}$, we have 
\begin{equation} \label{E: ergodic_theorem_sliced}
\lim_{R \to \infty} R^{1 - d} \int_{Q^{e}(0,R)} f(se + x^{\perp}) \, d x^{\perp} = \int_{\mathbb{T}^{d}} f(y) \, dy.
\end{equation}
\end{prop}  

\begin{proof}  By the ergodic theorem, there is a Lebesgue measurable, $\{\tilde{T}_{y}\}_{y \in \langle e \rangle^{\perp}}$-invariant set $B \subseteq \mathbb{R}^{d}$ such that $\mathcal{L}^{d}(\mathbb{R}^{d} \setminus B) = 0$ and
\begin{equation*}
\lim_{R \to \infty} R^{1 - d} \int_{Q^{e}(0,R)} f(x + x^{\perp}) \, d x^{\perp} = \int_{\mathbb{T}^{d}} f(y) \, dy \quad \text{if} \, \, x \in B.
\end{equation*}

Define $A \subseteq \mathbb{R}$ by 
\begin{equation*}
A = \left\{s \in \mathbb{R} \, \mid \, s e + w^{\perp} \in B \, \, \text{for almost every} \, \, w^{\perp} \in \langle e \rangle^{\perp} \right\}.
\end{equation*}  
We claim that $\mathcal{L}^{1}(\mathbb{R} \setminus A) = 0$.  Indeed, by Fubini's Theorem, we can write
\begin{equation*}
0 = \mathcal{L}^{d}(\mathbb{R}^{d} \setminus B)= \int_{\mathbb{R} \setminus A} \mathcal{H}^{d - 1}(\{w \in \langle e \rangle^{\perp} \, \mid \, se + w \notin B\}) \, ds.
\end{equation*}
We are left to conclude that $\mathcal{L}^{1}(\mathbb{R} \setminus A) = 0$.

Finally, we claim that if $s \in A$, then \eqref{E: ergodic_theorem_sliced} holds.  To see this, observe that there is a $w^{\perp} \in \langle e \rangle^{\perp}$ so that $se + w^{\perp} \in B$ and, thus, by $\{\tilde{T}_{y}\}_{y \in \langle e \rangle^{\perp}}$-invariance, $se \in B$.  Therefore, \eqref{E: ergodic_theorem_sliced} follows.    \qed \end{proof}  

\subsection{Irrational directions --- Main results}  We now establish the integral decomposition in the irrational case.  

\begin{prop} \label{P: proof of ergodic lemma} Suppose $e \in S^{d - 1} \setminus \mathbb{R} \mathbb{Z}^{d}$ and $F \in L^{1}(\mathbb{R} \times \mathbb{T}^{d - 1})$.  Let $\{f_{\zeta}\}_{\zeta \in \mathbb{R}}$ be the functions generated by $F$.  For almost every $\zeta \in \mathbb{R}$, we have
\begin{equation} \label{E: easy}
\int_{-\infty}^{\infty} \int_{\mathbb{T}^{d}} F(s,x) \, dx \, ds = \lim_{R \to \infty} R^{1 - d} \int_{Q^{e}(0,R) \oplus_{e} \mathbb{R}} f_{\zeta}(x) \, dx.
\end{equation}
\end{prop}

\begin{proof}  Let $\zeta \in \mathbb{R}$ be a free parameter.  First, we make some simplifications to the right-hand side of \eqref{E: easy}:
\begin{align*}
\int_{Q^{e}(0,R) \oplus_{e} \mathbb{R}} f_{\zeta}(x) \, dx &= \int_{Q^{e}(0,R)} \int_{-\infty}^{\infty} F(y - \zeta, ye + x^{\perp}) \, dy \, dx^{\perp} \\
		&= \int_{Q^{e}(0,R)} \int_{-\infty}^{\infty} F(s, (s + \zeta)e + x^{\perp}) \, ds \, dx^{\perp}.
\end{align*} 
Since $F \in L^{1}(\mathbb{R} \times \mathbb{T}^{d})$, it follows that $y \mapsto \int_{-\infty}^{\infty} F(s, (s + \zeta)e + y) \, ds$ is in $L^{1}(\mathbb{T}^{d})$, no matter the choice of $\zeta$.  Therefore, the previous lemma implies that almost every $\zeta \in \mathbb{R}$ satisfies
\begin{align*}
\lim_{R \to \infty} R^{1 - d} \int_{Q^{e}(0,R) \oplus_{e} \mathbb{R}} f_{\zeta}(x) \, dx &= \lim_{R \to \infty} R^{1 - d} \int_{Q^{e}(0,R)} \int_{-\infty}^{\infty} F(s, (s + \zeta)e + x^{\perp}) \, ds \, dx^{\perp} \\
	&= \int_{\mathbb{T}^{d}} \left(\int_{-\infty}^{\infty} F(s,(s + \zeta)e + y) \, ds \right) \, dy \\
	&= \int_{-\infty}^{\infty} \int_{\mathbb{T}^{d}} F(s,y) \, dy \, ds.
\end{align*} \qed
\end{proof}  

Finally, though we will not provide a complete proof of Proposition \ref{P: algebra}, by now the following observation is well within reach:

\begin{prop} \label{P: rational_irrational}  If $e \notin \mathbb{R} \mathbb{Z}^{d}$, then $M_{e}$ has rank less than $d - 1$ and $\{\langle k, e \rangle \, \mid \, k \in \mathbb{Z}^{d}\}$ is a dense subgroup of $\mathbb{R}$.  \end{prop}  

	\begin{proof}  Define $\i : \mathbb{Z}^{d} \to \mathbb{R}$ by 
		\begin{equation*}
			\i(k) = \langle k, e \rangle.
		\end{equation*}
	Notice that $\i$ is a group homomorphism.  Therefore, $\{\langle k, e \rangle \, \mid \, k \in \mathbb{Z}^{d}\} = \i(\mathbb{Z}^{d})$ is a subgroup of $\mathbb{R}$.  Recall that any subgroup of $\mathbb{R}$ with rank greater than one is necessarily dense.  Therefore, we will prove that $\i(\mathbb{Z}^{d})$ has rank greater than one.
	
	As before, let $\{k_{1},\dots,k_{r}\} \subseteq \mathbb{Q}^{d}$ be a $\mathbb{Q}$-basis of $\text{span}_{\mathbb{Q}}M_{e}$ satisfying $\langle k_{i}, k_{j} \rangle = 0$ if $i \neq j$.  Next, fix an orthogonal set $\{k_{r + 1},\dots,k_{d}\} \subseteq \mathbb{Q}^{d}$ such that $\{k_{1},\dots,k_{d}\}$ spans $\mathbb{Q}^{d}$.  Multiplying by a scalar if necessary, we can assume that $\{k_{r + 1},\dots,k_{d}\} \subseteq \mathbb{Z}^{d}$.  Evidently, $e \in \text{span}_{\mathbb{R}}\{k_{r + 1},\dots,k_{d}\}$, and the fact that $e \notin \mathbb{R} \mathbb{Z}^{d}$ implies $r < d - 1$.
	
	We claim that $\{\langle k_{r + 1},e \rangle, \dots, \langle k_{d}, e \rangle\}$ is independent over $\mathbb{Z}$.  Indeed, given integers $m_{r+1},\dots,m_{d}$, if $\sum_{i = r + 1}^{d} m_{i} \langle k_{i}, e \rangle = 0$, then $\sum_{i = r + 1}^{d} m_{i} k_{i} \in M_{e}$.  In view of the fact that the set $\{k_{r + 1},\dots,k_{d}\}$ is linearly independent over $\mathbb{Q}$, this implies 	
	\begin{equation*}
		m_{r + 1} = \dots = m_{d} = 0.
	\end{equation*}
	
We conclude that $\text{rk}(\i(\mathbb{Z}^{d})) = d - r > 1$.  Therefore, $\i(\mathbb{Z}^{d})$ is dense as claimed.  Finally, notice that $M_{e}$ has rank $r < d - 1$.  \qed
	\end{proof}  
	
\subsection{Rational directions}  The argument for rational directions is similar to the irrational case, except that the translations considered earlier now have many ergodic invariant measures.

Assume that $e \in \mathbb{R} \mathbb{Z}^{d}$ and define $\{\tilde{T}_{y}\}_{y \in \langle e \rangle^{\perp}}$ as before.  To understand how these translations act on $\mathbb{T}^{d}$, it is convenient to observe that $\mathbb{T}^{d}$ can be decomposed as
	\begin{equation*}
		\mathbb{T}^{d} = \bigcup_{m \in [0,m_{e})} \mathbb{T}^{d-1}_{e}(m), 
	\end{equation*}
where the hypersurface $\mathbb{T}^{d-1}_{e}(m)$ is defined by
	\begin{equation*}
		\mathbb{T}^{d-1}_{e}(m) = \left\{y \in \mathbb{T}^{d} \, \mid \, \langle y,e \rangle = m + \langle k,e \rangle \, \, \text{for some} \, \, k \in \mathbb{Z}^{d} \right\}
	\end{equation*}
and $m_{e} > 0$ is the period defined by \eqref{E: rational period}

Notice that $\{\mathbb{T}^{d-1}_{e}(m) \, \mid \, m \in [0,m_{e})\}$ is precisely the family of all orbits of $\{\tilde{T}_{y}\}_{y \in \langle e \rangle^{\perp}}$.  Further, considering the case when $e$ is one of the coordinate vectors, the following result is natural:

	\begin{prop}  For each $m \in [0,m_{e})$, the normalized $(d-1)$-dimensional Hausdorff measure on $\mathbb{T}^{d-1}_{e}(m)$ is an ergodic invariant probability measure of $\{\tilde{T}_{y}\}_{y \in \langle e \rangle^{\perp}}$.  These are the only ergodic invariant probability measures.  \end{prop}   
	
By considering each $\mathbb{T}^{d-1}_{e}(m)$ as a torus in its own right, it is thus not hard to show that if $N \in L^{1}(\mathbb{T}^{d})$, then 
	\begin{equation*}
		\lim_{R \to \infty} R^{1-d} \int_{Q^{e}(0,R)} N(x + \xi) \, \mathcal{H}^{d-1}(d \xi) = \fint_{\mathbb{T}^{d-1}_{e}(\langle x, e \rangle)} N(\eta) \, \mathcal{H}^{d-1}(d \eta) \quad \text{for a.e.} \, \, x \in \mathbb{T}^{d}.
	\end{equation*}
Thus, in the proof of Proposition \ref{P: proof of ergodic lemma}, we find, for a.e.\ $\zeta \in [0,m_{e})$,
	\begin{align*}
		\lim_{R \to \infty} R^{1 - d} \int_{Q^{e}(0,R) \oplus_{e} \mathbb{R}} f_{\zeta}(x) \, dx &= \lim_{R \to \infty} R^{1 - d} \int_{Q^{e}(0,R)} \int_{-\infty}^{\infty} F(s, (s + \zeta)e + x^{\perp}) \, ds \, dx^{\perp} \\
			&= \fint_{\mathbb{T}^{d-1}_{e}(\zeta)} \left(\int_{-\infty}^{\infty} F(s,se + \xi) \, ds \right) \, \mathcal{H}^{d-1}(d \xi).
	\end{align*}
Since $f_{\zeta}$ is a function in $\mathbb{T}_{e}^{d-1} \oplus_{e} \mathbb{R}$, the left-hand side is readily identified:
	\begin{equation*}
		\mathcal{H}^{d-1}(Q_{e})^{-1} \int_{Q_{e} \oplus_{e} \mathbb{R}} f_{\zeta}(x) \, dx = \lim_{R \to \infty} R^{1 - d} \int_{Q^{e}(0,R) \oplus_{e} \mathbb{R}} f_{\zeta}(x) \, dx.
	\end{equation*}
Combining these and averaging in $\zeta$, we conclude
	\begin{align*}
		m_{e}^{-1} \mathcal{H}^{d-1}(Q_{e})^{-1} \int_{0}^{m_{e}} \int_{Q_{e} \oplus_{e} \mathbb{R}} f_{\zeta}(x) \, dx \, d\zeta  = \int_{\mathbb{R} \times \mathbb{T}^{d}} F(s,x) \, dx \, ds.
	\end{align*}
	
\section{Tubular Neighborhoods of Graphs}  \label{A: tubular_neighborhoods}

In this appendix, we construct a tubular neighborhood of a smooth graph in $\mathbb{R}^{d}$.  The existence of such a tubular neighborhood is an essential ingredient in the proof of Lemma \ref{L: propagation}.  Since technical considerations arise that are not present when compact hypersurfaces are considered instead of graphs, we provide the details for the convenience of the reader.

	In what follows, if $\Omega \subseteq \mathbb{R}^{d + 1}$ is an open set and $n \in \mathbb{N}$, then $BUC^{n}(\Omega)$ is the space of $C^{n}$ functions $f : \Omega \to \mathbb{R}$ such that $D^{m}f$ is bounded and uniformly continuous in $\Omega$ for each $m \in \{0,1,\dots,n\}$.  

	\begin{prop} \label{P: surface_fact}  Suppose $\Phi : \mathbb{R}^{d-1} \times (-1,1) \to \mathbb{R}$ is $C^{5}$ with bounded and uniformly continuous derivatives.
	For each $t \in (-1,1)$, let $\mathcal{U}_{t}$ be the epigraph defined by 
		\begin{equation*}
			\mathcal{U}_{t} = \{(x_{e},x') \in \mathbb{R}^{d} \, \mid \, x_{e} > \Phi(x',t) \}.
		\end{equation*}
	If $d : \mathbb{R}^{d} \times (-1,1) \to \mathbb{R}$ is defined so that $d(\cdot,t)$ is the signed distance function to $\partial \mathcal{U}_{t}$, positive in $\mathcal{U}_{t}$, then there is a positive number $\gamma > 0$ such that:
		\begin{itemize}
			\item[(i)] $d$ is $C^{4}$ in the set $\{(x,t) \in \mathbb{R}^{d} \times (-1,1) \, \mid \, |d(x,t)| < \gamma \}$.
			\item[(ii)]  For each $\delta \in (0,\gamma)$, $d \in BUC^{4}(\Omega_{\delta})$, where $\Omega_{\delta} = \{(x,t) \in \mathbb{R}^{d} \times (-1,1) \, \mid \, |d(x,t) | < \gamma - \delta\}$.
		\end{itemize}      
	\end{prop}

The following preliminary fact will be used in the proof:

	\begin{lemma} \label{L: balls}  Suppose $\varphi : \mathbb{R}^{d -1} \to \mathbb{R}$ is $C^{2}$ and there is a constant $C > 0$ such that $\|D^{2}\varphi(x')\| \leq C$ for all $x' \in \mathbb{R}^{d-1}$.  Let $\mathcal{S} = \{(x_{e},x') \in \mathbb{R}^{d} \, \mid \, x_{e} > \varphi(x') \}$.  If $(x_{e},x') \in \partial \mathcal{S}$ and $B$ is the open ball of radius $C^{-1}$ tangent to $\partial \mathcal{S}$ at $(x_{e},x')$ from inside $\mathcal{S}$, then $B \subseteq \mathcal{S}$ and $\partial B \cap \partial \mathcal{S} = \{(x_{e},x')\}$.    \end{lemma}
	
	We defer the proof of the lemma to the end of this section and proceed with the   
	
		\begin{proof}[Proof of Proposition \ref{P: surface_fact}]  To start with, let us define $\gamma$ by 
			\begin{equation*}
				\frac{1}{\gamma} = \|D^{2}\Phi\|_{L^{\infty}(\mathbb{R}^{d - 1} \times (-1,1))}.
			\end{equation*}
			
		\textbf{Step 1: Tubular neighborhoods}  
	
	For each $r \in (-1,1)$, define the parametrization $\psi_{r} : \mathbb{R}^{d-1} \to \partial \mathcal{U}_{r}$ by 
		\begin{equation*}
			\psi_{r}(y) = (\Phi(y,r),y).
		\end{equation*} 
Next, define $\Psi_{r} : \mathbb{R}^{d-1} \times \mathbb{R} \to \mathbb{R}^{d}$ by 
		\begin{equation*}
			\Psi_{r}(y,\xi) = \psi_{r}(y) + \xi n(\psi_{r}(y),r).
		\end{equation*}
	where $n(\cdot,r)$ is the normal vector to $\partial \mathcal{U}_{r}$ pointing into $\mathcal{U}_{r}$.   Explicitly, $n(\cdot,r)$ can be computed as
		\begin{equation*}
			n(x,r) = \frac{(1,-D\Phi(x',r))}{\sqrt{1 + \|D\Phi(x',r)\|^{2}}}.
		\end{equation*}	
The assumptions on $\Phi$ imply that the map $\Psi : (y,\xi,s) \mapsto \Psi_{s}(y,\xi)$ is in $C^{4}$.  We claim that the map $\Psi_{r} : \mathbb{R}^{d-1} \times (-\gamma,\gamma) \to \mathbb{R}^{d}$ is a diffeomorphism, no matter the choice of $r \in (-1,1)$.
	
	To see this, first, notice that $D\Psi_{r}$ can be represented in matrix form as
		\begin{equation*}
			D\Psi_{r}(y,\xi) = \left( \begin{array}{c c c c c}
									\frac{\partial \psi_{r}}{\partial y_{1}} + \xi Dn(\psi_{r}) \frac{\partial \psi_{r}}{\partial y_{1}} & \dots & \frac{\partial \psi_{r}}{\partial y_{d-1}} + \xi Dn(\psi_{r}) \frac{\partial \psi_{r}}{\partial y_{d-1}} & n(\psi_{r}(y))
								\end{array} \right).
		\end{equation*}
	Since $\psi_{r}$ parametrizes $\partial \mathcal{U}_{r}$, $\left \{\frac{\partial \psi_{r}}{\partial y_{1}}, \frac{\partial \psi_{r}}{\partial y_{2}}, \dots, \frac{\partial \psi_{r}}{\partial y_{d-1}} \right\}$ spans $\langle n \rangle^{\perp}$ at each point.  Moreover, recall that $Dn$ maps $\langle n \rangle^{\perp}$ into itself, and the definition of $n$ implies 
		\begin{equation*}
			\|Dn\|_{L^{\infty}(\mathbb{R}^{d-1} \times [-1,1])} \leq \gamma^{-1}.
		\end{equation*}  
Thus, if $|\xi| < \gamma$, then $\text{Id} + \xi Dn(\psi_{r})$ is an invertible operator on $\langle n \rangle^{\perp}$.  In particular, this shows $D\Psi_{r}$ is invertible in $\mathbb{R}^{d-1} \times (-\gamma,\gamma)$. 
	
	In addition, we claim that $\Psi_{r} : \mathbb{R}^{d-1} \times (-\gamma,\gamma) \to \mathbb{R}^{d}$ is injective.  To see this, suppose $\Psi_{r}(y,\xi) = \Psi_{r}(\tilde{y},\tilde{\xi})$.  Notice that, in general, $\Psi_{r}(\cdot,\xi)$ maps $\mathbb{R}^{d-1}$ into $\mathcal{U}_{r}$ if $\xi > 0$ and into $\mathbb{R}^{d} \setminus \overline{\mathcal{U}_{r}}$ if $\xi < 0$.  Thus, we know that $\xi$ and $\tilde{\xi}$ have the same sign.  Of course, if $\xi = \tilde{\xi} = 0$, then $y = \tilde{y}$ follows from the injectivity of the parametrization $\psi_{r}$.  Therefore, let us assume without loss of generality that $0 \leq \xi \leq \tilde{\xi}$ with $\tilde{\xi} \neq 0$.  Since $\tilde{\xi} < \gamma$, Lemma \ref{L: balls} implies that the open ball $B(\Psi_{r}(\tilde{y},\tilde{\xi}), \tilde{\xi})$ is entirely contained in $\mathcal{U}_{r}$ and its boundary intersects $\partial \mathcal{U}_{r}$ only at $\psi_{r}(\tilde{y})$.  On the other hand, $\|\Psi_{r}(\tilde{y},\tilde{\xi}) - \psi_{r}(y)\| = \|\Psi_{r}(y,\xi) - \psi_{r}(y)\| = \xi \leq \tilde{\xi}$.  Since $\psi_{r}(y) \in \partial \mathcal{U}_{r}$, the only way these two observations can be consistent is if $\xi = \tilde{\xi}$ and $\psi_{r}(y) = \psi_{r}(\tilde{y})$.  Therefore, we conclude that $\Psi_{r}$ is injective in $\mathbb{R}^{d-1} \times (-\gamma,\gamma)$ as claimed.  
	
Putting together the results of the previous two paragraphs, we see that $\Psi_{r} : \mathbb{R}^{d-1} \times (-\gamma,\gamma) \to \mathbb{R}^{d}$ is a diffeomorphism onto its range.
	
Finally, observe that if $(\tilde{x},s) \in \mathbb{R}^{d} \times (-1,1)$ satisfies $\text{dist}(\tilde{x},\partial \mathcal{U}_{s}) < \gamma$, we can take any $x \in \partial \mathcal{U}_{s}$ satisfying $\|x - \tilde{x}\| = \text{dist}(\tilde{x},\partial \mathcal{U}_{s})$ and then, arguing as before, we see that $\tilde{x} = \Psi_{s}(x)$ and
		\begin{equation} \label{E: explicit formula for d}
			d(\tilde{x},s) = \pi_{2}(\Psi_{s}^{-1}(\tilde{x})).
		\end{equation}      
(Here $\pi_{2} : \mathbb{R}^{d-1} \times (-\gamma,\gamma) \to (-\gamma,\gamma)$ is the projection onto the second factor.)  
		
	\textbf{Step 2: Regularity of $d$}  
	
	Since $(r,y,\xi) \mapsto \Psi_{r}(y,\xi)$ is $C^{4}$ in all three variables, it is not hard to prove that $(\tilde{x},s) \mapsto \Psi_{s}^{-1}(\tilde{x})$ is $C^{4}$ in both variables in $\{|d| < \gamma\}$.  Of course, from this and \eqref{E: explicit formula for d}, it follows that $d$ is $C^{4}$ in both variables.  Note, in addition, that if $\delta \in (0,\gamma)$, then the assumptions on $\Phi$ and the fact that $\|D\Psi_{s}^{-1}\|$ is bounded in $\Omega_{\delta}$ together imply that all four of the derivatives of $(\tilde{x},s) \mapsto \Psi_{s}^{-1}$ are bounded and uniformly continuous in $\Omega_{\delta}$.  Thus, \eqref{E: explicit formula for d} implies $d \in BUC^{4}(\Omega_{\delta})$.            	\qed \end{proof}  
	
It only remains to treat the
	
	\begin{proof}[Proof of Lemma \ref{L: balls}]  Given such a point $(x_{e},x')$, we can write, for an arbitrary $\tilde{x}' \in \mathbb{R}^{d-1}$,
		\begin{equation*}
			\varphi(\tilde{x}') \leq \varphi(x') + \langle D\varphi(x'), \tilde{x}' - x' \rangle + \frac{C \|\tilde{x}' - x'\|^{2}}{2}.
		\end{equation*}
	Thus, if we define $\psi: \mathbb{R}^{d-1} \to \mathbb{R}$ by $\psi(\tilde{x}') = \varphi(x') + \langle D \varphi(x'), \tilde{x}' - x' \rangle + \frac{C \|\tilde{x}' - x'\|^{2}}{2}$, then the epigraph $\mathscr{P} = \{(\tilde{x}_{e},\tilde{x}') \, \mid \, \tilde{x}_{e} > \psi(\tilde{x}')\}$ is contained in $\mathcal{S}$.
	
	Now $\psi$ is a paraboloid of opening $C$, and a calculus exercise shows that the open ball $B$ of radius $C^{-1}$ tangent to $\partial \mathscr{P}$ at $(x_{e},x')$ from inside $\mathscr{P}$ satisfies $B \subseteq \mathscr{P}$ and $\partial B \cap \partial \mathscr{P} = \{(x_{e},x')\}$.  Since the normal vectors of $\mathscr{P}$ and $\mathcal{S}$ at $(x_{e},x')$ coincide and $\mathscr{P} \subseteq \mathcal{S}$, we arrive at the desired conclusion.  \qed \end{proof}

\section*{Acknowledgements}  

The author is grateful to his thesis adviser, P.E.\ Souganidis, for introducing him to this subject and for many enlightening discussions, and to the anonymous reviewers for suggestions on improving the paper.

\end{document}